\documentclass[10pt,a4paper,reqno]{amsart}
\usepackage{mystyle}

\title[Hitting time for loose Hamilton cycles]{Hitting time for loose Hamilton cycles in the random hypergraph process}

\author[Le\'on]{Jared Le\'on}
\address[Le\'on]{Mathematics Institute, University of Warwick, Coventry CV4 7AL, United Kingdom}
\email{jared.leon@warwick.ac.uk}
\thanks{J.L. was supported by the Warwick Mathematics Institute Centre
  for Doctoral Training. This research was carried out while he was
  visiting the School of Mathematical Sciences at the University of
  Science and Technology of China, which he thanks for its support.}

\author[Liu]{Xizhi Liu}
\address[Liu]{School of Mathematical Sciences, University of Science and Technology of China, Hefei, China}
\email{liuxizhi@ustc.edu.cn}
\thanks{X.L. was supported by the Excellent Young Talents Program
  (Overseas) of the National Natural Science Foundation of China.}

  
\hypersetup{pdftitle={Hitting time for loose Hamilton cycles in the random hypergraph process},pdfauthor={Jared Le\'on, Xizhi Liu}}

\begin{document}
\begin{abstract}
  The random hypergraph process starts with an empty hypergraph on $n$
  vertices and adds a uniformly chosen missing edge at each step. For
  even $n$, we prove that, with high probability, a $3$-uniform loose
  Hamilton cycle appears when the last isolated vertex disappears. The
  approach developed here adapts arguments from Kahn's work on perfect
  matchings. It can also be extended to obtain the corresponding
  hitting-time theorems for $r=2$ and every fixed $r\ge4$ satisfying
  $(r-1)\mid n$. Furthermore, it can be used to estimate the numbers
  of loose Hamilton cycles at these hitting times for every fixed
  $r\ge2$, up to a factor of $\exp\{o(n)\}$.
\end{abstract}

\maketitle

\section{Introduction}
\label{sec:introduction}
A \emph{loose Hamilton cycle} in a $3$-graph is a spanning
subhypergraph whose edges can be cyclically ordered so that
consecutive edges intersect in exactly one vertex, these intersection
vertices are distinct, and non-consecutive edges are disjoint. Observe
that if $n$ is odd, then no $n$-vertex $3$-graph contains a loose
Hamilton cycle.

Let $(\mathcal H_t)_{t=0}^{\binom n3}$ be the random $3$-uniform hypergraph process on $[n]\coloneqq\{1,\ldots,n\}$. More precisely, set $\mathcal H_0=\varnothing$. For each $0\le t<\binom n3$, conditional on $\mathcal H_t$, choose a triple uniformly from $\binom{[n]}3\setminus\mathcal H_t$ and add it to obtain $\mathcal H_{t+1}$. For even $n$, we define two hitting times in the random hypergraph process $(\mathcal H_t)$: $\tau_1$ is the time when the last isolated vertex disappears, and $\tau_{\LH}$ is the time when a loose Hamilton cycle first appears. That is,
\[
 \tau_1\coloneqq\min\{t\colon \delta_1(\mathcal H_t)\ge1\},\quad\text{and}\quad
 \tau_{\LH}\coloneqq\min\{t\colon \mathcal H_t\text{ contains a loose Hamilton cycle}\},
\]
where $\delta_1(H)$ denotes the minimum vertex degree of $H$. Observe that a spanning loose cycle has no isolated vertex, so deterministically $\tau_{\LH}\ge\tau_1$.

For every fixed $r\ge3$ and $n$ divisible by $r$, Kahn~\cite[Theorem~1.3]{Kahn2022} proved that, with high probability, a perfect matching first appears in the random $r$-uniform hypergraph process when the last isolated vertex disappears. We prove the analogous statement for loose Hamilton cycles in $3$-graphs, which was posed as a problem by K\"uhn and Osthus~\cite[Section~5.2]{KuhnOsthus2014} and revisited by Frieze and P\'erez-Gim\'enez~\cite[Section~5]{FriezePerez2025}.
\begin{theorem}\label{thm:main}
As $n$ tends to infinity through even integers, $\PP[\tau_{\LH}=\tau_1]$ tends to $1$.
\end{theorem}

We also provide a Lean~4 formalisation~\cite{Leon2026Formalization} of
the corresponding hitting-time theorem for every fixed $r\ge3$ with
$(r-1)\mid n$. This is registered as
\href{https://palomar-registry.org/entry?id=PALOMAR-2026-10-03-000003&version=1}{Palomar
  entry \texttt{PALOMAR-2026-10-03-000003}, version~1} with source
code available at
\url{https://github.com/WozMit/hitting-time-loose-hamilton}.

The proof of \cref{thm:main} adapts the entropy, conditioned sampling,
and random edge-deletion arguments used by
Kahn~\cite{Kahn2023,Kahn2022} for perfect matchings. This approach
extends to every fixed $r\ge4$ with $(r-1)\mid n$. It can also be
adapted to give a proof of the graph hitting-time theorem at minimum
degree $2$ without P\'osa's the well-known rotation--extension
method. For every fixed $r\ge2$, the counting argument can be extended
to estimate the number of loose Hamilton cycles at the corresponding
hitting time up to a factor $\exp\{o(n)\}$. See the concluding remarks
and \cref{thm:counting-cycles} for these extensions.

We now review related work on Hamilton cycles and perfect matchings in random graphs and hypergraphs.

For ordinary graphs, Koml\'os and
Szemer\'edi~\cite{KomlosSzemeredi1983} proved that, for each fixed
$c\in\mathbb R$, a uniformly random $n$-vertex graph with
$\bigl\lfloor\tfrac n2(\log n+\log\log n+c)\bigr\rfloor$ edges
contains a Hamilton cycle with probability tending to
$\exp\{-\e^{-c}\}$. Bollob\'as~\cite{Bollobas1984} and Ajtai,
Koml\'os, and Szemer\'edi~\cite{AjtaiKomlosSzemeredi1985} proved that,
with high probability, a random graph process becomes Hamiltonian when
its minimum degree first reaches $2$.

In the binomial $r$-uniform hypergraph model, each $r$-set is selected independently with probability $p$. For a sufficiently large constant $C$, Frieze~\cite{Frieze2010} proved that, with high probability, a random hypergraph from this model contains a loose Hamilton cycle when $r=3$, $4\mid n$, and $p\ge C\log n/n^2$. Dudek and Frieze~\cite{DudekFrieze2011} proved the same conclusion for fixed $r\ge4$ when $pn^{r-1}/\log n\to\infty$ and $2(r-1)\mid n$. Dudek, Frieze, Loh, and Speiss~\cite{DudekFriezeLohSpeiss2012} removed the extra factor of $2$ from these divisibility conditions. Ferber~\cite{Ferber2015} later gave a coupling proof for fixed $r\ge3$, $(r-1)\mid n$, and $pn^{r-1}/\log n\to\infty$. The threshold order $n^{-(r-1)}\log n$ also follows from the general threshold theorem of Frankston, Kahn, Narayanan, and Park~\cite[Theorem~5.1]{FrankstonKahnNarayananPark2021}.

Recently, Frieze and P\'erez-Gim\'enez~\cite{FriezePerez2025}
determined the threshold for loose Hamilton cycles up to a factor of
$1+o(1)$ in the uniform $m$-edge hypergraph model. More precisely,
they proved that for every fixed $r\ge3$ and $ \eps > 0$, with high
probability, the uniform random $r$-graph with $m\ge(1+\eps)n\log n/r$
edges contains a loose Hamilton cycle, provided $(r-1)\mid n$. For
$r=3$, this threshold is asymptotic to $n\log n/3$. For Hamilton
cycles whose consecutive edges overlap in $\ell\ge2$ vertices,
Narayanan and Schacht~\cite{NarayananSchacht2020} determined the sharp
threshold in the binomial $r$-uniform hypergraph model, for fixed
$r>\ell$ and $(r-\ell)\mid n$. This threshold is an explicit constant
times $n^{-(r-\ell)}$. In particular, when $r=3$ and $\ell=2$, for
every fixed $0<\varepsilon<1$, with high probability, such a cycle
exists if $p\ge(1+\varepsilon)\e/n$ and does not exist if
$p\le(1-\varepsilon)\e/n$.

In the random regular $r$-uniform hypergraph model, one chooses
uniformly from all simple $r$-graphs on $[n]$ in which every vertex
has degree $d$. For every fixed $r\ge3$ and integers $n,d$ satisfying
$r\mid nd$ and $(r-1)\mid n$, Dudek, Frieze, Ruci\'nski, and
\v{S}ileikis~\cite{DudekFriezeRucinskiSileikis2015} proved that, with
high probability, such a hypergraph contains a loose Hamilton cycle if
$d=\omega(\log n)$ and $d=o(n^{1/2})$. For $r=3$, the lower bound on
$d$ can be weakened to $d\ge C\log n$ for a sufficiently large
constant $C$. Their proof couples the random regular hypergraph with a
uniform random hypergraph having a prescribed number of edges, so that
the latter is a spanning subhypergraph of the former with high
probability. Their later work~\cite{DudekFriezeRucinskiSileikis2017}
extended this embedding to $d$-regular hypergraphs for
$\log n\ll d\ll n^{r-1}$, with the embedded hypergraph containing
almost all edges of the random regular hypergraph. For fixed
uniformity and fixed degree, Altman, Greenhill, Isaev, and
Ramadurai~\cite[Theorem~1.1]{Altman2020} determined the degree
threshold for loose Hamiltonicity. In particular, for fixed $d\ge2$, a
uniform simple $d$-regular $3$-graph contains a loose Hamilton cycle
with high probability if $d\ge4$, and does not if $d\le3$, along
integers satisfying $2\mid n$ and $3\mid dn$.

We now consider perfect matchings. For graphs of even order,
Bollob\'as and Thomason~\cite{BollobasThomason1985} proved that, in
the random graph process, with high probability a perfect matching
appears when the last isolated vertex disappears. For perfect
matchings in random uniform hypergraphs, early threshold bounds were
obtained by Schmidt and Shamir~\cite{SchmidtShamir1983}, and
subsequently improved by Frieze and Janson~\cite{FriezeJanson1995} and
by Kim~\cite{Kim2003}. For fixed $r\ge3$ and $r\mid n$, Johansson,
Kahn, and Vu~\cite{JohanssonKahnVu2008} established the threshold
$n^{-(r-1)}\log n$ for the appearance of a perfect matching. For these
values of $r,n$ and every fixed $\varepsilon>0$, Kahn~\cite{Kahn2023}
proved that, with high probability, a uniformly random $m$-edge
$r$-graph has a perfect matching if $m>(1+\varepsilon)n\log n/r$.

Cuckler and Kahn~\cite{CucklerKahn2009Entropy} developed entropy
bounds for perfect matchings and Hamilton cycles in graphs using their
edge marginals. In related work, Cuckler and
Kahn~\cite{CucklerKahn2009Dirac} obtained asymptotic lower bounds for
the number of Hamilton cycles in graphs with minimum degree at least
$n / 2$, using entropy-maximising fractional matchings.

Our proof of \Cref{thm:main} adapts Kahn's overall proof strategy, as
well as his entropy inequality. Our main contribution is to make the
strategy work for a loose Hamilton cycle while accommodating vertices
of degree one. Frieze and P\'erez-Gim\'enez \cite{FriezePerez2025}
adapt Kahn's matching argument to find several independent uniform
perfect matchings inside suitable random subhypergraphs. The union of
these perfect matchings can then be treated using a Hamiltonicity
theorem for random regular hypergraphs. This route requires minimum
degree at least two. The obstruction at the hitting time is that a
vertex that was isolated immediately before the last edge arrived now
has degree one, so the hypergraph cannot contain a spanning regular
subhypergraph of degree at least two. It is not even possible to find
two edge-disjoint perfect matchings. However, a loose cycle can
accommodate this vertex, and it must occur in just one cycle edge, as
a private vertex in the cycle, rather than as an intersection vertex.
Removing such vertices first leaves the problem of fitting their
incident edges into a spanning cycle. This is the route that we take
in this work.

\subsection*{Overview of the argument}
\label{sec:proof-snapshot}

Kahn~\cite{Kahn2023,Kahn2022} proves the existence of perfect matchings by counting them.\footnote{The
  paper \emph{Asymptotics for Shamir's problem} proves the asymptotic
  threshold and gives the hitting time reduction. The companion paper
  \emph{Hitting times for Shamir's problem} establishes the hitting
  time statement.} Starting from the complete hypergraph, he tracks the number of perfect matchings remaining as edges are deleted uniformly at random. Entropy estimates first control \emph{most} edge marginals, where an edge's marginal is the fraction of perfect matchings containing it. Resampling and comparisons between the numbers of perfect matchings left after deleting different candidate vertex sets then extend this control to \emph{every} edge marginal. A martingale estimate controls the accumulated decrease in the logarithm of the number of perfect matchings. For the hitting-time argument, he first reveals the low-degree set and its incident edges, then reduces to counting perfect matchings in a random hypergraph subject to vertex-degree lower bounds. We adapt this strategy to loose Hamilton cycles.

To explain the counting step, suppose the current hypergraph contains $\Phi>0$ objects, such as perfect matchings or cycles, each using $k$ of the $j$ edges subject to deletion. For each deletable edge $e$, let $q_e$ be the fraction of these objects containing $e$. Deleting $e$ leaves $\Phi(1-q_e)$ objects, and averaging over the $j$ deletable edges gives $\frac1j\sum_e q_e=k/j$. Thus $k/j$ is the expected fraction of objects lost under uniform deletion. Since $\log(1-q_e)\approx-q_e$ for small $q_e$, summing these average losses gives a reference value for $\log\Phi$. To control deviations from this reference value, we require $q_e=O(k/j)$ for \emph{every} deletable edge, bounding the proportion of objects removed at each step. Near the hitting time, $k/j$ has order $1/\log n$.

The first modification concerns the reduction at low-degree vertices. In Kahn's reduction~\cite{Kahn2023}, disjoint edges cover the low-degree vertices. He removes all vertices of these edges, finds a perfect matching on the remaining vertices, and then restores the reserved edges. In the cycle problem, for each low-degree vertex $v$, we choose an incident triple $\{v,a_v,b_v\}$, with the chosen triples pairwise disjoint. We delete the low-degree vertices and prescribe the pairs $a_vb_v$, calling the remaining hypergraph the \emph{core}. We count cycles spanning the core vertices, containing every prescribed pair, and otherwise using triples of the core. We call these \emph{mixed Hamilton cycles}. We keep all prescribed pairs throughout edge deletion. Expanding each pair $a_vb_v$ back to $\{v,a_v,b_v\}$ then restores $v$ as a private vertex (see \cref{fig:core-reduction}).

To determine the conditional distribution of the core, we reveal the hitting time, the low-degree set, and all edges incident to that set. Conditional on this information, the core is uniform among hypergraphs with its fixed number of edges and prescribed vertex-degree lower bounds of order $\log n$. \Cref{thm:conditioned-core} gives a lower bound on the number of mixed Hamilton cycles containing all prescribed pairs under this conditional law.

The second modification concerns the entropy argument for edge marginals. We choose a mixed Hamilton cycle uniformly from the $\Phi$ cycles under consideration. Its entropy is $\log\Phi$, so a lower bound on the number of cycles becomes an entropy lower bound. Kahn's inequality applies to perfect matchings, so we first encode each mixed Hamilton cycle as a perfect matching in an auxiliary hypergraph. We orient the cycle and condition on its set of intersection vertices and the directions of its prescribed pairs. We then split each intersection vertex into incoming and outgoing copies. A triple traversed from $u$ to $v$, with private vertex $w$, becomes $\{u^{\mathrm{out}},v^{\mathrm{in}},w\}$. For each prescribed pair directed from $a$ to $b$, we omit $a^{\mathrm{out}}$ and $b^{\mathrm{in}}$. The auxiliary edges corresponding to the cycle's triples then form a perfect matching. We apply Kahn's inequality to the matching distribution induced by the mixed Hamilton cycles. Because the matching is induced by a sampled mixed Hamilton cycle, it corresponds to a single connected cycle, even though other auxiliary perfect matchings may correspond to several disjoint cycles. Combining entropy stability with the host's degree, codegree, and edge-distribution estimates, we show that a lower bound on $\log\Phi$ close to its reference value forces \emph{most} directed edge marginals to be nearly equal.

To bound \emph{every} edge marginal, we compare the numbers of mixed Hamilton cycles as the deleted vertex and the additional prescribed pair vary. Kahn counts perfect matchings of $H-U$, obtained by deleting the vertices of a candidate edge $U$, whether or not $U$ belongs to $H$. He compares these numbers as the vertices of $U$ change. For distinct vertices $w,u,v$ not covered by the prescribed pairs, we instead let $Y(w;u,v)$ count mixed Hamilton cycles after deleting $w$ and adding $uv$ to the prescribed matching. This number is defined even when $\{w,u,v\}$ is not an edge of the host hypergraph. If $\{w,u,v\}$ is an edge, expanding $uv$ gives cycles using that edge with $w$ private. Summing the three counts corresponding to the choices of private vertex gives $\Phi q_{\{w,u,v\}}$.

We use resampling to show that $Y(w;u,v)/\Phi\approx1/(2d)$ for all but a vanishing fraction of these choices, including choices that do not form an edge of the host hypergraph, where $d$ is the host's average degree. To prove this, we describe one edge-deletion experiment in two orders. In the forward description, we reveal the current host $H$ and delete a uniformly chosen set $T$ of a prescribed size that is small compared with $|H|$. An entropy bound shows that two independent uniformly chosen mixed Hamilton cycles share few triples on average. This limits correlations between their survival events and bounds the variance of the number of mixed Hamilton cycles left after deletion. The same argument applies to each sufficiently large family counted by $Y(w;u,v)$. For smaller families, an upper bound on the expected number of cycles left after deletion suffices. These estimates compare $Y(w;u,v)/\Phi$ before deletion with the corresponding ratio after deleting $T$.

In the reverse description, we reveal the remaining host $H-T$ and the core after removing the edges of $T$. We show that, with high probability, the core after deletion still satisfies the prescribed vertex-degree lower bounds. Conditional on these two remaining hypergraphs, if the degree lower bounds hold, then $T$ is uniform among sets of its prescribed size consisting of triples absent from $H-T$. If the ratios $Y(w;u,v)/\Phi$ were atypical for many choices before deletion, then the corresponding ratios would remain atypical for many choices after deletion. Yet the entropy bound and the forward estimates show that only a small proportion of the triples in $T$ admit a choice of private vertex for which the ratio after deletion is atypical. Uniform sampling makes this discrepancy unlikely. This proves that $Y(w;u,v)/\Phi$ is close to $1/(2d)$ for almost all admissible choices.

The preceding estimate does not rule out choices of $w,u,v$ for which $Y(w;u,v)$ is much larger than $\Phi/(2d)$. We obtain an upper bound for every $Y(w;u,v)$ by contracting and rebuilding short cycle segments. These modifications move the omitted vertex or an endpoint of the added pair while preserving connectivity and the fixed prescribed pairs. They provide the comparisons needed in place of Kahn's comparisons for different candidate edges $U$. If $Y(w;u,v)$ were too large for one choice of the three vertices, then the modifications would force $Y(w;u,v)$ to exceed $\Phi/(2d)$ by a fixed factor for a positive proportion of the other choices, contradicting the estimate for almost all admissible choices. For triples meeting a prescribed pair, we sum the numbers of cycles obtained by the corresponding contractions to obtain the remaining edge-marginal bounds.

Finally, we combine the uniform edge-marginal bound with the martingale estimate. We examine the first possible failure of either the counting lower bound or the edge-marginal bound. If the counting lower bound holds at a state, then the uniform edge-marginal bound applies there. If the marginal bound holds at all preceding states, then the stopped martingale estimate keeps $\log\Phi$ closer to its reference value than the counting lower bound requires. The probability that this martingale estimate fails remains $o(1)$ after conditioning on the core's degree lower bounds. Consequently, with high probability, both bounds hold throughout edge deletion. In particular, the core contains a mixed Hamilton cycle, and expanding the prescribed pairs proves \cref{thm:main}.

This method also estimates the number of loose Hamilton cycles at the hitting time, up to a factor $\exp\{o(n)\}$. The arguments proving \cref{thm:conditioned-core} can be adapted to graphs and to every fixed uniformity $r\ge4$ with $(r-1)\mid n$, giving corresponding results on hitting times and the numbers of loose Hamilton cycles. This hitting time extension
was developed semi-automatically and then formalised automatically in
Lean 4 \cite{Leon2026Formalization} by a large language model. We
supply the informal proofs separately.

\subsection*{Paper organisation}
\label{sec:proof-snapshot-paper}
\cref{sec:proof-sketch} contains a proof sketch and a diagram of the
main proof dependencies. The basic counting and probability tools are
collected in \cref{sec:preliminaries}. In \cref{sec:mixed}, we reduce
\cref{thm:main} to a conditioned counting theorem.

\cref{sec:martingale} counts mixed cycles by random edge deletion from
the complete hypergraph down to the core. A martingale estimate then
protects the cycle count throughout deletion. To obtain this bound,
\cref{sec:matching-entropy,sec:entropy} encode cycles as perfect
matchings and use Kahn's entropy inequality to show that sufficiently
large cycle families use most edges in nearly equal proportions.  The
tools for random hypergraphs with degree lower bounds appear in
\cref{sec:degree-constrained}. \cref{sec:switchings} provides the
conditional laws and edge-distribution estimates needed throughout the
deletion process.

To pass from most edges to every edge, \cref{sec:resampling} compares
the forward and reverse description of a small random edge-deletion
experiment. This gives nearly equal cycle counts for most ways of
removing a vertex and prescribing an additional pair. In
\cref{sec:migration}, local re-routings preserve the prescribed pairs
and connectivity while spreading any exceptionally large cycle family
across many choices. Combined with resampling, this bounds every
edge's share of the cycles and completes the martingale counting
argument. Finally, \cref{sec:conclusion} discusses extensions to other
uniformities and the resulting counting conclusions.

\section{Proof sketch}
\label{sec:proof-sketch}
In this section we sketch the main ideas of the proof. The reader can
return to this sketch while reading the detailed proofs to keep the
overall strategy in view. \Cref{fig:proof-dependencies} gives a map of
the main dependencies.

\begin{figure}[!htp]
\centering
\scalebox{0.7}{%
\begin{tikzpicture}[
  font=\small,
  result/.style={draw=black!50,rounded corners=2pt,fill=black!2,
    align=center,text width=4.6cm,inner xsep=5pt,inner ysep=4pt,
    execute at begin node={\hyphenpenalty=10000}},
  central/.style={result,draw=darkblue,fill=blue!4,text width=5.1cm},
  dependency/.style={-{Stealth[length=4pt]},draw=black!65,line width=.6pt}
]
  \node[central] (main) at (0,0) {
    \textbf{\cref{thm:main}}\\
    Loose Hamilton cycle at $\tau_1$};
  \node[central] (counting) at (0,-2.1) {
    \textbf{\cref{thm:conditioned-core}}\\
    Lower bound on the number of mixed Hamilton cycles in the core};

  \node[result,text width=4.2cm] (counting-estimates) at (-2.8,-4.4) {
    \textbf{\cref{thm:stopped-counting,lem:beta}}\\
    Martingale estimate and probability of satisfying the degree lower bounds};
  \node[central] (marginals) at (2.8,-4.4) {
    \textbf{\cref{thm:uniform-marginal}}\\
    Uniform upper bounds on edge marginals};

  \node[result] (resampling) at (-2.8,-6.7) {
    \textbf{\cref{prop:uniform-completion}}\\
    Nearly equal numbers of cycles for most vertex choices};
  \node[result] (migration) at (2.8,-6.7) {
    \textbf{\cref{lem:private-migration,lem:intersection-migration}}\\
    Comparing numbers of cycles by migration};

  \node[result] (cycle-entropy) at (-2.8,-9) {
    \textbf{\cref{thm:entropy-stability,lem:overlap}}\\
    Nearly equal directed marginals and a bound on expected overlap};
  \node[result] (process-estimates) at (2.8,-9) {
    \textbf{\cref{prop:path-regularity,lem:simultaneous-link}}\\
    Degrees, codegrees and links during deletion};

  \node[result,text width=3.8cm,left=6mm of cycle-entropy] (matching-entropy) {
    \textbf{\cref{thm:kahn,prop:matching-stability}}\\
    Kahn's entropy inequality and stability};
  \node[result,text width=3.8cm,right=6mm of process-estimates] (core-estimates) {
    \textbf{\cref{thm:terminal-regularity}}\\
    Degree and codegree bounds for the core};

  \draw[dependency] (counting) -- (main);
  \draw[dependency] (counting-estimates) -- (counting);
  \draw[dependency] (marginals) -- (counting);
  \draw[dependency] (resampling) -- (marginals);
  \draw[dependency] (migration) -- (marginals);
  \draw[dependency] (resampling.east) -- (migration.west);
  \draw[dependency] (cycle-entropy) -- (resampling);
  \draw[dependency] (process-estimates.north west) -- (resampling.south east);
  \draw[dependency] (process-estimates) -- (migration);
  \draw[dependency] (matching-entropy.east) -- (cycle-entropy.west);
  \draw[dependency] (core-estimates.west) -- (process-estimates.east);
\end{tikzpicture}
}
\caption{Dependencies among the results used in the proof.}
\label{fig:proof-dependencies}
\end{figure}
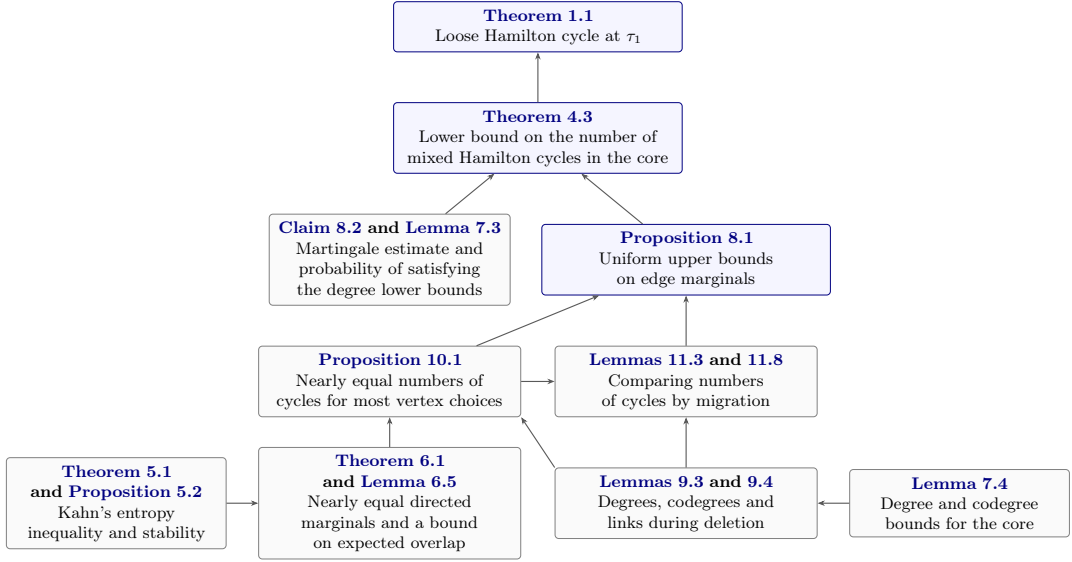

We prove existence by counting. At $\tau_1$, when the last isolated vertex disappears, we seek a loose Hamilton cycle. We first reduce \cref{thm:main} to the conditioned counting theorem, \cref{thm:conditioned-core}. We explain this reduction and then describe the estimates used to prove that counting theorem.


Let $B$ denote the set of vertices of degree at most $\log n/100$ at the hitting time $\tau_1$. With high probability, $B$ is small, no edge of $\mathcal H_{\tau_1}$ contains two vertices of $B$, and edges through distinct vertices of $B$ are disjoint. We can therefore choose disjoint edges $\{v,a_v,b_v\}$ of $\mathcal H_{\tau_1}$, one through each $v\in B$. We remove $B$, leaving the \emph{core} $F_*=\mathcal H_{\tau_1}-B$, and prescribe a matching $\cM_0=\{a_vb_v\colon v\in B\}$. Write $n_\ast=n-|B|$ and $s=|\cM_0|$. A \emph{mixed Hamilton cycle} spans the core vertices, uses every pair of $\cM_0$, and otherwise uses edges of $F_*$. Its edges admit a cyclic ordering in which consecutive edges meet in distinct single vertices and non-consecutive edges are disjoint. We call the shared vertices \emph{intersection vertices} and vertices in just one edge \emph{private vertices}. Expanding each pair $a_vb_v$ to $\{a_v,v,b_v\}$ restores $v$ as a private vertex. This is illustrated in \cref{fig:core-reduction}.

We reveal $\tau_1$, $B$, and all edges of $\mathcal H_{\tau_1}$ meeting $B$. Once these edges have been revealed, we fix the selected triples and the resulting matching $\cM_0$. Conditional on $\tau_1=m$, an $m$-edge $3$-graph without isolated vertices is sampled with probability proportional to its number of edges containing a degree-one vertex. Since every degree-one vertex lies in $B$, the exposure determines this number. Conditional on the revealed information, $F_*$ is uniform among $3$-graphs with $m_*=|F_*|$ edges satisfying prescribed degree lower bounds $\ell_x$ at their vertices $x$. These bounds are near $\log n/100$. We now sketch the proof of \cref{thm:conditioned-core}, which gives a lower bound on the number of mixed Hamilton cycles under this conditional law.

We count cycles by starting with the complete $3$-graph on the $n_\ast$ core vertices and repeatedly deleting one uniformly chosen remaining edge. Let $H_j$ be the hypergraph when $j$ edges remain, and let $N=\binom{n_\ast}{3}$. We follow this nested sequence from $H_N$ down to $H_{m_\ast}$. Let $\PP$ be its unconditioned law, and let $\mathcal D$ be the event that every vertex $x$ of $H_{m_\ast}$ has degree at least $\ell_x$. Under the conditioned law $\QQ=\PP[\,\cdot\mid\mathcal D]$, the hypergraph $H_{m_\ast}$ has the same distribution as the exposed core. We henceforth write $F_*=H_{m_\ast}$, so $F_*\subseteq H_j$ throughout.

The prescribed matching remains fixed throughout deletion. We call an edge of $H_N$ \emph{allowed} if it contains at most one vertex of $V(\cM_0)$, the set of endpoints of the prescribed pairs. Let $G_j$ be the hypergraph of allowed edges present in $H_j$, and write $\Xc(H_j)$ for the total number of mixed Hamilton cycles using $G_j$ and $\cM_0$. This restriction keeps the later cycle modifications compatible with $\cM_0$. Every counted cycle uses $k=(n_\ast-s)/2$ edges of $G_j$. When $\Xc(H_j)>0$, the \emph{edge marginal} $\Qc_e(H_j)$ is the fraction of these cycles using $e$. Its average over the $j$ edges of $H_j$ is $k/j$, which is also the expected fraction of cycles lost at the next deletion under $\PP$. If a deletion removes a fraction $q$ of the cycles, it multiplies their number by $1-q$. For small $q$, the change in the logarithm is $\log(1-q)\approx-q$. These logarithmic changes add over successive deletions. Subtracting the expected fractional losses from the initial logarithmic count gives the reference value
\[
 L_j\coloneqq\log \Xc(H_N)-k\sum_{h=j+1}^{N}\frac1h.
\]
We use $L_j$ as a reference value for $\log\Xc(H_j)$ and bound $L_j-\log\Xc(H_j)$ throughout deletion.

Our next goal is \cref{thm:uniform-marginal}, which gives, with high probability under $\QQ$, the bound $\Qc_e(H_j)=O(k/j)$ for every edge $e\in H_j$ at every state satisfying $\log\Xc(H_j)\ge L_j-n_\ast/\sqrt{\log n_\ast}$. With this marginal bound, \cref{thm:stopped-counting} keeps $\log\Xc(H_j)$ above $L_j-O(n_\ast/\log n_\ast)$ up to and including the first state at which the edge-marginal bound fails. We first bound, under $\PP$, the probability that this logarithmic lower bound fails at any state up to that stopping state. Under $\PP$, where deletions are uniform, the martingale estimate bounds this failure probability by $\exp\{-cn_\ast/\log n_\ast\}$ for an absolute $c>0$. By \cref{lem:beta}, $\PP[\mathcal D]\ge\exp\{-n_\ast^{1/10}\}$. Conditioning on $\mathcal D$ increases the probability of failure of the logarithmic lower bound by at most $\PP[\mathcal D]^{-1}$. Consequently, the lower bound still holds with high probability under $\QQ$. We return to this counting argument after proving the marginal bound.

Before the entropy argument, we establish degree, codegree, and edge-distribution estimates along the deletion sequence. \Cref{thm:terminal-regularity} controls degrees and codegrees in the core $F_*$. \Cref{prop:path-regularity} extends this control throughout deletion, even after removing a bounded number of vertices. Conditional on $F_*$, the extra edges of $H_j$ are a uniformly chosen set of $j-m_\ast$ triples absent from $F_*$, so standard concentration estimates for sampling apply. Degrees stay within constant factors of $d_j\coloneqq3j/n_\ast$, the average degree of $H_j$, and codegrees are small compared with $d_j$. For each partition $A\cup\overline A$ in the range covered by \cref{prop:path-regularity}(ii), the parts have nearly equal sizes, and about $3j/8$ triples of $H_j$ have two vertices in $A$ and one in $\overline A$. These estimates give the degree and edge-distribution hypotheses for the entropy argument.

We now apply entropy estimates. Choose a mixed Hamilton cycle uniformly from those counted by $\Xc(H_j)$. Its entropy is $\log\Xc(H_j)$. Fix a direction on one prescribed pair and orient every cycle to agree with it. Conditional on the set of intersection vertices and the directions of the prescribed pairs, split each intersection vertex into incoming and outgoing copies. For each prescribed pair directed from $a$ to $b$, omit $a^{\rm out}$ and $b^{\rm in}$. The cycle's $3$-edges then give a perfect matching in an auxiliary hypergraph $\widetilde G$, as illustrated in \cref{fig:cycle-matching}. This matching consists of disjoint edges of $\widetilde G$ covering all auxiliary vertices. Kahn's inequality, \cref{thm:kahn}, bounds the entropy of the induced matching distribution in terms of its edge marginals. The stability result, \cref{prop:matching-stability}, shows that an entropy close to the upper bound forces the edge marginals to be nearly equal. The lower bound on $\log\Xc(H_j)$ forces both the entropy of the intersection set and pair directions and the average conditional matching entropy to be close to their respective upper bounds. Returning to cycles gives \cref{thm:entropy-stability}, which controls directed edge marginals for most triples of $G_j$ disjoint from $V(\cM_0)$. The same encoding gives \cref{lem:overlap}: two independent cycles chosen uniformly from a sufficiently large family share only $O(n_\ast/\log d_j)$ $3$-edges on average. We use this bound to control the variance after random edge deletions.

To extend the marginal estimates to triples absent from $H_j$, fix distinct $w,u,v$ outside $V(\cM_0)$. We count mixed Hamilton cycles in $G_j-w$ using the prescribed matching $\cM_0\cup\{uv\}$, and call this number $Y_j(w;u,v)$. This \emph{completion count} is defined even when $\{w,u,v\}\notin H_j$. If $\{w,u,v\}\in H_j$, then expanding $uv$ identifies these auxiliary cycles with cycles counted by $\Xc(H_j)$ having $w$ private on $\{w,u,v\}$. Partitioning by the private vertex therefore gives
\[
 \Xc(H_j)\Qc_{\{w,u,v\}}(H_j)
 =Y_j(w;u,v)+Y_j(u;w,v)+Y_j(v;w,u).
\]
It suffices to bound $Y_j(w;u,v)$ from above. Ignoring the relative direction of the added pair, \cref{prop:uniform-completion} gives $Y_j(w;u,v)=(1+o(1))\Xc(H_j)/(2d_j)$ for all but a vanishing fraction of the choices of $w,u,v$, under the logarithmic counting lower bound above. The proposition is stated for directed pairs and also permits the bounded vertex deletions and changes to the prescribed matching used in the migration arguments.

We prove \cref{prop:uniform-completion} by describing the same edge-deletion experiment in two orders. In the \emph{forward description}, we fix a realisation of the deletion sequence and delete a uniformly chosen set $T\subseteq H_j$ of a prescribed size small compared with $j$. We count the cycles left from the family counted by $\Xc(H_j)$ and from each auxiliary family counted by $Y_j(w;u,v)$. The overlap bound in \cref{lem:overlap} limits correlations between the survival indicators of two cycles and hence bounds the variance of these numbers of remaining cycles. For each sufficiently large family, we obtain concentration around its expected number of remaining cycles. For smaller families, an upper bound on that expectation suffices. The entropy estimate shows that few normalized completion counts associated with triples sampled into $T$ are atypical before deletion. We normalize by $\Xc(H_j)$ before deletion and by $\Xc(H_j-T)$ after deletion. The concentration estimates show that most of these normalized counts remain close to their reference value after deleting $T$, whereas most initially atypical normalized counts remain atypical.

In the \emph{reverse description}, we fix the remaining hypergraphs $H_j-T$ and $F_*-T$, treating $T$ as the unknown edges to add back. Conditional on these two hypergraphs, if $F_*-T$ satisfies the prescribed degree lower bounds, then $T$ is uniform among sets of its prescribed size consisting of triples of $H_N$ absent from $H_j-T$. The degree estimates for the core ensure that $F_*-T$ satisfies these lower bounds with high probability in the forward experiment. If too many completion counts are atypical before deletion, then many triples absent from $H_j-T$ still have an admissible choice of private vertex with an atypical normalized completion count. Yet the forward estimates show that only a small proportion of the triples in $T$ have such a choice. A uniform sample is unlikely to contain so few triples with this property. Comparing the forward and reverse probabilities proves \cref{prop:uniform-completion}.

We now use the estimates for most candidates to bound every edge marginal. One basic reconstruction takes a mixed Hamilton cycle in $G_j-\{w,w'\}$ through the prescribed matching $\cM_0\cup\{uv,ab\}$, with $\{a,w,b\}\in G_j$, and replaces $ab$ by $\{a,w,b\}$. This restores $w$ and gives a cycle counted by $Y_j(w';u,v)$. In the output cycle, $w$ is private on the unique triple $\{a,w,b\}$. Contracting this triple to the prescribed pair $ab$ recovers the input cycle. The reconstruction is therefore injective even when summed over the eligible pairs $ab$.

We need many pairs $ab$ for which this reconstruction yields enough cycles. The available pairs lie in the \emph{link graph} $L_{H_j}(w)$, whose edges are the pairs $ab$ with $\{a,w,b\}\in H_j$. For a fixed $w'$, we exclude pairs that violate the vertex restrictions or give too few auxiliary mixed Hamilton cycles. The excluded pair set is determined by $H_j-w$ and the fixed prescribed pairs, before the link at $w$ is revealed. The estimates for completion counts make this excluded set small for most choices of $w'$, and \cref{lem:simultaneous-link} supplies many link edges outside it. Contracting and reconstructing short segments in this way gives \cref{lem:private-migration,lem:intersection-migration}. These lemmas compare the numbers of cycles when the omitted vertex or an endpoint of the added pair is changed to most other vertices outside the prescribed matching and the fixed labels, with a loss of at most a constant factor. If $Y_j(w;u,v)$ exceeded a sufficiently large constant multiple of $\Xc(H_j)/d_j$, successive migrations would force a positive proportion of the other completion counts to exceed $\Xc(H_j)/(2d_j)$ by a fixed factor, contradicting \cref{prop:uniform-completion}.

For a triple $\{y_0,w,u\}\in G_j$ with $y_0z_0\in\cM_0$, consider counted cycles in which $w$ is private on that triple. We contract the prescribed pair $y_0z_0$ together with its adjacent triple $\{y_0,w,u\}$ in the cycle to the pair $z_0u$, deleting $y_0,w$. Each counted cycle has a unique $3$-edge at $y_0$ and a unique private vertex on that edge. Summing the numbers of cycles obtained from these contractions therefore recovers $\Xc(H_j)$. The two migration lemmas, together with the link estimate, compare each summand with many other summands and give the required edge-marginal bound. Edges of $H_j\setminus G_j$ are used by no counted cycle and therefore have marginal zero. Together, these cases prove \cref{thm:uniform-marginal}.

We can now finish the counting argument. If the counting lower bound first failed at some state, then every preceding state would satisfy the edge-marginal bound. The martingale estimate would give an error $O(n_\ast/\log n_\ast)$, smaller than $n_\ast/\sqrt{\log n_\ast}$, the maximum permitted value of $L_j-\log\Xc(H_j)$ in the counting lower bound. This is a contradiction. Conversely, \cref{thm:uniform-marginal} rules out a first failure of the edge-marginal bound at a state satisfying the counting lower bound. With high probability, both bounds therefore hold throughout deletion. At $H_{m_\ast}=F_*$, we obtain the lower bound in \cref{thm:conditioned-core} and in particular $\Xc(H_{m_\ast})>0$. Expanding the prescribed pairs of $\cM_0$ restores the vertices of $B$ as private vertices and proves \cref{thm:main}.

\section{Preliminaries}
\label{sec:preliminaries}
We collect the basic definitions and cycle-counting facts, followed by the probability and entropy tools used in the proof.

All hypergraphs are simple. For a finite set $V$, write $\binom Vr$ for the family of its $r$-element subsets. We usually identify a hypergraph $\mathcal{G}$ with its edge set, and hence use $|\mathcal{G}|$ to denote the number of edges in it. For a hypergraph $\mathcal{G}$ and $S\subseteq V(\mathcal{G})$, let $\mathcal{G}-S$ be the hypergraph on $V(\mathcal{G})\setminus S$ with edge set $\{e\in\mathcal{G}\colon e\cap S=\varnothing\}$. For a single vertex $v$, write $\mathcal{G}-v$ for $\mathcal{G}-\{v\}$.

For a $3$-graph $\mathcal G$ and disjoint sets $A,B\subseteq V(\mathcal G)$, define
\[
 \mathcal G[A,B]\coloneqq
 \left\{e\in\mathcal G\colon |e\cap A|=2\text{ and }|e\cap B|=1\right\}.
\]
In the notation $\mathcal G[A,\overline A]$, we take $\overline A\coloneqq V(\mathcal G)\setminus A$.

For a hypergraph $\mathcal{G}$, the vertex degree $d_{\mathcal{G}}(v)$ is the number of edges containing $v$. For distinct vertices $u,v$, the codegree $d_{\mathcal{G}}(u,v)$ counts the edges containing both. Write $\delta_1(\mathcal{G})$ for the minimum vertex degree, $\Delta_1(\mathcal{G})$ for the maximum vertex degree, and $\Delta_2(\mathcal{G})$ for the maximum codegree. For a $3$-graph $\mathcal{G}$, the link $L_{\mathcal{G}}(v)$ is the graph on $V(\mathcal{G})\setminus\{v\}$ whose edges are the pairs $ab$ such that $vab\in \mathcal{G}$. For a $3$-graph $\mathcal{G}$, write $d(\mathcal{G})\coloneqq3|\mathcal{G}|/|V(\mathcal{G})|$ for its average degree.

A \emph{Berge path} of length $h$ consists of distinct vertices $v_0,\ldots,v_h$ and distinct edges $e_1,\ldots,e_h$ with $\{v_{i-1},v_i\}\subseteq e_i$ for each $i$.

In a loose cycle $C$, we call vertices belonging to two edges of $C$ \emph{intersection vertices} and those belonging to only one edge of $C$ \emph{private vertices}. For a $3$-graph $\mathcal{G}$, define its set of \emph{oriented triples} by
\[
 \vec{\mathcal{G}}\coloneqq\left\{(w;u\to v)\colon \{u,v,w\}\in\mathcal{G}\right\}.
\]
Note that each edge of a $3$-uniform loose Hamilton cycle contains one private vertex and two intersection vertices. Here we use $(w;u\to v)$ to designate $w$ as the private vertex and $u,v$ as the initial and final intersection vertices, respectively.
For $S\subseteq V(\mathcal{G})$, define $\vec{\mathcal{G}}-S\coloneqq\{(w;u\to v)\in\vec{\mathcal{G}}\colon \{u,v,w\}\cap S=\varnothing\}$.

We also need spanning cycles that contain a prescribed matching of pairs.

\begin{definition}\label{def:mixed}
Let $V$ be a vertex set of size $n$. Let $\mathcal{G}$ be a $3$-graph on $V$, and let $\cM$ be a matching of $s$ pairs on $V$. We use the following terminology and notation.
\begin{enumerate}[label=(\textit{\roman*})]
\item A \emph{$(\mathcal{G},\cM)$-respecting mixed Hamilton cycle} is a hypergraph on $V$ consisting of all the pairs in $\cM$ and triples from $\mathcal{G}$. Its edges cover $V$ and admit a cyclic ordering in which consecutive edges intersect in exactly one vertex, all such intersection vertices are distinct, and nonconsecutive edges are disjoint.
\item A vertex is \emph{$\cM$-saturated} if it belongs to a pair of $\cM$, and \emph{$\cM$-unsaturated} otherwise.
\item Write $\cyc(\mathcal{G},\cM)$ for the number of $(\mathcal{G},\cM)$-respecting mixed Hamilton cycles.
\end{enumerate}
\end{definition}

The lower-right diagram in \cref{fig:core-reduction} illustrates such
a cycle. The coloured segments represent the prescribed pairs and the
grey triangles represent triples.

\begin{fact}\label{fact:mixed-cycle-parameters}
  Suppose that a $(\mathcal G,\cM)$-respecting mixed Hamilton cycle on
  $n$ vertices contains $k$ triples, where $|\cM|=s$. Then, the
  following statements hold.
\begin{enumerate}[label=(\textit{\roman*})]
\item Every $\cM$-saturated vertex belongs to exactly one triple edge and one pair of $\cM$.
\item The cycle has $k$ private vertices and $k+s$ intersection vertices, of which $2s$ are $\cM$-saturated and $k-s$ are $\cM$-unsaturated.
\item We have $n=2k+s$, so $k=(n-s)/2$ and $n+s$ is even.
\end{enumerate}
\end{fact}
\begin{proof}
  Each pair of $\cM$ has two distinct intersection vertices, and the
  matching property forces both neighbouring edges to be triples. Each
  triple has one private vertex, while the cyclic order of the $k+s$
  edges gives $k+s$ distinct intersection vertices. These observations
  give the stated numbers of vertices and the identity $n=2k+s$.
\end{proof}

When $\cM\ne\varnothing$, a chosen pair of $\cM$ is called the \emph{root pair}. Fix an arbitrary direction on this pair. The \emph{root orientation} of a mixed Hamilton cycle containing the root pair is the unique orientation in which it traverses that pair in the fixed direction.

The following lemma counts the number of $(\mathcal{G},\cM)$-respecting mixed Hamilton cycles when the host hypergraph $\mathcal{G}$ is complete and estimates the loss from excluding triples containing more than one vertex of $V(\cM)$. It also computes the expected number of these cycles in a uniformly random hypergraph with a prescribed number of edges.
\begin{lemma}\label{prop:complete-count}
Let $\cM$ be a matching of $s\ge1$ pairs on an $n$-vertex set $V$, and suppose that $n-s$ is even and $n\ge\max\{s+6,3s\}$. The following statements hold.
\begin{enumerate}[label=(\textit{\roman*})]
\item $\displaystyle \cyc\bigl(\tbinom V3,\cM\bigr)=\frac{2^{s-1}(n-2s)!\bigl(\frac{n-s-2}{2}\bigr)!}{\bigl(\frac{n-3s}{2}\bigr)!}$.

\item If $n\ge5s$, then restricting to triples containing at most one vertex of $V(\cM)$ gives
\[
\frac{\cyc\bigl(\{e\in\tbinom V3\colon |e\cap V(\cM)|\le1\},\cM\bigr)}
 {\cyc\bigl(\tbinom V3,\cM\bigr)}
 =\frac{\bigl(\frac{n-3s}{2}\bigr)!\bigl(\frac{n-3s-2}{2}\bigr)!}{\bigl(\frac{n-s-2}{2}\bigr)!\bigl(\frac{n-5s}{2}\bigr)!}
 \ge1-\frac{2s(s-1)}{n-s-2}.
\]

\item Let $m$ be an integer with $0\le m\le\binom n3$, and choose $\mathcal G$ uniformly among all $m$-edge $3$-graphs on $V$. Then
\[
  \EE[\cyc(\mathcal G,\cM)] = \cyc\left(\binom V3,\cM\right) \cdot
  \frac{\binom{m}{(n-s)/2}}{\binom{\binom n3}{(n-s)/2}}.
\]
If $s\le n^{1/10}$ and $|m-n\log n/3|\le n\log\log n$, then, as
$n\to\infty$,
\[
 \log\EE[\cyc(\mathcal G,\cM)]
 =\frac{n-s}{2}\log\frac{6m}{n}-n+s+O\left(\frac{n}{\log n}\right).
\]
The expectation identity also holds when both cycle numbers are restricted to cycles whose triples contain at most one vertex of $V(\cM)$. Under the same asymptotic assumptions, the logarithmic estimate holds for this restricted family as well.
\end{enumerate}
\end{lemma}
Note that for $s=o(\sqrt n)$, the number of cycles given in
\cref{prop:complete-count}(i) is
$(n-1)!\allowbreak n^{-s}\allowbreak\exp\{O(s^2/n)\}$. Therefore, by
Stirling's formula we have
$\cyc\bigl(\tbinom
V3,\cM\bigr)=(1+o(1))\sqrt{2\pi/n}\,n^{-s}(n/\e)^n$. Also, if
$s=o(\sqrt n)$, then the ratio in \cref{prop:complete-count}(ii) is
$1-o(1)$. That is, deleting all triples containing at least two
$\cM$-saturated vertices from the complete $3$-graph on $V$ does not
change the number of $\cM$-respecting mixed Hamilton cycles by more
than an $o(1)$ fraction.

\begin{proof}[Proof of \cref{prop:complete-count}]
For \cref{prop:complete-count}(i), we count cycles by arranging blocks in a cyclic order and assigning one private vertex between consecutive blocks, where each block is either a pair in $\cM$ or a single $\cM$-unsaturated intersection vertex. To count these arrangements, fix a pair in $\cM$ and choose a direction on it. Orient each cycle so that it traverses this pair in the chosen direction. Choose the $(n-3s)/2$ $\cM$-unsaturated intersection vertices from $V\setminus V(\cM)$. These choices give $(n-s)/2$ blocks.

Starting with the block corresponding to the fixed directed pair, the remaining $(n-s-2)/2$ blocks can be ordered after it in $((n-s-2)/2)!$ ways, and the remaining pairs in $\cM$ can be oriented in $2^{s-1}$ ways. Between each pair of consecutive blocks, form a triple using the last vertex of the first block, the first vertex of the next block, and one private vertex. Assigning the $(n-s)/2$ private vertices to these positions gives $((n-s)/2)!$ choices. Since there are at least three blocks, these data are uniquely recoverable from the edge set and the chosen direction of the fixed pair. Multiplying these numbers of choices gives
\[
 \cyc\bigl(\tbinom V3,\cM\bigr)
 =\binom{n-2s}{(n-3s)/2}\,2^{s-1}
 \left(\frac{n-s-2}{2}\right)!
 \left(\frac{n-s}{2}\right)!.
\]
Expanding the binomial coefficient and cancelling the factor $((n-s)/2)!$ gives \cref{prop:complete-count}(i).

For \cref{prop:complete-count}(ii), the key observation is that every triple contains at most one vertex of $V(\cM)$ exactly when no two blocks arising from pairs in $\cM$ are consecutive in the cyclic order. Indeed, a triple contains two $\cM$-saturated vertices precisely when it joins two such blocks. To count the possible sets of $s$ non-consecutive positions, distinguish one chosen position and record the clockwise distances between successive chosen positions. These $s$ distances are at least $2$ and sum to $(n-s)/2$. Subtracting $2$ from each distance gives $s$ nonnegative integers with sum $(n-5s)/2$, so there are $\binom{(n-3s-2)/2}{s-1}$ possible distance sequences. There are $(n-s)/2$ choices for the distinguished position, and every set is counted $s$ times. Thus the number of sets is
\[
 \frac{n-s}{2s}\binom{(n-3s-2)/2}{s-1}
 =\frac{n-s}{n-3s}\binom{(n-3s)/2}{s}.
\]
Dividing by $\binom{(n-s)/2}{s}$ gives the proportion of block orders in which no two of these $s$ blocks are consecutive. By rotational symmetry, this proportion is unchanged when the block containing the fixed directed pair occupies the first position. The orientations and assignments of private vertices contribute the same factor when counting either family, so this proportion is the required ratio; simplifying gives the formula in \cref{prop:complete-count}(ii). For the lower bound, any two specified blocks are adjacent in a uniformly chosen cyclic order with probability $2/((n-s)/2-1)=4/(n-s-2)$. Hence the union bound gives
\[
 \PP[\text{some two prescribed-pair blocks are adjacent}]
 \le\binom{s}{2}\frac{4}{n-s-2}
 =\frac{2s(s-1)}{n-s-2}.
\]
Taking the complementary probability gives the lower bound in \cref{prop:complete-count}(ii).

For \cref{prop:complete-count}(iii), each mixed Hamilton cycle contains exactly $(n-s)/2$ triples. The probability that a fixed set of this many triples belongs to $\mathcal G$ is $\binom{m}{(n-s)/2}/\binom{\binom n3}{(n-s)/2}$, so linearity of expectation gives the stated identity. The same argument applies when we count only cycles whose triples contain at most one vertex of $V(\cM)$.

Under the asymptotic assumptions in \cref{prop:complete-count}(iii), the logarithm of this probability has main term $\frac{n-s}{2}\log(m/\binom n3)$. Throughout the summation range below, $i/m=O(1/\log n)$. Expanding $\log(1-x)=-x+O(x^2)$ therefore gives
\[
\begin{aligned}
 \sum_{i=0}^{(n-s-2)/2}\left[\log\left(1-\frac{i}{m}\right)-\log\left(1-\frac{i}{\binom n3}\right)\right]
 &=O\left(\left(\frac1m+\frac1{\binom n3}\right)\sum_{i=0}^{(n-s-2)/2}i\right)
 =O\left(\frac{n^2}{m}\right).
\end{aligned}
\]
By \cref{prop:complete-count}(i) and Stirling's formula, the number of cycles in the complete $3$-graph satisfies
\[
 \log\cyc\left(\binom V3,\cM\right)
 =\frac{n-s}{2}\log\bigl((n-1)(n-2)\bigr)-n+s+O(s\log n+\log n).
\]
Adding the logarithm of the probability factor and using $\binom n3=n(n-1)(n-2)/6$ yields
\[
 \log\EE[\cyc(\mathcal G,\cM)]
 =\frac{n-s}{2}\log\frac{6m}{n}-n+s
 +O\left(s\log n+\log n+\frac{n^2}{m}\right).
\]
The assumptions $s\le n^{1/10}$ and $|m-n\log n/3|\le n\log\log n$ give $s\log n+\log n+n^2/m=O(n/\log n)$, proving the logarithmic estimate. Finally, \cref{prop:complete-count}(ii) shows that requiring each triple to contain at most one vertex of $V(\cM)$ multiplies the number of cycles in the complete $3$-graph by $1-O(s^2/n)$ under these assumptions. The common probability factor gives the same ratio between the two expectations, so their logarithms differ by $O(s^2/n)$, which is absorbed by the stated error.
\end{proof}

We next relate the numbers of mixed Hamilton cycles before and after deleting a vertex and adding a prescribed pair.

Let $\mathcal{G}$ be a host $3$-graph on $n$ vertices, let $\cM$ be a matching of $s\ge1$ pairs on $V(\mathcal{G})$, and fix a root pair in $\cM$ and its direction. For distinct vertices $x,y,z\in V(\mathcal{G})\setminus V(\cM)$, let $\comp_{\mathcal{G},\cM}(x;y\to z)$ be the number of $(\mathcal{G}-x,\cM\cup\{yz\})$-respecting mixed Hamilton cycles in which the added pair $yz$ is traversed from $y$ to $z$ in the root orientation.
Note that this number is defined whether or not $xyz\in\mathcal{G}$ and depends only on $\mathcal{G}-x$. The next fact relates these numbers to $\cyc(\mathcal{G},\cM)$ and computes the ratio $\comp_{\mathcal{G},\cM}(x;y\to z)/\cyc(\mathcal{G},\cM)$ when $\mathcal G$ is complete. This identifies the reference value for comparing the numbers $\comp_{\mathcal G,\cM}(x;y\to z)$.
\begin{fact}\label{fact:completion-identities}
For distinct vertices $x,y,z\in V(\mathcal G)\setminus V(\cM)$, the following statements hold.
\begin{enumerate}[label=(\textit{\roman*})]
\item We have $\cyc\bigl(\mathcal{G}-x,\cM\cup\{yz\}\bigr)=\comp_{\mathcal{G},\cM}(x;y\to z)+\comp_{\mathcal{G},\cM}(x;z\to y)$.
\item If $xyz\in\mathcal{G}$, the number of $(\mathcal{G},\cM)$-respecting mixed Hamilton cycles containing $xyz$ is $\cyc\bigl(\mathcal{G}-x,\cM\cup\{yz\}\bigr)+\cyc\bigl(\mathcal{G}-y,\cM\cup\{xz\}\bigr)+\cyc\bigl(\mathcal{G}-z,\cM\cup\{xy\}\bigr)$.
\item If $\cyc(\mathcal{G},\cM)>0$ and $xyz\in\mathcal{G}$, then the proportion of $(\mathcal{G},\cM)$-respecting mixed Hamilton cycles using $\{x,y,z\}$ from $y$ to $z$ with $x$ private in the root orientation is $\comp_{\mathcal{G},\cM}(x;y\to z)/\cyc(\mathcal{G},\cM)$.
\item If $\mathcal{G}=\binom{V(\mathcal{G})}{3}$, $n-s$ is even, $n-s\ge8$, and $n-3s\ge4$, then
\begin{equation}\label{eq:Yratio}
 \frac{\comp_{\mathcal{G},\cM}(x;y\to z)}{\cyc(\mathcal{G},\cM)}
 =\frac{(n-3s)(n-3s-2)}
 {2(n-s-2)(n-2s)(n-2s-1)(n-2s-2)}.
\end{equation}
If also $s=o(n)$, this ratio is $(1+o(1))/(2n^2)$.
\end{enumerate}
\end{fact}
\begin{proof}
Partition the cycles counted by $\cyc\bigl(\mathcal{G}-x,\cM\cup\{yz\}\bigr)$ according to the direction of the added pair $yz$ in the root orientation. This proves \cref{fact:completion-identities}(i). For \cref{fact:completion-identities}(ii), partition the cycles containing $xyz$ by its private vertex. When $x$ is private, contracting $xyz$ to the pair $yz$ gives a bijection to the cycles counted by $\cyc\bigl(\mathcal{G}-x,\cM\cup\{yz\}\bigr)$; expansion is the inverse. The other two choices of private vertex give the remaining terms. This bijection preserves the direction in which the triple or pair is traversed. Thus the cycles using the oriented triple $(x;y\to z)$ are counted by $\comp_{\mathcal{G},\cM}(x;y\to z)$, and division by $\cyc(\mathcal{G},\cM)$ proves \cref{fact:completion-identities}(iii).

For \cref{fact:completion-identities}(iv), the cycles counted by $\cyc\bigl(\mathcal{G}-x,\cM\cup\{yz\}\bigr)$ have $n-1$ vertices, $s+1$ pairs and $(n-s-2)/2$ triples. Interchange $y$ and $z$ in every edge, fixing all other vertices. Since the host is complete, this operation preserves the family of mixed cycles and the fixed direction of the root pair, while reversing the direction of the added pair $yz$. It is its own inverse, so it gives a bijection between the two orientation classes. Thus $\comp_{\mathcal{G},\cM}(x;y\to z)$ is half the number of cycles given in \cref{prop:complete-count}(i) with these smaller parameters. Substituting $n-1$ and $s+1$ into that formula, and dividing by the original number of cycles, gives
\[
 \frac{\comp_{\mathcal{G},\cM}(x;y\to z)}{\cyc(\mathcal{G},\cM)}
 =\frac{(n-2s-3)!}{(n-2s)!}
 \frac{\bigl(\frac{n-s-4}{2}\bigr)!}{\bigl(\frac{n-s-2}{2}\bigr)!}
 \frac{\bigl(\frac{n-3s}{2}\bigr)!}{\bigl(\frac{n-3s-4}{2}\bigr)!}.
\]
Cancelling the factorials yields \eqref{eq:Yratio}; the asymptotic estimate follows when $s=o(n)$.
\end{proof}

We now introduce the probabilistic notation used below. All logarithms are natural.
\begin{itemize}
\item For real-valued random variables $X,Y$, we write $X\preceq Y$ if $\PP[X\ge t]\le\PP[Y\ge t]$ for every real $t$; in other words, $Y$ \emph{stochastically dominates} $X$.
\item An event holds \emph{with high probability} (\whp) if its probability tends to $1$ as the number of vertices tends to infinity.
\item For a nonempty finite set $S$, write $\Unif(S)$ for the uniform distribution on $S$; that is, if $X\sim\Unif(S)$, then $\PP[X=x]=1/|S|$ for every $x\in S$.
\item If $T$ is a uniformly chosen $t$-element subset of a finite set $U$ and $A\subseteq U$ is fixed, then $|T\cap A|$ has a \emph{hypergeometric distribution}; that is, for every integer $k$,
\[
 \PP[|T\cap A|=k]=
 \begin{cases}
 \displaystyle\frac{\binom{|A|}{k}\binom{|U\setminus A|}{t-k}}{\binom{|U|}{t}},
 &\text{if }\max\{0,t-|U\setminus A|\}\le k\le\min\{t,|A|\},\\[4pt]
 0,&\text{otherwise}.
 \end{cases}
\]
\item We write $\Poi(\mu)$ for the Poisson distribution with mean $\mu\ge0$; that is, if $X\sim\Poi(\mu)$, then for every integer $k\ge0$ (with $0^0=1$),
\[
 \PP[X=k]=\exp\{-\mu\}\frac{\mu^k}{k!}.
\]
\item For a finite-valued random vector $Z$ under a probability law $\PP$, we call a value $z$ \emph{$\PP$-admissible} if $\PP[Z=z]>0$.
\end{itemize}

\begin{lemma}[Bayes's formula]\label{lem:bayes}
Let $A,B$ be events of positive probability. Then $\PP[A\mid B]=\frac{\PP[B\mid A]\PP[A]}{\PP[B]}$.
\end{lemma}

We also record elementary moment bounds; see Janson, \L uczak, and Ruci\'nski~\cite[Section~1.2]{JansonLuczakRucinski2000} for Markov's and Chebyshev's inequalities.
\begin{lemma}\label{lem:moment-bounds}
The following bounds hold.
\begin{enumerate}[label=(\textit{\roman*})]
\item If $X$ is nonnegative and $t>0$, then $\PP[X\ge t]\le\EE[X]/t$.
\item If $X$ has finite variance and $t>0$, then $\PP[|X-\EE[X]|\ge t]\le\Var[X]/t^2$.
\item If $X\preceq\Poi(\mu)$ for some $\mu\ge0$, then $\PP[X\ge t]\le\mu^t/t!$ for every integer $t\ge1$.
\end{enumerate}
\end{lemma}

The following lemma bounds the variance of the number of members of a set family that avoid a uniformly chosen subset. Applied to the sets of triples used by mixed Hamilton cycles, it bounds the variance of the number of cycles remaining after random edge deletions.
\begin{lemma}\label{lem:batch-variance}
Let $U$ be a finite set of size $m$, and let $\mathcal F\subseteq\binom Uh$ be a fixed nonempty family, where $1\le h\le m/4$. Choose $T$ uniformly among the subsets of $U$ of a prescribed size $|T|\le m/4$. Let $Y_T$ count the members of $\mathcal F$ disjoint from $T$, that is, $Y_T\coloneqq|\{F\in\mathcal F\colon F\cap T=\varnothing\}|$.
If $F_1,F_2$ are independent uniform members of $\mathcal F$, then
\[
\frac{\Var[Y_T]}{\bigl(\EE[Y_T]\bigr)^2}
 \le\frac{\EE[|F_1\cap F_2|]}{h}\exp\left\{\frac{4|T| h}{m}\right\}.
\]
\end{lemma}
Here the mean and variance of $Y_T$ are taken over the random choice of $T$, with $\mathcal F$ fixed, whereas the overlap expectation is taken over the independent uniform choices of $F_1,F_2$ from $\mathcal F$.

\begin{proof}
Each $F\in\mathcal F$ is disjoint from $T$ with probability $\binom{m-h}{|T|}/\binom{m}{|T|}$, so $\EE[Y_T]=|\mathcal F|\binom{m-h}{|T|}/\binom{m}{|T|}$. For fixed $F_1,F_2$, the ratio of the probability that both are disjoint from $T$ to the product of the two individual probabilities is
\[
 \prod_{i=0}^{|T|-1}
 \frac{(m-2h+|F_1\cap F_2|-i)(m-i)}{(m-h-i)^2}.
\]
Each factor equals
\[
 1+\frac{|F_1\cap F_2|(m-i)-h^2}{(m-h-i)^2}
 \le1+\frac{4|F_1\cap F_2|}{m}
 \le\exp\left\{\frac{4|F_1\cap F_2|}{m}\right\}.
\]
Averaging the product over independent uniform $F_1,F_2$ gives $\EE[Y_T^2]/(\EE[Y_T])^2$. Since $0\le|F_1\cap F_2|\le h$, convexity yields
\[
\begin{aligned}
 \frac{\Var[Y_T]}{(\EE[Y_T])^2}
 =\frac{\EE[Y_T^2]}{(\EE[Y_T])^2}-1
 &\le\EE\left[\exp\left\{\frac{4|T||F_1\cap F_2|}{m}\right\}-1\right]\\
 &\le\frac{\EE[|F_1\cap F_2|]}{h}
 \left(\exp\left\{\frac{4|T|h}{m}\right\}-1\right)
 \le\frac{\EE[|F_1\cap F_2|]}{h}\exp\left\{\frac{4|T|h}{m}\right\},
\end{aligned}
\]
where the expectations involving $F_1,F_2$ are over their independent uniform choices. This proves the variance bound.
\end{proof}

The exponential tail bounds below will be used for degrees and random samples.
\begin{lemma}[Chernoff bounds]\label{lem:chernoff}
Let $X$ be binomial, hypergeometric, or Poisson with mean $\mu>0$, and put $\varphi(a)\coloneqq1-a+a\log a$ for $a>0$. Then the following bounds hold.
\begin{enumerate}[label=(\textit{\roman*})]
\item For $0<a<1$, we have $\PP[X\le a\mu]\le\exp\{-\mu\varphi(a)\}$.
\item For $t\ge\mu$, we have $\PP[X\ge t]\le\exp\{-\mu\varphi(t/\mu)\}\le(\e\mu/t)^t$.
\item For $0<\delta\le1$, we have $\PP[|X-\mu|\ge\delta\mu]\le2\exp\{-\delta^2\mu/3\}$.
\end{enumerate}
\end{lemma}
These are standard Chernoff bounds; see Janson, \L uczak, and Ruci\'nski~\cite[Section~2.1]{JansonLuczakRucinski2000}. The hypergeometric case follows from Hoeffding's comparison theorem~\cite[Theorem~4]{Hoeffding1963}.

An event determined by independent edge indicators is \emph{increasing} if adding edges preserves it. We use Harris's inequality~\cite{Harris1960} in the following form.
\begin{lemma}[{\cite[Theorem~2.12]{JansonLuczakRucinski2000}}]\label{lem:harris}
If $A_1,\ldots,A_m$ are increasing events determined by independent Bernoulli variables, then $\PP[\bigcap_{i=1}^m A_i]\ge\prod_{i=1}^m\PP[A_i]$.
\end{lemma}

We need to control the logarithmic losses at all intermediate sizes of the edge-deletion process, so we use the maximal form of the Azuma--Hoeffding inequality~\cite{Azuma1967}. A \emph{filtration} $(\mathcal F_i)_{i=0}^m$ is a sequence of $\sigma$-algebras on a common probability space satisfying $\mathcal F_0\subseteq\mathcal F_1\subseteq\cdots\subseteq\mathcal F_m$. Here $\mathcal F_i$ represents the information revealed up to step $i$.
\begin{lemma}\label{lem:maximal-azuma}
Let $(\mathcal F_i)_{i=0}^m$ be a filtration, and let $D_i$ be $\mathcal F_i$-measurable with $\EE[D_i\mid\mathcal F_{i-1}]=0$ and $|D_i|\le c_i$ almost surely, where the $c_i$ are deterministic nonnegative numbers and $\sum_{i=1}^m c_i^2>0$. For every $a>0$,
\[
 \PP\left[\max_{1\le j\le m}\sum_{i=1}^jD_i\ge a\right]
 \le\exp\left\{-\frac{a^2}{2\sum_{i=1}^m c_i^2}\right\}.
\]
\end{lemma}
\begin{proof}
The conditional exponential-moment bound of Azuma~\cite[Section~2, Remark~1]{Azuma1967} gives $\EE[\e^{\lambda D_i}\mid\mathcal F_{i-1}]\le\exp\{\lambda^2c_i^2/2\}$ for $\lambda>0$. Thus $\exp\{\lambda\sum_{i=1}^jD_i-\lambda^2\sum_{i=1}^jc_i^2/2\}$ is a nonnegative supermartingale starting at $1$. Stop at the first $j$ with $\sum_{i=1}^jD_i\ge a$, or at $m$ if no such $j$ exists. The expected value of the supermartingale at this bounded stopping time is at most $1$. On the event that a partial sum reaches $a$, the value of the supermartingale at the stopping time is at least $\exp\{\lambda a-\lambda^2\sum_{i=1}^m c_i^2/2\}$. By nonnegativity, it follows that
\[
 \exp\left\{\lambda a-\frac{\lambda^2}{2}\sum_{i=1}^m c_i^2\right\}
 \PP\left[\max_{1\le j\le m}\sum_{i=1}^jD_i\ge a\right]\le1.
\]
Taking $\lambda=a/\sum_{i=1}^m c_i^2$ proves the result.
\end{proof}

We next recall the definitions and basic properties of entropy. For a random variable $Z$ on a finite set, its \emph{entropy} is $\HH[Z]\coloneqq-\sum_z\PP[Z=z]\log\PP[Z=z]$, with $0\log0\coloneqq0$. Its conditional entropy given a finite-valued random variable $Y$ is $\HH[Z\mid Y]\coloneqq\sum_y\PP[Y=y]\HH[Z\mid Y=y]$, where the sum is over $\PP$-admissible values $y$ of $Y$.

For probability distributions $P,Q$ on the same finite set, their \emph{relative entropy} is
\[
 \KL[P\Vert Q]\coloneqq\sum_z P(z)\log\frac{P(z)}{Q(z)},
\]
where terms with $P(z)=0$ are interpreted as zero. It is infinite if $P(z)>0=Q(z)$ for some $z$.

We use the following entropy identities and inequalities, with natural logarithms.
\begin{lemma}[Chain rule {\cite[Section~2.5]{CoverThomas2006}}]\label{lem:entropy-chain-rule}
For finite-valued random variables $Y,Z$, we have $\HH[Y,Z]=\HH[Y]+\HH[Z\mid Y]$.
\end{lemma}

\begin{lemma}[Log-sum inequality {\cite[Section~2.7]{CoverThomas2006}}]\label{lem:log-sum}
For nonnegative numbers $a_1,\ldots,a_m$ and $b_1,\ldots,b_m$ with positive sums, we have
\[
 \sum_{i=1}^m a_i\log\frac{a_i}{b_i}
 \ge\left(\sum_{i=1}^m a_i\right)\log\frac{\sum_{i=1}^m a_i}{\sum_{i=1}^m b_i},
\]
where terms with $a_i=0$ are zero and a term with $a_i>0=b_i$ is infinite.
\end{lemma}

\begin{lemma}[Pinsker's inequality {\cite[Lemma~11.6.1]{CoverThomas2006}}]\label{lem:pinsker}
For probability distributions $P,Q$ on a common finite set, we have $\sum_z|P(z)-Q(z)|\le\sqrt{2\KL[P\Vert Q]}$.
\end{lemma}

Shearer's inequality bounds joint entropy by the entropies of overlapping groups of coordinates.
\begin{lemma}[{\cite[Section~V]{ChungEtAl1986}}]\label{lem:shearer}
Let $(Z_1,\ldots,Z_m)$ be a finite-valued random vector, and let $\mathcal A$ be a collection of subsets of $[m]$, with repetitions allowed. If each coordinate occurs in at least $t\ge1$ members of $\mathcal A$, then
\[
 \HH[Z_1,\ldots,Z_m]\le\frac1t\sum_{A\in\mathcal A}\HH[(Z_i)_{i\in A}].
\]
Moreover, for any finite-valued random variable $Y$ and any $\PP$-admissible value $y$ of $Y$, we have
\[
 \HH[Z_1,\ldots,Z_m\mid Y=y]
 \le\frac1t\sum_{A\in\mathcal A}\HH[(Z_i)_{i\in A}\mid Y=y].
\]
\end{lemma}

\section{Proof of \texorpdfstring{\cref{thm:main}}{Theorem 1.1} from the conditioned core}
\label{sec:mixed}
We record the conditioned-core counting theorem together with the
definitions needed to state it. Assuming this theorem, the proof of
\cref{thm:main} below uses only the mixed-cycle definitions and the
elementary probability bounds already given in \cref{sec:preliminaries}.

We study the uniform distribution on the following family of hypergraphs with prescribed degree lower bounds.
\begin{definition}[Hypergraphs with degree lower bounds]\label{def:degree-constrained-model}
For a positive integer $n$, a nonnegative integer $m$, and a vector $\boldsymbol\ell=(\ell_v)_{v\in[n]}$ of nonnegative integers, define
\[
 \Omega(n,m,\boldsymbol\ell)\coloneqq
 \left\{F\subseteq\tbinom{[n]}3\colon |F|=m\text{ and }d_F(v)\ge\ell_v\text{ for every }v\in[n]\right\}.
\]

We call a parameter triple $(n,m,\boldsymbol\ell)$
\begin{itemize}[beginpenalty=10000,midpenalty=10000]
\item \emph{admissible} if $\Omega(n,m,\boldsymbol\ell)\ne\varnothing$;
\item \emph{near-critical} if $|m-n\log n/3|\le n\log\log n$ and $|\ell_v-\lfloor0.01\log n\rfloor|\le2$ for every $v\in[n]$.
\end{itemize}
The bound $2$ in the definition of near-criticality may be replaced by any fixed positive constant, with the implicit constants and convergence rates allowed to depend on that constant.
\end{definition}

We use the following assumptions for the conditioned counting problem.
\begin{assumption}[Conditioned core]\label{ass:core}
The following conditions hold.
\begin{enumerate}[label=(\textit{\roman*})]
\item The parameter triple $(n_\ast,m_\ast,\boldsymbol\ell)$ is admissible and near-critical.
\item A matching $\cM_0$ is fixed on $[n_\ast]$ and consists of $s$ pairs, where $1\le s\le n_\ast^{1/10}$ and $n_\ast+s$ is even.
\item A hypergraph $F_*$ is chosen uniformly at random from $\Omega(n_\ast,m_\ast,\boldsymbol\ell)$.
\end{enumerate}
\end{assumption}
We use the following notation for cycles respecting this matching.
\begin{enumerate}[label=(\textit{\roman*})]
\item The set of \emph{$\cM_0$-compatible triples} is defined as
\[
 \Kc\coloneqq\left\{e\in\binom{[n_\ast]}3\colon |e\cap V(\cM_0)|\le1\right\}.
\]
\item For a $3$-graph $H$ on $[n_\ast]$, write $\Xc(H)\coloneqq \cyc(H\cap\Kc,\cM_0)$ for the number of $(H\cap\Kc,\cM_0)$-respecting mixed Hamilton cycles.
\end{enumerate}
Note that every such cycle uses exactly $(n_\ast-s)/2$ triples.
We use the convention $\log0\coloneqq-\infty$.

The following theorem gives a lower bound on the number $\Xc(F_*)$ of mixed Hamilton cycles. Its proof in \cref{sec:martingale} reduces to a uniform bound on edge marginals, which is established in \cref{sec:migration}.
\begin{theorem}[Conditioned counting]
  \label{thm:conditioned-core}
  There is an absolute constant $C>0$ such that, under
  \cref{ass:core}, with probability $1-o(1)$,
\[
\log\Xc(F_*)\ge
 \frac{n_\ast-s}{2}\log\frac{6m_\ast}{n_\ast}-n_\ast+s-C\frac{n_\ast}{\log n_\ast}.
\]
The $o(1)$ term depends only on $n_\ast$.
\end{theorem}

With the same probability, $\Xc(F_*)\ge(6m_\ast/(\e^2n_\ast))^{(n_\ast-s)/2}\exp\{-O(n_\ast/\log n_\ast)\}$. By \cref{prop:complete-count}(iii), this lower bound agrees, up to a factor $\exp\{o(n_\ast)\}$, with the expected number of mixed Hamilton cycles respecting the same matching $\cM_0$ and using only $\cM_0$-compatible triples in a uniform $m_\ast$-edge $3$-graph on $[n_\ast]$.

We now prove \cref{thm:main} assuming the conditioned counting theorem. Conditional on $\tau_1=m$, the probability that $\mathcal H_{\tau_1}$ equals a given $m$-edge hypergraph with no isolated vertices is proportional to its number of edges incident to a degree-$1$ vertex. Exposing all edges incident to low-degree vertices makes this weight constant across all hypergraphs consistent with the exposure.

\begin{proof}[Proof of \cref{thm:main} assuming \cref{thm:conditioned-core}]
Throughout this proof, we refer the reader to \cref{fig:core-reduction} for an illustration of the reduction.

\begin{figure}[!ht]
\centering
\begin{tikzpicture}[
  x=.85cm,y=.85cm,font=\small,
  vertex/.style={circle,fill=black,inner sep=1.4pt},
  core triple/.style={draw=black!55,fill=black!5,line width=.6pt},
  set boundary/.style={draw=black!30,densely dashed,line width=.4pt}
]
  \foreach \step/\xoffset/\yoffset/\heading in {
    1/0/0/{1. Choose the incident triples},
    2/8.4/0/{2. Keep the pairs of $\cM_0$},
    3/8.4/-7.5/{3. Find a mixed Hamilton cycle},
    4/0/-7.5/{4. Expand to a loose Hamilton cycle}} {
    \begin{scope}[shift={(\xoffset,\yoffset)}]
      \node at (0,3.3) {\heading};
      \ifnum\step=1
        \path[fill=black!4,even odd rule]
          (0,0) circle (2.75) (0,0) circle (1.98);
        \draw[set boundary] (0,0) circle (2.75);
        \node at (135:2.4) {$B$};
      \fi
      \ifnum\step=4
        \path[fill=black!4,even odd rule]
          (0,0) circle (2.75) (0,0) circle (1.98);
        \draw[set boundary] (0,0) circle (2.75);
        \node at (135:2.4) {$B$};
      \fi
      \draw[set boundary] (0,0) circle (1.98);
      \foreach \angle/\edgecolor in {
        90/blue!75!black,
        0/orange!80!black,
        -90/teal!75!black,
        -180/violet!80!black} {
        \coordinate (a) at ({\angle+15}:1.5);
        \coordinate (b) at ({\angle-15}:1.5);
        \coordinate (junction) at ({\angle-45}:1.5);
        \coordinate (next) at ({\angle-75}:1.5);
        \coordinate (private1) at ({\angle-30}:1.78);
        \coordinate (private2) at ({\angle-60}:1.78);
        \coordinate (exceptional) at (\angle:2.4);
        \ifnum\step>2
          \draw[core triple] (b)--(private1)--(junction)--cycle;
          \draw[core triple] (junction)--(private2)--(next)--cycle;
        \fi
        \ifnum\step=1
          \draw[draw=\edgecolor,fill=\edgecolor,fill opacity=.1,line width=.8pt]
            (a)--(exceptional)--(b)--cycle;
          \node[vertex,fill=\edgecolor] at (exceptional) {};
        \else
          \ifnum\step=4
            \draw[draw=\edgecolor,fill=\edgecolor,fill opacity=.1,line width=.8pt]
              (a)--(exceptional)--(b)--cycle;
            \node[vertex,fill=\edgecolor] at (exceptional) {};
          \else
            \draw[\edgecolor,line width=1.4pt] (a)--(b);
          \fi
        \fi
        \foreach \point in {a,b,junction,private1,private2}
          \node[vertex] at (\point) {};
      }
      \node at (0,.12) {$[n]\setminus B$};
      \ifnum\step=2
        \node at (0,-.42) {pairs $\cM_0$};
      \fi
      \ifnum\step=3
        \node at (0,-.42) {core triples};
      \fi
    \end{scope}
  }
  \draw[-{Stealth[length=5pt]},line width=.7pt] (3.05,0)--(5.35,0)
    node[midway,above=4pt] {delete $B$};
  \draw[-{Stealth[length=5pt]},line width=.7pt] (8.4,-2.3)--(8.4,-3.85)
    node[midway,right=5pt] {\cref{thm:conditioned-core}};
  \draw[-{Stealth[length=5pt]},line width=.7pt] (5.35,-7.5)--(3.05,-7.5)
    node[midway,above=4pt] {restore $B$};
\end{tikzpicture}
\caption{Reduction to a mixed cycle and expansion to a loose Hamilton cycle. The outer ring represents the low-degree set $B$, and the coloured pairs form $\cM_0$. Replace each selected triple by the pair obtained by removing its vertex in $B$. Expanding these pairs restores the vertices of $B$ as private vertices.}
\label{fig:core-reduction}
\end{figure}
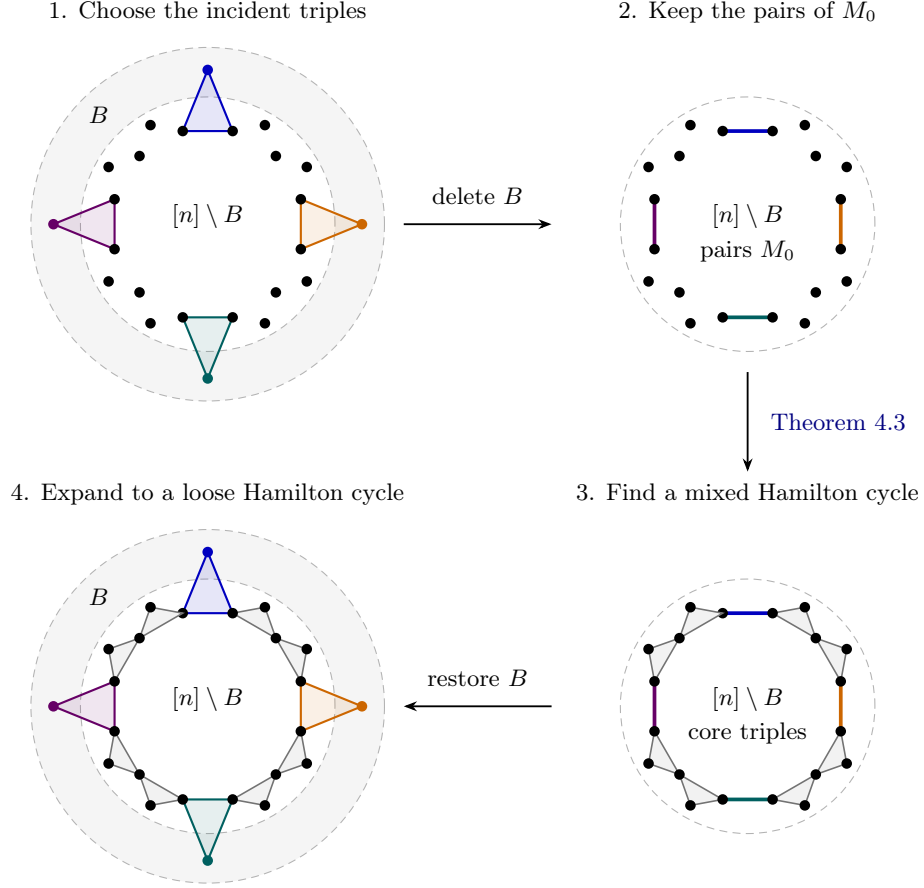

Let $n$ be even and sufficiently large. We first determine the conditional distribution of the remaining hypergraph and then control the exposed edges.

Set $d_0\coloneqq\lfloor0.01\log n\rfloor$. We reveal the hitting time $\tau_1$, writing its value as $m$, the set $B$ of vertices whose degree in $\mathcal H_{\tau_1}$ is at most $d_0$, and all edges meeting $B$. Let $\mathcal T\coloneqq\{e\in\mathcal H_{\tau_1}\colon e\cap B\ne\varnothing\}$ denote this set of revealed edges. The edges contained entirely in $[n]\setminus B$ remain unrevealed.

The remaining hypergraph $\mathcal H_{\tau_1}-B$ has $n_\ast\coloneqq n-|B|$ vertices and $m_\ast\coloneqq m-|\mathcal T|$ edges. Every $w\in[n]\setminus B$ has total degree at least $d_0+1$. The revealed edges contribute $d_{\mathcal T}(w)$ to this degree, so the remaining edges must contribute at least $\ell_w\coloneqq\max\{0,d_0+1-d_{\mathcal T}(w)\}$.
To specify the prescribed matching in the reduction, choose deterministically, from the exposed edges, an incident triple $e_v=\{v,a_v,b_v\}$ for each $v\in B$. If these triples are vertex-disjoint, then $\cM_0\coloneqq\{a_vb_v\colon v\in B\}$ is a matching on $[n]\setminus B$. In that case, a $(\mathcal H_{\tau_1}-B,\cM_0)$-respecting mixed Hamilton cycle becomes a loose Hamilton cycle in $\mathcal H_{\tau_1}$ when each pair $a_vb_v$ is replaced by its selected triple $e_v$. Every restored vertex $v\in B$ is private in this cycle. This replacement is illustrated in \cref{fig:core-reduction}.

\begin{claim}\label{prop:exact-core}
Conditional on any $\PP$-admissible exposure $(m,B,\mathcal T)$, $\mathcal H_{\tau_1}-B$ is uniform among simple $3$-graphs $\mathcal{G}$ on $[n]\setminus B$ with $m_\ast$ edges and satisfying $d_{\mathcal{G}}(w)\ge\ell_w$ for every $w\in[n]\setminus B$.
\end{claim}
\begin{proof}[Proof of \cref{prop:exact-core}]
For an $m$-edge $3$-graph $F$ on $[n]$ with minimum degree at least $1$, let $b(F)$ count the edges containing a degree-$1$ vertex. Then $\PP[\mathcal H_{\tau_1}=F]=\frac{b(F)}{m\binom{\binom n3}{m}}$. Indeed, conditional on $\mathcal H_m=F$, the edges of $F$ appear in a uniformly random order, so the edge $e$ added at step $m$ is uniform in $F$. Since $F$ has no isolated vertices, $\tau_1=m$ exactly when the preceding hypergraph $F-e$ has an isolated vertex. This occurs precisely when $e$ contains a degree-$1$ vertex of $F$, giving $\PP[\tau_1=m\mid\mathcal H_m=F]=b(F)/m$.

We now apply this probability formula after conditioning on $(m,B,\mathcal T)$. All degree-$1$ vertices lie in $B$, and all their incident edges are recorded in $\mathcal T$. Thus $b(F)$ is fixed by the exposure.  Every hypergraph consistent with the exposure has $m$ edges. For each $w\in[n]\setminus B$, the requirement that $w$ lie outside the low-degree set is equivalent to $d_{\mathcal{G}}(w)+d_{\mathcal T}(w)\ge d_0+1$. Hence all hypergraphs consistent with the exposure and these degree constraints have equal probability.
\end{proof}

We next establish the separation and degree bounds needed to make the selected triples vertex-disjoint and to ensure that deleting $B$ removes at most two incident edges at each remaining vertex.

\begin{claim}\label{prop:low-at-tau}
There exists an absolute constant $C>0$ such that, with probability $1-o(1)$, the following statements hold simultaneously.
\begin{enumerate}[label=(\textit{\roman*}),beginpenalty=10000,midpenalty=10000]
\item $1\le|B|\le n^{1/12}$;
\item $\Delta_1(\mathcal H_{\tau_1})\le C\log n$;
\item $\Delta_2(\mathcal H_{\tau_1})\le2$;
\item no two vertices of $B$ are joined by a Berge path of length at most two in $\mathcal H_{\tau_1}$.
\end{enumerate}
\end{claim}
\begin{proof}[Proof of \cref{prop:low-at-tau}]
We first show that $\mathcal H_{\tau_1}$ lies between two coupled binomial hypergraphs. Assign independent uniform $[0,1]$ arrival times to all triples, and let $\mathcal H(p)$ consist of those with arrival time at most $p$. Put
\[
 p_{\mathrm{lo}}\coloneqq\frac{\log n-\log\log n}{\binom{n-1}{2}},\quad\text{and}\quad
 p_{\mathrm{hi}}\coloneqq\frac{\log n+\log\log n}{\binom{n-1}{2}}.
\]
The expected numbers of isolated vertices at $p_{\mathrm{hi}}$ and $p_{\mathrm{lo}}$ are $(1+o(1))/\log n$ and $(1+o(1))\log n$, respectively. The first moment bound shows that $\mathcal H(p_{\mathrm{hi}})$ has no isolated vertices \whp. Let $Z$ count the isolated vertices of $\mathcal H(p_{\mathrm{lo}})$. The joint isolation probability for two vertices divided by the product of their isolation probabilities is $(1-p_{\mathrm{lo}})^{-(n-2)}=1+O(\log n/n)$. Expanding $\Var[Z]$ as the sum of the variances and covariances of the isolation indicators gives
\[
 \Var[Z]\le\EE[Z]+(\EE[Z])^2\left(\frac{1}{(1-p_{\mathrm{lo}})^{n-2}}-1\right).
\]
Chebyshev's inequality therefore gives
\[
 \PP[Z=0]\le\frac{\Var[Z]}{(\EE[Z])^2}
 \le\frac1{\EE[Z]}+O\left(\frac{\log n}{n}\right)
 =\frac{1+o(1)}{\log n}+O\left(\frac{\log n}{n}\right)=o(1).
\]
Thus $\mathcal H(p_{\mathrm{lo}})$ still has an isolated vertex \whp. Under the arrival-time coupling, these two conclusions give $\mathcal H(p_{\mathrm{lo}})\subseteq\mathcal H_{\tau_1}\subseteq\mathcal H(p_{\mathrm{hi}})$ \whp.

Let $B_{\mathrm{lo}}\coloneqq\{v\colon d_{\mathcal H(p_{\mathrm{lo}})}(v)\le d_0\}$. The binomial lower-tail bound \cref{lem:chernoff}(i) gives $\PP[v\in B_{\mathrm{lo}}]\le n^{-\varphi(0.01)+o(1)}$, where $\varphi(0.01)=0.943948\ldots$. Thus $\EE[|B_{\mathrm{lo}}|]\le n^{0.056052\ldots+o(1)}$, and \cref{lem:moment-bounds}(i) implies $|B_{\mathrm{lo}}|\le n^{1/12}$ \whp. On the event $\mathcal H(p_{\mathrm{lo}})\subseteq\mathcal H_{\tau_1}\subseteq\mathcal H(p_{\mathrm{hi}})$, we have $B\subseteq B_{\mathrm{lo}}$.

To prove the separation assertion, we count Berge paths in $\mathcal H(p_{\mathrm{hi}})$ whose ends lie in $B_{\mathrm{lo}}$. For $h\in\{1,2\}$, a labelled Berge path of length $h$ with distinct ends in the complete $3$-graph on $[n]$ is specified by its $h+1$ path vertices in order and the third vertex of each of its $h$ triples. Hence there are at most $n^{h+1}n^h=n^{2h+1}$ choices. Condition on the edges of a fixed such path belonging to $\mathcal H(p_{\mathrm{hi}})$. To bound the probability that both ends lie in $B_{\mathrm{lo}}$, count at each end only incident triples in $\mathcal H(p_{\mathrm{lo}})$ that are not edges of the fixed path and avoid the other end. These numbers of incident triples depend on disjoint sets of independent edge indicators, so they are independent; each has mean $(1+o(1))\log n$. Thus the conditional probability is at most $n^{-2\varphi(0.01)+o(1)}$, and the expected number of such paths with ends in $B_{\mathrm{lo}}$ is at most
\[
 O(n^{2h+1})p_{\mathrm{hi}}^h n^{-2\varphi(0.01)+o(1)}
 =n^{1-2\varphi(0.01)+o(1)}(\log n)^h=o(1).
\]
For each vertex $v$, the degree $d_{\mathcal H(p_{\mathrm{hi}})}(v)$ is binomial with mean $\log n+\log\log n$. By \cref{lem:chernoff}(ii), for a sufficiently large fixed constant $C$ and all sufficiently large $n$, we have \[
 \PP[d_{\mathcal H(p_{\mathrm{hi}})}(v)>C\log n]
 \le\left(\frac{\e(\log n+\log\log n)}{C\log n}\right)^{C\log n}
 \le n^{-3}.
\] A union bound over the vertices gives $\PP[\Delta_1(\mathcal H(p_{\mathrm{hi}}))>C\log n]\le n\cdot n^{-3}=n^{-2}=o(1)$. The inclusion $\mathcal H_{\tau_1}\subseteq\mathcal H(p_{\mathrm{hi}})$ therefore gives the required maximum-degree bound.  Also $\PP[\Delta_2(\mathcal H(p_{\mathrm{hi}}))\ge3] \le\binom n2\binom{n-2}{3}p_{\mathrm{hi}}^3=O(\log^3n/n)$.

Since no Berge path of length at most two joins distinct vertices of $B$, every $w\in[n]\setminus B$ shares an edge with at most one vertex of $B$. By the codegree bound, it therefore belongs to at most two edges of $\mathcal T$. Finally, the last arriving edge covers at least one previously isolated vertex, which has degree $1$ at $\tau_1$.  Hence $B$ is nonempty.
\end{proof}

To apply \cref{thm:conditioned-core} conditional on the revealed data, call an exposure $(m,B,\mathcal T)$ good if the following conditions hold.
\begin{enumerate}[label=(\textit{\roman*})]
\item The selected triples $e_v$ are vertex-disjoint.
\item Every vertex in $[n]\setminus B$ belongs to at most two edges of $\mathcal T$.
\item The resulting triple $(n_\ast,m_\ast,\boldsymbol\ell)$ is admissible and near-critical.
\item $1\le |B|\le n_\ast^{1/10}$.
\end{enumerate}

We verify these conditions using \cref{prop:low-at-tau}. On the event in that claim, no selected triple $e_v$ meets $B$ outside $v$. If two selected triples intersected, they would give a Berge path of length two between distinct vertices of $B$. Hence the selected triples are vertex-disjoint, and $\cM_0$ is a matching on $[n]\setminus B$. The separation and codegree bounds also imply that every vertex in $[n]\setminus B$ belongs to at most two edges of $\mathcal T$.

After identifying the vertex set $[n]\setminus B$ with $[n_\ast]$, we verify \cref{ass:core} for $\mathcal H_{\tau_1}-B$, conditional on the exposure $(m,B,\mathcal T)$. The space $\Omega(n_\ast,m_\ast,\boldsymbol\ell)$ is nonempty by \cref{prop:exact-core}. On the event in \cref{prop:low-at-tau}, $|\mathcal T|\le Cn^{1/12}\log n=o(n)$ and $\ell_w\in\{d_0-1,d_0,d_0+1\}$. Since $\log n-\log n_\ast=o(1)$, for large $n$ we have $|\ell_w-\lfloor0.01\log n_\ast\rfloor|\le2$.

To control the numbers of edges, for each $p\in\{p_{\mathrm{lo}},p_{\mathrm{hi}}\}$ we have $\Var[|\mathcal H(p)|]\le\binom n3p=O(n\log n)$. Chebyshev's inequality in \cref{lem:moment-bounds}(ii) gives
\[
 \PP\left[\left||\mathcal H(p)|-\binom n3p\right|>n^{3/4}\right]
 \le\frac{\binom n3p}{n^{3/2}}
 =O\left(\frac{\log n}{\sqrt n}\right)=o(1).
\]
Thus, with probability $1-o(1)$, we have $|\mathcal H(p)|=\binom n3p+o(n)$ at both values of $p$. The inclusions $\mathcal H(p_{\mathrm{lo}})\subseteq\mathcal H_{\tau_1}\subseteq\mathcal H(p_{\mathrm{hi}})$ therefore give $|3m/n-\log n|\le\log\log n+o(1)$. Since $n_\ast=n-|B|$ and $m_\ast=m-|\mathcal T|$, the bounds on $|B|$ and $|\mathcal T|$ give
\[
 \frac{3m_\ast}{n_\ast}-\frac{3m}{n}
 =\frac{3(m|B|-n|\mathcal T|)}{n(n-|B|)}
 =O\left(\frac{\log n}{n^{11/12}}\right)=o(1).
\]
Together with $\log n-\log n_\ast=o(1)$, this proves $|3m_\ast/n_\ast-\log n_\ast|\le3\log\log n_\ast$ for large $n$. Also $s=|B|\le n_\ast^{1/10}$ for large $n$, and $n_\ast+s=n$ is even. Thus the parameter and matching conditions in \cref{ass:core} hold.

All four conditions depend only on the exposure and hold with probability $1-o(1)$ by \cref{prop:low-at-tau} and the preceding parameter estimates.

Conditional on any good exposure, \cref{prop:exact-core} shows that $\mathcal H_{\tau_1}-B$ is uniform among the hypergraphs with the prescribed number of edges and degree lower bounds. Apply \cref{thm:conditioned-core} with the matching $\cM_0$ determined by this exposure. The lower bound on the number of mixed Hamilton cycles fails with conditional probability $o(1)$, uniformly over all good exposures. Whenever this bound holds, choose a $(\mathcal H_{\tau_1}-B,\cM_0)$-respecting mixed Hamilton cycle and replace each prescribed pair $a_vb_v$ by $\{a_v,v,b_v\}$. This inserts the vertices of $B$ as distinct private vertices and preserves the cyclic order of the intersection vertices, giving a loose Hamilton cycle at $\tau_1$. Thus
\[
\begin{aligned}
 \PP[\tau_{\LH}>\tau_1]
 &\le\PP[(\tau_1,B,\mathcal T)\text{ is not good}]\\
 &{}+\sup_{\substack{(m,B_0,\mathcal T_0)\\\text{good and }\PP\text{-admissible}}}
 \PP[\tau_{\LH}>\tau_1\mid(\tau_1,B,\mathcal T)=(m,B_0,\mathcal T_0)],
\end{aligned}
\]
which is $o(1)$ by the preceding discussion.
Since $\tau_{\LH}\ge\tau_1$ deterministically, the theorem follows.
\end{proof}

\section{Kahn's matching entropy inequality and its stability}
\label{sec:matching-entropy}
Kahn proved the following entropy inequality, stated here for random perfect matchings in $3$-graphs.
In this section, we use Kahn's entropy inequality to derive a stability result for the edge marginals of a random perfect matching. We will apply this result in \cref{sec:entropy} to the auxiliary matchings associated with mixed Hamilton cycles.

\begin{theorem}[{\cite[Theorem~4.2]{Kahn2023}}]\label{thm:kahn}
Let $\mathcal G$ be a $3$-graph on $n$ vertices, where $3\mid n$, and let $\mathbf F$ be sampled from an arbitrary probability distribution on the perfect matchings of $\mathcal G$. For each vertex $u$, let $\mathbf F(u)$ be the edge of $\mathbf F$ containing $u$. For each edge $e\in\mathcal G$, set $p_e\coloneqq\PP[e\in\mathbf F]$ and $\gamma_e\coloneqq\PP[\exists f\in\mathbf F\colon |e\cap f|\ge2]$; that is, $\gamma_e$ is the probability that the vertices of $e$ fail to lie in three distinct edges of $\mathbf F$. Then
\[
\HH[\mathbf F]\le\frac13\sum_u\HH[\mathbf F(u)]-\frac{2n}{3}
 +O\left(\sum_u\sum_{e\ni u}p_e\sqrt{\gamma_e}\right)+O(\log n).
\]
\end{theorem}

Since $\sum_{e\in\mathcal G}p_e=\EE[|\mathbf F|]=n/3$, the vector $(3p_e/n)_{e\in\mathcal G}$ is a probability distribution on the edge set of $\mathcal G$. This is the distribution of an edge chosen uniformly from $\mathbf F$ after the matching has been sampled. The following consequence bounds the relative entropy of this distribution with respect to $\Unif(\mathcal G)$, as well as the entropy of $\mathbf F$. Keeping both terms will allow us to apply the estimate under conditional laws and then average.
\begin{proposition}\label{prop:matching-stability}
Fix $C\ge1$. Let $\mathcal G$ be a $3$-graph on $n$ vertices, where $3\mid n$, with $\Delta_1(\mathcal G)\le C d(\mathcal G)$ and $\Delta_2(\mathcal G)=o(d(\mathcal G))$ as $n\to\infty$. Let $\mathbf F$ be sampled from an arbitrary probability distribution on the perfect matchings of $\mathcal G$, and set $p_e\coloneqq\PP[e\in\mathbf F]$. Then
\[
 \HH[\mathbf F]+\frac{n}{6}\KL\left[\left(\frac{3p_e}{n}\right)_{e\in\mathcal G}\,\middle\Vert\,\Unif(\mathcal G)\right]
 \le \frac{n}{3}\log d(\mathcal G)-\frac{2n}{3}
 +O_C\left(\frac{n}{\log\left(\frac{d(\mathcal G)}{\Delta_2(\mathcal G)}\right)}+\log n\right).
\]
In particular, if $\HH[\mathbf F]\ge (n/3)\log d(\mathcal G)-2n/3-\varepsilon n$ for some $\varepsilon\ge0$, we have
\[
 \sum_{e\in\mathcal G}\left|p_e-\frac{1}{d(\mathcal G)}\right|
 =O_C\left(n\sqrt{\varepsilon+\frac1{\log\left(\frac{d(\mathcal G)}{\Delta_2(\mathcal G)}\right)}+\frac{\log n}{n}}\right).
\]
\end{proposition}
\begin{proof}
We first bound $\sum_u\sum_{e\ni u}p_e\sqrt{\gamma_e}$, which controls the error term in \cref{thm:kahn}. We then express $\frac13\sum_u\HH[\mathbf F(u)]$ and the bound in \cref{lem:collision} in terms of the same relative entropy, retaining a positive coefficient of this term to obtain the stability result.

Since $\Delta_1(\mathcal G)\ge d(\mathcal G)$, the codegree assumption gives $3\Delta_2(\mathcal G)/\Delta_1(\mathcal G)=o(1)$. Thus the next claim applies for sufficiently large $n$. It is a quantitative form of the argument in Kahn~\cite[Lemma~4.3]{Kahn2023}; we use $p_e$ and $\gamma_e$ as defined in \cref{thm:kahn}.
\begin{claim}\label{lem:collision}
Let $\mathcal G$ be a $3$-graph on $n$ vertices, where $3\mid n$, and suppose that $0<3\Delta_2(\mathcal G)/\Delta_1(\mathcal G)<1/4$. Let $\mathbf F$ be sampled from an arbitrary probability distribution on the perfect matchings of $\mathcal G$. Then
\[
\sum_u\sum_{e\ni u}p_e\sqrt{\gamma_e}
 \le n\left(\frac{3\Delta_2(\mathcal G)}{\Delta_1(\mathcal G)}\right)^{1/4}
 +\frac{2\bigl(n\log\Delta_1(\mathcal G)-\sum_u\HH[\mathbf F(u)]+n\log2\bigr)}{\log\left(\frac{\Delta_1(\mathcal G)}{3\Delta_2(\mathcal G)}\right)}.
\]
\end{claim}
\begin{proof}[Proof of \cref{lem:collision}]
We split the sum $\sum_u\sum_{e\ni u}p_e\sqrt{\gamma_e}$ according to the size of $\gamma_e$. Let
\[
 B\coloneqq\left\{(u,e)\colon u\in e\in\mathcal G\text{ and }\gamma_e>\sqrt{\frac{3\Delta_2(\mathcal G)}{\Delta_1(\mathcal G)}}\right\}.
\]
On the complement of $B$, we have $\sqrt{\gamma_e}\le(3\Delta_2(\mathcal G)/\Delta_1(\mathcal G))^{1/4}$, while $\sqrt{\gamma_e}\le1$ on $B$. Hence
\[
\begin{aligned}
 \sum_u\sum_{e\ni u}p_e\sqrt{\gamma_e}
 \le\left(\frac{3\Delta_2(\mathcal G)}{\Delta_1(\mathcal G)}\right)^{1/4}
 \sum_u\sum_{\substack{e\ni u\\(u,e)\notin B}}p_e
 +\sum_{(u,e)\in B}p_e
 \le n\left(\frac{3\Delta_2(\mathcal G)}{\Delta_1(\mathcal G)}\right)^{1/4}
 +\sum_{(u,e)\in B}p_e.
\end{aligned}
\]
If $B=\varnothing$, then this proves the claim. Suppose henceforth that $B\ne\varnothing$. It remains to bound $\sum_{(u,e)\in B}p_e$. We first bound $|B|$ and then use the entropies $\HH[\mathbf F(u)]$ to control this sum.

Fix a perfect matching in the support of the law of $\mathbf F$. An edge of the host whose vertices do not lie in three distinct matching edges contains a pair from a single matching edge. There are at most $(n/3)\binom32\Delta_2(\mathcal G)=n\Delta_2(\mathcal G)$ such edges of the host, each with three incidences. Since $\gamma_e$ is the probability of this event for $e$, averaging this bound over the distribution of $\mathbf F$ gives $\sum_u\sum_{e\ni u}\gamma_e\le3n\Delta_2(\mathcal G)$.
By the definition of $B$ and the preceding bound on $\sum_u\sum_{e\ni u}\gamma_e$, we have $|B|\sqrt{3\Delta_2(\mathcal G)/\Delta_1(\mathcal G)}\le\sum_{(u,e)\in B}\gamma_e\le\sum_u\sum_{e\ni u}\gamma_e\le3n\Delta_2(\mathcal G)$, and hence
\[
 |B|\le\frac{3n\Delta_2(\mathcal G)}{\sqrt{\frac{3\Delta_2(\mathcal G)}{\Delta_1(\mathcal G)}}}
 =n\Delta_1(\mathcal G)\sqrt{\frac{3\Delta_2(\mathcal G)}{\Delta_1(\mathcal G)}}.
\]

To express the difference $n\log\Delta_1(\mathcal G)-\sum_u\HH[\mathbf F(u)]$ as relative entropy, at each vertex $u$ append zero entries to the array $(p_e)_{e\ni u}$ so that it has $\Delta_1(\mathcal G)$ entries. Since $\sum_{e\ni u}p_e=1$ at each vertex, the weights $p_e/n$, together with the added zero entries, form a probability distribution on the resulting $n\Delta_1(\mathcal G)$ positions. Equivalently, choose $u$ uniformly from $V(\mathcal G)$ and sample $\mathbf F$ independently; the incidence $(u,e)$ is selected when $\mathbf F(u)=e$, with probability $p_e/n$. The appended positions have probability zero. The relative entropy of this distribution with respect to the uniform distribution on these positions equals
\[
 \sum_{u\in V(\mathcal G)}\sum_{e\ni u}\frac{p_e}{n}\log\bigl(p_e\Delta_1(\mathcal G)\bigr)
 =\log\Delta_1(\mathcal G)-\frac1n\sum_{u\in V(\mathcal G)}\HH[\mathbf F(u)].
\]
Since the weights sum to $1$ and the appended positions have weight $0$, the total weights of $B$ and its complement are
\[
 \sum_{(u,e)\in B}\frac{p_e}{n}=\frac1n\sum_{(u,e)\in B}p_e,
 \quad
 \sum_{u\in V(\mathcal G)}\sum_{\substack{e\ni u\\(u,e)\notin B}}\frac{p_e}{n}
 =1-\frac1n\sum_{(u,e)\in B}p_e.
\]
Under the uniform distribution on all positions, the weight of $B$ is $\sum_{(u,e)\in B}1/(n\Delta_1(\mathcal G))=|B|/(n\Delta_1(\mathcal G))$, and its complement has weight $1-|B|/(n\Delta_1(\mathcal G))$. We apply \cref{lem:log-sum} to the weights $p_e/n$ and $1/(n\Delta_1(\mathcal G))$ separately on $B$ and its complement, including the appended positions in the latter group. This gives
\[
\begin{aligned}
 \sum_{(u,e)\in B}\frac{p_e}{n}\log\bigl(p_e\Delta_1(\mathcal G)\bigr)
 &\ge\left(\frac1n\sum_{(u,e)\in B}p_e\right)
 \log\frac{\frac1n\sum_{(u,e)\in B}p_e}{\frac{|B|}{n\Delta_1(\mathcal G)}},\\
 \sum_{u\in V(\mathcal G)}\sum_{\substack{e\ni u\\(u,e)\notin B}}\frac{p_e}{n}\log\bigl(p_e\Delta_1(\mathcal G)\bigr)
 &\ge\left(1-\frac1n\sum_{(u,e)\in B}p_e\right)
 \log\frac{1-\frac1n\sum_{(u,e)\in B}p_e}{1-\frac{|B|}{n\Delta_1(\mathcal G)}}.
\end{aligned}
\]
If either group has total weight $0$, its inequality reads $0\ge0$. The probabilities of $B$ and its complement form a distribution on two points, whose entropy is at most $\log2$. Thus
\[
 -\left(\frac1n\sum_{(u,e)\in B}p_e\right)\log\left(\frac1n\sum_{(u,e)\in B}p_e\right)
 -\left(1-\frac1n\sum_{(u,e)\in B}p_e\right)\log\left(1-\frac1n\sum_{(u,e)\in B}p_e\right)
 \le\log2.
\]
Multiplying by $-n$ bounds the corresponding two logarithmic terms below by $-n\log2$. The term $\bigl(n-\sum_{(u,e)\in B}p_e\bigr)\log\bigl(n\Delta_1(\mathcal G)/(n\Delta_1(\mathcal G)-|B|)\bigr)$ is nonnegative.
Adding the log-sum inequalities, multiplying by $n$, using the relative-entropy identity, and then applying these two observations gives
\[
\begin{aligned}
 & n\log\Delta_1(\mathcal G)-\sum_u\HH[\mathbf F(u)] \\
 &\quad =n\sum_{u\in V(\mathcal G)}\sum_{e\ni u}\frac{p_e}{n}\log\bigl(p_e\Delta_1(\mathcal G)\bigr)\\
 &\quad \ge\left(\sum_{(u,e)\in B}p_e\right)\log\frac{\Delta_1(\mathcal G)\sum_{(u,e)\in B}p_e}{|B|}
 +\left(n-\sum_{(u,e)\in B}p_e\right)\log\frac{\Delta_1(\mathcal G)\left(n-\sum_{(u,e)\in B}p_e\right)}{n\Delta_1(\mathcal G)-|B|}\\
 &\quad \ge\left(\sum_{(u,e)\in B}p_e\right)\log\frac{n\Delta_1(\mathcal G)}{|B|}-n\log2.
\end{aligned}
\]
Since $\log\bigl(n\Delta_1(\mathcal G)/|B|\bigr)\ge\tfrac12\log\bigl(\Delta_1(\mathcal G)/(3\Delta_2(\mathcal G))\bigr)>0$, the entropy estimate above yields
\[
 \sum_{(u,e)\in B}p_e
 \le\frac{2\left(n\log\Delta_1(\mathcal G)-\sum_u\HH[\mathbf F(u)]+n\log2\right)}
 {\log\left(\frac{\Delta_1(\mathcal G)}{3\Delta_2(\mathcal G)}\right)}.
\]
Substituting this bound for $\sum_{(u,e)\in B}p_e$ in the initial decomposition of $\sum_u\sum_{e\ni u}p_e\sqrt{\gamma_e}$ proves the claim.
\end{proof}

We now return to the proof of \cref{prop:matching-stability} and express the main entropy term in terms of relative entropy. For the host $\mathcal G$, each perfect matching has $n/3$ edges and $|\mathcal G|=nd(\mathcal G)/3$, so $(3p_e/n)_{e\in\mathcal G}$ is a probability distribution. Expanding its relative entropy gives
\[
 \KL\left[\left(\frac{3p_e}{n}\right)_{e\in\mathcal G}\,\middle\Vert\,\Unif(\mathcal G)\right]
 =\frac3n\sum_{e\in\mathcal G}p_e\log\bigl(p_e d(\mathcal G)\bigr)
 =\frac3n\sum_{e\in\mathcal G}p_e\log p_e+\log d(\mathcal G).
\]
Since each edge contributes at its three vertices, we have $\sum_u\HH[\mathbf F(u)]=-3\sum_{e\in\mathcal G}p_e\log p_e$. Combining these identities yields
\[
 \frac13\sum_u\HH[\mathbf F(u)]
 =\frac{n}{3}\log d(\mathcal G)-\frac{n}{3}\KL\left[\left(\frac{3p_e}{n}\right)_{e\in\mathcal G}\,\middle\Vert\,\Unif(\mathcal G)\right].
\]
By the preceding identity, the entropy difference $n\log\Delta_1(\mathcal G)-\sum_u\HH[\mathbf F(u)]$ in \cref{lem:collision} equals $n\log\frac{\Delta_1(\mathcal G)}{d(\mathcal G)}+n\KL\bigl[(3p_e/n)_{e\in\mathcal G}\Vert\Unif(\mathcal G)\bigr]$. Since $\Delta_1(\mathcal G)/d(\mathcal G)\le C$ and $d(\mathcal G)/\Delta_2(\mathcal G)\to\infty$, the fourth-root term in \cref{lem:collision} is absorbed by the term of order $n/\log(d(\mathcal G)/\Delta_2(\mathcal G))$, giving
\[
 \sum_u\sum_{e\ni u}p_e\sqrt{\gamma_e}
 \le O_C\left(\frac{n}{\log\left(\frac{d(\mathcal G)}{\Delta_2(\mathcal G)}\right)}\right)
 \left(1+\KL\left[\left(\frac{3p_e}{n}\right)_{e\in\mathcal G}\,\middle\Vert\,\Unif(\mathcal G)\right]\right).
\]
Substituting this bound and the displayed expression for $\frac13\sum_u\HH[\mathbf F(u)]$ into \cref{thm:kahn}, and moving the relative-entropy terms to the left, gives
\begin{multline*}
 \HH[\mathbf F]+\left(\frac n3-O_C\left(\frac n{\log\left(\frac{d(\mathcal G)}{\Delta_2(\mathcal G)}\right)}\right)\right)
 \KL\left[\left(\frac{3p_e}{n}\right)_{e\in\mathcal G}\,\middle\Vert\,\Unif(\mathcal G)\right]\\
 \le\frac n3\log d(\mathcal G)-\frac{2n}{3}
 +O_C\left(\frac n{\log\left(\frac{d(\mathcal G)}{\Delta_2(\mathcal G)}\right)}+\log n\right).
\end{multline*}
Since $d(\mathcal G)/\Delta_2(\mathcal G)\to\infty$, the coefficient of the relative entropy on the left is at least $n/6$ for all sufficiently large $n$. This proves the first inequality.

For the second inequality, subtract the assumed entropy lower bound from the first inequality and divide by $n/6$ to obtain
\[
 \KL\left[\left(\frac{3p_e}{n}\right)_{e\in\mathcal G}\,\middle\Vert\,\Unif(\mathcal G)\right]
 \le6\varepsilon+O_C\left(\frac1{\log\left(\frac{d(\mathcal G)}{\Delta_2(\mathcal G)}\right)}+\frac{\log n}{n}\right).
\]
Since $n/(3|\mathcal G|)=1/d(\mathcal G)$, \cref{lem:pinsker} gives
\[
\begin{aligned}
 \sum_{e\in\mathcal G}\left|p_e-\frac1{d(\mathcal G)}\right|
 =\frac n3\sum_{e\in\mathcal G}\left|\frac{3p_e}{n}-\frac1{|\mathcal G|}\right|
 \le\frac n3\sqrt{2\KL\left[\left(\frac{3p_e}{n}\right)_{e\in\mathcal G}\,\middle\Vert\,\Unif(\mathcal G)\right]}.
\end{aligned}
\]
Substituting the preceding relative-entropy bound proves the second inequality.
\end{proof}

The entropy bound in \cref{prop:matching-stability} also applies after conditioning on a finite-valued random variable. The host may depend on the conditioning value, but is deterministic once that value is fixed. Suppose that these hosts satisfy the hypotheses of \cref{prop:matching-stability}, with the same constant $C$ and $\Delta_2(\mathcal G)/d(\mathcal G)\to0$ uniformly over the conditioning values of positive probability. We may then apply the first inequality under each conditional matching law and average both sides, including the error term. We use this form in \cref{claim:conditional-entropy-balance}, where a lower bound for the average conditional matching entropy suffices.

\section{Entropy stability for mixed Hamilton cycles}
\label{sec:entropy}
We derive an entropy stability result for mixed Hamilton cycles in a fixed host by applying \cref{prop:matching-stability} to the associated perfect matchings. We then bound the expected number of triples shared by two independent cycles chosen uniformly from a sufficiently large family; this bound will control the variance of the number of mixed Hamilton cycles after random edge deletions in \cref{sec:resampling}.

We begin with the entropy estimate. Let $\mathcal{G}$ be a $3$-graph on $n$ vertices with $m$ edges, where $m/n\to\infty$, and let $\cM$ be a matching of $s\ge1$ pairs on the same vertex set. For the following discussion, fix a root pair in $\cM$ and its direction. Let $\mathbf C$ be sampled from an arbitrary probability distribution on the $(\mathcal G,\cM)$-respecting mixed Hamilton cycles. We give $\mathbf C$ its root orientation. For an oriented triple $(w;u\to v)\in\vec{\mathcal{G}}-V(\cM)$, define the \emph{directed edge marginal} by
\[
 \pi(w;u\to v)\coloneqq
 \PP[\mathbf C\text{ traverses }\{u,v,w\}\text{ from }u\text{ to }v\text{ with }w\text{ private}].
\]
When $\mathbf C$ is uniform on the $(\mathcal G,\cM)$-respecting mixed Hamilton cycles, \cref{fact:completion-identities}(iii) gives $\pi(w;u\to v)=\comp_{\mathcal G,\cM}(w;u\to v)/\cyc(\mathcal G,\cM)$.

\begin{figure}[htbp]
\centering
\begin{tikzpicture}[
  x=1cm,y=1cm,font=\small,
  vertex/.style={circle,fill=black,inner sep=1.5pt},
  traversal/.style={postaction={decorate},decoration={markings,mark=at position .55 with {\arrow{Stealth[length=4pt]}}}}
]
  \begin{scope}[xshift=1.8cm,xscale=1.35,yscale=.9]
    \node at (0,2.6) {Oriented mixed cycle};
    \coordinate (a) at (-1.1,-.85);
    \coordinate (b) at (-1.1,.85);
    \coordinate (c) at (0,1.7);
    \coordinate (d) at (1.1,.85);
    \coordinate (e) at (1.1,-.85);
    \coordinate (f) at (0,-1.7);
    \coordinate (w1) at (-.98,1.8);
    \coordinate (w2) at (.98,1.8);
    \coordinate (w3) at (1.9,0);
    \coordinate (w4) at (.98,-1.8);
    \coordinate (w5) at (-.98,-1.8);
    \foreach \i/\tailvertex/\headvertex/\edgecolor in {
      1/b/c/blue!75!black,2/c/d/orange!75!black,3/d/e/teal!75!black,
      4/e/f/violet!75!black,5/f/a/red!75!black} {
      \draw[draw=\edgecolor,fill=\edgecolor,fill opacity=.1,line width=.7pt]
        (\tailvertex)--(w\i)--(\headvertex)--cycle;
      \draw[draw=\edgecolor,line width=.7pt,traversal] (\tailvertex)--(\headvertex);
    }
    \draw[line width=1.1pt,traversal] (a)--(b);
    \foreach \v/\position in {a/left,b/left,c/above,d/right,e/right,f/below}
      \node[vertex,label=\position:$\v$] at (\v) {};
    \foreach \i/\position in {1/above left,2/above right,3/right,4/below right,5/below left}
      \node[vertex,label=\position:$w_{\i}$] at (w\i) {};
    \node at (0,-2.55) {root pair $a\to b$};
  \end{scope}
  \draw[-{Stealth[length=5pt]},line width=.6pt] (5.1,0)--(6.05,0);
  \begin{scope}[xshift=8.3cm,xscale=1.35,yscale=.9]
    \node at (0,2.6) {Corresponding perfect matching};
    \coordinate (ai) at (-1.1,-.85);
    \coordinate (bo) at (-1.1,.85);
    \coordinate (ci) at (-.38,1.48);
    \coordinate (co) at (.38,1.48);
    \coordinate (di) at (.86,1.1);
    \coordinate (do) at (1.1,.42);
    \coordinate (ei) at (1.1,-.42);
    \coordinate (eo) at (.86,-1.1);
    \coordinate (fi) at (.38,-1.48);
    \coordinate (fo) at (-.38,-1.48);
    \coordinate (w1) at (-.98,1.8);
    \coordinate (w2) at (.98,1.8);
    \coordinate (w3) at (1.9,0);
    \coordinate (w4) at (.98,-1.8);
    \coordinate (w5) at (-.98,-1.8);
    \foreach \i/\tailvertex/\headvertex/\edgecolor in {
      1/b/c/blue!75!black,2/c/d/orange!75!black,3/d/e/teal!75!black,
      4/e/f/violet!75!black,5/f/a/red!75!black}
      \draw[draw=\edgecolor,fill=\edgecolor,fill opacity=.1,line width=.7pt]
        (\tailvertex o)--(w\i)--(\headvertex i)--cycle;
    \foreach \v/\position in {a/left,c/above,d/above right,e/left,f/below}
      \node[vertex,label=\position:$\v^{\rm in}$] at (\v i) {};
    \foreach \v/\position in {b/left,c/above,d/left,e/right,f/below}
      \node[vertex,label=\position:$\v^{\rm out}$] at (\v o) {};
    \foreach \i/\position in {1/above left,2/above right,3/right,4/below right,5/below left}
      \node[vertex,label=\position:$w_{\i}$] at (w\i) {};
    \node at (0,-2.55) {omit $a^{\rm out},b^{\rm in}$};
  \end{scope}
\end{tikzpicture}
\caption{A mixed cycle with five triples and the corresponding perfect matching.}
\label{fig:cycle-matching}
\end{figure}
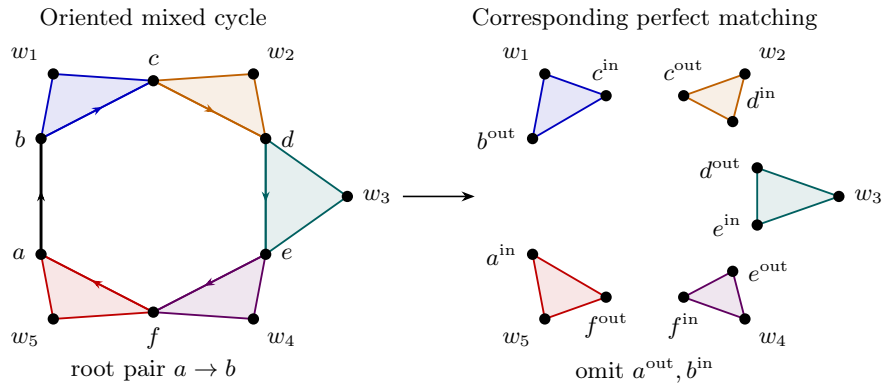

The next theorem shows that an entropy lower bound close to $\frac{n-s}{2}\log \frac{6m}{n} -(n-s)$ forces most directed edge marginals to be nearly equal.

\begin{theorem}\label{thm:entropy-stability}
With $\mathcal G,\cM$, and $\mathbf C$ as above, fix a constant $C$ and nonnegative parameters $\xi,\delta$ tending to $0$. Suppose that the following conditions hold.
\begin{enumerate}[label=(\textit{\roman*})]
\item We have $|\cM|=o(n)$, $\Delta_1(\mathcal{G})\le3Cm/n$, and $\Delta_2(\mathcal{G})=o(m/n)$.
\item For every $A$ with $V(\cM)\subseteq A\subseteq V(\mathcal G)$ and $|A|=(n+s)/2$, we have
\[
 \left|\bigl|\mathcal G[A,\overline A]\bigr|-\frac{3m}{8}\right|\le3\delta m.
\]
\item The entropy $\HH[\mathbf C]$ satisfies $\HH[\mathbf C]\ge\frac{n-s}{2}\log \frac{6m}{n} -(n-s)-\xi n$.
\end{enumerate}
Then
\[
 \sum_{(w;u\to v)\in\vec{\mathcal{G}}-V(\cM)}\left|\pi(w;u\to v)-\frac{n}{12m}\right|
 =O_C\!\left(n\sqrt{\xi+\delta+\frac{s}{n}+\frac1{\log\left(\frac{3m}{n\Delta_2(\mathcal{G})}\right)}}\right).
\]
\end{theorem}
\begin{proof}
Fixing the intersection set and the directions of the prescribed pairs converts a mixed Hamilton cycle into a perfect matching. We control the conditional matching marginals and the probabilities that two specified vertices of a triple are intersection vertices and the third is private, then combine these estimates to recover the directed edge marginals.

To construct the auxiliary hypergraph, we fix a set of $\cM$-unsaturated intersection vertices and directions of the pairs in $\cM$.  Let $J_0\subseteq V(\mathcal{G})\setminus V(\cM)$, and let $\theta$ assign directions to the $s-1$ pairs in $\cM$ other than the root pair. We call $(J_0,\theta)$ \emph{feasible} if there exists a $(\mathcal{G},\cM)$-respecting mixed Hamilton cycle whose set of $\cM$-unsaturated intersection vertices is $J_0$ and whose pairs other than the root pair have directions $\theta$ in the root orientation.

For each feasible pair $(J_0,\theta)$, define an auxiliary $3$-graph $\widetilde{\mathcal G}$ as follows (see \cref{fig:cycle-matching} for a concrete example).
\begin{itemize}
\item Replace each $\cM$-unsaturated intersection vertex $u\in J_0$ by an incoming copy $u^{\rm in}$ and an outgoing copy $u^{\rm out}$. For each pair in $\cM$ directed from $a$ to $b$, keep only $a^{\rm in}$ and $b^{\rm out}$. The private vertices are those in $V(\mathcal{G})\setminus(J_0\cup V(\cM))$; keep them unchanged.
\item The edges of $\widetilde{\mathcal G}$ are exactly the sets $\{u^{\rm out},v^{\rm in},w\}$ for which $\{u,v,w\}\in\mathcal{G}$, $w$ is private, and both copies $u^{\rm out},v^{\rm in}$ are present.
\end{itemize}
A pair in $\cM$ directed from $a$ to $b$ leaves $a$ and enters $b$. Thus the triple at $a$ must enter $a$, while the triple at $b$ must leave $b$, explaining why we keep $a^{\rm in}$ and $b^{\rm out}$. The auxiliary hypergraph is fixed once $J_0$ and $\theta$ are given. We record the properties needed to apply \cref{prop:matching-stability}.
\begin{claim}\label{claim:auxiliary-properties}
For every feasible pair $(J_0,\theta)$, the following statements hold.
\begin{enumerate}[label=(\textit{\roman*})]
\item The incoming copies, outgoing copies, and private vertices form three vertex classes, each of size $(n-s)/2$, and every edge of $\widetilde{\mathcal G}$ meets each class once.
\item We have $\Delta_1(\widetilde{\mathcal G})\le2\Delta_1(\mathcal G)$ and $\Delta_2(\widetilde{\mathcal G})\le\Delta_2(\mathcal G)$.
\item The average degree $d(\widetilde{\mathcal G})$ satisfies $d(\widetilde{\mathcal G})=\frac{3m}{2n}\bigl(1+O_C(\delta+\frac{s}{n})\bigr)$.
\end{enumerate}
\end{claim}
\begin{proof}[Proof of \cref{claim:auxiliary-properties}]
By \cref{fact:mixed-cycle-parameters}, any cycle realising $(J_0,\theta)$ has $(n-s)/2$ triples, each contributing one private vertex. The remaining $(n+s)/2$ vertices are intersection vertices, of which exactly $2s$ are $\cM$-saturated. Hence $|J_0|=(n+s)/2-2s=(n-3s)/2$. Each of the incoming and outgoing vertex classes has $|J_0|+s=(n-s)/2$ vertices, proving \cref{claim:auxiliary-properties}(i).

For each private vertex $w$, every triple of $\mathcal G$ containing $w$ contributes at most two auxiliary edges incident to $w$. For each copy of an intersection vertex, every original triple contributes at most one incident auxiliary edge. Any two vertices contained in an auxiliary edge correspond to two distinct vertices of $\mathcal G$, and each original triple contributes at most one auxiliary edge containing both auxiliary vertices. This proves \cref{claim:auxiliary-properties}(ii).

For \cref{claim:auxiliary-properties}(iii), take $A=J_0\cup V(\cM)$. Each triple meeting $A$ in exactly two vertices produces two auxiliary edges unless it meets $V(\cM)$. Hence
\[
 \left||\widetilde{\mathcal G}|-2\bigl|\mathcal G[A,\overline A]\bigr|\right|
 \le2\sum_{v\in V(\cM)}d_{\mathcal G}(v)
 \le4s\Delta_1(\mathcal G)=O_C\left(\frac{sm}{n}\right).
\]
Assumption (ii) of \cref{thm:entropy-stability} now gives $|\widetilde{\mathcal G}|=3m/4+O_C(\delta m+sm/n)$. Hence $d(\widetilde{\mathcal G})=2|\widetilde{\mathcal G}|/(n-s)=\frac{3m}{2n}(1+O_C(\delta+s/n))$, proving \cref{claim:auxiliary-properties}(iii).
\end{proof}

The random cycle $\mathbf C$ determines the random set $\mathbf J$ of $\cM$-unsaturated intersection vertices and the random vector $\boldsymbol{\Theta}$ recording the directions of the non-root pairs in the root orientation. For each triple traversed by $\mathbf C$ from $u$ to $v$ with private vertex $w$, include $\{u^{\rm out},v^{\rm in},w\}$ in $\mathbf F$. Every auxiliary vertex belongs to exactly one selected edge, so $\mathbf F$ is a perfect matching (see \cref{fig:cycle-matching} for a concrete example). From this induced matching, replacing each vertex copy by its original vertex and adding the pairs of $\cM$ recovers $\mathbf C$.

For every $\PP$-admissible value of $(\mathbf J,\boldsymbol{\Theta})$, the conditional law of $\mathbf C$ induces a distribution on the perfect matchings of the corresponding fixed auxiliary hypergraph. Thus $\mathbf C$ and $(\mathbf J,\boldsymbol{\Theta},\mathbf F)$ determine each other. The chain rule in \cref{lem:entropy-chain-rule} gives
\[
 \HH[\mathbf C]=\HH[\mathbf J,\boldsymbol{\Theta}]+\HH[\mathbf F\mid \mathbf J,\boldsymbol{\Theta}].
\]
In the conditional sums below, $\widetilde{\mathcal G}$ is the auxiliary hypergraph determined by the realised value of $(\mathbf J,\boldsymbol{\Theta})$; the outer expectation averages over these values. Put $\varepsilon\coloneqq\xi+\delta+s/n+1/\log(3m/(n\Delta_2(\mathcal{G})))$. The next claim shows that $\HH[\mathbf J,\boldsymbol{\Theta}]$ is within $O_C(n\varepsilon)$ of the upper bound $\log\binom{n-2s}{(n-3s)/2}+(s-1)\log2$, and bounds the deviation of the conditional matching marginals.
\begin{claim}\label{claim:conditional-entropy-balance}
The following statements hold.
\begin{enumerate}[label=(\textit{\roman*})]
\item The joint entropy $\HH[\mathbf J,\boldsymbol{\Theta}]$ satisfies $\log\binom{n-2s}{(n-3s)/2}+(s-1)\log2-\HH[\mathbf J,\boldsymbol{\Theta}]=O_C(n\varepsilon)$.
\item The conditional edge marginals $\PP[f\in\mathbf F\mid \mathbf J,\boldsymbol{\Theta}]$ satisfy
\[
 \EE\left[\sum_{f\in\widetilde{\mathcal G}}\left|\PP[f\in\mathbf F\mid \mathbf J,\boldsymbol{\Theta}]-\frac{2n}{3m}\right|\right]=O_C(n\sqrt{\varepsilon}).
\]
\end{enumerate}
\end{claim}
\begin{proof}[Proof of \cref{claim:conditional-entropy-balance}]
We compare the assumed entropy lower bound for $\mathbf C$ with the conditional matching entropy bound in \cref{prop:matching-stability}. The chain rule will bound the sum of two nonnegative terms, namely the difference in \cref{claim:conditional-entropy-balance}(i) and the expected relative entropy of the conditional matching marginals. Controlling the first term proves (i); controlling the second gives (ii) through Pinsker's inequality.

By \cref{claim:auxiliary-properties}(ii)--(iii), the auxiliary hypergraph has maximum degree $O_C(d(\widetilde{\mathcal G}))$, and the ratio of its maximum codegree to $d(\widetilde{\mathcal G})$ is $O_C(n\Delta_2(\mathcal G)/m)=o(1)$. Thus \cref{prop:matching-stability} applies to every conditional matching law. The error in \cref{prop:matching-stability} involving $1/\log\bigl(d(\widetilde{\mathcal G})/\Delta_2(\widetilde{\mathcal G})\bigr)$ is $O_C(n\varepsilon)$. Since $1\le\Delta_2(\mathcal G)$ and $3m/(n\Delta_2(\mathcal G))<n^2/2$ for large $n$, the $O(\log n)$ term is also absorbed by this error bound.

Under each conditional law, $\mathbf F$ contains exactly $(n-s)/2$ edges, so dividing its edge marginals by $(n-s)/2$ gives a probability distribution on all edges of $\widetilde{\mathcal G}$. Apply \cref{prop:matching-stability} conditional on each $\PP$-admissible value of $(\mathbf J,\boldsymbol{\Theta})$, and then average with respect to the distribution of $(\mathbf J,\boldsymbol{\Theta})$ induced by $\mathbf C$. Substituting the number of auxiliary vertices $3(n-s)/2$ and the average-degree estimate in \cref{claim:auxiliary-properties}(iii) gives
\[
\begin{aligned}
 \HH[\mathbf F\mid \mathbf J,\boldsymbol{\Theta}]
 &{}+\frac{n-s}{4}\EE\left[\KL\left[\left(\frac{2\PP[f\in\mathbf F\mid \mathbf J,\boldsymbol{\Theta}]}{n-s}\right)_{f\in\widetilde{\mathcal G}}\,\middle\Vert\,\Unif(\widetilde{\mathcal G})\right]\right]
 \\
 &\le\frac{n-s}{2}\log\frac{3m}{2n}-(n-s)+O_C(n\varepsilon).
\end{aligned}
\]
By the chain rule, the first term on the left is $\HH[\mathbf C]-\HH[\mathbf J,\boldsymbol{\Theta}]$. Substitute the lower bound for $\HH[\mathbf C]$ from assumption (iii) of \cref{thm:entropy-stability}. The linear terms cancel, while the difference of the logarithmic main terms is $\frac{n-s}{2}(\log\frac{6m}{n}-\log\frac{3m}{2n})=(n-s)\log2$. The number of possible values of $(\mathbf J,\boldsymbol{\Theta})$ is at most $\binom{n-2s}{(n-3s)/2}2^{s-1}$, whose logarithm is at most $(n-s)\log2$. Replacing $(n-s)\log2$ on the left by this smaller quantity gives
\[
\begin{aligned}
 & \log\binom{n-2s}{(n-3s)/2}+(s-1)\log2-\HH[\mathbf J,\boldsymbol{\Theta}]\\
 &{}+\frac{n-s}{4}\EE\biggl[\KL\biggl[\biggl(\frac{2\PP[f\in\mathbf F\mid \mathbf J,\boldsymbol{\Theta}]}{n-s}\biggr)_{f\in\widetilde{\mathcal G}}\big\Vert\Unif(\widetilde{\mathcal G})\biggr]\biggr]=O_C\!(n\varepsilon).
\end{aligned}
\]
The difference in \cref{claim:conditional-entropy-balance}(i) and the relative-entropy term on the left are both nonnegative, proving \cref{claim:conditional-entropy-balance}(i) and bounding the expected relative entropy by $O_C(\varepsilon)$.

For \cref{claim:conditional-entropy-balance}(ii), fix a $\PP$-admissible value of $(\mathbf J,\boldsymbol{\Theta})$ and apply \cref{lem:pinsker} to the normalised conditional marginals and the uniform distribution on $\widetilde{\mathcal G}$. Multiplying by $(n-s)/2$ converts the normalised marginals into matching probabilities, while the uniform probability becomes $(n-s)/(2|\widetilde{\mathcal G}|)=1/d(\widetilde{\mathcal G})$. Squaring the resulting inequality gives
\[
\begin{aligned}
 \left(\sum_{f\in\widetilde{\mathcal G}}\left|\PP[f\in\mathbf F\mid \mathbf J,\boldsymbol{\Theta}]-\frac1{d(\widetilde{\mathcal G})}\right|\right)^2
\le\frac{(n-s)^2}{2}\KL\left[\left(\frac{2\PP[f\in\mathbf F\mid \mathbf J,\boldsymbol{\Theta}]}{n-s}\right)_{f\in\widetilde{\mathcal G}}\,\middle\Vert\,\Unif(\widetilde{\mathcal G})\right].
\end{aligned}
\]
Taking expectations over $(\mathbf J,\boldsymbol{\Theta})$ bounds the mean square on the left by $O_C(n^2\varepsilon)$. Cauchy--Schwarz therefore gives $\EE[\sum_{f\in\widetilde{\mathcal G}}|\PP[f\in\mathbf F\mid \mathbf J,\boldsymbol{\Theta}]-1/d(\widetilde{\mathcal G})|]=O_C(n\sqrt{\varepsilon})$. To replace $1/d(\widetilde{\mathcal G})$ by $2n/(3m)$, note that $|\widetilde{\mathcal G}|=(n-s)d(\widetilde{\mathcal G})/2$. Thus \cref{claim:auxiliary-properties}(iii) gives
\[
\begin{aligned}
 \sum_{f\in\widetilde{\mathcal G}}
 \left|\frac1{d(\widetilde{\mathcal G})}-\frac{2n}{3m}\right|
 =\frac{n-s}{2}
 \left|1-\frac{2n d(\widetilde{\mathcal G})}{3m}\right|
 =O_C(n\delta+s).
\end{aligned}
\]
This bound holds for every $\PP$-admissible value of $(\mathbf J,\boldsymbol{\Theta})$. Taking expectations and applying the triangle inequality adds only $O_C(n\delta+s)$ to the preceding $\ell^1$ bound. This error is absorbed by $O_C(n\sqrt{\varepsilon})$, proving \cref{claim:conditional-entropy-balance}(ii).
\end{proof}

To pass from the conditional matching marginals to directed edge marginals, we also estimate the probability that $u,v\in \mathbf J$ and $w\notin \mathbf J$ for an oriented triple $(w;u\to v)\in\vec{\mathcal G}-V(\cM)$.
\begin{claim}\label{claim:junction-role-balance}
The probabilities $\PP[u,v\in \mathbf J\text{ and }w\notin \mathbf J]$ satisfy
\[
 \sum_{(w;u\to v)\in\vec{\mathcal{G}}-V(\cM)}\left|\PP[u,v\in \mathbf J\text{ and }w\notin \mathbf J]-\frac18\right|
 =O_C(m\sqrt{\varepsilon}).
\]
\end{claim}
\begin{proof}[Proof of \cref{claim:junction-role-balance}]
We first show that the indicator vector of $\mathbf J$ has entropy close to $(n-2s)\log2$. Shearer's inequality then shows that, on average over triples in $\mathcal G-V(\cM)$, the three indicators of membership in $\mathbf J$ have joint entropy close to $3\log2$. Pinsker's inequality bounds the probabilities of the specified indicator patterns.

The indicator vector of $\mathbf J$ has $n-2s$ coordinates. Since $\HH[\mathbf J,\boldsymbol{\Theta}]\le\HH[\mathbf J]+(s-1)\log2$, \cref{claim:conditional-entropy-balance}(i) gives
\[
\begin{aligned}
 0\le\log\binom{n-2s}{(n-3s)/2}-\HH[\mathbf J]
 \le\log\binom{n-2s}{(n-3s)/2}+(s-1)\log2-\HH[\mathbf J,\boldsymbol{\Theta}]
 =O_C(n\varepsilon).
\end{aligned}
\]
As in the proof of \cref{claim:conditional-entropy-balance}, $\varepsilon\ge1/(2\log n)$. Thus, for sufficiently large $n$, we have $s^2/n\le s\le n\varepsilon$ and $\log n=O(n\varepsilon)$. Combining these estimates with Stirling's formula gives
\[
 (n-2s)\log2-\HH[\mathbf J]
 =O\left(\frac{s^2}{n}+\log n\right)+O_C(n\varepsilon)
 =O_C(n\varepsilon).
\]

For every $e\in\mathcal{G}-V(\cM)$, let $\nu_e$ be the distribution on $\{0,1\}^e$ of the random vector $(\mathbf1_{\{v\in\mathbf J\}})_{v\in e}$, whose coordinate at $v$ indicates whether $v$ is an intersection vertex of $\mathbf C$. To apply Shearer's inequality with every coordinate covered equally often, for each $v\in V(\mathcal G)\setminus V(\cM)$ add $\Delta_1(\mathcal G)-d_{\mathcal G-V(\cM)}(v)$ copies of the singleton set $\{v\}$ to the triple family $\mathcal G-V(\cM)$. Every coordinate then occurs exactly $\Delta_1(\mathcal G)$ times. Since $\KL[\nu_e\Vert\Unif(\{0,1\}^e)]=3\log2-\HH[(\mathbf1_{\{v\in \mathbf J\}})_{v\in e}]$, Shearer's inequality in \cref{lem:shearer} gives
\begin{multline*}
  \sum_{e\in\mathcal G-V(\cM)}\KL[\nu_e\Vert\Unif(\{0,1\}^e)]
 +\sum_{v\in V(\mathcal G)\setminus V(\cM)}
 \bigl(\Delta_1(\mathcal G)-d_{\mathcal G-V(\cM)}(v)\bigr)
 \bigl(\log2-\HH[\mathbf1_{\{v\in \mathbf J\}}]\bigr)\\
 \le\Delta_1(\mathcal G)\bigl((n-2s)\log2-\HH[\mathbf J]\bigr)
 =O_C(m\varepsilon).
\end{multline*}
The singleton terms are nonnegative, so $\sum_{e\in\mathcal G-V(\cM)}\KL[\nu_e\Vert\Unif(\{0,1\}^e)]=O_C(m\varepsilon)$. For an oriented triple $(w;u\to v)$ with $e=\{u,v,w\}$, the event $u,v\in\mathbf J$ and $w\notin\mathbf J$ specifies one indicator pattern. Since the differences between two probability distributions sum to zero, the absolute difference on any single pattern is at most half their $\ell^1$ distance. Pinsker's inequality therefore gives
\[
 \left|\PP[u,v\in\mathbf J\text{ and }w\notin\mathbf J]-\frac18\right|
 \le\frac12\sum_{\omega\in\{0,1\}^e}\left|\nu_e(\omega)-\frac18\right|
 \le\sqrt{\frac12\KL[\nu_e\Vert\Unif(\{0,1\}^e)]}.
\]
Each triple has six orientations, so Cauchy--Schwarz gives
\[
\begin{aligned}
 & \sum_{(w;u\to v)\in\vec{\mathcal G}-V(\cM)}
 \left|\PP[u,v\in \mathbf J\text{ and }w\notin \mathbf J]-\frac18\right|\\
 &\le6\sum_{e\in\mathcal G-V(\cM)}
 \sqrt{\frac12\KL[\nu_e\Vert\Unif(\{0,1\}^e)]}\\
 &\le\sqrt{18m\sum_{e\in\mathcal G-V(\cM)}\KL[\nu_e\Vert\Unif(\{0,1\}^e)]}
 =O_C(m\sqrt{\varepsilon}).
\end{aligned}
\]
This proves the claim.
\end{proof}

We now transfer the estimate for auxiliary edges to oriented triples in the original hypergraph. Fix an oriented triple $(w;u\to v)\in\vec{\mathcal{G}}-V(\cM)$. If $u,v\in \mathbf J$ and $w\notin \mathbf J$, this oriented triple corresponds to the auxiliary edge $\{u^{\rm out},v^{\rm in},w\}$, regardless of the directions of the pairs in $\cM$. The law of total probability therefore gives
\[
\begin{aligned}
 & \pi(w;u\to v)-\frac{2n}{3m}\PP[u,v\in \mathbf J\text{ and }w\notin \mathbf J]
\\
 &=\EE\left[\mathbf{1}_{\{u,v\in \mathbf J,\ w\notin \mathbf J\}}
 \left(\PP[\{u^{\rm out},v^{\rm in},w\}\in\mathbf F\mid \mathbf J,\boldsymbol{\Theta}]-\frac{2n}{3m}\right)\right],
\end{aligned}
\]
where the conditional matching probability is set to zero unless $u,v\in \mathbf J$ and $w\notin \mathbf J$. For each fixed value of $(\mathbf J,\boldsymbol{\Theta})$, the oriented triples $(w;u\to v)\in\vec{\mathcal G}-V(\cM)$ with $u,v\in\mathbf J$ and $w\notin\mathbf J$ correspond to distinct auxiliary edges. Hence
\begin{align*}
   \sum_{\substack{(w;u\to v)\in\vec{\mathcal G}-V(\cM)\\u,v\in\mathbf J\,,\ w\notin\mathbf J}}
 \left|\PP[\{u^{\rm out},v^{\rm in},w\}\in\mathbf F\mid\mathbf J,\boldsymbol{\Theta}]-\frac{2n}{3m}\right|
 \le\sum_{f\in\widetilde{\mathcal G}}
 \left|\PP[f\in\mathbf F\mid\mathbf J,\boldsymbol{\Theta}]-\frac{2n}{3m}\right|.
\end{align*}
Taking absolute values in the preceding identity from the law of total probability, applying the triangle inequality for expectations, and using \cref{claim:conditional-entropy-balance}(ii) to bound the expected sum on the right gives
\[
 \sum_{(w;u\to v)\in\vec{\mathcal{G}}-V(\cM)}\left|\pi(w;u\to v)-\frac{2n}{3m}\PP[u,v\in \mathbf J\text{ and }w\notin \mathbf J]\right|
 =O_C(n\sqrt{\varepsilon}).
\]
Since $\frac{2n}{3m}\cdot\frac18=\frac{n}{12m}$, the triangle inequality gives
\begin{multline*}
 \sum_{(w;u\to v)\in\vec{\mathcal G}-V(\cM)}
 \left|\pi(w;u\to v)-\frac{n}{12m}\right|\\
 \le\sum_{(w;u\to v)\in\vec{\mathcal G}-V(\cM)}
 \left|\pi(w;u\to v)-\frac{2n}{3m}\PP[u,v\in\mathbf J\text{ and }w\notin\mathbf J]\right|\\
 {}+\frac{2n}{3m}\sum_{(w;u\to v)\in\vec{\mathcal G}-V(\cM)}
 \left|\PP[u,v\in\mathbf J\text{ and }w\notin\mathbf J]-\frac18\right|.
\end{multline*}
The preceding estimate bounds the first sum by $O_C(n\sqrt{\varepsilon})$. By \cref{claim:junction-role-balance}, the second contribution is $\frac{2n}{3m}O_C(m\sqrt{\varepsilon})=O_C(n\sqrt{\varepsilon})$. This proves the theorem.
\end{proof}

We next use the auxiliary matching construction to bound the expected number of triples shared by two independently chosen cycles.

\begin{lemma}\label{lem:overlap}
Let $\mathcal G$ be a $3$-graph on $n$ vertices, and let $\cM$ be a matching of $s\ge1$ pairs on the same vertex set. Let $\mathcal C$ be a nonempty family of $(\mathcal G,\cM)$-respecting mixed Hamilton cycles, and choose $C_1,C_2$ independently and uniformly from $\mathcal C$.
For every fixed $A\ge1$ and $B\ge0$, there exists a constant $C>0$, depending only on $A,B$, such that the following holds. Suppose that $s=o(n)$, $d(\mathcal G)\to\infty$, $\Delta_1(\mathcal G)\le A d(\mathcal G)$, and $\log|\mathcal C|\ge\tfrac{n-s}{2}\log d(\mathcal G)-Bn$. Then, for sufficiently large $n$, we have
\[
\EE[|C_1\cap C_2|-s]\le C\frac{n}{\log d(\mathcal G)}.
\]
\end{lemma}
\begin{proof}
By independence of $C_1$ and $C_2$, we have
\[
 \EE[|C_1\cap C_2|-s]
 =\sum_{e\in\mathcal G}\PP[e\in C_1]\PP[e\in C_2]
 =\sum_{e\in\mathcal G}\PP[e\in C_1]^2.
\]
We first bound the squared cycle marginals in terms of the largest conditional matching probabilities at the auxiliary vertices. We then use the entropy lower bound to control the expected sum of these probabilities. Fix a pair in $\cM$ and a direction on it, and orient every cycle accordingly. Choose $\mathbf C$ uniformly from $\mathcal C$. Let $\mathbf J$ be its set of $\cM$-unsaturated intersection vertices and let $\boldsymbol{\Theta}$ record the directions of the remaining pairs in $\cM$. Construct the corresponding perfect matching $\mathbf F$ in the auxiliary hypergraph $\widetilde{\mathcal G}$, as in the proof of \cref{thm:entropy-stability}.

For each $\PP$-admissible value $(J_0,\theta)$ of $(\mathbf J,\boldsymbol{\Theta})$, the auxiliary hypergraph $\widetilde{\mathcal G}$ is fixed. For each auxiliary vertex $u$, put $q_u^{\max}\coloneqq\max_{f\ni u}\PP[f\in\mathbf F\mid \mathbf J=J_0,\boldsymbol{\Theta}=\theta]$.

Conditional on any $\PP$-admissible value of $(\mathbf J,\boldsymbol{\Theta})$, the intersection vertices of $\mathbf C$ are precisely $\mathbf J\cup V(\cM)$. Each triple having positive conditional probability of belonging to $\mathbf C$ has exactly one vertex outside this set, which is its private vertex. For a triple $\{u,v,w\}$ with $w$ private, the possible auxiliary edges are $\{u^{\rm out},v^{\rm in},w\}$ and $\{v^{\rm out},u^{\rm in},w\}$, corresponding to the two orders of its intersection vertices. They share $w$, so at most one belongs to $\mathbf F$. Thus
\begin{align*}
 \PP[\{u,v,w\}\in\mathbf C\mid\mathbf J,\boldsymbol{\Theta}]
 =\PP[\{u^{\rm out},v^{\rm in},w\}\in\mathbf F\mid\mathbf J,\boldsymbol{\Theta}]
 +\PP[\{v^{\rm out},u^{\rm in},w\}\in\mathbf F\mid\mathbf J,\boldsymbol{\Theta}],
\end{align*}
with probability zero for an auxiliary edge absent from $\widetilde{\mathcal G}$. Applying $(a+b)^2\le2(a^2+b^2)$ to this identity gives the first inequality below. At each auxiliary vertex $u$, the incident matching probabilities sum to $1$, so their squared sum is at most $q_u^{\max}$. Since each auxiliary edge contains three vertices, we obtain
\[
 \sum_{e\in\mathcal G}\PP[e\in\mathbf C\mid \mathbf J,\boldsymbol{\Theta}]^2\le2\sum_{f\in\widetilde{\mathcal G}}\PP[f\in\mathbf F\mid \mathbf J,\boldsymbol{\Theta}]^2
 \le\frac{2}{3}\sum_{u\in V(\widetilde{\mathcal G})}q_u^{\max}.
\]
By the law of total probability, Jensen's inequality, and the preceding conditional bound, we have
\[
\begin{aligned}
 \sum_{e\in\mathcal G}\PP[e\in\mathbf C]^2
 =\sum_{e\in\mathcal G}\left(\EE\left[\PP[e\in\mathbf C\mid\mathbf J,\boldsymbol{\Theta}]\right]\right)^2
 \le\EE\left[\sum_{e\in\mathcal G}\PP[e\in\mathbf C\mid\mathbf J,\boldsymbol{\Theta}]^2\right]
 \le\frac23\EE\left[\sum_{u\in V(\widetilde{\mathcal G})}q_u^{\max}\right].
\end{aligned}
\]

It remains to bound the expectation on the right. The random set $\mathbf J$ is a subset of the $n-2s$ vertices outside $V(\cM)$, and $\boldsymbol{\Theta}$ records $s-1$ pair directions. Hence $\HH[\mathbf J,\boldsymbol{\Theta}]\le(n-2s)\log2+(s-1)\log2=O(n)$. Since $\HH[\mathbf C]=\log|\mathcal C|$, the chain rule gives
\[
 \HH[\mathbf F\mid \mathbf J,\boldsymbol{\Theta}]
 =\log|\mathcal C|-\HH[\mathbf J,\boldsymbol{\Theta}]
 \ge\frac{n-s}{2}\log d(\mathcal G)-O_B(n).
\]
Take $D=\lceil2Ad(\mathcal G)\rceil$. The degree estimate in the proof of \cref{claim:auxiliary-properties}(ii) follows directly from the auxiliary construction and gives $\Delta_1(\widetilde{\mathcal G})\le2\Delta_1(\mathcal G)\le2Ad(\mathcal G)\le D$ for every auxiliary hypergraph. Fix a $\PP$-admissible value $(J_0,\theta)$ of $(\mathbf J,\boldsymbol{\Theta})$, so that the auxiliary hypergraph $\widetilde{\mathcal G}$ is fixed. Under this conditional law, apply Shearer's inequality to the edge-indicator vector of $\mathbf F$, using the sets of edges incident to each auxiliary vertex. Each edge coordinate occurs in three such sets, and the indicators at $u$ determine and are determined by $\mathbf F(u)$. Thus
\[
 3\HH[\mathbf F\mid \mathbf J=J_0,\boldsymbol{\Theta}=\theta]
 \le\sum_{u\in V(\widetilde{\mathcal G})}\HH[\mathbf F(u)\mid \mathbf J=J_0,\boldsymbol{\Theta}=\theta].
\]

At each auxiliary vertex $u$, fix one edge attaining $q_u^{\max}$. Apply the chain rule by first recording whether $\mathbf F(u)$ equals the chosen edge. Under the fixed conditional law, this indicator has entropy at most $\log2$. If that edge is chosen, then $\mathbf F(u)$ is determined; otherwise, it has at most $D$ possible values. Thus $\HH[\mathbf F(u)\mid \mathbf J=J_0,\boldsymbol{\Theta}=\theta]\le\log2+(1-q_u^{\max})\log D$. Since $|V(\widetilde{\mathcal G})|=3(n-s)/2$, summing over the auxiliary vertices and applying the preceding form of Shearer's inequality yields
\[
\begin{aligned}
 (\log D)\sum_{u\in V(\widetilde{\mathcal G})}q_u^{\max}
 &\le\sum_{u\in V(\widetilde{\mathcal G})}
 \bigl(\log D+\log2-\HH[\mathbf F(u)\mid \mathbf J=J_0,\boldsymbol{\Theta}=\theta]\bigr)\\
 &\le\frac{3(n-s)}2(\log D+\log2)-3\HH[\mathbf F\mid \mathbf J=J_0,\boldsymbol{\Theta}=\theta].
\end{aligned}
\]
Now average over the $\PP$-admissible values $(J_0,\theta)$ with weights $\PP[\mathbf J=J_0,\boldsymbol{\Theta}=\theta]$, taking $q_u^{\max}$ from the corresponding conditional law. By the definition of conditional entropy, the averaged entropy term is $\HH[\mathbf F\mid \mathbf J,\boldsymbol{\Theta}]$. The lower bound for this conditional entropy above and $\log D=\log d(\mathcal G)+O_A(1)$ therefore give
\[
\begin{aligned}
 \EE\left[\sum_{u\in V(\widetilde{\mathcal G})}q_u^{\max}\right]
 &\le\frac{\frac{3(n-s)}2(\log D+\log2)-3\HH[\mathbf F\mid \mathbf J,\boldsymbol{\Theta}]}{\log D}\\
 &\le\frac{\frac{3(n-s)}2\bigl(\log D-\log d(\mathcal G)+\log2\bigr)+O_B(n)}{\log D}
 =O_{A,B}\left(\frac{n}{\log d(\mathcal G)}\right).
\end{aligned}
\]

Substitute this bound for $\EE[\sum_u q_u^{\max}]$ into $\sum_{e\in\mathcal G}\PP[e\in\mathbf C]^2\le\frac23\EE[\sum_u q_u^{\max}]$, and use the identity for $\EE[|C_1\cap C_2|-s]$ at the start of the proof. This gives the required bound.
\end{proof}

\section{Random hypergraphs with degree lower bounds}
\label{sec:degree-constrained}
In this section we prove a few facts about hypergraphs chosen
uniformly from $\Omega(n,m,\boldsymbol\ell)$, the family with
prescribed degree lower bounds from
\cref{def:degree-constrained-model}. We bound degrees and codegrees,
estimate the probability that a uniform $m$-edge hypergraph meets
these constraints, and control how far degrees can fall below their
lower bounds after deleting a bounded number of vertices.

We first bound the upper tails of vertex degrees in a uniformly chosen member of $\Omega(n,m,\boldsymbol\ell)$.

\begin{lemma}
\label{thm:degree-domination}
Let $(n,m,\boldsymbol\ell)$ be admissible. Choose $\mathcal{G}$ uniformly from $\Omega(n,m,\boldsymbol\ell)$ and fix a vertex $v$. Then
\[
 \max\Biggl\{0,d_{\mathcal{G}}(v)-\max\biggl\{\ell_v,\Bigl\lceil\frac{3m}{n}\Bigr\rceil\biggr\}\Biggr\}
 \preceq\Poi\left(\frac{3m}{n}\right).
\]
Consequently, for every integer $t\ge3m/n>0$ we have
\[
 \PP\Biggl[d_{\mathcal{G}}(v)\ge\max\biggl\{\ell_v,\Bigl\lceil\frac{3m}{n}\Bigr\rceil\biggr\}+t\Biggr]
 \le\exp\left\{-t\log\frac{nt}{3m}+t-\frac{3m}{n}\right\}.
\]
\end{lemma}
\begin{proof}
If $m=0$, the assertion is immediate, so assume $m>0$. Let $A_d$ be the number of hypergraphs in $\Omega(n,m,\boldsymbol\ell)$ with $d_{\mathcal{G}}(v)=d$. We compare $A_d$ with $A_{d-1}$ by the following switching. For $d>\max\{\ell_v,3m/n\}$, select an edge $e=\{v\}\cup S$ and a vertex $w\notin e$. If $S\cup\{w\}$ is absent, replace $e$ by that edge. Only $v$ loses degree, so the resulting hypergraph remains in $\Omega(n,m,\boldsymbol\ell)$ and has degree $d-1$ at $v$.

There are $d(n-3)$ initial choices for the edge $e$ containing $v$ and the vertex $w\in[n]\setminus e$. A forbidden choice has $S\cup\{w\}\in \mathcal{G}$. Each existing edge not containing $v$ accounts for at most three forbidden choices, one for each possible choice of $w$ in that edge. Hence at least $d(n-3)-3(m-d)=nd-3m$ choices are legal. For a reverse switching from a hypergraph counted by $A_{d-1}$, choose one of its $m-d+1$ edges not containing $v$ and one of the three vertices of that edge to replace by $v$. There are at most $3(m-d+1)$ such choices. Double-counting labelled switchings yields $(nd-3m)A_d\le3(m-d+1)A_{d-1}$, and hence $A_d\le\frac{3m}{nd-3m}A_{d-1}$.

To deduce stochastic domination from this recurrence, write $p_q\coloneqq\PP[\max\{0,d_{\mathcal{G}}(v)-\max\{\ell_v,\lceil3m/n\rceil\}\}=q]$. For every integer $q\ge2$, take $d=\max\{\ell_v,\lceil3m/n\rceil\}+q$. Since $q-1>0$, the probabilities $p_q$ and $p_{q-1}$ correspond to degrees exactly $d$ and $d-1$, respectively. The uniform law on $\Omega(n,m,\boldsymbol\ell)$ therefore gives $p_q=A_d/|\Omega(n,m,\boldsymbol\ell)|$ and $p_{q-1}=A_{d-1}/|\Omega(n,m,\boldsymbol\ell)|$. Moreover, $nd-3m=nq+n\max\{\ell_v,\lceil3m/n\rceil\}-3m\ge nq$. Dividing the recurrence for $A_d$ by $|\Omega(n,m,\boldsymbol\ell)|$ now yields $p_q\le\frac{3m}{nd-3m}p_{q-1}\le\frac{3m}{nq}p_{q-1}$. At $q=1$, we have $p_0=\PP[d_{\mathcal G}(v)\le\max\{\ell_v,\lceil3m/n\rceil\}]$, so the recurrence gives
\[
 p_1\le\frac{3m}{n}\PP\left[d_{\mathcal G}(v)=\max\left\{\ell_v,\left\lceil\frac{3m}{n}\right\rceil\right\}\right]\le\frac{3m}{n}p_0.
\]
Thus $p_q/(\e^{-3m/n}(3m/n)^q/q!)$ is nonincreasing in $q$. To compare upper tails, fix an integer $t\ge1$. If $p_t\le \e^{-3m/n}(3m/n)^t/t!$, then $p_q\le \e^{-3m/n}(3m/n)^q/q!$ for every $q\ge t$, so summing gives the required tail bound. Otherwise, the same monotonicity gives $p_q\ge \e^{-3m/n}(3m/n)^q/q!$ for every $0\le q<t$. Since both distributions have total probability $1$, we obtain
\[
 \sum_{q=t}^{\infty}\left(p_q-\exp\left\{-\frac{3m}{n}\right\}\frac{\left(\frac{3m}{n}\right)^q}{q!}\right)
 =-\sum_{q=0}^{t-1}\left(p_q-\exp\left\{-\frac{3m}{n}\right\}\frac{\left(\frac{3m}{n}\right)^q}{q!}\right)\le0.
\]
Thus in either case the upper tail of $(p_q)$ is bounded by the corresponding Poisson upper tail, proving stochastic domination. For $t>0$, the degree event in the final assertion is precisely the event that the nonnegative degree excess above is at least $t$. The tail bound therefore follows from \cref{lem:chernoff}(ii) applied to the dominating Poisson random variable.
\end{proof}

We next bound sums of codegrees. Conditional on any $\PP$-admissible degree sequence, a uniform member of $\Omega(n,m,\boldsymbol\ell)$ is uniform among the simple hypergraphs realising that degree sequence.
\begin{lemma}\label{thm:association}
Fix a degree sequence $\boldsymbol d$ on $[n]$ that has at least one simple $3$-graph realisation, and choose $\mathcal{G}$ uniformly among these realisations. Fix $B\subseteq[n]$ and $y\in[n]\setminus B$. If $\sum_{v\in[n]\setminus B}d_v>9\Delta_1(\mathcal{G})+2\Delta_1(\mathcal{G})^2$, then
\[
 \sum_{v\in B}d_{\mathcal{G}}(y,v)
 \preceq\Poi\left(\frac{2d_y\sum_{v\in B}d_v}{\sum_{v\in[n]\setminus B}d_v-9\Delta_1(\mathcal{G})-2\Delta_1(\mathcal{G})^2}\right).
\]
\end{lemma}
\begin{proof}
For each integer $t\ge0$, let $a_t$ count the simple $3$-graph realisations of $\boldsymbol d$ with $\sum_{v\in B}d_{\mathcal{G}}(y,v)=t$.  For $t\ge1$, we compare $a_t$ with $a_{t-1}$ by a switching that preserves the degree sequence. Choose a pair $(e,z)$ with $e\in\mathcal{G}$, $y,z\in e$, and $z\in B$, and write $e=\{y,z,w\}$. Choose also a disjoint edge $f=\{a,b,c\}$ with distinguished vertex $a\notin B$.  Replace $e$ and $f$ by $\{y,a,w\}$ and $\{z,b,c\}$, respectively, provided both new edges are absent.  The entire degree sequence is unchanged and $\sum_{v\in B}d_{\mathcal{G}}(y,v)$ decreases by exactly $1$.

The number of choices of $(e,z)$ is $|\{(e,z)\colon e\in\mathcal G,\ y,z\in e,\ z\in B\}|=\sum_{z\in B}d_{\mathcal G}(y,z)=t$. For each such choice, we start with all choices of an edge $f\in\mathcal G$ and a distinguished vertex $a\in f$. There are $\sum_{v\in[n]}d_v$ such choices. We exclude the following choices.
\begin{enumerate}[label=(\textit{\roman*})]
\item At most $\sum_{v\in B}d_v$ choices have $a\in B$.
\item At most $3\Delta_1(\mathcal{G})$ edges $f$ meet $e$, and each permits three choices of the distinguished vertex $a$. Thus these edges account for at most $9\Delta_1(\mathcal{G})$ choices.
\item If $\{y,a,w\}$ is present, then for each such $a$ there are at most $d_{\mathcal G}(a)$ choices of $f$. Thus the number of forbidden choices of $(f,a)$ is at most
\[
 \sum_{\substack{a\in[n]\\\{y,a,w\}\in\mathcal G}}d_{\mathcal G}(a)
 \le d_{\mathcal G}(y,w)\Delta_1(\mathcal G)
 \le\Delta_1(\mathcal G)^2.
\]
\item If $\{z,b,c\}$ is present, then for each such pair $\{b,c\}\in L_{\mathcal G}(z)$ there are at most $d_{\mathcal G}(b,c)$ choices of $f$. The remaining vertex of $f$ determines $a$, so the number of forbidden choices is at most
\[
 \sum_{\{b,c\}\in L_{\mathcal G}(z)}d_{\mathcal G}(b,c)
 \le d_{\mathcal G}(z)\Delta_1(\mathcal G)
 \le\Delta_1(\mathcal G)^2.
\]
\end{enumerate}
Thus each selected incidence admits at least $\sum_{v\in[n]\setminus B}d_v-9\Delta_1(\mathcal{G})-2\Delta_1(\mathcal{G})^2$ valid forward switchings.

For a reverse switching, choose an edge containing $y$, distinguish one of its other vertices as $a$, and write the edge as $\{y,a,w\}$. Choose also an incidence $(f',z)$ with $z\in f'\cap B$, and write $f'=\{z,b,c\}$. These choices specify the inverse replacement of $\{y,a,w\},\{z,b,c\}$ by $\{y,z,w\},\{a,b,c\}$. There are therefore at most $2d_y\sum_{v\in B}d_v$ reverse switchings. Double-counting gives, for $t\ge1$,
\[
 t\left(\sum_{v\in[n]\setminus B}d_v-9\Delta_1(\mathcal{G})-2\Delta_1(\mathcal{G})^2\right)a_t
 \le2d_y a_{t-1}\sum_{v\in B}d_v.
\]
If $d_y\sum_{v\in B}d_v=0$, the preceding recurrence gives $a_t=0$ for every $t\ge1$, so the sum of codegrees is identically zero. Otherwise, normalise the numbers $a_t$ to probabilities. The recurrence shows that their ratio to the Poisson probabilities in the statement is nonincreasing in $t$. The upper-tail comparison in the proof of \cref{thm:degree-domination} therefore gives the stochastic domination.
\end{proof}

The following lemma gives a lower bound on the probability that a uniform $m$-edge $3$-graph satisfies the prescribed degree constraints.

\begin{lemma}\label{lem:beta}
There exists an integer $n_0$ such that the following holds whenever $n\ge n_0$. Let $(n,m,\boldsymbol\ell)$ be near-critical. The probability that a uniform $m$-edge $3$-graph on $[n]$ belongs to $\Omega(n,m,\boldsymbol\ell)$ satisfies
\[
 \frac{|\Omega(n,m,\boldsymbol\ell)|}{\binom{\binom{n}{3}}{m}}
 \ge\exp\{-n^{1/10}\}.
\]
\end{lemma}
\begin{proof}
We first estimate the probability of the degree constraints in a binomial hypergraph, then transfer the lower bound to the uniform $m$-edge $3$-graph model. Let $\mathcal{G}$ be the binomial $3$-graph on $[n]$ with edge probability $(1-n^{-1/4})m/\binom{n}{3}$. For every vertex $v$, we have $\EE[d_{\mathcal G}(v)]=(1-n^{-1/4})3m/n=\log n+O(\log\log n)$. Since $\ell_v/\EE[d_{\mathcal G}(v)]=0.01+O(\log\log n/\log n)$ uniformly in $v$, continuity of $\varphi$ gives
\[
 \EE[d_{\mathcal G}(v)]\,\varphi\left(\frac{\ell_v}{\EE[d_{\mathcal G}(v)]}\right)
 =\bigl(\varphi(0.01)+o(1)\bigr)\log n.
\]
Substituting into the lower-tail bound \cref{lem:chernoff}(i) yields $\PP[d_{\mathcal{G}}(v)<\ell_v]\le n^{-\varphi(0.01)+o(1)}$ uniformly in $v$. The events $d_{\mathcal{G}}(v)\ge\ell_v$ are increasing. Applying \cref{lem:harris} and $\log(1-x)\ge-2x$ for $0\le x\le1/2$, and absorbing the factor $2$ into the final $o(1)$ term, gives
\[
 \begin{aligned}
 \PP[d_{\mathcal{G}}(v)\ge\ell_v\text{ for all }v]
 &\ge\prod_{v\in[n]}\PP[d_{\mathcal{G}}(v)\ge\ell_v]\\
 &\ge\bigl(1-n^{-\varphi(0.01)+o(1)}\bigr)^n
 \\
 &=\exp\Bigl\{n\log\bigl(1-n^{-\varphi(0.01)+o(1)}\bigr)\Bigr\}\\
 &\ge\exp\{-2n^{1-\varphi(0.01)+o(1)}\}
 =\exp\{-n^{1-\varphi(0.01)+o(1)}\}.
 \end{aligned}
\]
To control the number of edges, note that 
\[
\EE[|\mathcal G|]=\left(1-\frac{1}{n^{1/4}}\right)m
\quad\text{and}\quad
 \frac{(m-\EE[|\mathcal G|])^2}{\EE[|\mathcal G|]}
 =\frac{m}{\sqrt n-n^{1/4}}
 =\Theta(n^{1/2}\log n).
\]
Since $(m-\EE[|\mathcal G|])/\EE[|\mathcal G|]=n^{-1/4}/(1-n^{-1/4})<1$ for sufficiently large $n$, \cref{lem:chernoff}(iii) gives $\PP[|\mathcal{G}|>m]\le\exp\{-\Omega(n^{1/2}\log n)\}$. Conditional on $|\mathcal{G}|=t$, the hypergraph $\mathcal{G}$ is uniform among all $t$-edge $3$-graphs. For $t\le m$, extend it by uniformly chosen missing triples to obtain a uniform $m$-edge $3$-graph. The degree constraints are increasing, so the extension satisfies them whenever $\mathcal{G}$ does. Averaging over the values $t\le m$ gives
\[
\begin{aligned}
 \frac{|\Omega(n,m,\boldsymbol\ell)|}{\binom{\binom n3}{m}}
 &\ge\sum_{t=0}^{m}\PP[|\mathcal G|=t]\,
 \PP[d_{\mathcal G}(v)\ge\ell_v\text{ for every }v\in[n]\mid|\mathcal G|=t]\\
 &=\PP[|\mathcal G|\le m,\ d_{\mathcal G}(v)\ge\ell_v\text{ for every }v\in[n]]
 \\
 &\ge\PP[d_{\mathcal G}(v)\ge\ell_v\text{ for every }v\in[n]]
 -\PP[|\mathcal G|>m].
\end{aligned}
\]
Since $1-\varphi(0.01)<1/10$, for all sufficiently large $n$ this difference is at least
\[
 \exp\{-n^{1-\varphi(0.01)+o(1)}\}-\exp\{-\Omega(n^{1/2}\log n)\}
 \ge\frac12\exp\{-n^{1-\varphi(0.01)+o(1)}\}
 \ge\exp\{-n^{1/10}\},
\]
which proves the lemma.
\end{proof}

The next lemma combines these estimates to control vertex degrees and codegrees of a hypergraph chosen uniformly from $\Omega(n,m,\boldsymbol\ell)$. It also bounds how far degrees can fall below their prescribed lower bounds after vertex deletion.

For a $3$-graph $\mathcal{G}$, a vector $\boldsymbol\ell=(\ell_v)_{v\in V(\mathcal{G})}$ of nonnegative integers, and $Z\subseteq V(\mathcal{G})$, define
\[
 \mathfrak D_Z(\mathcal{G},\boldsymbol\ell)
 \coloneqq\sum_{v\in V(\mathcal{G})\setminus Z}\max\bigl\{0,\ell_v-d_{\mathcal{G}-Z}(v)\bigr\}.
\]

\begin{lemma}
\label{thm:terminal-regularity}
There exists an absolute constant $C>0$ with the following property. Let $(n,m,\boldsymbol\ell)$ be admissible and near-critical. For $F_*\sim\Unif\Omega(n,m,\boldsymbol\ell)$, put $B_*\coloneqq\{v\colon d_{F_*}(v)\le0.03\log n\}$. The bounds in (i) and (ii) below hold simultaneously with probability $1-o(1)$, where the $o(1)$ depends only on $n$.
\begin{enumerate}[label=(\textit{\roman*})]
\item We have $\Delta_1(F_*)\le C\log n$, $|B_*|\le n^{1/4}$, and $\Delta_2(F_*)\le2$.
\item For every vertex $y\in[n]$, we have $\sum_{v\in B_*\setminus\{y\}}d_{F_*}(y,v)\le1$.
\item For every fixed integer $h\ge0$, if $n$ is sufficiently large in terms of $h$ and the bounds in (i) and (ii) hold, then every $Z\subseteq V(F_*)$ with $|Z|\le h$ satisfies $\mathfrak D_Z(F_*,\boldsymbol\ell)\le |Z|$.
\end{enumerate}
\end{lemma}
\begin{proof}
We first bound the maximum degree, then estimate the number of low-degree vertices by passing to the binomial model. Conditioning on the degree sequence then controls codegrees and incidences with low-degree vertices; the vertex-deletion bound follows deterministically.

For \cref{thm:terminal-regularity}(i), \cref{thm:degree-domination} shows that, for every fixed $A>0$, there is a constant $C_A$ such that $\PP[\Delta_1(F_*)>C_A\log n]\le n^{-A}$. Indeed, $3m/n\sim\log n$ and $\ell_v<3m/n$ for sufficiently large $n$. Choose $C_A>2$ with $\varphi(C_A-1)>A+2$, and take $t=\lfloor C_A\log n\rfloor+1-\lceil3m/n\rceil$. Then $t=(C_A-1+o(1))\log n\ge3m/n$ for sufficiently large $n$, so the tail bound in \cref{thm:degree-domination} applies, with exponent
\[
 -t\log\frac{nt}{3m}+t-\frac{3m}{n}
 =-\bigl(\varphi(C_A-1)+o(1)\bigr)\log n
 \le-(A+1)\log n.
\]
Thus $\PP[d_{F_*}(v)>C_A\log n]\le n^{-A-1}$ for every $v\in[n]$. A union bound over the vertices therefore yields $\PP[\Delta_1(F_*)>C_A\log n]\le n\cdot n^{-A-1}=n^{-A}$.

We next show that $\PP[|B_*|>n^{1/4}]\le\exp\{-cn^{1/4}\log n\}$ for a fixed $c>0$ and all sufficiently large $n$. Let $\mathcal G$ be a binomial random $3$-graph on $[n]$ with edge probability $p=m/\binom{n}{3}$, and fix $S\subseteq[n]$ with $t=\lceil n^{1/4}\rceil$ vertices. For each $v\in S$, the degree $d_{\mathcal G-(S\setminus\{v\})}(v)$ counts precisely the triples containing $v$ and two vertices outside $S$, and has mean $\binom{n-t}{2}p=(1+o(1))\log n$. These degrees depend on disjoint sets of triples and are therefore mutually independent. Since deleting vertices cannot increase a degree, \cref{lem:chernoff}(i) gives
\[
\begin{aligned}
 \PP[d_{\mathcal G}(v)\le0.03\log n\text{ for every }v\in S]
 \le\prod_{v\in S}\PP[d_{\mathcal G-(S\setminus\{v\})}(v)\le0.03\log n]
 \le n^{-(\varphi(0.03)+o(1))t}.
\end{aligned}
\]
Since $t=\lceil n^{1/4}\rceil$, we have $\log(en/t)=(3/4+o(1))\log n$. Also $\varphi(0.03)>3/4$, so choose a fixed $0<c_1<\varphi(0.03)-3/4$. Summing over the choices of $S$ then gives
\[
 \begin{aligned}
 \binom nt n^{-(\varphi(0.03)+o(1))t}
 &\le\exp\left\{t\log\left(\frac{en}{t}\right)-(\varphi(0.03)+o(1))t\log n\right\}\\
 &=\exp\{-(\varphi(0.03)-3/4+o(1))t\log n\}
 \le\exp\{-c_1t\log n\}.
 \end{aligned}
\]
By Stirling's formula, the binomial hypergraph has exactly $m$ edges with probability $\binom{\binom n3}{m}p^m(1-p)^{\binom n3-m}=(1+o(1))/\sqrt{2\pi m(1-p)}=\Theta((n\log n)^{-1/2})$. Conditioning first on this prescribed number of edges and then on the degree lower bounds gives the law of $F_*$. By \cref{lem:beta}, the two conditioning steps yield
\begin{multline*}
  \PP[|B_*|>n^{1/4}]
 \le O(\sqrt{n\log n})\,\exp\{n^{1/10}\}\exp\{-c_1t\log n\}
 =\exp\{O(\log n)+n^{1/10}-c_1t\log n\} \\
 \le\exp\{-ct\log n\}
\end{multline*}
for some fixed $c>0$ and all sufficiently large $n$.

We next bound the maximum codegree $\Delta_2(F_*)$. Condition on a degree sequence $\boldsymbol d$ with $\max_v d_v\le C\log n$. For $B=\{x\}$, the denominator in \cref{thm:association} is $\sum_{v\in[n]\setminus\{x\}}d_{F_*}(v)-9\Delta_1(F_*)-2\Delta_1(F_*)^2=3m-O(\log^2n)=(1-o(1))3m>0$. Thus \cref{thm:association} shows that $d_{F_*}(y,x)$ is stochastically dominated by a Poisson random variable. Since $d_x,d_y\le C\log n$ and $3m\sim n\log n$, the mean of this variable satisfies
\[
 \frac{2d_yd_x}{3m-d_x-9\Delta_1(F_*)-2\Delta_1(F_*)^2}
 =O\left(\frac{\log^2n}{n\log n}\right)
 =O\left(\frac{\log n}{n}\right).
\]
By \cref{lem:moment-bounds}(iii) with $t=3$ and a union bound over pairs,
\[
 \PP[\Delta_2(F_*)\ge3\mid(d_{F_*}(v))_{v\in[n]}=\boldsymbol d]
 \le\binom n2 O\left(\left(\frac{\log n}{n}\right)^3\right)
 =O\left(\frac{\log^3n}{n}\right).
\]

For \cref{thm:terminal-regularity}(ii), condition on a degree sequence $\boldsymbol d$ satisfying $\max_v d_v\le C\log n$ and $|B_*|\le n^{1/4}$. For each such sequence, $B_*$ is fixed. With $B=B_*\setminus\{y\}$, the denominator in \cref{thm:association} satisfies
\[
\begin{aligned}
 \sum_{v\in[n]\setminus(B_*\setminus\{y\})}d_{F_*}(v)-9\Delta_1(F_*)-2\Delta_1(F_*)^2
 =3m-O(n^{1/4}\log n+\log^2n)
 =(1-o(1))3m>0.
\end{aligned}
\]
Hence \cref{thm:association} applies. Since $d_y\le C\log n$ and $\sum_{v\in B_*\setminus\{y\}}d_v\le Cn^{1/4}\log n$, the Poisson mean for $\sum_{v\in B_*\setminus\{y\}}d_{F_*}(y,v)$ satisfies
\[
 \frac{2d_y\sum_{v\in B_*\setminus\{y\}}d_v}
 {3m-\sum_{v\in B_*\setminus\{y\}}d_v-9\Delta_1(F_*)-2\Delta_1(F_*)^2}
 =O\left(\frac{n^{1/4}\log^2n}{n\log n}\right)
 =O\left(\frac{\log n}{n^{3/4}}\right).
\]
By \cref{lem:moment-bounds}(iii) with $t=2$ and a union bound over $y$,
\[
\begin{aligned}
 \PP\left[\max_{y\in[n]}\sum_{v\in B_*\setminus\{y\}}d_{F_*}(y,v)\ge2
 \,\middle|\,(d_{F_*}(v))_{v\in[n]}=\boldsymbol d\right]
\le n\,O\left(\left(\frac{\log n}{n^{3/4}}\right)^2\right)
 =O\left(\frac{\log^2n}{n^{1/2}}\right).
\end{aligned}
\]
Averaging over degree sequences and combining the preceding estimates proves that \cref{thm:terminal-regularity}(i)--(ii) hold simultaneously with probability $1-o(1)$.

For \cref{thm:terminal-regularity}(iii), suppose that the bounds in \cref{thm:terminal-regularity}(i)--(ii) hold. Since $\Delta_2(F_*)\le2$, every $Z\subseteq V(F_*)$ with $|Z|\le h$ and every $v\in V(F_*)\setminus Z$ satisfy $d_{F_*}(v)-d_{F_*-Z}(v)\le\sum_{y\in Z}d_{F_*}(v,y)\le2|Z|\le2h$. Every vertex $v\in V(F_*)\setminus(B_*\cup Z)$ has $d_{F_*}(v)>0.03\log n>\ell_v+2h$ and hence $d_{F_*-Z}(v)\ge\ell_v$. Since $d_{F_*}(v)\ge\ell_v$, for every $v\in B_*\setminus Z$ we have
\[
 \max\{0,\ell_v-d_{F_*-Z}(v)\}
 \le d_{F_*}(v)-d_{F_*-Z}(v)
 \le\sum_{y\in Z}d_{F_*}(v,y).
\]
Summing over these vertices and using \cref{thm:terminal-regularity}(ii) gives
\[
 \mathfrak D_Z(F_*,\boldsymbol\ell)
 \le\sum_{v\in B_*\setminus Z}\sum_{y\in Z}d_{F_*}(v,y)
 \le\sum_{y\in Z}\sum_{v\in B_*\setminus\{y\}}d_{F_*}(y,v)\le|Z|.
\]
This proves \cref{thm:terminal-regularity}(iii).
\end{proof}

\section{Counting mixed Hamilton cycles by edge deletion}
\label{sec:martingale}
Under \cref{ass:core}, we prove \cref{thm:conditioned-core} assuming the uniform edge-marginal proposition stated below. Initially, there are $\Xc(\binom{[n_\ast]}3)$ counted mixed Hamilton cycles. We first estimate the cumulative logarithmic loss under uniform edge deletion, then transfer the estimate to the law conditioned on the degree lower bounds.

We use the following probability laws.
\begin{itemize}[beginpenalty=10000]
\item $\PP$ is the law of the edge-deletion process $\binom{[n_\ast]}3=H_{\binom{n_\ast}{3}} \supsetneq\cdots\supsetneq H_{m_\ast+1}\supsetneq H_{m_\ast}$, where each step deletes a uniformly chosen edge.
\item $\QQ$ is the law of the same process $\binom{[n_\ast]}3=H_{\binom{n_\ast}{3}}
 \supsetneq\cdots\supsetneq H_{m_\ast+1}\supsetneq H_{m_\ast}=F_*$, conditional on $H_{m_\ast}\in\Omega(n_\ast,m_\ast,\boldsymbol\ell)$. 
\end{itemize}
Under $\PP$, $H_{m_\ast}$ is a uniform $m_\ast$-edge $3$-graph on $[n_\ast]$. Let $\beta$ denote the probability that $H_{m_\ast}$ belongs to $\Omega(n_\ast,m_\ast,\boldsymbol\ell)$, that is,
\[
 \beta=\PP[H_{m_\ast}\in\Omega(n_\ast,m_\ast,\boldsymbol\ell)]=\frac{|\Omega(n_\ast,m_\ast,\boldsymbol\ell)|}{\binom{\binom{n_\ast}{3}}{m_\ast}}.
\]
Conditioning on these degree constraints increases a failure probability by at most $\beta^{-1}$. The marginal proposition is proved in \cref{sec:migration}.

A realisation of this process under $\QQ$ can also be sampled by first choosing $F_*$ as in \cref{ass:core} and then adding the triples in $\binom{[n_\ast]}3\setminus F_*$ in a uniformly random order. We analyse the numbers $\Xc(H_j)$ of counted mixed Hamilton cycles in deletion order, from the complete $3$-graph down to $H_{m_\ast}$, with $\cM_0$ fixed and $\Xc$ defined as in \cref{sec:mixed}.

For a $3$-graph $H$ on $[n_\ast]$ with $\Xc(H)>0$ and every triple $e\in H$, define
\[
\Qc_e(H)\coloneqq
 \frac{\bigl|\{C\colon C\text{ is counted by }\Xc(H)\text{ and }e\in C\}\bigr|}{\Xc(H)}.
\]
This \emph{edge marginal} is the probability that a uniformly chosen $(H\cap\Kc,\cM_0)$-respecting mixed Hamilton cycle uses $e$.

For a host $H$ with $\Xc(H)>0$, we use the following identities.
\begin{enumerate}[label=(\textit{\roman*})]
\item Since every counted cycle contains exactly $(n_\ast-s)/2$ triples, we have $\sum_{e\in H}\Qc_e(H)=(n_\ast-s)/2$.
\item For every $e\in H$, deleting $e$ gives $\Xc(H-e)=\Xc(H)(1-\Qc_e(H))$. Thus, if $|H|=j$ and the deleted edge is chosen uniformly from $H$, the expected fraction of cycles lost is $j^{-1}\sum_{e\in H}\Qc_e(H)=(n_\ast-s)/(2j)$.
\end{enumerate}
Bounding every $\Qc_e(H)$ by a constant multiple of the average $(n_\ast-s)/(2j)$ controls the fluctuations and the quadratic error in approximating $\log(1-\Qc_e(H))$ by $-\Qc_e(H)$. This leads to the reference value $L_j$ for the logarithm of the number of counted mixed Hamilton cycles in $H_j$. With the convention $\log0\coloneqq-\infty$, define
\begin{equation}\label{eq:baseline}
 \mathsf A_j\coloneqq
 \biggl\{\log\Xc(H_j)\ge L_j-\frac{n_\ast}{\sqrt{\log n_\ast}}\biggr\},
 \quad\text{where}\quad
 L_j\coloneqq\log\Xc\biggl(\binom{[n_\ast]}3\biggr)-\frac{n_\ast-s}{2}\sum_{h=j+1}^{\binom{n_\ast}{3}}\frac1h.
\end{equation}
We first estimate $L_j$.

By \cref{prop:complete-count}, for $s\le n_\ast^{1/10}$, deleting all triples outside $\Kc$ from the complete $3$-graph $\binom{[n_\ast]}3$ preserves a $1-o(1)$ fraction of its $\cM_0$-respecting mixed Hamilton cycles. Stirling's formula then gives
\[
 \log\Xc\left(\binom{[n_\ast]}3\right)=\frac{n_\ast-s}{2}\log\bigl((n_\ast-1)(n_\ast-2)\bigr)-n_\ast+s+O(s\log n_\ast+\log n_\ast).
\]

Since $\sum_{h=j+1}^{\binom{n_\ast}{3}} h^{-1}=\log\bigl(\binom{n_\ast}{3}/j\bigr)+O(1/j)$ and $(n_\ast-s)/(2j)=O(1/\log n_\ast)$, the quantity $L_j$ in \eqref{eq:baseline} satisfies, for every $m_\ast\le j\le\binom{n_\ast}{3}$,
\begin{equation}\label{eq:benchmark}
L_j=\frac{n_\ast-s}{2}\log\frac{6j}{n_\ast}-n_\ast+s+O(s\log n_\ast+\log n_\ast).
\end{equation}

The following proposition bounds the edge marginals $\Qc_e(H_j)$ whenever $\mathsf A_j$ holds.

\begin{proposition}\label{thm:uniform-marginal}
There exists an absolute constant $C_0\ge1$ such that the following holds for every choice of parameters and matching satisfying \cref{ass:core}. With $\QQ$-probability $1-o(1)$, simultaneously for every $m_\ast\le j\le\binom{n_\ast}{3}$ such that $\mathsf A_j$ holds, we have
\[
 \max_{e\in H_j}\Qc_e(H_j)\le\frac{C_0(n_\ast-s)}{2j}.
\]
\end{proposition}

\begin{proof}[Proof of \cref{thm:conditioned-core} assuming \cref{thm:uniform-marginal}]
We use the method of counting by random edge deletion developed by Johansson, Kahn, and Vu~\cite[Sections~3 and~7]{JohanssonKahnVu2008}, in the form refined by Kahn~\cite[Sections~2 and~3]{Kahn2023}. We first stop the deletion process when no counted cycle remains or an edge marginal exceeds the prescribed bound, so that the martingale increments remain bounded. We then use \cref{thm:uniform-marginal} to exclude a failure of the counting lower bound before reaching $m_\ast$ edges.

Fix $C_0\ge1$ as in \cref{thm:uniform-marginal}. For sufficiently large $n_\ast$, \cref{ass:core} gives $n_\ast-s\ge2$, $m_\ast<\binom{n_\ast}{3}$, and $C_0(n_\ast-s)\le m_\ast$. Since $\sum_{j>m_\ast}j^{-2}\le1/m_\ast$ and $m_\ast\sim n_\ast\log n_\ast/3$, the variance parameter for the martingale estimate satisfies
\[
 \sigma^2\coloneqq\frac{C_0^2(n_\ast-s)^2}{4}\sum_{j=m_\ast+1}^{\binom{n_\ast}{3}}\frac{1}{j^2}=O\left(\frac{n_\ast}{\log n_\ast}\right).
\]

Starting from the complete $3$-graph, we inspect the states $H_j$ in decreasing order of $j$. Before deleting an edge from $H_j$, we stop if either
\[
 \Xc(H_j)=0,\quad\text{or}\quad
 \Xc(H_j)>0\quad\text{and}\quad\max_{e\in H_j}\Qc_e(H_j)>\frac{C_0(n_\ast-s)}{2j}.
\]
Let $j_0$ be the number of edges of the first state at which we stop, with $j_0=m_\ast$ if neither condition occurs before reaching $H_{m_\ast}$. Thus the process makes the deletions $H_j\to H_{j-1}$ for $j>j_0$ and ends at $H_{j_0}$.
\begin{claim}\label{thm:stopped-counting}
For every $a>0$,
\[
\QQ\left[\min_{j_0\le j\le\binom{n_\ast}{3}}
   \bigl\{\log\Xc(H_j)-L_j\bigr\}<-a-\sigma^2\right]
 \le\beta^{-1}\exp\left\{-\frac{a^2}{2\sigma^2}\right\}.
\]
\end{claim}
\begin{proof}[Proof of \cref{thm:stopped-counting}]
Work first under $\PP$. Let $\mathcal F_j$ be the $\sigma$-algebra generated by $(H_h)_{h=j}^{\binom{n_\ast}{3}}$. Since $\mathcal F_j$ records $H_j$ and all previously inspected states, whether the process makes the deletion from $H_j$ is $\mathcal F_j$-measurable. Before stopping, at a deletion step from size $j$ to size $j-1$, let $e_j$ be the deleted triple and put $D_j\coloneqq \Qc_{e_j}(H_j)-(n_\ast-s)/(2j)$. Conditional on $\mathcal F_j$, $e_j$ is uniform in $H_j$. Hence $\EE[D_j\mid\mathcal F_j]=j^{-1}\sum_{e\in H_j}\Qc_e(H_j)-(n_\ast-s)/(2j)=0$. After stopping, set $D_j=0$, so the same equality holds. Thus the increments form a martingale-difference sequence in deletion order. Moreover $|D_j|\le C_0(n_\ast-s)/(2j)$, and the sum of these squared bounds is $\sigma^2$. Applying \cref{lem:maximal-azuma} in deletion order gives
\[
 \PP\left[\max_{j_0\le j\le\binom{n_\ast}{3}}\sum_{h=j+1}^{\binom{n_\ast}{3}} D_h>a\right]
 \le\exp\left\{-\frac{a^2}{2\sigma^2}\right\}.
\]
We next translate this deviation bound into a lower bound for $\log\Xc(H_j)$. For each deletion before stopping, the exact identity is $\Xc(H_j-e_j)=\Xc(H_j)(1-\Qc_{e_j}(H_j))$. Since $\Qc_{e_j}(H_j)\le C_0(n_\ast-s)/(2j)\le1/2$, each deletion preserves at least half the cycles. The initial number of counted cycles is positive by \cref{prop:complete-count}, so $\Xc(H_j)>0$ at every state reached by the stopped process, including $H_{j_0}$, and we may take logarithms. For $0\le q\le1/2$, $-\log(1-q)\le q+q^2$. For every $j_0\le j\le\binom{n_\ast}{3}$, the bound $\Qc_{e_h}(H_h)\le C_0(n_\ast-s)/(2h)$ and the definition of $\sigma^2$ give $\sum_{h=j+1}^{\binom{n_\ast}{3}}\Qc_{e_h}(H_h)^2\le\sigma^2$. Summing the logarithmic identities and substituting the definitions of $L_j$ and $D_h$ therefore gives
\[
\begin{aligned}
 \log\Xc(H_j)-L_j
 =\sum_{h=j+1}^{\binom{n_\ast}{3}}\left(\log(1-\Qc_{e_h}(H_h))+\frac{n_\ast-s}{2h}\right)
 & \ge-\sum_{h=j+1}^{\binom{n_\ast}{3}}D_h-\sum_{h=j+1}^{\binom{n_\ast}{3}}\Qc_{e_h}(H_h)^2 \\
 & \ge-\sum_{h=j+1}^{\binom{n_\ast}{3}}D_h-\sigma^2.
\end{aligned}
\]
Whenever all partial sums of the $D_h$ are at most $a$, this lower bound is at least $-a-\sigma^2$. The maximal deviation estimate thus bounds the failure probability under $\PP$ by $\exp\{-a^2/(2\sigma^2)\}$. Finally, conditioning on the degree constraints at size $m_\ast$ increases this probability by at most a factor $\beta^{-1}$.
\end{proof}

With $a=n_\ast/\log n_\ast$, \cref{thm:stopped-counting} bounds the failure probability by $\beta^{-1}\exp\{-a^2/(2\sigma^2)\}$. Since $\sigma^2=O(n_\ast/\log n_\ast)$ and \cref{lem:beta} gives $\beta^{-1}\le\exp\{n_\ast^{1/10}\}$, we obtain the bound $\exp\{n_\ast^{1/10}-\Omega(n_\ast/\log n_\ast)\}$, which is $\exp\{-\Omega(n_\ast/\log n_\ast)\}=o(1)$.

We finish by ruling out a first failure of either the lower bound on the number of mixed Hamilton cycles or the edge-marginal bound. By \cref{thm:uniform-marginal}, with $\QQ$-probability $1-o(1)$, every size $j$ for which $\mathsf A_j$ holds satisfies $\max_{e\in H_j}\Qc_e(H_j)\le C_0(n_\ast-s)/(2j)$. By \cref{thm:stopped-counting} and the preceding probability estimate, with $\QQ$-probability $1-o(1)$, we also have $\log\Xc(H_j)\ge L_j-a-\sigma^2$ for every $j_0\le j\le\binom{n_\ast}{3}$. Fix a realisation of the edge-deletion process on which both properties hold. Starting at the complete $3$-graph, where $\mathsf A_{\binom{n_\ast}{3}}$ holds, suppose that the first violation occurs at size $j$ as the size decreases. There are two cases.
\begin{enumerate}[label=(\textit{\roman*})]
\item If $\mathsf A_j$ holds, \cref{thm:uniform-marginal} rules out a violation of the edge-marginal bound.
\item If $\mathsf A_j$ fails, all preceding states satisfy the marginal bound by \cref{thm:uniform-marginal}. Thus the stopped deletion process reaches $H_j$, and \cref{thm:stopped-counting} gives $\log\Xc(H_j)\ge L_j-O(n_\ast/\log n_\ast)$, contradicting $\log\Xc(H_j)<L_j-n_\ast/\sqrt{\log n_\ast}$.
\end{enumerate}
Thus neither kind of failure occurs. Since $O(s\log n_\ast+\log n_\ast)=o(n_\ast/\log n_\ast)$ and $a+\sigma^2=O(n_\ast/\log n_\ast)$, at $j=m_\ast$ the martingale estimate and \eqref{eq:benchmark} give $\log\Xc(F_*)\ge L_{m_\ast}-a-\sigma^2\ge\frac{n_\ast-s}{2}\log(6m_\ast/n_\ast)-n_\ast+s-O(n_\ast/\log n_\ast)$, as required.
\end{proof}

\section{Conditional laws, degree bounds, and edge distribution}
\label{sec:switchings}
Under \cref{ass:core}, we study the conditioned edge-deletion process from \cref{sec:martingale}. We first describe how conditioning changes the distribution of $H_j$ and determine the joint conditional law after revealing all edges meeting a fixed vertex set. We then establish vertex-degree, codegree, and edge-distribution bounds throughout the process. Finally, we estimate the number of pairs remaining in a specified link after excluding a small set determined by the hypergraph outside that vertex.

To describe how conditioning changes the distribution of $H_j$, for $|H|=j$ let $Z_{m_\ast}(H)$ be the number of members of $\Omega(n_\ast,m_\ast,\boldsymbol\ell)$ contained in $H$, namely
\[
 Z_{m_\ast}(H)\coloneqq\Bigl|\bigl\{F\in\Omega(n_\ast,m_\ast,\boldsymbol\ell)\colon F\subseteq H\bigr\}\Bigr|.
\]
As in Kahn~\cite[Observations~2.1 and~2.2]{Kahn2022}, we obtain the following formulas for the conditioned deletion law.
\begin{fact}\label{fact:conditioned-deletion}
Let $m_\ast\le j\le\binom{n_\ast}{3}$, and let $H\subseteq\binom{[n_\ast]}3$ have $j$ edges. The following statements hold.
\begin{enumerate}[label=(\textit{\roman*})]
\item We have $\frac{\QQ[H_j=H]}{\PP[H_j=H]}=\frac{Z_{m_\ast}(H)}{\beta\binom{j}{m_\ast}}$.
\item If $j>m_\ast$ and $Z_{m_\ast}(H)>0$, then, for every $e\in H$,
\[
 \QQ[H_{j-1}=H-e\mid H_j=H]
 =\frac{Z_{m_\ast}(H-e)}{(j-m_\ast)Z_{m_\ast}(H)}.
\]
\end{enumerate}
\end{fact}
Indeed, for \cref{fact:conditioned-deletion}(i), under $\PP$, conditional on $H_j=H$, the hypergraph $H_{m_\ast}$ is a uniform $m_\ast$-subset of $H$. Thus $\PP[H_{m_\ast}\in\Omega(n_\ast,m_\ast,\boldsymbol\ell)\mid H_j=H]=Z_{m_\ast}(H)/\binom{j}{m_\ast}$. Bayes's formula in \cref{lem:bayes}, with conditioning probability $\beta$, gives the first identity.

For \cref{fact:conditioned-deletion}(ii), under $\QQ$, conditional on $H_j=H$, the hypergraph $H_{m_\ast}$ is uniform among the $Z_{m_\ast}(H)$ members of $\Omega(n_\ast,m_\ast,\boldsymbol\ell)$ contained in $H$. Conditional also on $H_{m_\ast}$, the next deleted edge is uniform in the $(j-m_\ast)$-element set $H\setminus H_{m_\ast}$. The members of $\Omega(n_\ast,m_\ast,\boldsymbol\ell)$ that are contained in $H$ and do not contain $e$ are precisely those counted by $Z_{m_\ast}(H-e)$. Thus
\[
\begin{aligned}
 \QQ[H_{j-1}=H-e\mid H_j=H]
 =\frac1{Z_{m_\ast}(H)}
 \sum_{\substack{F\in\Omega(n_\ast,m_\ast,\boldsymbol\ell)\\F\subseteq H}}
 \frac{\mathbf1_{\{e\notin F\}}}{j-m_\ast}
 =\frac{Z_{m_\ast}(H-e)}{(j-m_\ast)Z_{m_\ast}(H)}.
\end{aligned}
\]

For fixed $j$, \cref{fact:conditioned-deletion}(i) shows that $\QQ[H_j=H]$ is proportional to $Z_{m_\ast}(H)$. Even among $j$-edge $3$-graphs satisfying $d_H(v)\ge\ell_v$ for every $v$, this weight may vary. Thus, under $\QQ$, $H_j$ need not be uniform on this family.

The cycle modifications later in the proof delete vertex sets of bounded size. The next lemma gives the joint conditional law after revealing all edges meeting a fixed vertex set, together with a deterministic comparison of the total amount by which degrees fall below their lower bounds. The final bound by $1$ additionally assumes that $F_*$ satisfies the bounds in \cref{thm:terminal-regularity}(i)--(ii).

Fix $\hat B\subseteq[n_\ast]$ and a time $m_\ast\le j\le\binom{n_\ast}{3}$. Write $F_*(\hat B)$ and $H_j(\hat B)$ for the sets of edges of $F_*$ and $H_j$, respectively, that meet $\hat B$; that is,
\[
 F_*(\hat B)\coloneqq\left\{e\in F_*\colon e\cap\hat B\ne\varnothing\right\},
 \quad\text{and}\quad
 H_j(\hat B)\coloneqq\left\{e\in H_j\colon e\cap\hat B\ne\varnothing\right\}.
\]
The induced hypergraphs on $[n_\ast]\setminus\hat B$ are
\[
 \hat F_\ast\coloneqq F_*-\hat B,
 \quad\text{and}\quad
 \hat H_j\coloneqq H_j-\hat B.
\]
For every $v\in[n_\ast]\setminus\hat B$, the resulting lower bound on $d_{\hat F_\ast}(v)$ is $\ell'_v\coloneqq\max\{0,\ell_v-d_{F_*(\hat B)}(v)\}$. We reveal $F_*(\hat B)$ and $H_j(\hat B)$; $\hat F_\ast$ and $\hat H_j$ remain unrevealed.

\begin{lemma}\label{lem:boundary}
The following statements hold.

\begin{enumerate}[label=(\textit{\roman*})]
\item Conditional on any $\QQ$-admissible value of $(F_*(\hat B),H_j(\hat B))$, $\hat F_\ast$ is uniform among simple $3$-graphs on $[n_\ast]\setminus \hat B$ with $m_\ast-|F_*(\hat B)|$ edges and degree lower bounds $\ell'_v$ for every $v\in[n_\ast]\setminus \hat B$.

\item Conditional on any $\QQ$-admissible value of $(F_*(\hat B),H_j(\hat B),\hat F_\ast)$, the hypergraph $\hat H_j$ is uniform among all $3$-graphs on $[n_\ast]\setminus\hat B$ that contain $\hat F_\ast$ and have $j-|H_j(\hat B)|$ edges.

\item For every $F_*\in\Omega(n_\ast,m_\ast,\boldsymbol\ell)$ and every vertex $y\in[n_\ast]\setminus \hat B$, we also have $\mathfrak D_{\{y\}}(\hat F_\ast,\boldsymbol\ell') \le\mathfrak D_{\{y\}}(F_*,\boldsymbol\ell)$. In particular, if the bounds in \cref{thm:terminal-regularity}(i)--(ii) hold, then $\mathfrak D_{\{y\}}(\hat F_\ast,\boldsymbol\ell')\le1$.
\end{enumerate}
\end{lemma}
\begin{proof}
For \cref{lem:boundary}(i)--(ii), under $\QQ$, $F_*$ is uniform on $\Omega(n_\ast,m_\ast,\boldsymbol\ell)$. Each possible $F_*$ has exactly $\binom{\binom{n_\ast}{3}-m_\ast}{j-m_\ast}$ extensions with $j$ edges, and conditional on $F_*$, these extensions are equally likely. Thus, before exposure, all pairs $(F_*,H_j)$ with $F_*\in\Omega(n_\ast,m_\ast,\boldsymbol\ell)$, $F_*\subseteq H_j$, and $|H_j|=j$ have equal probability.  After $F_*(\hat B)$ and $H_j(\hat B)$ are fixed, the remaining lower bounds are exactly $\boldsymbol\ell'$.  Every realisation of $\hat F_\ast$ has $m_\ast-|F_*(\hat B)|$ edges, so the number of extensions with $j-|H_j(\hat B)|$ edges is $\binom{\binom{n_\ast-|\hat B|}{3}-m_\ast+|F_*(\hat B)|}
 {j-m_\ast-|H_j(\hat B)|+|F_*(\hat B)|}$. 
This depends only on the revealed edge sets, proving the conditional laws.

For \cref{lem:boundary}(iii), fix $v\in[n_\ast]\setminus(\hat B\cup\{y\})$. The degree after both deletions is $d_{\hat F_\ast-y}(v)=d_{F_*}(v)-d_{F_*(\hat B)}(v)-d_{\hat F_\ast}(v,y)$. Since $d_{\hat F_\ast-y}(v)\ge0$, taking the positive part in the definition of $\ell'_v$ does not change the outer positive part below. Substituting the degree identity and the definition of $\ell'_v$, and using that deleting vertices cannot increase codegrees, gives
\[
\begin{aligned}
 \max\{0,\ell'_v-d_{\hat F_\ast-y}(v)\}
 &=\max\{0,\ell_v-d_{F_*}(v)+d_{\hat F_\ast}(v,y)\}\\
 &\le\max\{0,\ell_v-d_{F_*}(v)+d_{F_*}(v,y)\}
 =\max\{0,\ell_v-d_{F_*-y}(v)\}.
\end{aligned}
\]
Summing over $v\in[n_\ast]\setminus(\hat B\cup\{y\})$ and including the nonnegative summands indexed by $\hat B$ on the right proves the inequality. If the bounds in \cref{thm:terminal-regularity}(i)--(ii) hold, then $\mathfrak D_{\{y\}}(F_*,\boldsymbol\ell)\le1$ by \cref{thm:terminal-regularity}(iii).
\end{proof}

The next lemma supplies the degree and codegree bounds, together with the estimates for triples meeting an approximately half-size vertex set in exactly two vertices, needed for \cref{thm:entropy-stability}. These estimates hold throughout the conditioned edge-deletion process and after deleting any bounded number of vertices.

\begin{lemma}\label{prop:path-regularity}
For every fixed integer $h\ge0$ and real $L\ge1$, there exist constants $c,C>0$, depending only on $h,L$, such that, under $\QQ$, with probability $1-o(1)$ the following statements hold simultaneously for every $m_\ast\le j\le\binom{n_\ast}{3}$ and every $\hat B\subseteq[n_\ast]$ with $|\hat B|\le h$. The $o(1)$ depends only on $n_\ast,h,L$.
\begin{enumerate}[label=(\textit{\roman*})]
\item For every $v\in[n_\ast]\setminus\hat B$, we have $\frac{3cj}{n_\ast}\le d_{\hat H_j}(v)\le\frac{3Cj}{n_\ast}$, and $\Delta_2(\hat H_j)\le\frac{3Cj}{n_\ast(\log n_\ast)^{1/4}}$.
\item For every $A\subseteq[n_\ast]\setminus\hat B$ with $\bigl||A|-(n_\ast-|\hat B|)/2\bigr|\le Ln_\ast^{1/10}$,
\[
 \left|\bigl|\hat H_j[A,\overline A]\bigr|-\frac{3j}{8}\right|
 \le\frac{3Cj}{(\log n_\ast)^{1/8}}.
\]
\item The degree and codegree upper bounds in (i) and the estimate in (ii) also hold with $\hat H_j\cap\Kc$ in place of $\hat H_j$.
\end{enumerate}
\end{lemma}
\begin{proof}
We first prove the estimates for $H_j$ itself, then bound the changes caused by deleting a bounded set of vertices and removing triples outside $\Kc$. To establish the vertex-degree bounds, fix a realisation of $F_*$ satisfying the bounds in \cref{thm:terminal-regularity}(i)--(ii). Under $\QQ$, conditional on this $F_*$, the set $H_j\setminus F_*$ is a uniform $(j-m_\ast)$-subset of $\binom{[n_\ast]}3\setminus F_*$. Hence $d_{H_j}(v)-d_{F_*}(v)$ is hypergeometric, with
\[
 \EE[d_{H_j}(v)-d_{F_*}(v)\mid F_*]
 =(j-m_\ast)\frac{\binom{n_\ast-1}{2}-d_{F_*}(v)}{\binom{n_\ast}{3}-m_\ast}
 \le\frac{4j}{n_\ast}.
\]
Since $j\ge m_\ast\sim n_\ast\log n_\ast/3$, we can choose $C$ sufficiently large that $d_{F_*}(v)\le Cj/n_\ast$ and $(2Cj/n_\ast)\log(C/(2\e))\ge6\log n_\ast$. Then \cref{lem:chernoff}(ii) gives
\[
\begin{aligned}
 \QQ\left[d_{H_j}(v)>\frac{3Cj}{n_\ast}\,\middle|\,F_*\right]
 &\le\QQ\left[d_{H_j}(v)-d_{F_*}(v)>\frac{2Cj}{n_\ast}\,\middle|\,F_*\right]
 \le\left(\frac{2\e}{C}\right)^{\frac{2Cj}{n_\ast}}\le n_\ast^{-6}.
\end{aligned}
\]

For the lower bound, fix an absolute constant $L_0>144$. Since $d_{F_*}(v)=O(\log n_\ast)$, the same mean formula gives $\EE[d_{H_j}(v)-d_{F_*}(v)\mid F_*]\ge2(j-m_\ast)/n_\ast\ge j/n_\ast$ for $j>L_0m_\ast$ and sufficiently large $n_\ast$. Thus $j/(2n_\ast)\le\frac12\EE[d_{H_j}(v)-d_{F_*}(v)\mid F_*]$. Moreover, $j>L_0m_\ast$ and $m_\ast\sim n_\ast\log n_\ast/3$ give $j/(8n_\ast)\ge(L_0/24+o(1))\log n_\ast$. Since $L_0>144$, \cref{lem:chernoff}(i) with $a=1/2$ gives
\[
 \QQ\left[d_{H_j}(v)-d_{F_*}(v)<\frac{j}{2n_\ast}\,\middle|\,F_*\right]
 \le\exp\left\{-\frac{j}{8n_\ast}\right\}
 \le n_\ast^{-\frac{L_0}{24}+o(1)}\le n_\ast^{-6}.
\]
By near-criticality, $\min_v\ell_v/(3L_0m_\ast/n_\ast)=(0.01+o(1))/L_0$. We may therefore choose a fixed $c>0$ such that $c\le1/6$ and $3cL_0m_\ast/n_\ast\le\min_v\ell_v$ for all sufficiently large $n_\ast$. For $j\le L_0m_\ast$, we have $d_{H_j}(v)\ge d_{F_*}(v)\ge\ell_v\ge3cL_0m_\ast/n_\ast\ge3cj/n_\ast$. For $j>L_0m_\ast$, outside the lower-tail event above, $d_{H_j}(v)\ge j/(2n_\ast)\ge3cj/n_\ast$. A union bound over vertices and times therefore gives
\[
\begin{aligned}
 \QQ\left[\bigcup_{j=m_\ast}^{\binom{n_\ast}{3}}\ \bigcup_{v\in[n_\ast]}
 \left\{d_{H_j}(v)\notin\left[\frac{3cj}{n_\ast},\frac{3Cj}{n_\ast}\right]\right\}
 \,\middle|\,F_*\right]
 \le2n_\ast\binom{n_\ast}{3}n_\ast^{-6}
 =O\left(\frac{1}{n_\ast^2}\right)=o(1).
\end{aligned}
\]

For distinct vertices $u,v$ and a fixed time $j$, let $X\coloneqq d_{H_j}(u,v)-d_{F_*}(u,v)$. Conditional on $F_*$, this variable is hypergeometric, with $\EE[X\mid F_*]=(j-m_\ast)(n_\ast-2-d_{F_*}(u,v))/(\binom{n_\ast}{3}-m_\ast)=O(j/n_\ast^2)$. Put $t\coloneqq Cj/(n_\ast(\log n_\ast)^{1/4})$. Since $j\ge m_\ast$, we have $t=\Omega((\log n_\ast)^{3/4})$. Also $\e\EE[X\mid F_*]/t=O((\log n_\ast)^{1/4}/n_\ast)$ and $d_{F_*}(u,v)\le2$. Hence \cref{lem:chernoff}(ii) gives
\[
 \QQ\left[d_{H_j}(u,v)>\frac{3Cj}{n_\ast(\log n_\ast)^{1/4}}\,\middle|\,F_*\right]
 \le\QQ[X\ge t\mid F_*]
 \le\left(\frac{\e\EE[X\mid F_*]}{t}\right)^t
 \le\exp\{-c(\log n_\ast)^{7/4}\}.
\]
Thus,
\[
\begin{aligned}
 \QQ\left[\bigcup_{j=m_\ast}^{\binom{n_\ast}{3}}
 \left\{\Delta_2(H_j)>\frac{3Cj}{n_\ast(\log n_\ast)^{1/4}}\right\}
 \,\middle|\,F_*\right]
 &\le\sum_{j=m_\ast}^{\binom{n_\ast}{3}}
 \sum_{\{u,v\}\in\binom{[n_\ast]}2}
 \QQ\left[d_{H_j}(u,v)>\frac{3Cj}{n_\ast(\log n_\ast)^{1/4}}\,\middle|\,F_*\right]\\
 &\le\binom{n_\ast}{2}\binom{n_\ast}{3}
 \exp\{-c(\log n_\ast)^{7/4}\}=o(1).
\end{aligned}
\]
The conditional estimates hold uniformly over the chosen realisations of $F_*$. Averaging and applying \cref{thm:terminal-regularity}(i)--(ii) proves the degree and codegree bounds for $H_j$ under $\QQ$.

We next count triples meeting a set $A$ in exactly two vertices. Suppose $\bigl||A|-n_\ast/2\bigr|\le(L+1)n_\ast^{1/10}$. Under $\PP$, the number of such triples has a hypergeometric distribution, and
\[
 \EE\left[\bigl|H_j[A,\overline A]\bigr|\right]
 =j\frac{\binom{|A|}{2}(n_\ast-|A|)}{\binom{n_\ast}{3}}
 =\frac{3j}{8}+O_L\left(\frac{j}{n_\ast^{9/10}}\right).
\]
For sufficiently large $n_\ast$, the last error is at most $j/(2(\log n_\ast)^{1/8})$. A deviation from $3j/8$ exceeding $j/(\log n_\ast)^{1/8}$ therefore gives a deviation from the mean exceeding $j/(2(\log n_\ast)^{1/8})$. This threshold is smaller than the mean for large $n_\ast$. Applying \cref{lem:chernoff}(iii), with $\bigl(j/(2(\log n_\ast)^{1/8})\bigr)^2/\bigl(3j/8+O_L(j/n_\ast^{9/10})\bigr)=\Theta(j/(\log n_\ast)^{1/4})$, yields
\[
 \PP\left[\left|\bigl|H_j[A,\overline A]\bigr|-\frac{3j}{8}\right|>
 \frac{j}{(\log n_\ast)^{1/8}}\right]
 \le2\exp\left\{-\frac{cj}{(\log n_\ast)^{1/4}}\right\}.
\]
We have $m_\ast\sim n_\ast\log n_\ast/3$, and \cref{lem:beta} gives $\beta^{-1}\le\exp\{n_\ast^{1/10}\}$. Using these estimates, the definition of $\QQ$, and a union bound, we obtain
\[
\begin{aligned}
 &\QQ\left[
 \bigcup_{j=m_\ast}^{\binom{n_\ast}{3}}
 \bigcup_{\substack{A\subseteq[n_\ast]\\
 \left||A|-n_\ast/2\right|\le(L+1)n_\ast^{1/10}}}
 \left\{\left|\bigl|H_j[A,\overline A]\bigr|-\frac{3j}{8}\right|
 >\frac{j}{(\log n_\ast)^{1/8}}\right\}\right]\\
 &\quad\le\beta^{-1}\sum_{j=m_\ast}^{\binom{n_\ast}{3}}
 \sum_{\substack{A\subseteq[n_\ast]\\
 \left||A|-n_\ast/2\right|\le(L+1)n_\ast^{1/10}}}
 \PP\left[\left|\bigl|H_j[A,\overline A]\bigr|-\frac{3j}{8}\right|
 >\frac{j}{(\log n_\ast)^{1/8}}\right]\\
 &\quad\le2\beta^{-1}2^{n_\ast}\binom{n_\ast}{3}
 \exp\left\{-\frac{cm_\ast}{(\log n_\ast)^{1/4}}\right\}=o(1).
\end{aligned}
\]

It remains to pass from $H_j$ to $\hat H_j$ and $\hat H_j\cap\Kc$. For $|\hat B|\le h$ and $v\in[n_\ast]\setminus\hat B$, we have $0\le d_{H_j}(v)-d_{\hat H_j}(v)\le h\Delta_2(H_j)=o(j/n_\ast)$, so decreasing $c$ preserves the degree lower bound. The numbers of deleted edges satisfy
\[
\begin{aligned}
 |H_j\setminus\hat H_j|&\le h\Delta_1(H_j)=O_{h,L}\left(\frac{j}{n_\ast}\right)
 \quad\text{and}\quad 
 |\hat H_j\setminus\Kc|&\le\frac12\sum_{u\in V(\cM_0)}d_{H_j}(u)
 =O_{h,L}\left(\frac{sj}{n_\ast}\right).
\end{aligned}
\]
Every set $A$ in \cref{prop:path-regularity}(ii) satisfies $\bigl||A|-n_\ast/2\bigr|\le Ln_\ast^{1/10}+h/2\le(L+1)n_\ast^{1/10}$ for large $n_\ast$. Since $s\le n_\ast^{1/10}$, deleting edges changes each such number of triples by at most $|H_j\setminus\hat H_j|+|\hat H_j\setminus\Kc|=O_{h,L}((s+1)j/n_\ast)=o(j/(\log n_\ast)^{1/8})$. Increasing $C$ now gives the estimates in \cref{prop:path-regularity}(ii)--(iii). Since deleting edges cannot increase vertex degrees or codegrees, their upper bounds also hold.
\end{proof}

The cycle modifications require triples through a specified vertex that avoid certain excluded pairs. We show that a small set of excluded pairs, determined without revealing that vertex's link, is unlikely to cover most of the link. The estimate also allows us to require a bounded number of specified triples through the vertex to be present. For the next lemma, fix a nonnegative integer $h$ and make the following choices.
\begin{enumerate}[label=(\alph*)]
\item A set $\hat B\subseteq[n_\ast]$ with $|\hat B|\le8$.
\item A vertex $y\in[n_\ast]\setminus\hat B$.
\item A time $m_\ast\le j\le\binom{n_\ast}{3}$.
\item A set $\mathcal P$ of at most $h$ prescribed triples containing $y$ and disjoint from $\hat B$.
\item A fixed rule defining $\Gamma\subseteq\binom{[n_\ast]\setminus(\hat B\cup\{y\})}2$ using only $\hat H_j-y$ and fixed vertex labels.
\end{enumerate}
Thus revealing $\hat H_j-y$ determines $\Gamma$ while leaving $L_{\hat H_j}(y)$ unrevealed.

\begin{lemma}\label{lem:simultaneous-link}
Under the setting above, there exists a constant $c>0$, depending only on $h$, such that, for all sufficiently large $n_\ast$ and every such choice of data, the $\QQ$-probability that all of the following conditions hold is at most $\exp\{-c(\log n_\ast)\log\log\log\log n_\ast\}$.
\begin{enumerate}[label=(\textit{\roman*})]
\item The hypergraph $F_*$ satisfies the bounds in \cref{thm:terminal-regularity}(i)--(ii), and all conclusions of \cref{prop:path-regularity}, with $h=8$ and $L=2$, hold.
\item Each triple in $\mathcal P$ is an edge of $\hat H_j$.
\item $|\Gamma|\le\binom{n_\ast-|\hat B|-1}{2}/(\log\log\log n_\ast)^{1/16}$.
\item $|L_{\hat H_j}(y)\setminus\Gamma|<d_{\hat H_j}(y)/3$.
\end{enumerate}
\end{lemma}

\begin{proof}[Proof of \cref{lem:simultaneous-link}]
Recall that $\ell'_v=\max\{0,\ell_v-d_{F_*(\hat B)}(v)\}$ is the required lower bound on the degree of $v$ in $\hat F_\ast$. Following the exposure argument in Kahn~\cite[Section~9]{Kahn2022}, reveal $F_*(\hat B)$, $H_j(\hat B)$, $\hat F_\ast-y$, and $\hat H_j-y$. The only unrevealed triples of $\hat F_\ast$ and $\hat H_j$ now contain $y$; equivalently, it remains to reveal the links $L_{\hat F_\ast}(y)$ and $L_{\hat H_j}(y)$. Fix a $\QQ$-admissible value of these four edge sets. For now, we do not condition on $\mathcal P\subseteq\hat H_j$.

We first determine the joint conditional distribution of the two links $L_{\hat F_\ast}(y)$ and $L_{\hat H_j}(y)$ given the revealed edge sets. When the bounds in \cref{thm:terminal-regularity}(i)--(ii) hold, \cref{lem:boundary}(iii) and \cref{thm:terminal-regularity}(iii) show that deleting $y$ leaves at most one vertex needing an additional incident edge; this permits a comparison with uniform link sampling by changing at most one pair. We then incorporate the prescribed triples and apply a hypergeometric tail bound.

\begin{claim}\label{claim:conditional-links}
Conditional on the revealed edge sets $F_*(\hat B)$, $H_j(\hat B)$, $\hat F_\ast-y$, and $\hat H_j-y$ under $\QQ$, the pair of links $(L_{\hat F_\ast}(y),L_{\hat H_j}(y))$ is uniform among all choices satisfying the following conditions.
\begin{enumerate}[label=(\textit{\roman*})]
\item The links satisfy $L_{\hat F_\ast}(y)\subseteq L_{\hat H_j}(y)\subseteq\binom{V(\hat H_j)\setminus\{y\}}2$.
\item The numbers of edges in the two links are $|L_{\hat F_\ast}(y)|=m_\ast-|F_*(\hat B)|-|\hat F_\ast-y|$ and $|L_{\hat H_j}(y)|=j-|H_j(\hat B)|-|\hat H_j-y|$.
\item For every $v\in V(\hat H_j)\setminus\{y\}$, we have $d_{\hat F_\ast-y}(v)+d_{L_{\hat F_\ast}(y)}(v)\ge\ell'_v$.
\end{enumerate}
\end{claim}
\begin{proof}[Proof of \cref{claim:conditional-links}]
The link sizes are fixed because
\[
\begin{aligned}
 |L_{\hat F_\ast}(y)|&=|\hat F_\ast|-|\hat F_\ast-y|
 =m_\ast-|F_*(\hat B)|-|\hat F_\ast-y|,\\
 |L_{\hat H_j}(y)|&=|\hat H_j|-|\hat H_j-y|
 =j-|H_j(\hat B)|-|\hat H_j-y|.
\end{aligned}
\]
For every $\QQ$-admissible exposure, the fixed size $|L_{\hat F_\ast}(y)|=d_{\hat F_\ast}(y)$ is at least $\ell'_y$. Thus every choice of $L_{\hat F_\ast}(y)$ with this fixed size satisfies the degree lower bound $d_{\hat F_\ast}(y)\ge\ell'_y$.

Since $\hat F_\ast\subseteq\hat H_j$, their links satisfy $L_{\hat F_\ast}(y)\subseteq L_{\hat H_j}(y)$. The remaining degree constraints concern the vertices other than $y$. For every $v\in[n_\ast]\setminus(\hat B\cup\{y\})$, the triples through $y$ contribute $d_{L_{\hat F_\ast}(y)}(v)$ to the degree of $v$, so the required inequality is
\[
 d_{\hat F_\ast}(v)
 =d_{\hat F_\ast-y}(v)+d_{L_{\hat F_\ast}(y)}(v)\ge\ell'_v.
\]
By \cref{lem:boundary}(i)--(ii), conditional on $F_*(\hat B)$ and $H_j(\hat B)$, all compatible pairs $(\hat F_\ast,\hat H_j)$ of the prescribed sizes are equally likely, since the number of extensions of $\hat F_\ast$ depends only on its number of edges. Fixing $\hat F_\ast-y$ and $\hat H_j-y$ restricts this equally likely family, and each remaining pair $(\hat F_\ast,\hat H_j)$ is uniquely determined by $L_{\hat F_\ast}(y)$ and $L_{\hat H_j}(y)$. Therefore, conditional under $\QQ$ on $F_*(\hat B)$, $H_j(\hat B)$, $\hat F_\ast-y$, and $\hat H_j-y$, the pair of links is uniform among all choices satisfying conditions (i)--(iii) of \cref{claim:conditional-links}.
\end{proof}

By \cref{lem:boundary}(iii) and \cref{thm:terminal-regularity}(iii), whenever $F_*$ satisfies the bounds in \cref{thm:terminal-regularity}(i)--(ii), we have
\[
 \mathfrak D_{\{y\}}(\hat F_\ast,\boldsymbol\ell')
 =\sum_{v\in V(\hat F_\ast)\setminus\{y\}}
 \max\{0,\ell'_v-d_{\hat F_\ast-y}(v)\}\le1.
\]
The sum on the left is determined by the four revealed edge sets. Values for which it exceeds $1$ cannot contribute to the event in \cref{lem:simultaneous-link}. It therefore suffices to treat the cases $\mathfrak D_{\{y\}}(\hat F_\ast,\boldsymbol\ell')=0$ and $\mathfrak D_{\{y\}}(\hat F_\ast,\boldsymbol\ell')=1$. Note that these cases mean that, in $\hat F_\ast-y$, either every vertex satisfies its degree lower bound, or exactly one vertex $v$ has degree $\ell'_v-1$ and every other vertex satisfies its degree lower bound.

The vertex-degree lower bound in \cref{prop:path-regularity}(i), required by \cref{lem:simultaneous-link}(i), and $j\ge m_\ast\sim n_\ast\log n_\ast/3$ give $d_{\hat H_j}(y)\ge c\log n_\ast$ for a fixed $c>0$. We therefore need only consider revealed edge sets satisfying this degree bound and the size condition $|\Gamma|\le\binom{n_\ast-|\hat B|-1}{2}/(\log\log\log n_\ast)^{1/16}$ in \cref{lem:simultaneous-link}(iii). The set $\Gamma$ is fixed by the four originally revealed edge sets.

Suppose first that $\mathfrak D_{\{y\}}(\hat F_\ast,\boldsymbol\ell')=0$. Keep the four revealed edge sets fixed, condition on $\mathcal P\subseteq\hat H_j$, and reveal $\mathcal P\cap\hat F_\ast$, considering only $\QQ$-admissible values. Every vertex other than $y$ already satisfies its degree lower bound in $\hat F_\ast-y$, so the degree inequalities in \cref{claim:conditional-links} impose no further restriction on the links. After removing the fixed pairs $L_{\mathcal P}(y)$, each possible remaining larger link contains exactly $\binom{d_{\hat H_j}(y)-|\mathcal P|}{d_{\hat F_\ast}(y)-|\mathcal P\cap\hat F_\ast|}$ possible remaining smaller links. Thus $L_0\coloneqq L_{\hat H_j}(y)\setminus L_{\mathcal P}(y)$ is uniform among the subsets of $\binom{V(\hat H_j)\setminus\{y\}}2\setminus L_{\mathcal P}(y)$ having $d_{\hat H_j}(y)-|\mathcal P|$ pairs.

Under the conditional law just described, $|L_0\cap\Gamma|$ is hypergeometric, and its mean satisfies
\[
\begin{aligned}
 \EE[|L_0\cap\Gamma|]
 =\frac{(d_{\hat H_j}(y)-|\mathcal P|)|\Gamma\setminus L_{\mathcal P}(y)|}
 {\binom{n_\ast-|\hat B|-1}{2}-|\mathcal P|}
 \le\frac{d_{\hat H_j}(y)|\Gamma|}{\binom{n_\ast-|\hat B|-1}{2}-h}
 \le\frac{2d_{\hat H_j}(y)}{(\log\log\log n_\ast)^{1/16}}.
\end{aligned}
\]
Since $|\mathcal P|\le h$, on the event $|L_{\hat H_j}(y)\setminus\Gamma|<d_{\hat H_j}(y)/3$ we have
\[
 |L_0\cap\Gamma|
 =|L_{\hat H_j}(y)\cap\Gamma|-|L_{\mathcal P}(y)\cap\Gamma|
 >\frac{2d_{\hat H_j}(y)}3-h
 \ge\frac{d_{\hat H_j}(y)}2
\]
for sufficiently large $n_\ast$. If the mean is zero, then this event is impossible. Otherwise, \cref{lem:chernoff}(ii) bounds its conditional probability by $\bigl(4\e/(\log\log\log n_\ast)^{1/16}\bigr)^{d_{\hat H_j}(y)/2}\le\exp\{-c(\log n_\ast)\log\log\log\log n_\ast\}$. This proves the conditional bound when $\mathfrak D_{\{y\}}(\hat F_\ast,\boldsymbol\ell')=0$, for every such value of $\mathcal P\cap\hat F_\ast$, with the four original revealed edge sets fixed and conditional on $\mathcal P\subseteq\hat H_j$.

In the remainder of the proof, we focus on values of the revealed edge sets for which $\mathfrak D_{\{y\}}(\hat F_\ast,\boldsymbol\ell')=1$. At this stage, we condition only on $F_*(\hat B)$, $H_j(\hat B)$, $\hat F_\ast-y$, and $\hat H_j-y$, postponing the requirement $\mathcal P\subseteq\hat H_j$. The next claim compares the actual link with uniform sampling. In its coupling construction and the final tail estimate, $\PP$ denotes the probability law of the auxiliary sampling experiment.

\begin{claim}\label{claim:link-coupling}
If $\mathfrak D_{\{y\}}(\hat F_\ast,\boldsymbol\ell')=1$, then under the conditional law in \cref{claim:conditional-links}, the link $L_{\hat H_j}(y)$ can be coupled with a uniformly chosen set of $d_{\hat H_j}(y)$ pairs in $\binom{V(\hat H_j)\setminus\{y\}}2$ so that the two sets differ by at most one pair replacement.
\end{claim}
\begin{proof}[Proof of \cref{claim:link-coupling}]
Recall that $\mathfrak D_{\{y\}}(\hat F_\ast,\boldsymbol\ell')=1$ means that there is a unique vertex $v\in V(\hat H_j)\setminus\{y\}$ with $d_{\hat F_\ast-y}(v)=\ell'_v-1$, and every other vertex of $\hat F_\ast-y$ satisfies its degree lower bound. To satisfy the degree lower bound at $v$, the smaller link $L_{\hat F_\ast}(y)$ must contain a pair incident to $v$. Let $S\subseteq\binom{V(\hat H_j)\setminus\{y\}}2$ be the set of all such pairs. The remaining degree constraint is exactly $L_{\hat F_\ast}(y)\cap S\ne\varnothing$.

For comparison, first choose $A_0$ uniformly among the subsets of $\binom{V(\hat H_j)\setminus\{y\}}2$ having $d_{\hat F_\ast}(y)$ pairs. Given $A_0$, choose $L_0$ uniformly among the subsets of $\binom{V(\hat H_j)\setminus\{y\}}2$ having $d_{\hat H_j}(y)$ pairs and containing $A_0$. Since the exposure is $\QQ$-admissible, the event $A_0\cap S\ne\varnothing$ has positive probability. With the revealed edge sets fixed, the required joint distribution of $(L_{\hat F_\ast}(y),L_{\hat H_j}(y))$ is precisely the law of $(A_0,L_0)$ conditional on this event. We first couple $A_0$ with a set $A$ having the law of $A_0$ conditional on $A_0\cap S\ne\varnothing$, so that they differ by at most one pair replacement. We then extend the two smaller sets to sets of size $d_{\hat H_j}(y)$, preserving this bound.

To couple the intersection sizes so that they differ by at most one, we first prove
\[
 |A_0\cap S|\preceq\bigl(|A_0\cap S|\mid |A_0\cap S|\ge1\bigr)\preceq |A_0\cap S|+1.
\]
We will use these inequalities to construct the coupled intersection sizes and then the sets themselves. If $|A_0\cap S|$ is positive almost surely, then the conditional and unconditional laws coincide. Otherwise, for consecutive values in its support, we have
\[
 \frac{\PP[|A_0\cap S|=i+1]}{\PP[|A_0\cap S|=i]}=\frac{(|S|-i)(d_{\hat F_\ast}(y)-i)}{(i+1)(\binom{n_\ast-|\hat B|-1}{2}-|S|-d_{\hat F_\ast}(y)+i+1)}.
\]
These ratios decrease with $i$. For every $i$ with $\PP[|A_0\cap S|=i]>0$, summing the point probabilities and dividing by $\PP[|A_0\cap S|=i]$ gives
\[
 \frac{\PP[|A_0\cap S|\ge i]}{\PP[|A_0\cap S|=i]}
 =1+\sum_{k=1}^{\min\{d_{\hat F_\ast}(y),|S|\}-i}\ \prod_{q=0}^{k-1}
 \frac{\PP[|A_0\cap S|=i+q+1]}{\PP[|A_0\cap S|=i+q]}.
\]
As $i$ increases, each factor and the number of summands are nonincreasing. Hence $\PP[|A_0\cap S|\ge i]/\PP[|A_0\cap S|=i]$ is nonincreasing wherever its denominator is positive, and
\[
 \frac{\PP[|A_0\cap S|\ge i+1]}{\PP[|A_0\cap S|\ge i]}
 =1-\frac{\PP[|A_0\cap S|=i]}{\PP[|A_0\cap S|\ge i]}
 \le\frac{\PP[|A_0\cap S|\ge1]}{\PP[|A_0\cap S|\ge0]}
 =\PP[|A_0\cap S|\ge1]
\]
for $i\ge0$ with $\PP[|A_0\cap S|\ge i]>0$. Thus $\PP[|A_0\cap S|\ge i]\le\PP[|A_0\cap S|\ge1]\PP[|A_0\cap S|\ge i-1]$ for every integer $i\ge1$.
Dividing by $\PP[|A_0\cap S|\ge1]>0$, for every integer $i\ge1$ we obtain
\[
 \PP[|A_0\cap S|\ge i\mid |A_0\cap S|\ge1]
 =\frac{\PP[|A_0\cap S|\ge i]}{\PP[|A_0\cap S|\ge1]}
 \le\PP[|A_0\cap S|\ge i-1]
 =\PP[|A_0\cap S|+1\ge i].
\]
Since the denominator is at most $1$, the same conditional probability is at least $\PP[|A_0\cap S|\ge i]$. This proves both stochastic-domination inequalities.

We now construct $Z$ with the law of $|A_0\cap S|$ conditional on $A_0\cap S\ne\varnothing$, keeping the already chosen set $A_0$ fixed. If $|A_0\cap S|=i$, then choose a number uniformly from
\[
 \bigl(\PP[|A_0\cap S|<i],\ \PP[|A_0\cap S|\le i]\bigr),
\]
and let $Z$ be the smallest integer whose cumulative probability under the law of $|A_0\cap S|$ conditional on $|A_0\cap S|\ge1$ is at least this number. Averaging over $i$, the chosen number is uniform in $(0,1)$, so $Z$ has the required conditional distribution. The preceding stochastic-domination inequalities give
\[
\begin{aligned}
 \PP[|A_0\cap S|\le i-1\mid |A_0\cap S|\ge1]
 &\le\PP[|A_0\cap S|\le i-1],\\
 \PP[|A_0\cap S|\le i+1\mid |A_0\cap S|\ge1]
 &\ge\PP[|A_0\cap S|\le i].
\end{aligned}
\]
When $|A_0\cap S|=i$, the sampled number is greater than $\PP[|A_0\cap S|\le i-1]$ and at most $\PP[|A_0\cap S|\le i]$. The displayed inequalities therefore place it above the conditional cumulative probability at $i-1$ and at most that at $i+1$. Hence $Z\in\{i,i+1\}$, giving $|A_0\cap S|\le Z\le|A_0\cap S|+1$. The conditional law of $Z$ given $A_0$ depends only on $|A_0\cap S|$.

We now construct $A$ from $A_0$. If $Z=|A_0\cap S|$, then set $A=A_0$. Otherwise, remove a uniformly chosen pair from $A_0\setminus S$ and add a uniformly chosen pair from $S\setminus A_0$, choosing independently given $(A_0,Z)$. Both choices are possible because $Z=|A_0\cap S|+1\le\min\{d_{\hat F_\ast}(y),|S|\}$. The resulting set has $d_{\hat F_\ast}(y)$ pairs and satisfies $|A\cap S|=Z$. The law of $A_0$ and this construction are invariant under permutations within $S$ and within $\binom{V(\hat H_j)\setminus\{y\}}2\setminus S$. Hence, conditional on $Z$, the set $A$ is uniform among the subsets having $d_{\hat F_\ast}(y)$ pairs, exactly $Z$ of which belong to $S$. To identify the marginal law of $A$, fix $U\subseteq\binom{V(\hat H_j)\setminus\{y\}}2$ with $|U|=d_{\hat F_\ast}(y)$ and $U\cap S\ne\varnothing$. The conditional uniformity gives the first equality below. The second follows from the law of $Z$, which is the law of $|A_0\cap S|$ conditional on $A_0\cap S\ne\varnothing$, and the hypergeometric point probabilities for $|A_0\cap S|$.
\[
\begin{aligned}
 \PP[A=U]
 &=\frac{\PP[Z=|U\cap S|]}
 {\displaystyle\binom{|S|}{|U\cap S|}\binom{\binom{n_\ast-|\hat B|-1}{2}-|S|}{d_{\hat F_\ast}(y)-|U\cap S|}}
 =\frac{\PP[A_0=U]}{\PP[A_0\cap S\ne\varnothing]}
 =\PP[A_0=U\mid A_0\cap S\ne\varnothing].
\end{aligned}
\]
If $U\cap S=\varnothing$, then both distributions assign probability zero to $U$. Thus $A_0$ and $A$ have the required marginal distributions and differ by at most one pair replacement.

We next couple the larger links by extending $A_0$ and $A$ to sets of size $d_{\hat H_j}(y)$. If they differ, let $\pi$ exchange the unique element of $A_0\setminus A$ with the unique element of $A\setminus A_0$, fixing all other elements of $\binom{V(\hat H_j)\setminus\{y\}}2$; otherwise let $\pi$ be the identity. Given $(A_0,A)$, choose $L_0$ uniformly among the subsets of $\binom{V(\hat H_j)\setminus\{y\}}2$ having $d_{\hat H_j}(y)$ pairs and containing $A_0$. For each fixed pair $(A_0,A)$, the map $\pi$ is a fixed bijection between the extensions of $A_0$ and those of $A$. Thus, for any fixed set $U\subseteq\binom{V(\hat H_j)\setminus\{y\}}2$ with $|U|=d_{\hat H_j}(y)$, on the event $A\subseteq U$ we have $\PP[\pi(L_0)=U\mid A_0,A]=1/\binom{\binom{n_\ast-|\hat B|-1}{2}-d_{\hat F_\ast}(y)}{d_{\hat H_j}(y)-d_{\hat F_\ast}(y)}$. This value does not depend on $A_0$, so averaging over the conditional distribution of $A_0$ given $A$ gives $\PP[\pi(L_0)=U\mid A]=1/\binom{\binom{n_\ast-|\hat B|-1}{2}-d_{\hat F_\ast}(y)}{d_{\hat H_j}(y)-d_{\hat F_\ast}(y)}$. The conditional law of $L_0$ given $A_0$ is uniform among the extensions of $A_0$ having $d_{\hat H_j}(y)$ pairs. For every set $U\subseteq\binom{V(\hat H_j)\setminus\{y\}}2$ with $|U|=d_{\hat H_j}(y)$, there are $\binom{d_{\hat H_j}(y)}{d_{\hat F_\ast}(y)}$ possible subsets $A_0\subseteq U$, so
\[
 \PP[L_0=U]
 =\frac{\binom{d_{\hat H_j}(y)}{d_{\hat F_\ast}(y)}}{\binom{\binom{n_\ast-|\hat B|-1}{2}}{d_{\hat F_\ast}(y)}\binom{\binom{n_\ast-|\hat B|-1}{2}-d_{\hat F_\ast}(y)}{d_{\hat H_j}(y)-d_{\hat F_\ast}(y)}}
 = \frac{1}{\binom{\binom{n_\ast-|\hat B|-1}{2}}{d_{\hat H_j}(y)}}.
\]
Thus $L_0$ is uniform among all subsets of $\binom{V(\hat H_j)\setminus\{y\}}2$ having $d_{\hat H_j}(y)$ pairs. By the law of $A$ and the uniform extension property proved above, $\pi(L_0)$ has the conditional distribution of $L_{\hat H_j}(y)$ described in \cref{claim:conditional-links}. Finally, $\pi(L_0)$ differs from $L_0$ by at most one pair replacement, so their intersections with any fixed set of pairs differ in size by at most $1$.

The construction uses only the sample sizes and the requirement that the smaller set intersect $S$. In \cref{claim:prescribed-link-coupling}, we apply it after removing the prescribed link pairs $L_{\mathcal P}(y)$, with the remaining sample sizes specified in that proof.
\end{proof}

We now incorporate the prescribed triples in $\mathcal P$. Keep the previously revealed edge sets fixed, condition on $\mathcal P\subseteq\hat H_j$, and reveal $\mathcal P\cap\hat F_\ast$. We consider only $\QQ$-admissible values of this additional information. If a triple in $\mathcal P\cap\hat F_\ast$ contains the unique vertex whose degree in $\hat F_\ast-y$ is below its lower bound, then this prescribed triple supplies the required edge at that vertex. Otherwise, the remaining part of $L_{\hat F_\ast}(y)$ must contain a pair incident to that vertex.

\begin{claim}\label{claim:prescribed-link-coupling}
If $\mathfrak D_{\{y\}}(\hat F_\ast,\boldsymbol\ell')=1$, then under this conditional law, the link $L_{\hat H_j}(y)$ can be coupled with a uniformly chosen set $L_0$ of $d_{\hat H_j}(y)-|\mathcal P|$ pairs in $\binom{V(\hat H_j)\setminus\{y\}}2\setminus L_{\mathcal P}(y)$ so that, for every fixed set of pairs $\Gamma$,
\[
 \bigl||L_{\hat H_j}(y)\cap\Gamma|-|L_0\cap\Gamma|\bigr|\le h+1.
\]
\end{claim}
\begin{proof}[Proof of \cref{claim:prescribed-link-coupling}]
Remove the fixed pairs $L_{\mathcal P}(y)$ from the set of available pairs. The remaining parts of $L_{\hat H_j}(y)$ and $L_{\hat F_\ast}(y)$ have sizes $d_{\hat H_j}(y)-|\mathcal P|$ and $d_{\hat F_\ast}(y)-|\mathcal P\cap\hat F_\ast|$, respectively. Let $v\in V(\hat H_j)\setminus\{y\}$ be the unique vertex with $d_{\hat F_\ast-y}(v)=\ell'_v-1$. For every $u\in V(\hat H_j)\setminus\{y,v\}$, we have $d_{\hat F_\ast-y}(u)\ge\ell'_u$. The prescribed triples in $\hat F_\ast$ contribute $d_{\mathcal P\cap\hat F_\ast}(u)$ to the degree of each $u\ne y$, so
\[
 \sum_{u\in V(\hat H_j)\setminus\{y\}}
 \max\{0,\ell'_u-d_{\hat F_\ast-y}(u)-d_{\mathcal P\cap\hat F_\ast}(u)\}
 =\max\{0,1-d_{\mathcal P\cap\hat F_\ast}(v)\}.
\]
We distinguish the following two cases.
\begin{enumerate}[label=(\textit{\roman*})]
\item If $d_{\mathcal P\cap\hat F_\ast}(v)\ge1$, then every vertex other than $y$ already satisfies its degree lower bound using the revealed edges. As in the case $\mathfrak D_{\{y\}}=0$ above, the remaining larger link is uniform among the subsets of $\binom{V(\hat H_j)\setminus\{y\}}2\setminus L_{\mathcal P}(y)$ having $d_{\hat H_j}(y)-|\mathcal P|$ pairs. Take $L_0$ to be this remaining link; no replacement is needed.
\item If $d_{\mathcal P\cap\hat F_\ast}(v)=0$, then the remaining smaller link must contain a pair incident to $v$. By \cref{claim:conditional-links}, after conditioning on the additional information, the two remaining links are uniform among all nested choices of the specified sizes on $\binom{V(\hat H_j)\setminus\{y\}}2\setminus L_{\mathcal P}(y)$ for which the smaller link contains such a pair. The $\QQ$-admissibility of the exposure ensures that this requirement is feasible. The coupling construction in \cref{claim:link-coupling}, with the set of available pairs and the sample sizes adjusted as above, gives a uniform sample $L_0$ that differs from the remaining larger link by at most one pair replacement.
\end{enumerate}
In either case, the number of pairs in $\Gamma$ differs by at most $1$ between the remaining larger link and $L_0$. Adding back the fixed link pairs $L_{\mathcal P}(y)$ changes the intersection size by at most $|L_{\mathcal P}(y)|=|\mathcal P|\le h$, proving the claimed bound.
\end{proof}

It remains to apply the hypergeometric tail bound in the case $\mathfrak D_{\{y\}}(\hat F_\ast,\boldsymbol\ell')=1$. Keep the four original revealed edge sets, the event $\mathcal P\subseteq\hat H_j$, and the value of $\mathcal P\cap\hat F_\ast$ fixed as above. The set $\Gamma$ is fixed before the remaining link pairs are sampled. Take the uniform comparison sample $L_0$ supplied by \cref{claim:prescribed-link-coupling}; its size lies between $d_{\hat H_j}(y)-h$ and $d_{\hat H_j}(y)$. If $|L_{\hat H_j}(y)\setminus\Gamma|<d_{\hat H_j}(y)/3$, then the coupling inequality in \cref{claim:prescribed-link-coupling} gives
\[
 |L_0\cap\Gamma|
 \ge |L_{\hat H_j}(y)\cap\Gamma|-(h+1)
 >\frac{2d_{\hat H_j}(y)}3-(h+1)
 \ge\frac{d_{\hat H_j}(y)}2
\]
for sufficiently large $n_\ast$. Under $\PP$, $|L_0\cap\Gamma|$ is hypergeometric and
\[
 \EE[|L_0\cap\Gamma|]
 \le\frac{d_{\hat H_j}(y)|\Gamma|}{\binom{n_\ast-|\hat B|-1}{2}-h}
 \le\frac{2d_{\hat H_j}(y)}{(\log\log\log n_\ast)^{1/16}}.
\]
Applying \cref{lem:chernoff}(ii) under $\PP$ gives
\[
\begin{aligned}
 \PP\left[|L_{\hat H_j}(y)\setminus\Gamma|<\frac{d_{\hat H_j}(y)}3\right]
 &\le\PP\left[|L_0\cap\Gamma|\ge\frac{d_{\hat H_j}(y)}2\right]\\
 &\le\left(\frac{4\e}{(\log\log\log n_\ast)^{1/16}}\right)^{d_{\hat H_j}(y)/2}\\
 &=\exp\left\{-\frac{d_{\hat H_j}(y)}{32}\log\log\log\log n_\ast+O(d_{\hat H_j}(y))\right\}\\
 &\le\exp\{-c(\log n_\ast)\log\log\log\log n_\ast\},
\end{aligned}
\]
for some fixed $c>0$ and all sufficiently large $n_\ast$. If the mean is zero, the failure probability is zero. Since the coupled link has the required conditional law, this proves the conditional bound when $\mathfrak D_{\{y\}}(\hat F_\ast,\boldsymbol\ell')=1$.

Both cases therefore give the same bound conditional on the four original revealed edge sets, the event $\mathcal P\subseteq\hat H_j$, and any $\QQ$-admissible value of $\mathcal P\cap\hat F_\ast$. Averaging over $\mathcal P\cap\hat F_\ast$ preserves this bound. In the following calculation, all probabilities are conditional on the four original revealed edge sets $F_*(\hat B),H_j(\hat B),\hat F_\ast-y,\hat H_j-y$. If $\mathcal P\subseteq\hat H_j$ has conditional probability zero, then the joint probability is zero. Otherwise,
\begin{multline*}
 \QQ\left[\mathcal P\subseteq\hat H_j,
 |L_{\hat H_j}(y)\setminus\Gamma|<\frac{d_{\hat H_j}(y)}3\right]
 =\QQ[\mathcal P\subseteq\hat H_j]\,
 \QQ\left[|L_{\hat H_j}(y)\setminus\Gamma|<\frac{d_{\hat H_j}(y)}3
 \,\middle|\,\mathcal P\subseteq\hat H_j\right]\\
 \le\QQ[\mathcal P\subseteq\hat H_j]
 \exp\{-c(\log n_\ast)\log\log\log\log n_\ast\}
 \le\exp\{-c(\log n_\ast)\log\log\log\log n_\ast\}.
\end{multline*}
Every realisation of the event in \cref{lem:simultaneous-link} satisfies the degree and $|\Gamma|$ bounds imposed above and falls into one of the two cases. Since the conditional joint bound is uniform over these values of the four original revealed edge sets, averaging over them proves the lemma.
\end{proof}

\section{Edge resampling for mixed Hamilton cycles}
\label{sec:resampling}
For a fixed prescribed matching $\cM$, we compare the numbers of mixed Hamilton cycles as the deleted vertex $w$ and the added pair $uv$ with prescribed direction $u\to v$ vary. The three vertices $w,u,v$ need not form an edge of the host hypergraph. In the conditioned edge-deletion process, we show that, with high probability, the bounds in \cref{thm:terminal-regularity}(i)--(ii) and \cref{prop:path-regularity}, together with the required lower bound on the number of these mixed Hamilton cycles, ensure that these numbers are nearly equal for most choices of the three vertices.

We work with the conditioned edge-deletion process $(H_j)$ under $\QQ$ and the fixed set $\Kc$ of $\cM_0$-compatible triples. Whenever we consider a prescribed matching $\cM$, we also choose a root pair in $\cM$ and fix an arbitrary direction on it. Every counted cycle is given its root orientation. From now on, $\comp_{H,\cM}(x;y\to z)$ counts only cycles using triples in $H\cap\Kc$. The total number of cycles with the added pair $yz$ is $\cyc((H\cap\Kc)-x,\cM\cup\{yz\})$. Adding an auxiliary pair does not change $\Kc$, which remains defined by the original matching $\cM_0$. We compute $d(H)$ using all triples of $H$, before intersecting $H$ with $\Kc$.

We fix the constants in \cref{prop:path-regularity} with $h=8$ and $L=2$. The bounds in \cref{thm:terminal-regularity}(i)--(ii) and \cref{prop:path-regularity} will be included in the event whose probability we estimate.

We state the bound for $\comp_{H,\cM}(w;u\to v)$ uniformly over choices of deleted vertices and prescribed pairs. We record these choices, together with the prescribed directions, in an auxiliary frame.

An \emph{auxiliary frame} $\Lambda$ is specified by the following choices.
\begin{itemize}[beginpenalty=10000,midpenalty=10000]
\item A set $\hat B\subseteq[n_\ast]$ of vertices to delete from $H_j$.
\item A matching $\cM$ on $[n_\ast]\setminus\hat B$.
\item A chosen root pair in $\cM$, together with an arbitrary fixed direction.
\item When needed, a pair $uv\in\cM$ other than the root pair, with prescribed direction $u\to v$ in the root orientation.
\end{itemize}
Cycles in the frame use triples only from $(H_j\cap\Kc)-\hat B$, where $\Kc$ remains the set of $\cM_0$-compatible triples. We consider frames satisfying
\begin{equation}\label{eq:frame-bounds-final}
 |\hat B|\le8,\quad
 |\cM\setminus\cM_0|\le2,\quad\text{and}\quad
 |\cM_0\setminus\cM|\le2.
\end{equation}
We require $\cM\ne\varnothing$ so that a root pair can be chosen. We also require $n_\ast-|\hat B|-|\cM|$ to be even, since every mixed Hamilton cycle in the frame contains $(n_\ast-|\hat B|-|\cM|)/2$ triples.

For a frame $\Lambda$ at time $j$, use its deleted vertex set $\hat B$ in the definitions of $\hat F_\ast$ and $\hat H_j$, and put $m\coloneqq|\hat H_j|$. For a host $F$ on $[n_\ast]\setminus\hat B$, write $\cyc_{\Lambda}(F)$ for the number of $(F\cap\Kc,\cM)$-respecting mixed Hamilton cycles satisfying the prescribed directions of pairs in $\cM$ in the root orientation. When a cycle is sampled from this family, the directions of the remaining pairs in $\cM$ are determined by that cycle.

The frame construction and deletion of an edge set $T\subseteq\hat H_j$ are illustrated in \cref{fig:batch-deletion}.

\begin{figure}[htb]
\centering
\begin{tikzpicture}[
  x=.86cm,y=.86cm,font=\small,
  vertex/.style={circle,fill=black,inner sep=1.4pt},
  Fstar triple/.style={draw=blue!65!black,fill=blue!12,line width=.7pt},
  current triple/.style={draw=black!50,fill=black!5,line width=.55pt},
  selected triple/.style={draw=red!80!black,densely dashed,line width=1.15pt},
  set boundary/.style={draw=black!30,densely dashed,line width=.4pt}
]
  \foreach \stage/\xoffset/\yoffset/\heading in {
    1/0/0/{1. Current hypergraph $H_j$},
    2/8.4/0/{2. Frame on $\hat H_j=H_j-\hat B$},
    3/8.4/-7.2/{3. After deleting $T$},
    4/0/-7.2/{4. A cycle counted by $\cyc_{\Lambda}(\hat H_j-T)$}} {
    \begin{scope}[shift={(\xoffset,\yoffset)}]
      \node at (0,3.3) {\heading};
      \foreach \i in {0,...,8}
        \coordinate (j\i) at ({90-40*\i}:1.6);
      \foreach \i in {1,2,3,5,6,7,8}
        \coordinate (p\i) at ({70-40*\i}:2.02);
      \coordinate (d0) at (-2.52,.7);
      \coordinate (d1) at (-2.52,-.7);
      \ifnum\stage=1
        \path[fill=black!4,even odd rule]
          (0,0) circle (2.85) (0,0) circle (2.18);
        \draw[set boundary] (0,0) circle (2.85);
        \node at (-1.95,1.9) {$\hat B$};
        \draw[Fstar triple] (d0)--(p7)--(j7)--cycle;
        \draw[current triple] (d1)--(p6)--(j6)--cycle;
        \draw[violet!80!black,line width=1.5pt] (d0)--(d1);
        \foreach \point in {d0,d1}
          \node[vertex] at (\point) {};
      \fi
      \draw[set boundary] (0,0) circle (2.18);
      \foreach \i/\next/\edgestyle in {
        1/2/Fstar triple,
        2/3/current triple,
        3/4/Fstar triple,
        5/6/current triple,
        6/7/Fstar triple,
        7/8/current triple,
        8/0/Fstar triple} {
        \draw[\edgestyle] (j\i)--(p\i)--(j\next)--cycle;
      }
      \ifnum\stage<4
        \draw[Fstar triple] (j6)--(p5)--(p6)--cycle;
        \draw[current triple] (j2)--(p1)--(p2)--cycle;
      \fi
      \ifnum\stage<3
        \draw[Fstar triple] (j3)--(p2)--(p3)--cycle;
        \draw[current triple] (j7)--(p6)--(p7)--cycle;
      \fi
      \ifnum\stage=2
        \draw[selected triple] (j3)--(p2)--(p3)--cycle;
        \draw[selected triple] (j7)--(p6)--(p7)--cycle;
      \fi
      \ifnum\stage=1
        \draw[violet!80!black,line width=1.5pt] (j0)--(j1);
        \node at (0,.12) {$[n_\ast]\setminus\hat B$};
        \node at (0,-3.25) {original pairs $\cM_0$};
      \else
        \draw[violet!80!black,line width=1.5pt,shorten >=2.2pt,-{Stealth[length=5.5pt,width=4.5pt]}]
          (j0)--(j1);
        \draw[orange!85!black,line width=1.5pt] (j4)--(j5);
        \node at (0,.12) {$[n_\ast]\setminus\hat B$};
        \ifnum\stage=2
          \node[blue!65!black] at (0,-2.6) {$\hat F_\ast=F_*-\hat B$};
        \fi
        \ifnum\stage=3
          \node at (0,-2.6) {$\hat F_\ast-T\subseteq\hat H_j-T$};
        \fi
        \ifnum\stage=4
          \node at (0,-2.6) {Mixed Hamilton cycle};
        \fi
      \fi
      \foreach \i in {0,...,8}
        \node[vertex] at (j\i) {};
      \foreach \i in {1,2,3,5,6,7,8}
        \node[vertex] at (p\i) {};
    \end{scope}
  }
  \draw[-{Stealth[length=5pt]},line width=.7pt] (3.05,0)--(5.85,0)
    node[midway,above=4pt] {delete $\hat B$}
    node[midway,below=4pt] {choose $\cM$};
  \draw[-{Stealth[length=5pt]},line width=.7pt] (8.4,-3.05)--(8.4,-3.65)
    node[midway,right=5pt] {delete $T\subseteq \hat H_j$};
  \draw[-{Stealth[length=5pt]},line width=.7pt] (5.85,-7.2)--(2.55,-7.2)
    node[midway,above=4pt] {a surviving cycle};
\end{tikzpicture}
\caption{An auxiliary frame and deletion of $T$. Blue triangles are triples of $F_*$; grey triangles belong to $H_j\setminus F_*$. Deleting $\hat B$ removes the outer pair of $\cM_0$, and adding the orange pair gives $\cM$; the purple arrow fixes the direction of the root pair. Red outlines mark the triples in $T$, sampled from $\hat H_j$. Deleting $T$ leaves $\hat F_\ast-T\subseteq\hat H_j-T$ with the frame unchanged. The last diagram illustrates a cycle counted by $\cyc_{\Lambda}(\hat H_j-T)$. Only selected host triples are shown.}
\label{fig:batch-deletion}
\end{figure}
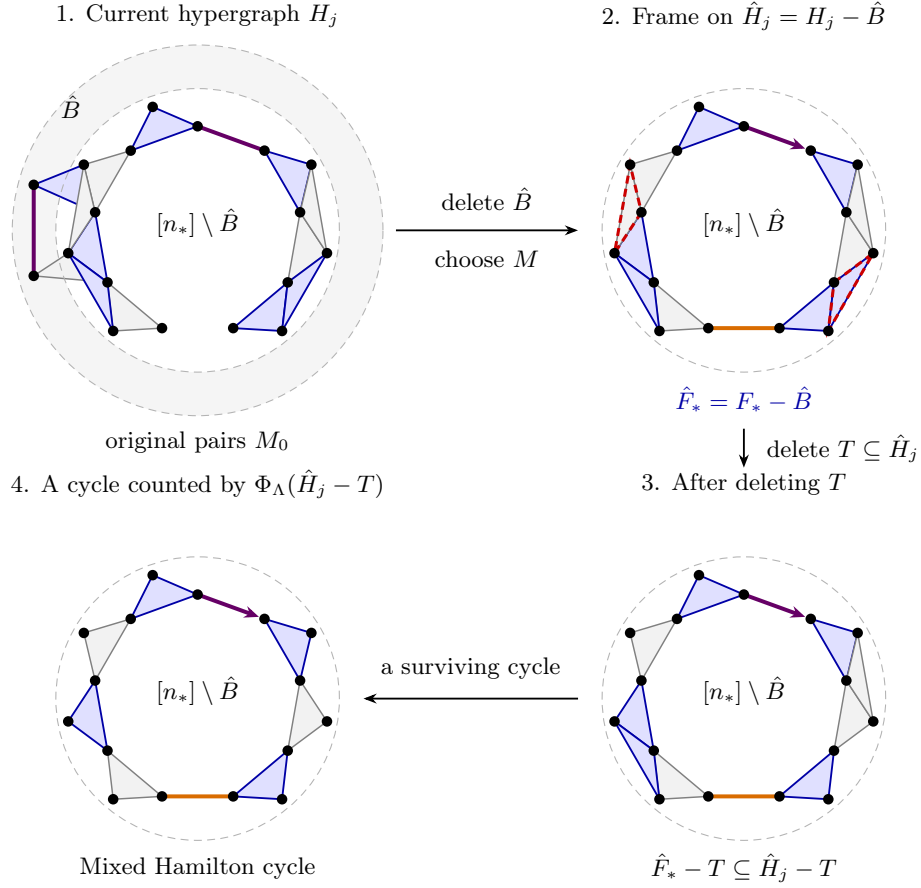

We define the set of \emph{legal directed candidates} by
\[
 \mathcal Q_{\Lambda}\coloneqq\left\{(w;u\to v)\colon w,u,v\in[n_\ast]\setminus\bigl(V(\cM)\cup V(\cM_0)\cup\hat B\bigr)\text{ are distinct}\right\}.
\]
For $(w;u\to v)\in\mathcal Q_{\Lambda}$, let $\comp_{F,\cM}(w;u\to v)$ count the $\bigl((F\cap\Kc)-w,\cM\cup\{uv\}\bigr)$-respecting mixed Hamilton cycles that satisfy every pair direction prescribed in the frame $\Lambda$ and traverse $uv$ from $u$ to $v$ in the root orientation.

The complete $3$-graph $F$ on $[n_\ast]\setminus\hat B$ has average degree $d(F)=(1+o(1))n_\ast^2/2$. For mixed cycles in $F$, with all triples allowed and only the direction of the root pair prescribed, applying \eqref{eq:Yratio} with $n=n_\ast-|\hat B|$ and $s=|\cM|$ gives, for every $(w;u\to v)\in\mathcal Q_\Lambda$,
\[
 \frac{\cyc\bigl(F-w,\cM\cup\{uv\}\bigr)}{2\cyc(F,\cM)}
 =\frac{1+o(1)}{4d(F)}.
\]
The factor $1/2$ accounts for the two equally sized direction classes of $uv$ in the root orientation. We use this value as the reference scale for normality in an auxiliary frame.

For $\eta>0$ and $\cyc_{\Lambda}(F)>0$, a legal candidate $(w;u\to v)$ is \emph{$\eta$-normal in $F$} if
\[
 \left|\frac{4d(F)\comp_{F,\cM}(w;u\to v)}{\cyc_{\Lambda}(F)}-1\right|\le\eta,
\]
and \emph{$\eta$-abnormal in $F$} otherwise.

The next event records that the bounds in \cref{thm:terminal-regularity}(i)--(ii) and \cref{prop:path-regularity} hold, the number $\cyc_\Lambda(\hat H_j)$ satisfies \eqref{eq:frame-large}, and too many legal candidates are abnormal. For a permitted frame $\Lambda$ and time $m_\ast\le j\le\binom{n_\ast}{3}$, let $\mathsf E_{\Lambda,j}$ be the event that all of the following hold.
\begin{enumerate}[label=(E\arabic*),ref=E\arabic*]
\item\label{item:completion-event-regularity} The hypergraph $F_*=H_{m_\ast}$ satisfies the bounds in \cref{thm:terminal-regularity}(i)--(ii), and the conclusions of \cref{prop:path-regularity}(i)--(iii), with $h=8$ and $L=2$, hold simultaneously at every time and after deleting any set of at most eight vertices.
\item\label{item:completion-event-count} The number $\cyc_{\Lambda}(\hat H_j)$ of mixed Hamilton cycles in the frame satisfies
\begin{equation}\label{eq:frame-large}
 \log \cyc_{\Lambda}(\hat H_j)\ge \frac{n_\ast-|\hat B|-|\cM|}{2}\log\bigl(2d(\hat H_j)\bigr)-n_\ast+|\hat B|+|\cM|-\frac{10n_\ast}{\sqrt{\log n_\ast}}.
\end{equation}
\item\label{item:completion-event-abnormal} More than $|\mathcal Q_{\Lambda}|/\sqrt{\log\log\log n_\ast}$ legal directed candidates $(w;u\to v)\in\mathcal Q_{\Lambda}$ are $1/\sqrt{\log\log\log n_\ast}$-abnormal in $\hat H_j$; that is, they satisfy
\[
 \left|\frac{4d(\hat H_j)\comp_{\hat H_j,\cM}(w;u\to v)}{\cyc_{\Lambda}(\hat H_j)}-1\right|>\frac{1}{\sqrt{\log\log\log n_\ast}}.
\]
\end{enumerate}

For each frame $\Lambda$, we adopt the notation $\hat H_j=H_j-\hat B$ and $\hat F_\ast=F_*-\hat B$ for the hypergraphs induced by $H_j$ and $F_*$ on $[n_\ast]\setminus\hat B$, respectively.

The following proposition extends the estimates for directed edge marginals to the numbers $\comp_{\hat H_j,\cM}(w;u\to v)$ for legal candidates, including those whose underlying triples are absent. We adapt the forward and reverse sampling argument of Kahn~\cite[Sections~10--12]{Kahn2022}.
\begin{proposition}\label{prop:uniform-completion}
There exists an absolute constant $c>0$ such that the following holds for all sufficiently large $n_\ast$. For every permitted frame $\Lambda$, every time $m_\ast\le j\le\binom{n_\ast}{3}$, and every $\QQ$-admissible value of $(F_*(\hat B),H_j(\hat B))$, we have
\[
 \QQ[\mathsf E_{\Lambda,j}\mid F_*(\hat B),H_j(\hat B)]\le\exp\{-c(\log n_\ast)\sqrt{\log\log\log n_\ast}\}.
\]
\end{proposition}

\begin{proof}
We bound the probability of $\mathsf E_{\Lambda,j}$ by showing that too many abnormal candidates would typically produce a sampling discrepancy that is exponentially unlikely. Specifically, if $\mathsf E_{\Lambda,j}$ occurs, deleting a uniformly chosen set $T\subseteq\hat H_j$ typically leaves many missing triples with an abnormal legal orientation in $\hat H_j-T$, but far fewer such triples in $T$ than uniform sampling from the missing triples would predict. We compare the forward and reverse descriptions of this same deletion experiment.
\begin{enumerate}[label=(\textit{\roman*})]
\item In the forward argument, we fix a realisation of the edge-deletion process in $\mathsf E_{\Lambda,j}$ and delete a uniformly chosen set $T\subseteq\hat H_j$ of triples of a prescribed size. We show that this discrepancy occurs with probability $1-o(1)$.
\item In the reverse argument, we fix $\hat H_j-T$ and $\hat F_\ast-T$. If $\hat F_\ast-T$ satisfies the degree lower bounds $\ell'_v$, then $T$ is uniform among the sets of missing triples of the prescribed size. A hypergeometric tail estimate shows that the discrepancy has exponentially small probability.
\end{enumerate}

Condition first on the exposed edge sets $F_*(\hat B)$ and $H_j(\hat B)$. By \cref{lem:boundary}(i)--(ii), the two induced hypergraphs have the following conditional distributions.
\begin{itemize}
\item $\hat F_\ast$ is uniform among simple $3$-graphs on $[n_\ast]\setminus\hat B$ with $m_\ast-|F_*(\hat B)|$ edges and degree lower bounds $\ell'_v$.
\item Given $\hat F_\ast$, the hypergraph $\hat H_j$ is uniform among all $m$-edge $3$-graphs on $[n_\ast]\setminus\hat B$ that contain $\hat F_\ast$.
\end{itemize}

To prepare for the forward estimate, we first apply the entropy bound to candidates whose underlying triples are edges of $\hat H_j$. The next claim bounds the exceptional proportion using \cref{thm:entropy-stability} and the explicit bounds on its parameters stated below.
\begin{claim}\label{fact:role-exceptions}
Fix $C\ge1$ and suppose that the hypotheses of \cref{thm:entropy-stability} hold with this constant. Suppose also that
\[
 \xi\le\frac{C}{(\log n)^{1/2}},\quad
 \delta\le\frac{C}{(\log n)^{1/8}},\quad
 \frac{n\Delta_2(\mathcal{G})}{m}\le\frac{C}{(\log n)^{1/4}},
 \quad\text{and}\quad s\le Cn^{1/10}.
\]
For every $\eta>0$, the fraction of oriented triples $(w;u\to v)\in\vec{\mathcal{G}}-V(\cM)$ satisfying $|\pi(w;u\to v)-n/(12m)|>\eta n/(12m)$ is $O_C(1/(\eta\sqrt{\log\log n}))$.
\end{claim}
\begin{proof}[Proof of \cref{fact:role-exceptions}]
By the codegree hypothesis, $\log\bigl(3m/(n\Delta_2(\mathcal G))\bigr)\ge\tfrac14\log\log n+\log(3/C)$. Since $\xi+\delta+s/n=o(1/\log\log n)$, the expression under the square root in \cref{thm:entropy-stability} is $O_C(1/\log\log n)$. Consequently, the total deviation satisfies
\[
 \sum_{(w;u\to v)\in\vec{\mathcal{G}}-V(\cM)}
 \left|\pi(w;u\to v)-\frac{n}{12m}\right|
 =O_C\left(\frac{n}{\sqrt{\log\log n}}\right).
\]
Each exceptional oriented triple contributes more than $\eta n/(12m)$ to this sum, so the number of such triples is at most
\[
 \frac{12m}{\eta n}\cdot O_C\left(\frac{n}{\sqrt{\log\log n}}\right)
 =O_C\left(\frac{m}{\eta\sqrt{\log\log n}}\right).
\]
At most $2s\Delta_1(\mathcal G)=o(m)$ host triples meet $V(\cM)$, so $|\vec{\mathcal G}-V(\cM)|=(6+o(1))m$. The proportion of exceptional oriented triples is therefore $O_C(1/(\eta\sqrt{\log\log n}))$.
\end{proof}

We next apply \cref{fact:role-exceptions} to $\hat H_j$, to control abnormal candidates whose underlying triples are edges of $\hat H_j$.
\begin{claim}\label{claim:present-candidates}
For every realisation in $\mathsf E_{\Lambda,j}$, the number of legal directed candidates $(w;u\to v)\in\mathcal Q_{\Lambda}$ with $\{u,v,w\}\in\hat H_j$ that are $1/(100\sqrt{\log\log\log n_\ast})$-abnormal in $\hat H_j$ is $O\Bigl(m\sqrt{\frac{\log\log\log n_\ast}{\log\log n_\ast}}\Bigr)=o\left(\frac{m}{\sqrt{\log\log\log n_\ast}}\right)$.
\end{claim}
\begin{proof}[Proof of \cref{claim:present-candidates}]
Fix a realisation in $\mathsf E_{\Lambda,j}$ and choose $\mathbf C$ uniformly from the family counted by $\cyc_{\Lambda}(\hat H_j)$. We will apply \cref{fact:role-exceptions} to its directed edge marginals. In this application, take $n=n_\ast-|\hat B|$, $m=|\hat H_j|$, and $s=|\cM|$.

Every legal candidate $(w;u\to v)$ has $\{u,v,w\}\in\Kc$. If $\{u,v,w\}\in\hat H_j$, then contracting this triple to the pair $uv$ in a cycle that uses it from $u$ to $v$ with $w$ private preserves all prescribed pair directions. Thus its directed edge marginal satisfies
\[
 \pi(w;u\to v)=\frac{\comp_{\hat H_j,\cM}(w;u\to v)}{\cyc_{\Lambda}(\hat H_j)}.
\]
Since $d(\hat H_j)=3m/n$, the comparison value $n/(12m)$ in \cref{fact:role-exceptions} equals $1/(4d(\hat H_j))$, the value used to define normality in the frame $\Lambda$. Hence the exceptional legal oriented triples in \cref{fact:role-exceptions} are exactly the abnormal legal candidates whose underlying triples belong to $\hat H_j$, at the same tolerance. It remains to verify the hypotheses of \cref{fact:role-exceptions} for this host and cycle distribution.

We first check the entropy lower bound in \cref{thm:entropy-stability}(iii). \Cref{thm:entropy-stability} allows an arbitrary distribution on mixed Hamilton cycles, including the uniform distribution on this subfamily using $\cM_0$-compatible triples and satisfying the prescribed pair directions. Since the distribution is uniform, condition (\ref{item:completion-event-count}) and $2d(\hat H_j)=6m/n$ give
\[
 \HH[\mathbf C]=\log\cyc_{\Lambda}(\hat H_j)
 \ge\frac{n-s}{2}\log\frac{6m}{n}-(n-s)-\frac{10n_\ast}{\sqrt{\log n_\ast}}.
\]
Thus the required entropy lower bound holds with
\[
 \xi=\frac{10n_\ast}{n\sqrt{\log n_\ast}}
 =\frac{10n_\ast}{(n_\ast-|\hat B|)\sqrt{\log n_\ast}}
 =O\left(\frac{1}{\sqrt{\log n_\ast}}\right),
\]
where the last estimate uses $|\hat B|\le8$.

We next check the degree and codegree bounds in \cref{thm:entropy-stability}(i). Condition (\ref{item:completion-event-regularity}) and $|\hat B|\le8$ give $j-m=|H_j(\hat B)|\le|\hat B|\Delta_1(H_j)\le24Cj/n_\ast$. Thus $m=j(1-O(1/n_\ast))$. Substituting this comparison and $n=n_\ast-|\hat B|$ into \cref{prop:path-regularity}(i) gives $\Delta_1(\hat H_j)=O(m/n)$ and $n\Delta_2(\hat H_j)/m=O((\log n_\ast)^{-1/4})$. Also, the frame bounds give $|\cM|\le|\cM_0|+2\le n_\ast^{1/10}+2$, so $s=O(n^{1/10})$.

For \cref{thm:entropy-stability}(ii), the required sets $A$ have size $(n+s)/2=(n_\ast-|\hat B|+|\cM|)/2$. Since $|\cM|\le n_\ast^{1/10}+2$, we have
\[
 \left||A|-\frac{n_\ast-|\hat B|}{2}\right|
 =\frac{|\cM|}{2}
 \le\frac{n_\ast^{1/10}+2}{2}
 \le2n_\ast^{1/10}
\]
for sufficiently large $n_\ast$. Thus \cref{prop:path-regularity}(ii) applies with $L=2$. Since $|3j/8-3m/8|=O(j/n_\ast)$, replacing the reference value $3j/8$ by $3m/8$ preserves the error bound required in \cref{thm:entropy-stability}(ii). Thus this condition holds with $\delta=O((\log n_\ast)^{-1/8})$. These estimates satisfy all the parameter bounds in \cref{fact:role-exceptions} with an absolute fixed constant.

We may now apply \cref{fact:role-exceptions} with tolerance $1/(100\sqrt{\log\log\log n_\ast})$. Since there are $O(m)$ oriented triples and $n=n_\ast-O(1)$, the number of abnormal legal candidates whose underlying triples belong to $\hat H_j$ is $O\left(m\sqrt{\frac{\log\log\log n_\ast}{\log\log n_\ast}}\right)$, completing the proof.
\end{proof}

Keep the exposed edge sets $F_*(\hat B)$ and $H_j(\hat B)$ fixed. We use $\widehat{\PP}^{*}$ to denote the joint law of the edge-deletion process and the set $T$ obtained from the following experiment.
\begin{itemize}
\item First sample a realisation $(H_h)_{h=m_\ast}^{\binom{n_\ast}{3}}$ of the edge-deletion process according to $\QQ[\cdot\mid F_*(\hat B),H_j(\hat B)]$.
\item Given this realisation, choose $T$ uniformly from all subsets of $\hat H_j=H_j-\hat B$ containing exactly $\bigl\lfloor m\log\log\log n_\ast/\bigl(50(n_\ast-|\hat B|-|\cM|)\bigr)\bigr\rfloor$ triples. For realisations in $\mathsf E_{\Lambda,j}$, this prescribed size is $(1+o(1))\frac{m\log\log\log n_\ast}{50n_\ast}$.
\end{itemize}

We now delete the triples in $T$ from $\hat H_j$. Note that this removes precisely $T\cap\hat F_\ast$ from its subhypergraph $\hat F_\ast$, so the remaining hypergraphs satisfy $\hat F_\ast-T\subseteq\hat H_j-T$. Since $\hat H_j-T$ has $m-|T|$ edges on the same $n_\ast-|\hat B|$ vertices, its average degree is $d(\hat H_j-T)=3(m-|T|)/(n_\ast-|\hat B|)$. \Cref{fig:batch-deletion} illustrates the frame, the deletion of $T$ in the forward experiment, and a surviving mixed Hamilton cycle.

To describe reverse sampling, we first specify the set of all triples missing from $\hat H_j-T$. Let $\Omega_{\hat H_j-T}\coloneqq\binom{[n_\ast]\setminus\hat B}{3}\setminus(\hat H_j-T)$ be the set of its missing triples. Note that $T\subseteq\Omega_{\hat H_j-T}$. When $\cyc_\Lambda(\hat H_j-T)>0$, define
\[
 \mathcal A\coloneqq
 \left\{e\in\Omega_{\hat H_j-T}\colon
 \text{some orientation of }e\text{ in }\mathcal Q_\Lambda\text{ is }\frac{1}{4\sqrt{\log\log\log n_\ast}}\text{-abnormal in }\hat H_j-T
 \right\}.
\]
When $\cyc_\Lambda(\hat H_j-T)=0$, set $\mathcal A\coloneqq\varnothing$.

Let $\mathsf W$ be the event that the following conditions hold.
\begin{enumerate}[label=(W\arabic*),ref=W\arabic*]
\item\label{item:reverse-event-degree-bounds} $d_{\hat F_\ast-T}(v)\ge\ell'_v$ for every $v\in[n_\ast]\setminus\hat B$.
\item\label{item:reverse-event-nonedges} $|\mathcal A|\ge \frac{|\Omega_{\hat H_j-T}|}{3\sqrt{\log\log\log n_\ast}}$.
\item\label{item:reverse-event-sampled} $|T\cap\mathcal A|\le \frac{|T|}{100\sqrt{\log\log\log n_\ast}}$.
\end{enumerate}

The event $\mathsf W$ requires $T$ to contain far fewer triples in $\mathcal A$ than their proportion among all missing triples would predict under uniform sampling. We also require $\hat F_\ast-T$ to satisfy the degree lower bounds, so that adding back deleted triples imposes no further degree restriction. Under this condition, \cref{prop:reverse-uniform}(\ref{item:reverse-set}) will show that, conditional on $\hat H_j-T$ and $\hat F_\ast-T$ with the exposed edge sets fixed, $T$ is uniform among sets of missing triples of the prescribed size.

The set $\mathcal A$ is determined by $\hat H_j-T$ and the frame $\Lambda$.

We now establish the forward estimate for $\mathsf W$. Fix a $\QQ$-admissible realisation in $\mathsf E_{\Lambda,j}$, and take probability only over the uniform choice of $T$. We first control the numbers of mixed Hamilton cycles after deletion using the overlap estimate in \cref{lem:overlap} and the variance bound in \cref{lem:batch-variance}, then show that many abnormal candidates correspond to missing triples, bound the number of abnormal orientations among the triples of $T$, and finally verify the degree lower bounds for $\hat F_\ast-T$. We apply \cref{lem:batch-variance} with $U=\hat H_j$, representing each cycle by its set of triples and omitting the prescribed pairs. For the following survival estimates, take $h=(n_\ast-|\hat B|-|\cM|)/2$ in the notation of \cref{lem:batch-variance}. Thus cycles in the frame have $h$ triples, whereas those counted by $\comp_{\hat H_j,\cM}(w;u\to v)$ have $h-1$.

For these forward estimates, call a legal candidate $(w;u\to v)$ \emph{light} if $\comp_{\hat H_j,\cM}(w;u\to v)<\cyc_{\Lambda}(\hat H_j)n_\ast^{-50}$. For the family counted by $\cyc_{\Lambda}(\hat H_j)$, or the family counted by $\comp_{\hat H_j,\cM}(w;u\to v)$ for a candidate that is not light, let $Y_T$ be the number of its cycles whose triples are disjoint from $T$.

\begin{claim}\label{claim:batch-stability}
With probability $1-o(1)$, the number of cycles in the frame satisfies
\begin{equation}\label{eq:X-stable-final}
 \cyc_{\Lambda}(\hat H_j-T)=\left(1\pm\frac{1}{10^4\sqrt{\log\log\log n_\ast}}\right)\frac{\binom{m-h}{|T|}}{\binom{m}{|T|}}\cyc_{\Lambda}(\hat H_j).
\end{equation}
For each of the cycle families specified above, the following deviation bound holds:
\begin{multline}\label{eq:count-concentration}
 \widehat{\PP}^{*}\left[
 |Y_T-\EE[Y_T\mid(H_i)_i]|>\frac{\EE[Y_T\mid(H_i)_i]}{10^4\sqrt{\log\log\log n_\ast}}
 \,\middle|\,(H_i)_i\right]\\
 \le10^8(\log\log\log n_\ast)\frac{\Var[Y_T\mid(H_i)_i]}{(\EE[Y_T\mid(H_i)_i])^2}
 =O\left(\frac{\log\log\log n_\ast}{(\log\log n_\ast)^{24/25}}\right)
 =o\left( \frac{1}{(\log\log\log n_\ast)^{1/2}} \right).
\end{multline}
\end{claim}
\begin{proof}[Proof of \cref{claim:batch-stability}]
By Chebyshev's inequality, it suffices to bound the relative variance $\Var[Y_T\mid(H_i)_i]/(\EE[Y_T\mid(H_i)_i])^2$ by $O((\log\log n_\ast)^{-24/25})$. By \cref{lem:batch-variance} and the prescribed size of $T$, this follows from a bound of order $n_\ast/\log d(\hat H_j)$ on the expected number of triples shared by two independent uniform cycles from the family under consideration. We therefore check the hypotheses of \cref{lem:overlap} for the family counted by $\cyc_{\Lambda}(\hat H_j)$ and each family counted by $\comp_{\hat H_j,\cM}(w;u\to v)$ for a candidate that is not light. The overlap estimate in \cref{lem:overlap} applies to any nonempty subfamily satisfying the prescribed pair directions, provided the host satisfies the degree hypotheses and the subfamily satisfies the logarithmic size lower bound.

For a candidate $(w;u\to v)$, the cycles counted by $\comp_{\hat H_j,\cM}(w;u\to v)$ lie in $(\hat H_j\cap\Kc)-w$ with matching $\cM\cup\{uv\}$ and have $h-1$ triples. Every triple in $\hat H_j\setminus\Kc$ contains at least two vertices of $V(\cM_0)$, which has $2s$ vertices, so there are at most $s\Delta_1(\hat H_j)$ such triples. Deleting these triples and then the vertex $w$ therefore removes $|\hat H_j\setminus\Kc|+d_{\hat H_j\cap\Kc}(w)\le(s+1)\Delta_1(\hat H_j)=O((s+1)d(\hat H_j))$ triples. Since the number of vertices decreases from $n_\ast-|\hat B|$ to $n_\ast-|\hat B|-1$, the average degrees satisfy
\[
\begin{aligned}
 \frac{d((\hat H_j\cap\Kc)-w)}{d(\hat H_j)}
 =\frac{n_\ast-|\hat B|}{n_\ast-|\hat B|-1}
 \left(1-\frac{|\hat H_j\setminus\Kc|+d_{\hat H_j\cap\Kc}(w)}{m}\right)
 =1+O\left(\frac{s+1}{n_\ast}\right).
\end{aligned}
\]
Since $h=O(n_\ast)$, the estimate above gives $(h-1)\log\bigl(d((\hat H_j\cap\Kc)-w)/d(\hat H_j)\bigr)=O(s+1)=o(n_\ast)$. For every candidate that is not light, \eqref{eq:frame-large} therefore gives
\[
\begin{aligned}
 \log\comp_{\hat H_j,\cM}(w;u\to v)
 \ge\log\cyc_{\Lambda}(\hat H_j)-50\log n_\ast
 &\ge(h-1)\log d(\hat H_j)-O(n_\ast)\\
 &=(h-1)\log d((\hat H_j\cap\Kc)-w)-O(n_\ast).
\end{aligned}
\]
By \cref{prop:path-regularity}(i), deleting edges preserves the maximum-degree upper bound for $\hat H_j$. Using $d((\hat H_j\cap\Kc)-w)/d(\hat H_j)=1+O((s+1)/n_\ast)$, we obtain $\Delta_1((\hat H_j\cap\Kc)-w)=O(d((\hat H_j\cap\Kc)-w))$. These estimates verify the hypotheses of \cref{lem:overlap} for the family counted by $\comp_{\hat H_j,\cM}(w;u\to v)$ whenever $(w;u\to v)$ is not light. For the family counted by $\cyc_{\Lambda}(\hat H_j)$, apply \cref{lem:overlap} with host $\hat H_j$ and prescribed matching $\cM$. The required vertex-degree bound follows from \cref{prop:path-regularity}(i), and the lower bound on $\cyc_\Lambda(\hat H_j)$ is \eqref{eq:frame-large}. In either case, choose $C_1,C_2$ independently and uniformly from the cycle family under consideration. With the realisation of the edge-deletion process fixed, the overlap estimate gives
\[
 \EE[|C_1\cap C_2\cap\hat H_j|]
 =O\left(\frac{n_\ast}{\log d(\hat H_j)}\right),
\]
where the intersection with $\hat H_j$ counts only the shared triples.

Also $h/m=O(1/\log n_\ast)$ and $|T|/m=O(\log\log\log n_\ast/n_\ast)$, so the size hypotheses of \cref{lem:batch-variance} hold. Since $4|T|h/m\le\frac1{25}\log\log\log n_\ast$ and $d(\hat H_j)\ge c\log n_\ast$, \cref{lem:overlap,lem:batch-variance} give
\[
\begin{aligned}
 \frac{\Var[Y_T\mid(H_i)_i]}{(\EE[Y_T\mid(H_i)_i])^2}
 &\le O\left(\frac1{\log d(\hat H_j)}\right)\exp\left\{\frac{4|T|h}{m}\right\}\\
 &\le O\left(\frac{(\log\log n_\ast)^{1/25}}{\log\log n_\ast}\right)
 =O\left( \frac{1}{(\log\log n_\ast)^{24/25}} \right).
\end{aligned}
\]
For the families counted by $\comp_{\hat H_j,\cM}(w;u\to v)$, apply \cref{lem:batch-variance} with $h-1$ triples in each family member and ground set $\hat H_j$; the displayed estimate is unchanged. By \cref{lem:moment-bounds}(ii), substituting this relative-variance bound into the Chebyshev calculation in \eqref{eq:count-concentration} proves that estimate.
These estimates are uniform over the realisations in $\mathsf E_{\Lambda,j}$ and apply to the family counted by $\cyc_{\Lambda}(\hat H_j)$ and every family counted by $\comp_{\hat H_j,\cM}(w;u\to v)$ for a candidate that is not light.

Each cycle counted by $\cyc_{\Lambda}(\hat H_j)$ has $h$ triples and survives with probability $\binom{m-h}{|T|}/\binom{m}{|T|}$. Applying \eqref{eq:count-concentration} with $Y_T=\cyc_{\Lambda}(\hat H_j-T)$ and $\EE[Y_T\mid(H_i)_i]=\binom{m-h}{|T|}\cyc_{\Lambda}(\hat H_j)/\binom{m}{|T|}$ gives \eqref{eq:X-stable-final} with probability $1-o((\log\log\log n_\ast)^{-1/2})$, completing the proof.\qedhere
\end{proof}

The bound in \eqref{eq:count-concentration} lets us compare the ratios defining normality before and after deletion.
\begin{claim}\label{claim:abnormal-persistence}
With probability $1-o(1)$, \eqref{eq:X-stable-final} holds and at least $9|\mathcal Q_{\Lambda}|/(10\sqrt{\log\log\log n_\ast})$ legal directed candidates are $1/(2\sqrt{\log\log\log n_\ast})$-abnormal in $\hat H_j-T$.
\end{claim}
\begin{proof}[Proof of \cref{claim:abnormal-persistence}]
We first compare the normalised ratios for candidates that are not light, then handle light candidates by monotonicity, and finally bound the number of concentration failures. For each legal candidate $(w;u\to v)$ that is not light, count a failure when
\begin{equation}\label{eq:candidate-deviation}
 \left|
 \frac{\binom{m}{|T|}\comp_{\hat H_j-T,\cM}(w;u\to v)}
 {\binom{m-h+1}{|T|}\comp_{\hat H_j,\cM}(w;u\to v)}-1
 \right|
 >\frac{1}{10^4\sqrt{\log\log\log n_\ast}}.
\end{equation}
For a legal candidate $(w;u\to v)$ that is not light, we will prove that, if \eqref{eq:X-stable-final} holds and \eqref{eq:candidate-deviation} does not occur, then
\[
 \frac{4d(\hat H_j-T)\comp_{\hat H_j-T,\cM}(w;u\to v)}{\cyc_{\Lambda}(\hat H_j-T)}
 =\left(1\pm\frac{1}{10\sqrt{\log\log\log n_\ast}}\right)
 \frac{4d(\hat H_j)\comp_{\hat H_j,\cM}(w;u\to v)}{\cyc_{\Lambda}(\hat H_j)}.
\]
This comparison will preserve abnormality when the tolerance is halved.

To establish the comparison, we apply \cref{claim:batch-stability} to the families counted by $\cyc_\Lambda(\hat H_j)$ and $\comp_{\hat H_j,\cM}(w;u\to v)$. Fix a legal candidate $(w;u\to v)$ that is not light. Each cycle counted by $\comp_{\hat H_j,\cM}(w;u\to v)$ has $h-1$ triples, so
\[
 \EE\left[\comp_{\hat H_j-T,\cM}(w;u\to v)\mid(H_i)_i\right]
 =\frac{\binom{m-h+1}{|T|}}{\binom{m}{|T|}}
 \comp_{\hat H_j,\cM}(w;u\to v).
\]
Applying \eqref{eq:count-concentration} with $Y_T=\comp_{\hat H_j-T,\cM}(w;u\to v)$ and the mean computed above shows that, with probability $1-o((\log\log\log n_\ast)^{-1/2})$,
\[
 \comp_{\hat H_j-T,\cM}(w;u\to v)
 =\left(1\pm\frac{1}{10^4\sqrt{\log\log\log n_\ast}}\right)
 \frac{\binom{m-h+1}{|T|}}{\binom{m}{|T|}}
 \comp_{\hat H_j,\cM}(w;u\to v).
\]

Dividing the probability that a fixed cycle with $h-1$ triples avoids $T$ by the corresponding probability for a cycle with $h$ triples and cancelling their common denominator $\binom{m}{|T|}$ gives
\begin{equation}\label{eq:survival-two-final}
 \frac{\binom{m-h+1}{|T|}}{\binom{m-h}{|T|}}
 =\frac{m-h+1}{m-h-|T|+1}
 =1+O\left(\frac{\log\log\log n_\ast}{n_\ast}\right).
\end{equation}

We now compare the ratios defining normality before and after deletion. Suppose that \eqref{eq:X-stable-final} holds and that \eqref{eq:candidate-deviation} does not occur. Dividing these estimates introduces a relative error at most $3/(10^4\sqrt{\log\log\log n_\ast})$ for large $n_\ast$. Combining this error with the degree ratio $d(\hat H_j-T)/d(\hat H_j)=1-|T|/m$ and the binomial-coefficient ratio in \eqref{eq:survival-two-final} gives
\begin{multline*}
  \frac{4d(\hat H_j-T)\comp_{\hat H_j-T,\cM}(w;u\to v)}{\cyc_{\Lambda}(\hat H_j-T)}
  =\left(1-\frac{|T|}{m}\right)\frac{m-h+1}{m-h-|T|+1}\\
 {}\times\left(1\pm\frac{3}{10^4\sqrt{\log\log\log n_\ast}}\right)
 \frac{4d(\hat H_j)\comp_{\hat H_j,\cM}(w;u\to v)}{\cyc_{\Lambda}(\hat H_j)}.
\end{multline*}
The factors $1-|T|/m$ and $(m-h+1)/(m-h-|T|+1)$ both differ from $1$ by $o((\log\log\log n_\ast)^{-1/2})$. Thus, for sufficiently large $n_\ast$, this proves the multiplicative comparison stated at the start of the proof.

For $0<\eta\le1$, we have $\left(1-\frac{\eta}{10}\right)(1+\eta)>1+\frac{\eta}{2}$ and $\left(1+\frac{\eta}{10}\right)(1-\eta)<1-\frac{\eta}{2}$. 
Taking $\eta=(\log\log\log n_\ast)^{-1/2}$ shows that every candidate that is not light and was $\eta$-abnormal before deletion remains $\eta/2$-abnormal whenever \eqref{eq:X-stable-final} holds and \eqref{eq:candidate-deviation} does not occur.

We now consider light candidates. Since $h/(m-|T|)=o(1)$, the inequality $-\log(1-x)\le2x$ for $0\le x\le1/2$ gives
\[
 \log\frac{\binom{m}{|T|}}{\binom{m-h}{|T|}}
 =\sum_{i=0}^{|T|-1}-\log\left(1-\frac{h}{m-i}\right)
 \le\frac{2h|T|}{m-|T|}
 =O(\log\log\log n_\ast).
\]
Hence $\binom{m}{|T|}/\binom{m-h}{|T|}=\exp\{O(\log\log\log n_\ast)\}=n_\ast^{o(1)}$. We also have $d(\hat H_j-T)=O(n_\ast^2)$. Thus, if $(w;u\to v)$ is light and \eqref{eq:X-stable-final} holds, then the bound $\comp_{\hat H_j-T,\cM}(w;u\to v)\le\comp_{\hat H_j,\cM}(w;u\to v)$ gives
\[
 \frac{4d(\hat H_j-T)\comp_{\hat H_j-T,\cM}(w;u\to v)}{\cyc_{\Lambda}(\hat H_j-T)}
 \le8d(\hat H_j-T)n_\ast^{-50}\frac{\binom{m}{|T|}}{\binom{m-h}{|T|}}=o(1).
\]

By \eqref{eq:count-concentration} and linearity of expectation, we have
\[
 \EE\left[\sum_{\substack{(w;u\to v)\in\mathcal Q_\Lambda\\
 \text{not light}}}
 \mathbf1_{\{\eqref{eq:candidate-deviation}\text{ holds}\}}\right]
 \le |\mathcal Q_\Lambda|\,o\left( \frac{1}{(\log\log\log n_\ast)^{1/2}} \right).
\]
Here and below the expectations in this forward calculation are over $T$, with this realisation fixed. Markov's inequality yields
\begin{multline*}
 \widehat{\PP}^{*}\left[
 \sum_{\substack{(w;u\to v)\in\mathcal Q_\Lambda\\
 \text{not light}}}
 \mathbf1_{\{\eqref{eq:candidate-deviation}\text{ holds}\}}
 >\frac{|\mathcal Q_\Lambda|}{100\sqrt{\log\log\log n_\ast}}
 \,\middle|\,(H_i)_i\right]\\
 \le\frac{100\sqrt{\log\log\log n_\ast}}{|\mathcal Q_\Lambda|}
 \EE\left[\sum_{\substack{(w;u\to v)\in\mathcal Q_\Lambda\\
 \text{not light}}}
 \mathbf1_{\{\eqref{eq:candidate-deviation}\text{ holds}\}}\right]=o(1).
\end{multline*}
If \eqref{eq:X-stable-final} holds, then the comparisons above show that every previously $(\log\log\log n_\ast)^{-1/2}$-abnormal candidate remains $1/(2\sqrt{\log\log\log n_\ast})$-abnormal in $\hat H_j-T$, except possibly the candidates that are not light and satisfy \eqref{eq:candidate-deviation}. By the preceding Markov bound, with probability $1-o(1)$ there are at most $|\mathcal Q_\Lambda|/(100\sqrt{\log\log\log n_\ast})$ such failures. Combining this bound with \eqref{eq:X-stable-final} and condition (\ref{item:completion-event-abnormal}) in the definition of $\mathsf E_{\Lambda,j}$, we find that at least $(1-1/100)|\mathcal Q_\Lambda|/\sqrt{\log\log\log n_\ast}\ge9|\mathcal Q_\Lambda|/(10\sqrt{\log\log\log n_\ast})$ candidates remain abnormal. This proves the claim.\qedhere
\end{proof}

By \cref{claim:abnormal-persistence}, many candidates remain abnormal after deletion, but their underlying triples need not be missing. We apply \cref{fact:role-exceptions} to $\hat H_j-T$ to show that only a small number of abnormal candidates have their underlying triples in $\hat H_j-T$; the others provide the missing triples needed for reverse sampling.
\begin{claim}\label{claim:batch-nonedges}
With probability $1-o(1)$, we have $|\mathcal A|\ge|\Omega_{\hat H_j-T}|/(3\sqrt{\log\log\log n_\ast})$.
\end{claim}
\begin{proof}[Proof of \cref{claim:batch-nonedges}]
By \cref{claim:batch-stability,claim:abnormal-persistence}, with probability $1-o(1)$, \eqref{eq:X-stable-final} holds and at least $\frac{9|\mathcal Q_\Lambda|}{10\sqrt{\log\log\log n_\ast}}$ legal directed candidates are $1/(2\sqrt{\log\log\log n_\ast})$-abnormal in $\hat H_j-T$. These candidates are also $1/(4\sqrt{\log\log\log n_\ast})$-abnormal. Since each triple has at most six legal orientations, it suffices to show that the number of legal candidates that are abnormal at this smaller tolerance and whose underlying triples belong to $\hat H_j-T$ is $o(n_\ast^3/\sqrt{\log\log\log n_\ast})$. The remaining candidates then give the required lower bound on $|\mathcal A|$.

To prove this bound, we apply \cref{fact:role-exceptions} to the uniform distribution on the cycles counted by $\cyc_\Lambda(\hat H_j-T)$. As in \cref{claim:present-candidates}, its directed edge marginals are the ratios defining normality. We first verify the entropy lower bound required for this application. By \eqref{eq:X-stable-final}, the change in the logarithm of the number of cycles in this family is, up to the stated concentration error, $\log\bigl(\binom{m-h}{|T|}/\binom m{|T|}\bigr)$. We compare this with $h\log\bigl(d(\hat H_j-T)/d(\hat H_j)\bigr)$, the change in $h\log\bigl(2d(\hat H_j)\bigr)-2h$, the main term in the counting lower bound \eqref{eq:frame-large}. The following expansions have the same leading term.
\[
\begin{aligned}
 \log\frac{\binom{m-h}{|T|}}{\binom{m}{|T|}}
 &=\sum_{i=0}^{|T|-1}\log\left(1-\frac{h}{m-i}\right)=-\frac{h|T|}{m}+O\left(\frac{h^2|T|+h|T|^2}{m^2}\right),\quad\text{and}\\
 h\log\frac{d(\hat H_j-T)}{d(\hat H_j)}
 &=h\log\left(1-\frac{|T|}{m}\right)=-\frac{h|T|}{m}+O\left(\frac{h|T|^2}{m^2}\right).
\end{aligned}
\]
Here $h^2|T|/m^2=O(\log\log\log n_\ast/d(\hat H_j))$ and $h|T|^2/m^2=O((\log\log\log n_\ast)^2/n_\ast)$. Taking logarithms in \eqref{eq:X-stable-final} contributes an additional error $O((\log\log\log n_\ast)^{-1/2})$. Substituting these estimates and the two expansions into the lower bound \eqref{eq:frame-large} on the logarithm of the number of mixed Hamilton cycles gives
\begin{multline*}
 \log\cyc_{\Lambda}(\hat H_j-T)
 \ge{}h\log\bigl(2d(\hat H_j-T)\bigr)-2h-\frac{10n_\ast}{\sqrt{\log n_\ast}}\\
 {}-O\left(\frac1{\sqrt{\log\log\log n_\ast}}+\frac{\log\log\log n_\ast}{d(\hat H_j)}+\frac{(\log\log\log n_\ast)^2}{n_\ast}\right).
\end{multline*}
Since $d(\hat H_j)\ge c\log n_\ast$, the additional error is $o(n_\ast/\sqrt{\log n_\ast})$. Hence $\log\cyc_{\Lambda}(\hat H_j-T)\ge h\log\bigl(2d(\hat H_j-T)\bigr)-2h-11n_\ast/\sqrt{\log n_\ast}$ for all sufficiently large $n_\ast$.

We next verify the degree bounds and the estimate in \cref{thm:entropy-stability}(ii). Maximum degrees and codegrees can only decrease when we delete $T$, while $d(\hat H_j-T)/d(\hat H_j)=1-O(\log\log\log n_\ast/n_\ast)$. For the host $\hat H_j-T$, apply \cref{thm:entropy-stability} with $n=n_\ast-|\hat B|$ vertices, $m-|T|$ edges, and matching size $s=|\cM|$. Deleting $T$ does not change the vertex set or the prescribed matching. Thus the size calculation for $A$ in the proof of \cref{claim:present-candidates} still applies, and we may use \cref{prop:path-regularity}(ii) with $L=2$. The comparison $m=j(1-O(1/n_\ast))$ established in the proof of \cref{claim:present-candidates} gives $|3j/8-3m/8|=O(j/n_\ast)$. This difference is absorbed by the allowed error, so before deleting $T$ we may replace the reference value $3j/8$ by $3m/8$. For each set $A$ allowed in \cref{thm:entropy-stability}(ii), deleting $T$ changes the number of such triples by at most $|T|$ and its reference value by $3|T|/8$, so
\[
 \left|\bigl|(\hat H_j-T)[A,\overline A]\bigr|-\frac{3(m-|T|)}8\right|\le\left|\bigl|\hat H_j[A,\overline A]\bigr|-\frac{3m}8\right|+2|T|.
\]
Since $|T|=O(d(\hat H_j)\log\log\log n_\ast)=o(n_\ast d(\hat H_j) (\log n_\ast)^{-1/8})$, these changes are absorbed by the allowed error. Thus $\Delta_1(\hat H_j-T)=O(d(\hat H_j-T))$, the ratio $\Delta_2(\hat H_j-T)/d(\hat H_j-T)$ remains $O((\log n_\ast)^{-1/4})$, and condition (ii) of \cref{thm:entropy-stability} holds with $\delta=O((\log n_\ast)^{-1/8})$, all with absolute implied constants. This verifies the remaining hypotheses needed to apply \cref{fact:role-exceptions}.

Apply \cref{fact:role-exceptions} to $\hat H_j-T$, using the entropy lower bound established above, whose additive error is $11n_\ast/\sqrt{\log n_\ast}$. As in \cref{claim:present-candidates}, we use the uniform law on the counted cycles, all of whose triples are $\cM_0$-compatible; every legal candidate has its underlying triple in $\Kc$. We have $m=O(n_\ast^3)$, $|\mathcal Q_\Lambda|=(1+o(1))n_\ast^3$, and $|\Omega_{\hat H_j-T}|\le\binom{n_\ast-|\hat B|}{3}$. The number of legal directed candidates that are $1/(4\sqrt{\log\log\log n_\ast})$-abnormal in $\hat H_j-T$ and whose underlying triples also belong to $\hat H_j-T$ is $O\Bigl(m\sqrt{\frac{\log\log\log n_\ast}{\log\log n_\ast}}\Bigr)=o(n_\ast^3/\sqrt{\log\log\log n_\ast})$. Removing these candidates from those supplied by \cref{claim:abnormal-persistence} and using at most six legal orientations per triple gives
\[
\begin{aligned}
 |\mathcal A|
 &\ge\frac16\left(\frac{9|\mathcal Q_\Lambda|}{10\sqrt{\log\log\log n_\ast}}
 -O\left(m\sqrt{\frac{\log\log\log n_\ast}{\log\log n_\ast}}\right)\right)\\
 &\ge\left(\frac3{20}-o(1)\right)\frac{n_\ast^3}{\sqrt{\log\log\log n_\ast}} 
\ge\frac{|\Omega_{\hat H_j-T}|}{3\sqrt{\log\log\log n_\ast}}.
\end{aligned}
\]
This proves the claim.
\end{proof}

To establish condition (\ref{item:reverse-event-sampled}) in the definition of $\mathsf W$, we next consider a candidate whose underlying triple belongs to the deleted set $T$. For these estimates, call a candidate \emph{initially normal} if it is $1/(100\sqrt{\log\log\log n_\ast})$-normal in $\hat H_j$, and \emph{initially abnormal} otherwise. For the fixed realisation of the edge-deletion process, define
\begin{align*}
 \mathcal N_\Lambda&\coloneqq\left\{(w;u\to v)\in\mathcal Q_\Lambda\colon \{u,v,w\}\in\hat H_j\text{ and }(w;u\to v)\text{ is initially normal}\right\},\\
 \mathcal B_\Lambda&\coloneqq\left\{(w;u\to v)\in\mathcal Q_\Lambda\colon \{u,v,w\}\in\hat H_j\text{ and }(w;u\to v)\text{ is initially abnormal}\right\}.
\end{align*}
For a candidate $(w;u\to v)\in\mathcal N_\Lambda$, we record a deviation from its mean conditional on $\{u,v,w\}\in T$ when
\begin{equation}\label{eq:sampled-candidate-deviation}
 \left|
 \frac{\binom{m-1}{|T|-1}\comp_{\hat H_j-T,\cM}(w;u\to v)}
 {\binom{m-h}{|T|-1}\comp_{\hat H_j,\cM}(w;u\to v)}-1
 \right|
 >\frac1{10^4\sqrt{\log\log\log n_\ast}}.
\end{equation}
\begin{claim}\label{claim:sampled-candidate-stability}
For each candidate $(w;u\to v)\in\mathcal N_\Lambda$, the following statements hold.
\begin{enumerate}[label=(\roman*),ref=\roman*,beginpenalty=10000]
\item\label{item:sampled-candidate-tail} The conditional failure probability satisfies
\[
 \widehat{\PP}^{*}\left[\eqref{eq:sampled-candidate-deviation}\text{ holds}
 \,\middle|\,(H_i)_i,\ \{u,v,w\}\in T\right]
 =o\left(\frac{1}{\sqrt{\log\log\log n_\ast}}\right).
\]
\item\label{item:sampled-candidate-normality} If $\{u,v,w\}\in T$, \eqref{eq:X-stable-final} holds, and \eqref{eq:sampled-candidate-deviation} does not hold,\newline then $(w;u\to v)$ is $1/(4\sqrt{\log\log\log n_\ast})$-normal in $\hat H_j-T$.
\end{enumerate}
\end{claim}
\begin{proof}[Proof of \cref{claim:sampled-candidate-stability}]
Fix $(w;u\to v)\in\mathcal N_\Lambda$. We first prove (i) by computing the mean and variance under the conditional sampling law, then prove (ii) by comparing the normalised ratios before and after deletion. Conditional on $\{u,v,w\}\in T$, the other $|T|-1$ triples of $T$ form a uniform subset of the remaining $m-1$ triples of $\hat H_j$. Every cycle counted by $\comp_{\hat H_j,\cM}(w;u\to v)$ omits $w$, so $\{u,v,w\}$ belongs to none of these cycles. Avoiding the $h-1$ triples of such a cycle therefore gives
\[
 \EE\left[\comp_{\hat H_j-T,\cM}(w;u\to v)
 \mid(H_i)_i,\ \{u,v,w\}\in T\right]
 =\frac{\binom{m-h}{|T|-1}}{\binom{m-1}{|T|-1}}
 \comp_{\hat H_j,\cM}(w;u\to v).
\]

For sufficiently large $n_\ast$, the initial normality of $(w;u\to v)$ gives $\comp_{\hat H_j,\cM}(w;u\to v)\ge\cyc_\Lambda(\hat H_j)/(8d(\hat H_j))>\cyc_\Lambda(\hat H_j)n_\ast^{-50}$. Thus this candidate is not light, so its cycle family satisfies the hypotheses of \cref{lem:overlap}. To obtain concentration under the conditional law, apply \cref{lem:batch-variance} to the triple sets of this family. Each cycle has $h-1$ triples, and $T\setminus\{\{u,v,w\}\}$ is sampled uniformly from the other $m-1$ host triples. Since $4(|T|-1)(h-1)/(m-1)\le4|T|h/m$, we obtain
\[
 \frac{\Var[\comp_{\hat H_j-T,\cM}(w;u\to v)\mid(H_i)_i,\ \{u,v,w\}\in T]}
 {\bigl(\EE[\comp_{\hat H_j-T,\cM}(w;u\to v)\mid(H_i)_i,\ \{u,v,w\}\in T]\bigr)^2}
 =O\left(\frac1{(\log\log n_\ast)^{24/25}}\right).
\]
By the Chebyshev calculation in \eqref{eq:count-concentration}, the probability of \eqref{eq:sampled-candidate-deviation}, conditional on $(H_i)_i$ and $\{u,v,w\}\in T$, is $o((\log\log\log n_\ast)^{-1/2})$, uniformly over $(w;u\to v)\in\mathcal N_\Lambda$, proving (i).

For (ii), we compare the normalised ratios defining normality before and after deletion. Dividing the probability that a fixed cycle counted by $\comp_{\hat H_j,\cM}(w;u\to v)$ avoids $T$, conditional on $\{u,v,w\}\in T$, by the unconditioned probability that a fixed cycle counted by $\cyc_\Lambda(\hat H_j)$ avoids $T$ gives
\[
 \frac{\binom{m-h}{|T|-1}\binom{m}{|T|}}
 {\binom{m-1}{|T|-1}\binom{m-h}{|T|}}
 =\frac{m}{m-h-|T|+1}
 =1+O\left(\frac{1}{d(\hat H_j)}+\frac{\log\log\log n_\ast}{n_\ast}\right).
\]
Combining \eqref{eq:X-stable-final} with the complement of \eqref{eq:sampled-candidate-deviation} and dividing the resulting bounds for $\comp_{\hat H_j-T,\cM}(w;u\to v)$ and $\cyc_\Lambda(\hat H_j-T)$ introduces a relative error at most $3/(10^4\sqrt{\log\log\log n_\ast})$. The degree ratio $d(\hat H_j-T)/d(\hat H_j)=1-|T|/m$ and the factor $m/(m-h-|T|+1)$ just computed therefore give the multiplicative factor
\[
 \left(1-\frac{|T|}{m}\right)\frac{m}{m-h-|T|+1}
 \left(1\pm\frac{3}{10^4\sqrt{\log\log\log n_\ast}}\right)
 =1\pm\frac{1}{10\sqrt{\log\log\log n_\ast}}
\]
for sufficiently large $n_\ast$. Thus $4d(\hat H_j-T)\comp_{\hat H_j-T,\cM}(w;u\to v)/\cyc_\Lambda(\hat H_j-T)$ is $1\pm1/(10\sqrt{\log\log\log n_\ast})$ times $4d(\hat H_j)\comp_{\hat H_j,\cM}(w;u\to v)/\cyc_\Lambda(\hat H_j)$.

For $0<\eta\le1$, we have $\left(1-\frac{\eta}{10}\right)\left(1-\frac{\eta}{100}\right)>1-\frac{\eta}{4}$ and $\left(1+\frac{\eta}{10}\right)\left(1+\frac{\eta}{100}\right)<1+\frac{\eta}{4}$.
With $\eta=(\log\log\log n_\ast)^{-1/2}$, these inequalities show that every candidate in $\mathcal N_\Lambda$ remains $1/(4\sqrt{\log\log\log n_\ast})$-normal in $\hat H_j-T$ whenever \eqref{eq:X-stable-final} holds and \eqref{eq:sampled-candidate-deviation} does not occur. This proves (ii).\qedhere
\end{proof}

We combine this conditional estimate with \cref{claim:present-candidates} to bound the number of deleted triples having an abnormal orientation.
\begin{claim}\label{claim:sampled-abnormal-bound}
With probability $1-o(1)$, \eqref{eq:X-stable-final} holds and at most $|T|/(100\sqrt{\log\log\log n_\ast})$ triples of $T$ have a legal orientation that is $1/(4\sqrt{\log\log\log n_\ast})$-abnormal in $\hat H_j-T$.
\end{claim}
\begin{proof}[Proof of \cref{claim:sampled-abnormal-bound}]
We estimate separately the candidates in $\mathcal B_\Lambda$ whose triples lie in $T$ and those in $\mathcal N_\Lambda$ whose triples lie in $T$ and for which \eqref{eq:sampled-candidate-deviation} holds. First, by \cref{claim:present-candidates},
\[
\begin{aligned}
 \EE\left[\sum_{(w;u\to v)\in\mathcal B_\Lambda}
 \mathbf1_{\{\{u,v,w\}\in T\}}\right]
=\frac{|T|}{m}\,
 O\left(m\sqrt{\frac{\log\log\log n_\ast}{\log\log n_\ast}}\right)
 =o\left(\frac{|T|}{\sqrt{\log\log\log n_\ast}}\right).
\end{aligned}
\]
Next consider the candidates in $\mathcal N_\Lambda$ for which \eqref{eq:sampled-candidate-deviation} occurs. Using \cref{claim:sampled-candidate-stability}(\ref{item:sampled-candidate-tail}) and the bound of six legal orientations per triple, we obtain
\[
\begin{aligned}
 &\sum_{(w;u\to v)\in\mathcal N_\Lambda}
 \widehat{\PP}^{*}\left[\{u,v,w\}\in T\text{ and }\eqref{eq:sampled-candidate-deviation}\text{ holds}
 \,\middle|\,(H_i)_i\right]\\
 &\quad=\frac{|T|}{m}
 \sum_{(w;u\to v)\in\mathcal N_\Lambda}
 \widehat{\PP}^{*}\left[\eqref{eq:sampled-candidate-deviation}\text{ holds}
 \,\middle|\,(H_i)_i,\ \{u,v,w\}\in T\right]\\
 &\quad\le6|T|\,o\left(\frac{1}{(\log\log\log n_\ast)^{1/2}}\right)
 =o\left(\frac{|T|}{\sqrt{\log\log\log n_\ast}}\right).
\end{aligned}
\]
If \eqref{eq:X-stable-final} holds, then \cref{claim:sampled-candidate-stability}(\ref{item:sampled-candidate-normality}) shows that every candidate in $\mathcal N_\Lambda$ whose underlying triple belongs to $T$ and whose deviation event in \eqref{eq:sampled-candidate-deviation} does not occur is $1/(4\sqrt{\log\log\log n_\ast})$-normal in $\hat H_j-T$. Every triple in $T\cap\mathcal A$ therefore has a legal orientation $(w;u\to v)$ that either belongs to $\mathcal B_\Lambda$ or belongs to $\mathcal N_\Lambda$ and satisfies \eqref{eq:sampled-candidate-deviation}. Hence
\[
\begin{aligned}
 |T\cap\mathcal A|
 &\le\sum_{(w;u\to v)\in\mathcal B_\Lambda}
 \mathbf1_{\{\{u,v,w\}\in T\}}
 +\sum_{(w;u\to v)\in\mathcal N_\Lambda}
 \mathbf1_{\{\{u,v,w\}\in T\}}
 \mathbf1_{\{\eqref{eq:sampled-candidate-deviation}\text{ holds}\}}.
\end{aligned}
\]
The preceding expectation bounds show that the expected right-hand side is $o(|T|/\sqrt{\log\log\log n_\ast})$. Markov's inequality therefore bounds the probability that the right-hand side exceeds $\frac{|T|}{100\sqrt{\log\log\log n_\ast}}$ by
\[
 \frac{100\sqrt{\log\log\log n_\ast}}{|T|}\,
 o\left(\frac{|T|}{\sqrt{\log\log\log n_\ast}}\right)=o(1).
\]
By \cref{claim:batch-stability}, \eqref{eq:X-stable-final} also holds with probability $1-o(1)$. Intersecting this event with the event supplied by the preceding Markov bound proves the claim.\qedhere
\end{proof}

We finally verify the degree lower bounds needed for reverse sampling.
\begin{claim}\label{claim:batch-degree-bounds}
With probability $1-o(1)$, we have $d_{\hat F_\ast-T}(v)\ge\ell'_v$ for every $v\in[n_\ast]\setminus\hat B$.
\end{claim}
\begin{proof}[Proof of \cref{claim:batch-degree-bounds}]
For the fixed realisation in $\mathsf E_{\Lambda,j}$, passing from $F_*$ to $\hat F_\ast=F_*-\hat B$ removes the edges incident to $\hat B$, where $|\hat B|\le8$. Since $\Delta_2(F_*)\le2$, we have $d_{F_*}(v)-d_{\hat F_\ast}(v)\le\sum_{u\in\hat B}d_{F_*}(v,u)\le2|\hat B|$ for every $v\in[n_\ast]\setminus\hat B$. Hence $\ell'_v=0.01\log n_\ast+O(1)$ for every $v\in[n_\ast]\setminus\hat B$. Also, the maximum-degree bound in \cref{thm:terminal-regularity}(i) gives $|F_*(\hat B)|\le|\hat B|\Delta_1(F_*)=O(\log n_\ast)$, so $|\hat F_\ast|=m_\ast-O(\log n_\ast)$. We treat the remaining vertices in $B_*$ and outside $B_*$ separately.
\begin{enumerate}[label=(\textit{\roman*})]
\item By the bounds on $|B_*|$ and $\Delta_1(F_*)$ in \cref{thm:terminal-regularity}(i) and the uniform choice of $T$,
\[
\begin{aligned}
 \widehat{\PP}^{*}\left[\exists e\in T\cap\hat F_\ast\text{ with }e\cap B_*\ne\varnothing
 \,\middle|\,(H_i)_i\right]
 &\le\bigl|\{e\in\hat F_\ast\colon e\cap B_*\ne\varnothing\}\bigr|\frac{|T|}{m} \\
 & \le |B_*|\Delta_1(F_*)\frac{|T|}{m}
 =O\left(\frac{\log n_\ast \log\log\log n_\ast}{n_\ast^{3/4}}\right).
\end{aligned}
\]
Outside this event, $d_{\hat F_\ast-T}(v)=d_{\hat F_\ast}(v)\ge\ell'_v$ for every $v\in B_*\setminus\hat B$.
\item For $v\in[n_\ast]\setminus(\hat B\cup B_*)$, we have $d_{\hat F_\ast}(v)>0.03\log n_\ast-2|\hat B|$ and $\ell'_v=0.01\log n_\ast+O(1)$. Thus a violation of the lower bound requires at least $\lfloor0.01\log n_\ast\rfloor$ incident edges of $\hat F_\ast$ to lie in $T$. The inclusion probability of any fixed collection of this size is at most $(|T|/m)^{\lfloor0.01\log n_\ast\rfloor}$. Using $\Delta_1(\hat F_\ast)=O(\log n_\ast)$ and $|T|/m=O(\log\log\log n_\ast/n_\ast)$, a union bound gives
\[
\begin{aligned}
 &\widehat{\PP}^{*}\left[\exists v\in[n_\ast]\setminus(\hat B\cup B_*)\colon
 d_{\hat F_\ast-T}(v)<\ell'_v\,\middle|\,(H_i)_i\right]\\
 &\quad\le\sum_{v\in[n_\ast]\setminus(\hat B\cup B_*)}
 \binom{d_{\hat F_\ast}(v)}{\lfloor0.01\log n_\ast\rfloor}
 \left(\frac{|T|}{m}\right)^{\lfloor0.01\log n_\ast\rfloor}\\
 &\quad\le n_\ast\left(\frac{C'|T|}{m}\right)^{\lfloor0.01\log n_\ast\rfloor}
 =\exp\left\{\log n_\ast+\lfloor0.01\log n_\ast\rfloor\log\frac{C'|T|}{m}\right\}\\
 &\quad\le\exp\{-(0.01-o(1))(\log n_\ast)^2\}
 =\exp\{-\Omega(\log^2n_\ast)\}.
\end{aligned}
\]
\end{enumerate}
The estimate $|T|/m=O(\log\log\log n_\ast/n_\ast)$ shows that these probability bounds are uniform over the permitted times $j$ and the fixed realisations under consideration. Thus $\hat F_\ast-T$ satisfies these degree lower bounds with probability at least
\[
 1-O\left(\frac{(\log n_\ast)\log\log\log n_\ast}{n_\ast^{3/4}}\right)-\exp\{-\Omega(\log^2n_\ast)\}=1-o(1).\qedhere
\]
\end{proof}

By \cref{claim:sampled-abnormal-bound,claim:batch-nonedges,claim:batch-degree-bounds}, for every fixed $\QQ$-admissible realisation in $\mathsf E_{\Lambda,j}$, the event $\mathsf W$ holds with probability $1-o(1)$ over the uniformly chosen set $T$. The error is uniform over these realisations and their exposed edge sets $(F_*(\hat B),H_j(\hat B))$.

For the reverse estimate, we return to the joint law $\widehat{\PP}^{*}$. We keep the exposed edge sets $F_*(\hat B)$ and $H_j(\hat B)$ fixed and condition on the two remaining hypergraphs $\hat H_j-T$ and $\hat F_\ast-T$. We do not additionally condition on $\mathsf E_{\Lambda,j}$; the bounds in condition (\ref{item:completion-event-regularity}) in its definition were used in the forward estimate. We now determine the conditional law of the original hypergraphs and the deleted set $T$. In reverse sampling, $T$ specifies the triples added back to $\hat H_j-T$, while $T\cap\hat F_\ast$ specifies those added back to $\hat F_\ast-T$. The original hypergraphs $\hat F_\ast,\hat H_j$ are no longer fixed, although their sizes are determined by $F_*(\hat B),H_j(\hat B)$.
The revealed edges $F_*(\hat B)$ contribute $d_{F_*(\hat B)}(v)$ to the degree of each remaining vertex $v$. Hence the reconstructed hypergraph $\hat F_\ast=(\hat F_\ast-T)\cup(T\cap\hat F_\ast)$ must satisfy
\begin{equation}\label{eq:reverse-degree-bounds}
 d_{(\hat F_\ast-T)\cup(T\cap\hat F_\ast)}(v)
 \ge\ell'_v=\max\{0,\ell_v-d_{F_*(\hat B)}(v)\}
 \quad\text{for every }v\in[n_\ast]\setminus\hat B.
\end{equation}
These lower bounds are fixed by $F_*(\hat B)$ throughout reverse sampling.
\begin{claim}\label{prop:reverse-uniform}
Conditional on any $\widehat{\PP}^{*}$-admissible value of $(\hat H_j-T,\hat F_\ast-T)$, the following statements hold.
\begin{enumerate}[label=(\roman*),ref=\roman*]
\item\label{item:reverse-nested} The pair $(T\cap\hat F_\ast,T)$ is uniformly distributed over all nested pairs $T\cap\hat F_\ast\subseteq T\subseteq\Omega_{\hat H_j-T}$ with $|T\cap\hat F_\ast|=|\hat F_\ast|-|\hat F_\ast-T|$ and $|T|=\bigl\lfloor m\log\log\log n_\ast/\bigl(50(n_\ast-|\hat B|-|\cM|)\bigr)\bigr\rfloor$, and satisfying \eqref{eq:reverse-degree-bounds}.
\item\label{item:reverse-set} If $d_{\hat F_\ast-T}(v)\ge\ell'_v$ for every $v\in[n_\ast]\setminus\hat B$, then $T$ is exactly uniform among the subsets of $\Omega_{\hat H_j-T}$ of the prescribed size.
\end{enumerate}
\end{claim}
\begin{proof}[Proof of \cref{prop:reverse-uniform}]
For \cref{prop:reverse-uniform}(\ref{item:reverse-nested}), keep $F_*(\hat B),H_j(\hat B)$ fixed, as in the definition of $\widehat{\PP}^{*}$. For any $\widehat{\PP}^{*}$-admissible value $(F,H,U)$ of $(\hat F_\ast,\hat H_j,T)$, \cref{lem:boundary}(i)--(ii) and the uniform choice of $T$ give
\[
 \widehat{\PP}^{*}[\hat F_\ast=F,\hat H_j=H,T=U]
 =\frac{\widehat{\PP}^{*}[\hat F_\ast=F]}
 {\displaystyle\binom{\binom{n_\ast-|\hat B|}{3}-|F|}{m-|F|}\binom{m}{|T|}}.
\]
The numerator is constant over the permitted realisations $F$, and the binomial factors depend only on the prescribed sizes. Thus all choices of $(\hat F_\ast,\hat H_j,T)$ with the prescribed sizes, satisfying $\hat F_\ast\subseteq\hat H_j$, $T\subseteq\hat H_j$, and $d_{\hat F_\ast}(v)\ge\ell'_v$ for every $v\in[n_\ast]\setminus\hat B$, have equal probability. Once $\hat H_j-T$ and $\hat F_\ast-T$ are fixed, these realisations correspond bijectively to nested pairs $(T\cap\hat F_\ast,T)$ through $\hat F_\ast=(\hat F_\ast-T)\cup(T\cap\hat F_\ast)$ and $\hat H_j=(\hat H_j-T)\cup T$. The only remaining restriction is \eqref{eq:reverse-degree-bounds}, proving \cref{prop:reverse-uniform}(\ref{item:reverse-nested}).

For \cref{prop:reverse-uniform}(\ref{item:reverse-set}), suppose that $d_{\hat F_\ast-T}(v)\ge\ell'_v$ for every $v\in[n_\ast]\setminus\hat B$. Adding edges preserves these inequalities, so the restriction in \eqref{eq:reverse-degree-bounds} disappears; every $T$ then admits $\binom{|T|}{|\hat F_\ast|-|\hat F_\ast-T|}$ possible sets $T\cap\hat F_\ast$. Hence, for every fixed $U\subseteq\Omega_{\hat H_j-T}$ with $|U|=|T|$,
\[
\begin{aligned}
 \widehat{\PP}^{*}[T=U\mid\hat H_j-T,\hat F_\ast-T]
 =\frac{\binom{|T|}{|\hat F_\ast|-|\hat F_\ast-T|}}
 {\binom{|\Omega_{\hat H_j-T}|}{|T|}\binom{|T|}{|\hat F_\ast|-|\hat F_\ast-T|}}
 = \frac{1}{\binom{|\Omega_{\hat H_j-T}|}{|T|}},
\end{aligned}
\]
as required.
\end{proof}

The next claim gives the reverse probability bound, with $F_*(\hat B),H_j(\hat B)$ still fixed.
\begin{claim}\label{claim:reverse-witness-bound}
There is an absolute constant $c>0$ such that, for every $\widehat{\PP}^{*}$-admissible value of $(\hat H_j-T,\hat F_\ast-T)$, we have
\[
 \widehat{\PP}^{*}[\mathsf W\mid\hat H_j-T,\hat F_\ast-T]
 \le\exp\left\{-\frac{c|T|}{\sqrt{\log\log\log n_\ast}}\right\}.
\]
\end{claim}
\begin{proof}[Proof of \cref{claim:reverse-witness-bound}]
If $d_{\hat F_\ast-T}(v)<\ell'_v$ for some $v\in[n_\ast]\setminus\hat B$ or $|\mathcal A|<|\Omega_{\hat H_j-T}|/(3\sqrt{\log\log\log n_\ast})$, then condition (\ref{item:reverse-event-degree-bounds}) or (\ref{item:reverse-event-nonedges}) in the definition of $\mathsf W$ fails, so this probability is zero. Otherwise, by \cref{prop:reverse-uniform}(\ref{item:reverse-set}), $T$ is exactly uniform among the subsets of $\Omega_{\hat H_j-T}$ of the prescribed size. The set $\mathcal A$ is determined by $\hat H_j-T$ and the frame $\Lambda$. Under this conditional law, $|T\cap\mathcal A|$ is hypergeometric, with
\[
 \EE[|T\cap\mathcal A|\mid\hat H_j-T,\hat F_\ast-T]
 =\frac{|T||\mathcal A|}{|\Omega_{\hat H_j-T}|}
 \ge\frac{|T|}{3\sqrt{\log\log\log n_\ast}}.
\]
The threshold in condition (\ref{item:reverse-event-sampled}) in the definition of $\mathsf W$ is at most $3/100$ of this mean. Thus \cref{lem:chernoff}(i), with $a=3/100$, gives
\[
\begin{aligned}
 \widehat{\PP}^{*}[\mathsf W\mid\hat H_j-T,\hat F_\ast-T]
 &\quad\le\widehat{\PP}^{*}\left[|T\cap\mathcal A|\le
 \frac3{100}\EE[|T\cap\mathcal A|\mid\hat H_j-T,\hat F_\ast-T]
 \,\middle|\,\hat H_j-T,\hat F_\ast-T\right]\\
 &\quad\le\exp\left\{-\varphi\left(\frac{3}{100}\right)\frac{|T||\mathcal A|}{|\Omega_{\hat H_j-T}|}\right\}
 \le\exp\left\{-\frac{c|T|}{\sqrt{\log\log\log n_\ast}}\right\}.
\end{aligned}
\]
This proves the claim.
\end{proof}

We now combine the forward estimate with \cref{claim:reverse-witness-bound} to bound the conditional probability of $\mathsf E_{\Lambda,j}$.
\begin{claim}\label{claim:conditional-completion-bound}
There is an absolute constant $c>0$ such that, for all $\QQ$-admissible values of $(F_*(\hat B),H_j(\hat B))$, we have
\[
 \QQ[\mathsf E_{\Lambda,j}\mid F_*(\hat B),H_j(\hat B)]
 \le\exp\{-c(\log n_\ast)\sqrt{\log\log\log n_\ast}\}.
\]
\end{claim}
\begin{proof}[Proof of \cref{claim:conditional-completion-bound}]
Keep $F_*(\hat B),H_j(\hat B)$ fixed. Then $m=j-|H_j(\hat B)|$ is fixed, and the prescribed size $|T|$ is fixed as well. If no realisation in $\mathsf E_{\Lambda,j}$ has these exposed edge sets, then the conditional probability to be bounded is zero. Otherwise, under $\widehat{\PP}^{*}$ the edge-deletion process has law $\QQ$ conditional on these edge sets. All expectations below are taken under $\widehat{\PP}^{*}$. Since the forward estimate holds uniformly over realisations in $\mathsf E_{\Lambda,j}$, averaging over the edge-deletion process gives
\[
\begin{aligned}
 (1-o(1))\QQ[\mathsf E_{\Lambda,j}\mid F_*(\hat B),H_j(\hat B)]
 \le\EE\left[\mathbf1_{\mathsf E_{\Lambda,j}}\,
 \widehat{\PP}^{*}[\mathsf W\mid(H_i)_i]\right]
 =\widehat{\PP}^{*}[\mathsf E_{\Lambda,j}\cap\mathsf W]
 \le\widehat{\PP}^{*}[\mathsf W].
\end{aligned}
\]
Similarly, averaging the bound in \cref{claim:reverse-witness-bound} over the two remaining hypergraphs gives
\[
 \widehat{\PP}^{*}[\mathsf W]
 =\EE\bigl[\widehat{\PP}^{*}[\mathsf W\mid\hat H_j-T,\hat F_\ast-T]\bigr]
 \le\exp\left\{-\frac{c|T|}{\sqrt{\log\log\log n_\ast}}\right\}.
\]

It remains to express this bound in terms of $n_\ast$. For the fixed exposed edge sets, choose a realisation in $\mathsf E_{\Lambda,j}$. The bound $d_{H_j}(v)\le3Cj/n_\ast$ for every $v\in[n_\ast]$ gives $|H_j(\hat B)|\le\sum_{v\in\hat B}d_{H_j}(v)\le3Cj|\hat B|/n_\ast=O(j/n_\ast)$, and hence $m=j-|H_j(\hat B)|=j(1-O(1/n_\ast))$. Since $j\ge m_\ast\sim n_\ast\log n_\ast/3$ and $n_\ast-|\hat B|-|\cM|=(1-o(1))n_\ast$, the prescribed size of $T$ yields
\[
 |T|=\left\lfloor\frac{m\log\log\log n_\ast}{50(n_\ast-|\hat B|-|\cM|)}\right\rfloor
 \ge c(\log n_\ast)\log\log\log n_\ast.
\]
Thus $|T|/\sqrt{\log\log\log n_\ast}\ge c(\log n_\ast)\sqrt{\log\log\log n_\ast}$. Substituting this inequality into the preceding probability bounds and decreasing $c$ to absorb the factor $1-o(1)$ proves the claim.
\end{proof}

\Cref{claim:conditional-completion-bound} gives the conditional estimate in \cref{prop:uniform-completion}. Averaging this bound over $F_*(\hat B),H_j(\hat B)$ gives the same bound for $\QQ[\mathsf E_{\Lambda,j}]$.
\end{proof}

\section{Cycle modifications and uniform edge marginals}
\label{sec:migration}
\Cref{prop:uniform-completion} shows that, with high probability, the numbers $\comp_{\hat H_j,\cM}(w;u\to v)$ are nearly equal for most legal directed candidates in each frame satisfying \eqref{eq:frame-large}. In this section, we strengthen this control to a uniform upper bound for every edge marginal whenever $\mathsf A_j$ holds, proving \cref{thm:uniform-marginal}. The proof uses two migration operations, one moving a private vertex and the other moving an intersection vertex. If one legal choice of the deleted vertex $w$ and the added directed pair $u\to v$ gives too many mixed Hamilton cycles, successive migrations force a positive proportion of legal candidates to have comparably large numbers of mixed Hamilton cycles, contradicting \cref{prop:uniform-completion}. For edges containing a $\cM_0$-saturated vertex, we instead combine the migration comparisons with a partition of the counted mixed Hamilton cycles. The estimate in \cref{lem:simultaneous-link} supplies the triples needed for the migrations.

The following lemma allows us to apply \cref{prop:uniform-completion} to auxiliary cycle families whose sizes are at least $\Xc(H_j)n_\ast^{-100}$.
\begin{lemma}\label{claim:frame-count-transfer}
Under \cref{ass:core}, the following holds for all sufficiently large $n_\ast$. Let $m_\ast\le j\le\binom{n_\ast}{3}$, and let $\Lambda$ be a frame satisfying \eqref{eq:frame-bounds-final}. Suppose that the following conditions hold.
\begin{enumerate}[label=(\textit{\roman*})]
\item The event $\mathsf A_j$ holds.
\item The vertex-degree bounds in \cref{prop:path-regularity}(i), with $h=8$ and $L=2$, hold at size $j$.
\item We have $\cyc_{\Lambda}(\hat H_j)\ge\Xc(H_j)n_\ast^{-100}$.
\end{enumerate}
Then $\log\cyc_{\Lambda}(\hat H_j)$ satisfies the lower bound \eqref{eq:frame-large}.
\end{lemma}
\begin{proof}[Proof of \cref{claim:frame-count-transfer}]
The assumed loss of a factor $n_\ast^{100}$ costs only $100\log n_\ast$ after taking logarithms. We show that passing from $H_j,\cM_0$ to $\hat H_j,\cM$ changes the main term in \eqref{eq:benchmark} by $O(\log n_\ast)$, so both losses are absorbed by the error allowed in \eqref{eq:frame-large}.

By \eqref{eq:frame-bounds-final}, $|\hat B|\le8$ and $\bigl||\cM|-s\bigr|\le2$. The degree bounds in \cref{prop:path-regularity}(i) give $|H_j\setminus\hat H_j|\le8\Delta_1(H_j)=O(d(H_j))$. Consequently, $d(\hat H_j)=d(H_j)(1+O(1/n_\ast))$ and $\log(d(\hat H_j)/d(H_j))=O(1/n_\ast)$. In comparing the main terms in \eqref{eq:benchmark} and \eqref{eq:frame-large}, the coefficient of the logarithm changes by $O(1)$, the logarithm itself changes by $O(1/n_\ast)$, and the linear term changes by $O(1)$. Since $\log(2d(H_j))=O(\log n_\ast)$, the two main terms differ by $O(\log n_\ast)$.

Using this comparison, $\mathsf A_j$, \eqref{eq:benchmark}, and hypothesis (iii), we obtain
\[
\begin{aligned}
 \log\cyc_{\Lambda}(\hat H_j)
 &\ge\log\Xc(H_j)-100\log n_\ast\\
 &\ge\frac{n_\ast-|\hat B|-|\cM|}{2}\log\bigl(2d(\hat H_j)\bigr)
 -n_\ast+|\hat B|+|\cM|\\
 &\qquad{}-\frac{2n_\ast}{\sqrt{\log n_\ast}}-100\log n_\ast-O(\log n_\ast).
\end{aligned}
\]
Since $100\log n_\ast+O(\log n_\ast)=o(n_\ast/\sqrt{\log n_\ast})$, the last bound implies \eqref{eq:frame-large} for sufficiently large $n_\ast$.
\end{proof}

We work with the conditioned edge-deletion process under $\QQ$ and \cref{ass:core}. We first collect the bounds from \cref{thm:terminal-regularity}(i)--(ii) and \cref{prop:path-regularity}, together with the estimates for legal directed candidates from \cref{prop:uniform-completion}, that will be used in both migration proofs. 

Let $\mathsf R$ be the event that the following three conditions hold.
\begin{enumerate}[label=(R\arabic*),ref=R\arabic*]
\item\label{item:migration-core-regularity} The hypergraph $F_*$ satisfies the bounds in \cref{thm:terminal-regularity}(i)--(ii).
\item\label{item:migration-process-regularity} The conclusions of \cref{prop:path-regularity}(i)--(iii), with $h=8$ and $L=2$, hold simultaneously for every $m_\ast\le j\le\binom{n_\ast}{3}$ and every $\hat B\subseteq[n_\ast]$ with $|\hat B|\le8$.
\item\label{item:migration-candidates} For every $m_\ast\le j\le\binom{n_\ast}{3}$, in every frame satisfying \eqref{eq:frame-bounds-final} and \eqref{eq:frame-large}, the fraction of legal directed candidates that are $(\log\log\log n_\ast)^{-1/2}$-abnormal in $\hat H_j$ is at most $(\log\log\log n_\ast)^{-1/2}$.
\end{enumerate}
\begin{fact}\label{fact:migration-common-event}
Under \cref{ass:core}, the following statements hold.
\begin{enumerate}[label=(\textit{\roman*})]
\item We have $\QQ[\mathsf R]=1-o(1)$.
\item If $\mathsf R$ holds, then there exists an absolute constant $c_0>0$ such that $d_{\hat H_j}(v)\ge100c_0d(\hat H_j)$ simultaneously for every $m_\ast\le j\le\binom{n_\ast}{3}$, every $\hat H_j=H_j-\hat B$ with $|\hat B|\le8$, and every $v\in V(\hat H_j)$, for all sufficiently large $n_\ast$.
\end{enumerate}
\end{fact}
\begin{proof}
For \cref{fact:migration-common-event}(i), by \cref{thm:terminal-regularity}(i)--(ii) and \cref{prop:path-regularity}, conditions (\ref{item:migration-core-regularity}) and (\ref{item:migration-process-regularity}) in the definition of $\mathsf R$ hold simultaneously with $\QQ$-probability $1-o(1)$. The bounds in \eqref{eq:frame-bounds-final} allow at most eight deleted vertices, two removed pairs of $\cM_0$, and two added pairs. A frame also chooses a root pair and fixes directions on at most two pairs. These data involve only a fixed number of vertex labels and direction choices, giving $n_\ast^{O(1)}$ possible frames. A union bound over these frames and all possible numbers of edges $j$, using \cref{prop:uniform-completion} for condition (\ref{item:migration-candidates}) in the definition of $\mathsf R$, gives
\[
 \QQ[\mathsf R^c]\le o(1)+n_\ast^{O(1)}\exp\{-\Omega((\log n_\ast)\sqrt{\log\log\log n_\ast})\}=o(1).
\]
For \cref{fact:migration-common-event}(ii), take $c,C$ from \cref{prop:path-regularity}(i) with $h=8$ and $L=2$. If $\mathsf R$ holds, then for every host and vertex in the statement,
\[
 \frac{d_{\hat H_j}(v)}{d(\hat H_j)}
 \ge\frac{\delta_1(\hat H_j)}{\Delta_1(\hat H_j)}
 \ge\frac{3cj}{3Cj}
 =\frac{c}{C}.
\]
Thus the absolute constant $c_0=c/(100C)>0$ gives the required inequality simultaneously for all these hosts and vertices.
\end{proof}

We now introduce \emph{migration of a private vertex} and \emph{migration of an intersection vertex}. The former changes the deleted vertex while keeping the added pair fixed; the latter changes one vertex of the added pair while keeping the deleted vertex fixed.

The migration lemmas compare mixed Hamilton cycle families obtained by deleting a vertex $x$ and adding a prescribed pair $yz$, with $\cyc\bigl((\hat H_j\cap\Kc)-x,\cM\cup\{yz\}\bigr)$ cycles in the initial family. At a size $m_\ast\le j\le\binom{n_\ast}{3}$, we use the following two choices of host and prescribed matching.
\begin{enumerate}[label=\textup{Case \arabic*:},ref=\arabic*,leftmargin=*]
\item\label{item:migration-unsaturated-case} Choose distinct vertices $x,y,z\in[n_\ast]\setminus V(\cM_0)$. Take $\hat B=\varnothing$ and $\cM=\cM_0$, so $\hat H_j=H_j$. If a counted mixed Hamilton cycle uses $\{x,y,z\}$ with $x$ private, then replacing this triple by the prescribed pair $yz$ gives a mixed Hamilton cycle on $[n_\ast]\setminus\{x\}$. The migration comparisons below do not require $\{x,y,z\}\in H_j$.
\item\label{item:migration-saturated-case} Fix a pair $y_0z_0\in\cM_0$ and choose distinct vertices $x,y\in[n_\ast]\setminus V(\cM_0)$. Take $\hat B=\{y_0\}$ and $\cM=\cM_0\setminus\{y_0z_0\}$, so $\hat H_j=H_j-y_0$, and set $z=z_0$. When $\{y_0,x,y\}\in H_j$, consider a cycle counted by $\Xc(H_j)$ using this triple with $x$ private. Contract $y_0z_0$ together with this triple to the prescribed pair $yz_0$, deleting $y_0$ and $x$. The migration comparisons below do not require $\{y_0,x,y\}\in H_j$.
\end{enumerate}

\begin{figure}[htbp]
\centering
\begin{tikzpicture}[
  y=.8cm,font=\small,
  vertex/.style={circle,fill=black,inner sep=1.5pt},
  retained/.style={draw=black!55,fill=black!6,line width=.6pt},
  marker/.style={line width=1.1pt},
  removed/.style={draw=red!70!black,line width=1.3pt},
  inserted/.style={draw=blue!75!black,fill=blue!12,line width=.9pt}
]
  \foreach \stage/\heading in {0/{1. Auxiliary cycle},1/{2. Remove $uv$},2/{3. Insert $\{u,x,v\}$}} {
    \begin{scope}[shift={({4.6*\stage},0)}]
      \node at (.45,2.45) {\heading};
      \coordinate (y) at (-1,-.7);
      \coordinate (z) at (-1,.7);
      \coordinate (u) at (1.4,.7);
      \coordinate (v) at (1.4,-.7);
      \coordinate (x) at (2.15,0);
      \coordinate (top) at (.2,1.35);
      \coordinate (bottom) at (.2,-1.35);
      \coordinate (ptl) at (-.65,1.75);
      \coordinate (ptr) at (1.05,1.75);
      \coordinate (pbl) at (-.65,-1.75);
      \coordinate (pbr) at (1.05,-1.75);
      \draw[retained] (z)--(ptl)--(top)--cycle;
      \draw[retained] (top)--(ptr)--(u)--cycle;
      \draw[retained] (v)--(pbr)--(bottom)--cycle;
      \draw[retained] (bottom)--(pbl)--(y)--cycle;
      \draw[marker,black!55] (y)--(z);
      \ifnum\stage=0
        \draw[removed] (u)--(v);
        \node at (.45,-2.2) {$x,x'$ absent};
      \else
        \ifnum\stage=2
          \draw[inserted] (u)--(x)--(v)--cycle;
          \node at (.45,-2.2) {$x'$ absent; $x$ private};
        \else
          \node at (.45,-2.2) {path with ends $u,v$};
        \fi
        \node[vertex,label=right:$x$] at (x) {};
      \fi
      \foreach \point in {top,bottom,ptl,ptr,pbl,pbr}
        \node[vertex] at (\point) {};
      \node[vertex,label=left:$y$] at (y) {};
      \node[vertex,label=left:$z$] at (z) {};
      \node[vertex,label=above right:$u$] at (u) {};
      \node[vertex,label=below right:$v$] at (v) {};
    \end{scope}
  }
  \draw[-{Stealth[length=5pt]},line width=.6pt] (2.75,0)--(3.15,0);
  \draw[-{Stealth[length=5pt]},line width=.6pt] (7.35,0)--(7.75,0);
\end{tikzpicture}
\caption{An example of migration of a private vertex with $uv\in L_{\hat H_j\cap\Kc}(x)$. Replace the red pair $uv$ by the blue triple $\{u,x,v\}$, making $x$ private; $x'$ remains absent. Grey triples and the pair $yz$ are unchanged.}
\label{fig:private-migration}
\end{figure}
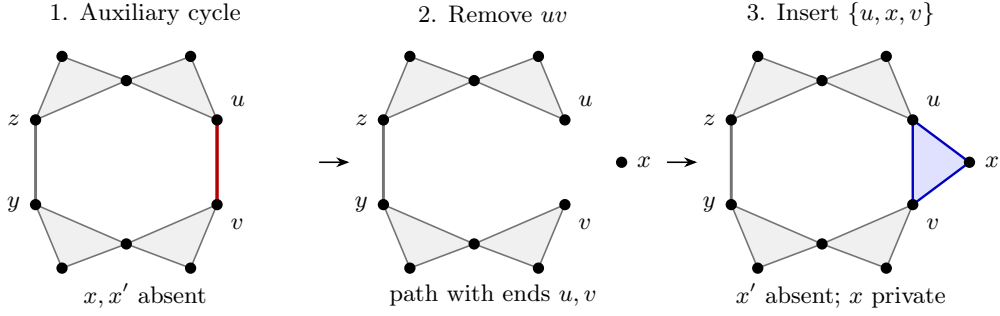

We first prove the comparison for migration of a private vertex, using the reconstruction illustrated in \cref{fig:private-migration}.

\begin{lemma}\label{lem:private-migration}
There exists an absolute constant $C_1>0$ such that, with $\QQ$-probability $1-o(1)$, the following holds simultaneously for every $m_\ast\le j\le\binom{n_\ast}{3}$ and every choice of $x,y,z,\hat B,\cM$ described in Cases~\ref{item:migration-unsaturated-case} and~\ref{item:migration-saturated-case}. If $\mathsf A_j$ holds and $\cyc\bigl((\hat H_j\cap\Kc)-x,\cM\cup\{yz\}\bigr)\ge\Xc(H_j)n_\ast^{-40}$, then for all but $C_1n_\ast/(\log\log\log n_\ast)^{1/4}$ vertices $x'\in[n_\ast]\setminus(V(\cM_0)\cup\{x,y,z\})$, we have
\[
\begin{aligned}
 \cyc\bigl((\hat H_j\cap\Kc)-x',\cM\cup\{yz\}\bigr)&\ge\frac{d_{\hat H_j}(x)}{24d(\hat H_j-x)}\cyc\bigl((\hat H_j\cap\Kc)-x,\cM\cup\{yz\}\bigr)\\
 &\ge c_0\cyc\bigl((\hat H_j\cap\Kc)-x,\cM\cup\{yz\}\bigr).
\end{aligned}
\]
\end{lemma}
\begin{proof}[Proof of \cref{lem:private-migration}]
Put $\mathcal{G}\coloneqq\hat H_j\cap\Kc$. We compare the two cycle families through auxiliary cycles omitting both $x$ and $x'$ and containing an added pair $uv$. For most $x'$, the candidate estimates make the excluded set of pairs small, and the link estimate supplies many remaining pairs $uv\in L_{\mathcal G}(x)$. Expanding $uv$ to $\{u,x,v\}$ restores $x$; the injection below allows the resulting numbers of cycles to be summed.

\begin{claim}\label{claim:private-injection}
For every vertex $x'\in[n_\ast]\setminus(V(\cM_0)\cup\{x,y,z\})$,
\[
 \cyc\bigl(\mathcal{G}-x',\cM\cup\{yz\}\bigr)\ge
 \sum_{uv\in L_{\mathcal{G}}(x)}
 \cyc\bigl(\mathcal{G}-\{x,x'\},\cM\cup\{yz,uv\}\bigr),
\]
where the sum runs over pairs disjoint from $V(\cM)\cup\{y,z,x'\}$.
\end{claim}

\begin{proof}[Proof of \cref{claim:private-injection}]
Take a pair $uv$ in the sum and a cycle $C$ counted by the corresponding summand. Replacing $uv$ by $\{u,x,v\}\in\mathcal G$ restores $x$ as a private vertex while $x'$ remains absent, giving a cycle counted by $\cyc\bigl(\mathcal{G}-x',\cM\cup\{yz\}\bigr)$; see \cref{fig:private-migration}. The unique triple containing $x$ in the resulting cycle determines $uv$, and replacing this triple by $uv$ recovers $C$. Thus the resulting cycle uniquely determines both $uv$ and $C$, so the map is injective.
\end{proof}

Fix a vertex $x'\in[n_\ast]\setminus(V(\cM_0)\cup\{x,y,z\})$. To use \cref{claim:private-injection}, we need many pairs $uv\in L_{\hat H_j}(x)$ that avoid $V(\cM_0)\cup\{y,z,x'\}$ and give sufficiently large summands. If $d(\hat H_j-x)=0$, then put $\Gamma=\binom{V(\hat H_j)\setminus\{x\}}2$. Otherwise, let $\Gamma$ be the union of the following two sets of pairs in $\binom{V(\hat H_j)\setminus\{x\}}2$.
\begin{enumerate}[label=(\textit{\roman*})]
\item Pairs $uv$ with $\{u,v\}\cap(V(\cM_0)\cup\{y,z,x'\})\ne\varnothing$.
\item Pairs $uv$ with $\{u,v\}\cap(V(\cM_0)\cup\{y,z,x'\})=\varnothing$ and satisfying
\begin{equation}\label{eq:private-pair-small}
 \cyc\bigl(\mathcal G-\{x,x'\},\cM\cup\{yz,uv\}\bigr)
 <\frac{\cyc\bigl(\mathcal G-x,\cM\cup\{yz\}\bigr)}{8d(\hat H_j-x)}.
\end{equation}
\end{enumerate}

\begin{claim}\label{claim:private-link-bound}
With $\QQ$-probability $1-o(1)$, the following holds simultaneously for every $m_\ast\le j\le\binom{n_\ast}{3}$, every choice of $x,y,z,\hat B,\cM$ in Cases~\ref{item:migration-unsaturated-case} and~\ref{item:migration-saturated-case}, and every $x'\in[n_\ast]\setminus(V(\cM_0)\cup\{x,y,z\})$. If $\mathsf R$ holds and $|\Gamma|\le\binom{n_\ast-|\hat B|-1}{2}/(\log\log\log n_\ast)^{1/16}$, then $|L_{\hat H_j}(x)\setminus\Gamma|\ge d_{\hat H_j}(x)/3$.
\end{claim}
\begin{proof}[Proof of \cref{claim:private-link-bound}]
Both $\cyc\bigl(\mathcal G-\{x,x'\},\cM\cup\{yz,uv\}\bigr)$ and $\cyc\bigl(\mathcal G-x,\cM\cup\{yz\}\bigr)$ are determined by $\hat H_j-x$ and the fixed labels. Since the denominator $8d(\hat H_j-x)$ is also determined by $\hat H_j-x$, the set $\Gamma$ is determined by $\hat H_j-x$ and the fixed labels. Apply \cref{lem:simultaneous-link} with the current deleted set $\hat B$, with $x$ as the vertex whose link is sampled, and with $h=0$ and $\mathcal P=\varnothing$. Since $j$ has at most $\binom{n_\ast}{3}$ possible values and each application specifies only a fixed number of vertices, there are $n_\ast^{O(1)}$ applications of the link estimate across Cases~\ref{item:migration-unsaturated-case} and~\ref{item:migration-saturated-case}. A union bound gives a failure probability at most $n_\ast^{O(1)}\exp\{-\Omega((\log n_\ast)\log\log\log\log n_\ast)\}=o(1)$.
\end{proof}

For the rest of the proof, fix a realisation in $\mathsf R$ for which the conclusion of \cref{claim:private-link-bound} holds simultaneously for all these choices. Fix a size $j$ for which $\mathsf A_j$ holds and a choice of $x,y,z,\hat B,\cM$ in Case~\ref{item:migration-unsaturated-case} or Case~\ref{item:migration-saturated-case} such that $\cyc\bigl(\mathcal G-x,\cM\cup\{yz\}\bigr)\ge\Xc(H_j)n_\ast^{-40}$.

\begin{claim}\label{claim:private-small-exception}
There exists an absolute constant $C_1>0$ such that, for all but $\frac{C_1n_\ast}{(\log\log\log n_\ast)^{1/4}}$ vertices $x'\in[n_\ast]\setminus(V(\cM_0)\cup\{x,y,z\})$, we have
\[
 |\Gamma|\le\binom{n_\ast-|\hat B|-1}{2}/(\log\log\log n_\ast)^{1/16}.
\]
\end{claim}
\begin{proof}[Proof of \cref{claim:private-small-exception}]
Consider the frame $\Lambda$ with host $\hat H_j-x$, prescribed matching $\cM\cup\{yz\}$, and root direction $y\to z$.

We first verify \eqref{eq:frame-bounds-final} and \eqref{eq:frame-large} for $\Lambda$ before applying condition (\ref{item:migration-candidates}) in the definition of $\mathsf R$. The host of the frame $\Lambda$ is $H_j-(\hat B\cup\{x\})$, obtained by deleting at most two vertices from $H_j$ (i.e. deleting $x$ in Case~\ref{item:migration-unsaturated-case} and $y_0,x$ in Case~\ref{item:migration-saturated-case}). The prescribed matching $\cM\cup\{yz\}$ is obtained from $\cM_0$ by removing at most one pair and adding $yz$, so $\Lambda$ satisfies \eqref{eq:frame-bounds-final}. The number of mixed Hamilton cycles in the frame $\Lambda$ is $\cyc_{\Lambda}(\hat H_j-x)=\cyc\bigl(\mathcal G-x,\cM\cup\{yz\}\bigr)$. The assumed lower bound on this number of cycles and \cref{claim:frame-count-transfer} give \eqref{eq:frame-large}.

For a permitted vertex $x'$, a legal pair here is $uv\in\binom{[n_\ast]\setminus(V(\cM_0)\cup\{x,y,z,x'\})}{2}$. We now show that each such pair satisfying \eqref{eq:private-pair-small} gives two abnormal directed candidates.

For each legal candidate $(x';u\to v)$, consider the mixed Hamilton cycles using triples of $\mathcal G-\{x,x'\}$ and containing the prescribed matching $\cM\cup\{yz,uv\}$. We give each cycle the root orientation determined by $y\to z$. The pair $uv$ is then traversed either from $u$ to $v$ or from $v$ to $u$, and these two classes are counted by $\comp_{\hat H_j-x,\cM\cup\{yz\}}(x';u\to v)$ and $\comp_{\hat H_j-x,\cM\cup\{yz\}}(x';v\to u)$, respectively. Thus
\[
 \cyc\bigl(\mathcal G-\{x,x'\},\cM\cup\{yz,uv\}\bigr)
 =\comp_{\hat H_j-x,\cM\cup\{yz\}}(x';u\to v)
 +\comp_{\hat H_j-x,\cM\cup\{yz\}}(x';v\to u).
\]
If \eqref{eq:private-pair-small} holds, then the number in each orientation class is at most the sum for both classes, so
\[
 \frac{4d(\hat H_j-x)\comp_{\hat H_j-x,\cM\cup\{yz\}}(x';u\to v)}
 {\cyc_{\Lambda}(\hat H_j-x)}
 <\frac12\le1-\frac1{\sqrt{\log\log\log n_\ast}}.
\]
The same inequality holds with $u$ and $v$ interchanged. Thus both candidates are\newline $(\log\log\log n_\ast)^{-1/2}$-abnormal in $\hat H_j-x$.

Each choice of $x'$ and a legal unordered pair $uv$ satisfying \eqref{eq:private-pair-small} gives two distinct abnormal candidates. Condition (\ref{item:migration-candidates}) in the definition of $\mathsf R$ therefore gives
\[
 \sum_{\substack{x'\text{ and }uv\text{ legal}\\
                 \eqref{eq:private-pair-small}\text{ holds}}}1
 \le\frac{|\mathcal Q_{\Lambda}|}{2\sqrt{\log\log\log n_\ast}}
 =O\left(\frac{n_\ast^3}{\sqrt{\log\log\log n_\ast}}\right).
\]
Every vertex $x'$ with more than $n_\ast^2/(\log\log\log n_\ast)^{1/4}$ such pairs contributes more than this quantity to the sum. Consequently, the number of these vertices is at most
\[
 \frac{(\log\log\log n_\ast)^{1/4}}{n_\ast^2}
 O\left(\frac{n_\ast^3}{\sqrt{\log\log\log n_\ast}}\right)
 =O\left(\frac{n_\ast}{(\log\log\log n_\ast)^{1/4}}\right).
\]

There are at most $2s+3$ vertices in $V(\cM_0)\cup\{y,z,x'\}$, each belonging to at most $n_\ast-1$ pairs. Hence pairs meeting this set contribute at most $(2s+3)(n_\ast-1)=O((s+1)n_\ast)$ further pairs. Since $s\le n_\ast^{1/10}$, this is $o(n_\ast^2/(\log\log\log n_\ast)^{1/16})$. Also, $(\log\log\log n_\ast)^{-1/4}=o((\log\log\log n_\ast)^{-1/16})$ and $|\hat B|\le1$ in the two cases, so every vertex $x'$ outside the exceptional set satisfies
\[
 |\Gamma|\le\frac{n_\ast^2}{(\log\log\log n_\ast)^{1/4}}+O((s+1)n_\ast)
 =o\left(\frac{\binom{n_\ast-|\hat B|-1}{2}}{(\log\log\log n_\ast)^{1/16}}\right).
\]
Choosing $C_1$ to cover the preceding bound on the number of exceptional vertices proves the claim.
\end{proof}

\begin{claim}\label{claim:private-sum-bound}
For every $x'\in[n_\ast]\setminus(V(\cM_0)\cup\{x,y,z\})$ satisfying the size bound on $\Gamma$ in \cref{claim:private-small-exception}, we have
\[
\begin{aligned}
 \sum_{uv\in L_{\mathcal G}(x)}\cyc\bigl(\mathcal G-\{x,x'\},\cM\cup\{yz,uv\}\bigr)
 &\ge\frac{d_{\hat H_j}(x)}{24d(\hat H_j-x)}\cyc\bigl(\mathcal G-x,\cM\cup\{yz\}\bigr)\\
 &\ge c_0\cyc\bigl(\mathcal G-x,\cM\cup\{yz\}\bigr),
\end{aligned}
\]
with the same summation range as in \cref{claim:private-injection}.
\end{claim}
\begin{proof}[Proof of \cref{claim:private-sum-bound}]
By \cref{claim:private-link-bound}, at least $d_{\hat H_j}(x)/3$ pairs belong to $L_{\hat H_j}(x)\setminus\Gamma$. Each such pair $uv$ is disjoint from $V(\cM_0)\cup\{y,z,x'\}$. Hence $\{x,u,v\}$ is $\cM_0$-compatible, so $uv\in L_{\mathcal G}(x)$ and belongs to the summation range. By the definition of $\Gamma$, the summand corresponding to $uv$ is at least $\cyc\bigl(\mathcal G-x,\cM\cup\{yz\}\bigr)/(8d(\hat H_j-x))$. Thus the sum is at least
\[
 \frac{d_{\hat H_j}(x)}3\,
 \frac{\cyc\bigl(\mathcal G-x,\cM\cup\{yz\}\bigr)}{8d(\hat H_j-x)}
 =\frac{d_{\hat H_j}(x)}{24d(\hat H_j-x)}\cyc\bigl(\mathcal G-x,\cM\cup\{yz\}\bigr).
\]
The hypergraph $\hat H_j-x$ has $|\hat H_j|-d_{\hat H_j}(x)$ triples on $|V(\hat H_j)|-1$ vertices, so $d(\hat H_j-x)\le\frac{|V(\hat H_j)|}{|V(\hat H_j)|-1}d(\hat H_j)\le2d(\hat H_j)$. Together with $d_{\hat H_j}(x)\ge100c_0d(\hat H_j)$ from \cref{fact:migration-common-event}(ii), this gives $d_{\hat H_j}(x)/(24d(\hat H_j-x))\ge100c_0/48\ge c_0$.
\end{proof}

By \cref{claim:private-small-exception}, all but $C_1n_\ast/(\log\log\log n_\ast)^{1/4}$ vertices $x'$ satisfy the size bound on $\Gamma$. Applying \cref{claim:private-injection,claim:private-sum-bound} to each of these vertices gives the inequalities in the lemma. By \cref{fact:migration-common-event}(i) and \cref{claim:private-link-bound}, these conclusions hold simultaneously in both cases on an event of $\QQ$-probability $1-o(1)$, proving the lemma.
\end{proof}

We next prove the comparison for migration of an intersection vertex. The contraction and reconstruction are illustrated in \cref{fig:endpoint-migration}.
\begin{lemma}\label{lem:intersection-migration}
There exists an absolute constant $C_1>0$ such that, with $\QQ$-probability $1-o(1)$, the following holds simultaneously for every $m_\ast\le j\le\binom{n_\ast}{3}$ and every choice of $x,y,z,\hat B,\cM$ described in Cases~\ref{item:migration-unsaturated-case} and~\ref{item:migration-saturated-case}. If $\mathsf A_j$ holds and $\cyc\bigl((\hat H_j\cap\Kc)-x,\cM\cup\{yz\}\bigr)\ge\Xc(H_j)n_\ast^{-40}$, then for all but $C_1n_\ast/(\log\log\log n_\ast)^{1/8}$ vertices $y'\in[n_\ast]\setminus(V(\cM_0)\cup\{x,y,z\})$, we have
\[
\begin{aligned}
 \cyc\bigl((\hat H_j\cap\Kc)-x,\cM\cup\{y'z\}\bigr)&\ge\frac{d_{\hat H_j}(y)}{100d(\hat H_j)}\cyc\bigl((\hat H_j\cap\Kc)-x,\cM\cup\{yz\}\bigr)\\
 &\ge c_0\cyc\bigl((\hat H_j\cap\Kc)-x,\cM\cup\{yz\}\bigr).
\end{aligned}
\]
\end{lemma}
\begin{proof}[Proof of \cref{lem:intersection-migration}]
Put $\mathcal{G}\coloneqq\hat H_j\cap\Kc$. We compare the numbers of cycles counted by $\cyc(\mathcal G-x,\cM\cup\{yz\})$ and $\cyc(\mathcal G-x,\cM\cup\{y'z\})$, keeping $x$ deleted. We first group the cycles containing $yz$ according to a short path starting at $y$. Within each group, we replace this path together with $yz$ by a single prescribed pair, obtaining auxiliary cycles from which the original cycles can be recovered (see \cref{fig:endpoint-migration}). Different paths may exclude different target vertices. We therefore sum the numbers of cycles in the excluded path classes, showing that for most targets only a small fraction of the original cycles is lost. An injective reconstruction using triples in the link of $y$ then gives the required lower bound.

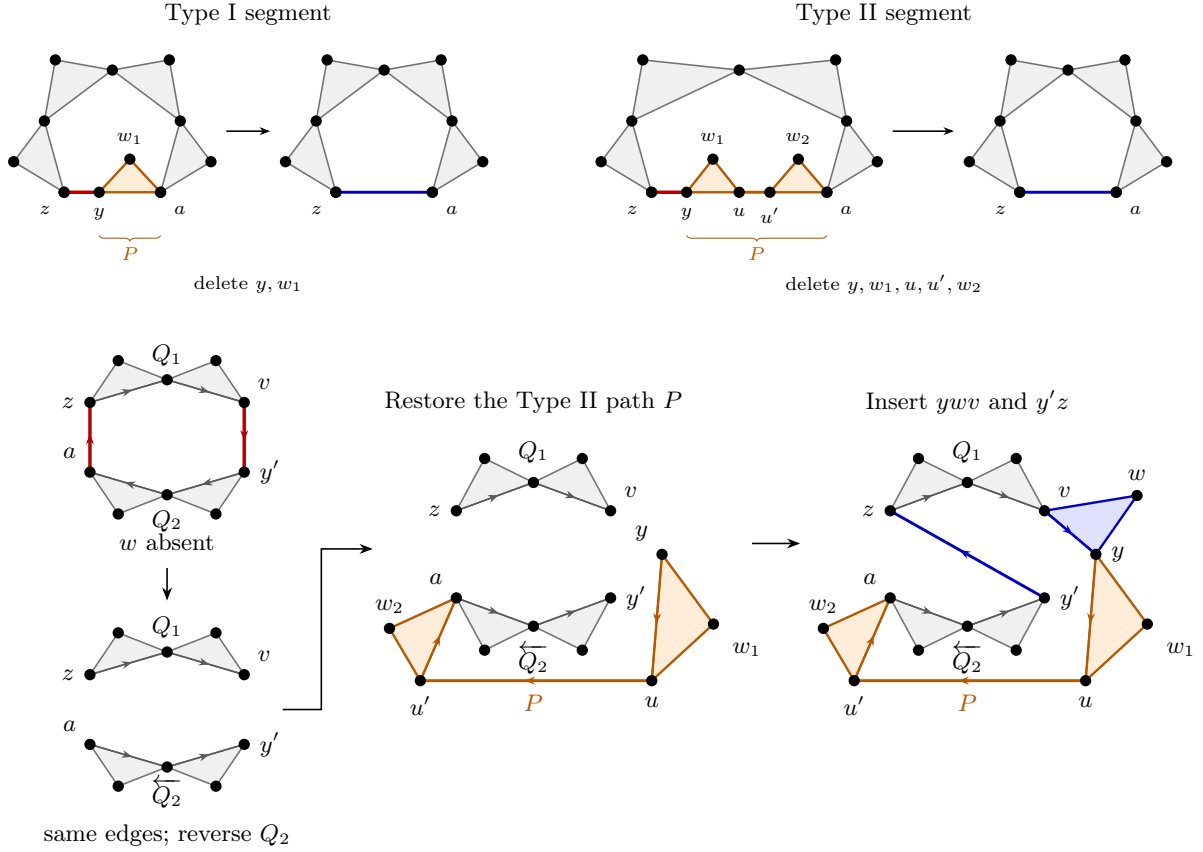
\begin{figure}[!htb]
\centering
\begin{tikzpicture}[
  x=.85cm,y=.72cm,font=\small,
  vertex/.style={circle,fill=black,inner sep=1.5pt},
  retained/.style={draw=black!55,fill=black!6,line width=.6pt},
  marker/.style={line width=1.1pt},
  removed/.style={draw=red!70!black,line width=1.3pt},
  inserted/.style={draw=blue!75!black,fill=blue!12,line width=.9pt},
  cut/.style={draw=orange!70!black,fill=orange!15,line width=.9pt},
  traversal/.style={postaction={decorate},decoration={markings,mark=at position .55 with {\arrow{Stealth[length=4pt]}}}}
]
  \begin{scope}[xshift=-.585cm,xscale=.682,yscale=.935,font=\scriptsize]
  \foreach \type in {1,2} {
    \begin{scope}[shift={({13.4*(\type-1)-.9},0)}]
      \pgfmathsetmacro{\cutwidth}{1.8*\type+.4}
      \pgfmathsetmacro{\aftershift}{2.2*\type+4}
      \foreach \offset in {0,\aftershift} {
        \begin{scope}[shift={(\offset,0)}]\pgfmathsetmacro{\panelwidth}{\offset==0 ? \cutwidth : 2.2}
          \coordinate (cz) at (0,0);
          \coordinate (ca) at (\panelwidth,0);
          \coordinate (cl) at (-.45,1.4);
          \coordinate (cm) at ({\panelwidth/2},2.4);
          \coordinate (cr) at ({\panelwidth+.45},1.4);
          \coordinate (cpl) at (-1.15,.6);
          \coordinate (cpm) at (-.2,2.6);
          \coordinate (cpn) at ({\panelwidth+.2},2.6);
          \coordinate (cpr) at ({\panelwidth+1.15},.6);
          \draw[retained] (cz)--(cpl)--(cl)--cycle;
          \draw[retained] (cl)--(cpm)--(cm)--cycle;
          \draw[retained] (cm)--(cpn)--(cr)--cycle;
          \draw[retained] (cr)--(cpr)--(ca)--cycle;
          \foreach \point in {cz,ca,cl,cm,cr,cpl,cpm,cpn,cpr}
            \node[vertex] at (\point) {};
        \end{scope}
      }
      \coordinate (z) at (0,0);
      \coordinate (y) at (.8,0);
      \draw[removed] (z)--(y);
      \ifnum\type=1
        \node[font=\small] at ({(2.2+\aftershift)/2},3.5) {Type I segment};
        \coordinate (a) at (2.2,0);
        \coordinate (p) at (1.5,.65);
        \draw[cut] (y)--(p)--(a)--cycle;
        \node[vertex,label=above:$w_1$] at (p) {};
        \draw[orange!70!black,decorate,decoration={brace,mirror,amplitude=2pt}]
          (.8,-.85)--(2.2,-.85) node[midway,below=1pt] {$P$};
        \node at ({(2.2+\aftershift)/2},-1.8) {delete $y,w_1$};
      \else
        \node[font=\small] at ({(2.2+\aftershift)/2},3.5) {Type II segment};
        \coordinate (u) at (2,0);
        \coordinate (up) at (2.7,0);
        \coordinate (a) at (4,0);
        \coordinate (p) at (1.4,.65);
        \coordinate (q) at (3.35,.65);
        \draw[cut] (y)--(p)--(u)--cycle;
        \draw[marker,orange!70!black] (u)--(up);
        \draw[cut] (up)--(q)--(a)--cycle;
        \node[vertex,label=above:$w_1$] at (p) {};
        \node[vertex,label=above:$w_2$] at (q) {};
        \node[vertex,label=below:$u$] at (u) {};
        \node[vertex,label=below:$u'$] at (up) {};
        \draw[orange!70!black,decorate,decoration={brace,mirror,amplitude=2pt}]
          (.8,-.85)--(4,-.85) node[midway,below=1pt] {$P$};
        \node at ({(2.2+\aftershift)/2},-1.8) {delete $y,w_1,u,u',w_2$};
      \fi
      \draw[-{Stealth[length=4pt]},line width=.6pt]
        ({\cutwidth+1.5},1.2)--({\aftershift-1.5},1.2);
      \draw[marker,blue!75!black] (\aftershift,0)--({\aftershift+2.2},0);
      \node[vertex,label=below left:$z$] at (\aftershift,0) {};
      \node[vertex,label=below right:$a$] at ({\aftershift+2.2},0) {};
      \node[vertex,label=below left:$z$] at (z) {};
      \node[vertex,label=below:$y$] at (y) {};
      \node[vertex,label=below right:$a$] at (a) {};
    \end{scope}
  }
  \end{scope}

  \foreach \stage/\sx/\sy/\heading in {
    0/-.9/-5.3/{},
    1/-.9/-10.3/{},
    2/4.8/-7.45/{Restore the Type II path $P$},
    3/11.55/-7.45/{Insert $ywv$ and $y'z$}} {
    \begin{scope}[shift={(\sx,\sy)},scale=.8]
      \ifnum\stage>1\node at (1.5,4.5) {\heading};\fi\ifnum\stage<2\pgftransformshift{\pgfpointxy{0}{.2}}\pgftransformyscale{.8}\fi
      \coordinate (a) at (0,0);
      \coordinate (z) at (0,2);
      \coordinate (v) at (3,2);
      \coordinate (t) at (3,0);
      \coordinate (top) at (1.5,2.65);
      \coordinate (bottom) at (1.5,-.65);
      \coordinate (ptl) at (.55,3.2);
      \coordinate (ptr) at (2.45,3.2);
      \coordinate (pbl) at (.55,-1.2);
      \coordinate (pbr) at (2.45,-1.2);
      \draw[retained] (z)--(ptl)--(top)--cycle;
      \draw[retained] (top)--(ptr)--(v)--cycle;
      \draw[retained] (t)--(pbr)--(bottom)--cycle;
      \draw[retained] (bottom)--(pbl)--(a)--cycle;
      \draw[black!65,line width=.6pt,traversal] (z)--(top);
      \draw[black!65,line width=.6pt,traversal] (top)--(v);
      \node at (1.5,3.35) {$Q_1$};
      \ifnum\stage=0
        \draw[removed,traversal] (a)--(z);
        \draw[removed,traversal] (v)--(t);
        \draw[black!65,line width=.6pt,traversal] (t)--(bottom);
        \draw[black!65,line width=.6pt,traversal] (bottom)--(a);
        \node at (1.5,-1.4) {$Q_2$};
        \node at (1.5,-2) {$w$ absent};
      \else
        \draw[black!65,line width=.6pt,traversal] (a)--(bottom);
        \draw[black!65,line width=.6pt,traversal] (bottom)--(t);
        \node at (1.5,-1.4) {$\overleftarrow{Q_2}$};
        \ifnum\stage=1
          \node at (1.5,-2.6) {same edges; reverse $Q_2$};
        \fi
      \fi
      \ifnum\stage>1
        \coordinate (y) at (4,1);
        \coordinate (u) at (3.8,-1.9);
        \coordinate (up) at (-.7,-1.9);
        \coordinate (p) at (5,-.6);
        \coordinate (q) at (-1.3,-.7);
        \draw[cut] (y)--(p)--(u)--cycle;
        \draw[marker,orange!70!black,traversal] (u)--(up);
        \draw[cut] (up)--(q)--(a)--cycle;
        \draw[orange!70!black,line width=.7pt,traversal] (y)--(u);
        \draw[orange!70!black,line width=.7pt,traversal] (up)--(a);
        \node[orange!70!black] at (1.5,-2.4) {$P$};
        \node[vertex,label=below:$u$] at (u) {};
        \node[vertex,label=below:$u'$] at (up) {};
        \node[vertex,label={[xshift=2pt,yshift=-2pt]below right:$w_1$}] at (p) {};
        \node[vertex,label=above:$w_2$] at (q) {};
        \ifnum\stage=3
          \coordinate (w) at (4.8,2.35);
          \draw[inserted] (v)--(w)--(y)--cycle;
          \draw[blue!75!black,line width=.9pt,traversal] (v)--(y);
          \draw[marker,blue!75!black,traversal] (t)--(z);
          \node[vertex,label=above:$w$] at (w) {};
        \fi
        \ifnum\stage=2 \node[vertex,label=above left:$y$] at (y) {};\else \node[vertex,label=right:$y$] at (y) {};\fi
      \fi
      \foreach \point in {top,bottom,ptl,ptr,pbl,pbr}
        \node[vertex] at (\point) {};
      \node[vertex,label=above left:$a$] at (a) {};
      \node[vertex,label=left:$z$] at (z) {};
      \node[vertex,label=above right:$v$] at (v) {};
      \node[vertex,label=right:$y'$] at (t) {};
    \end{scope}
  }
  \draw[-{Stealth[length=5pt]},line width=.6pt] (.3,-6.9)--(.3,-7.5);
  \draw[-{Stealth[length=5pt]},line width=.6pt] (2.1,-9.5)--(2.7,-9.5)--(2.7,-6.55)--(3.5,-6.55);\draw[-{Stealth[length=5pt]},line width=.6pt] (9.4,-6.45)--(10.15,-6.45);
\end{tikzpicture}
\caption{Migration of an intersection vertex. Above, contract $yz$ and the orange path $P$ from $y$ to $a$ to the pair $az$. Below, the auxiliary cycle omits $w$ and contains $az,vy'$; its paths $Q_1,Q_2$ run from $z$ to $v$ and from $y'$ to $a$, respectively. Remove the red pairs, restore the orange path, and insert the blue triple $ywv\in\mathcal G$ and pair $y'z$. The grey paths are unchanged, with $Q_2$ traversed in reverse; $x$ remains absent.}
\label{fig:endpoint-migration}
\end{figure}

For a cycle counted by $\cyc(\mathcal G-x,\cM\cup\{yz\})$, start at $y$ and move away from the pair $yz$ along the unique triple containing $y$. This records the path that must be restored after contraction. We end the segment at an $\cM$-unsaturated intersection vertex $a$, so that the contracted pair $az$ is disjoint from every remaining prescribed pair. If the first triple reaches an $\cM$-saturated vertex, we include its prescribed pair and the following triple. We thus consider the following two types of initial segment.
\begin{enumerate}[label=\textup{Type \Roman*:},leftmargin=*]
\item The segment consists of the single triple $\{y,w_1,a\}$, traversed from $y$ to an $\cM$-unsaturated intersection vertex $a\ne z$, with $w_1$ as its private vertex.
\item The first triple is $\{y,w_1,u\}$, where $w_1$ is private and $u$ belongs to a pair $uu'\in\cM$. Follow this pair from $u$ to $u'$, and then follow the triple $\{u',w_2,a\}$ to an $\cM$-unsaturated intersection vertex $a\ne z$, with $w_2$ private. Since $uu'\in\cM\subseteq\cM_0$ and $\{u',w_2,a\}\in\Kc$, we have $a,w_2\notin V(\cM_0)$. Thus $a$ is $\cM$-unsaturated, and the segment ends after this triple.
\end{enumerate}
Let $P$ range over all labelled initial paths of Types I and II that occur in a cycle counted by $\cyc(\mathcal G-x,\cM\cup\{yz\})$, with the private vertex in each triple specified. The path $P$ is determined by $(w_1,a)$ in Type I and $(w_1,u,u',w_2,a)$ in Type II. Each path $P$ has one edge in Type I and three edges in Type II. Fix such a path $P$ and consider the cycles containing $yz$ whose initial path is $P$. In each cycle, the pair $yz$ and the segment form a path from $z$ to $a$. Replace this path by the pair $az$. These cycles already omit $x$. The contraction additionally deletes $y,w_1$ in Type I and $y,w_1,u,u',w_2$ in Type II. Once $\hat H_j,x,y,z,\cM$ are fixed, the resulting host and matching are determined by the vertex labels of $P$. We define the host $H_P$ and prescribed matching $M_P$ as follows.
\begin{enumerate}[label=\textup{Type \Roman*:},leftmargin=*]
\item Define $H_P\coloneqq (\hat H_j-x)-\{y,w_1\}$ and $M_P\coloneqq\cM\cup\{az\}$.
\item Define $H_P\coloneqq (\hat H_j-x)-\{y,w_1,u,u',w_2\}$ and $M_P\coloneqq(\cM\setminus\{uu'\})\cup\{az\}$.
\end{enumerate}
Take $az$ as the root pair for the frame on $H_P$ with matching $M_P$ and fix the direction $a\to z$.

We now decompose the family counted by $\cyc(\mathcal G-x,\cM\cup\{yz\})$ according to these segments.
\begin{claim}\label{claim:cut-decomposition}
The following statements hold.
\begin{enumerate}[label=(\roman*),ref=\roman*]
\item Every cycle counted by $\cyc(\mathcal G-x,\cM\cup\{yz\})$ has a unique initial segment of Type I or II.
\item For each fixed path $P$, contracting $P$ together with $yz$ gives a bijection from the cycles containing $yz$ with initial path $P$ to the cycles counted by $\cyc(H_P\cap\Kc,M_P)$. Thus $\cyc\bigl(\mathcal{G}-x,\cM\cup\{yz\}\bigr)=\sum_P\cyc(H_P\cap\Kc,M_P)$.
\item The number of possible labelled initial paths $P$ is at most $2d_{\hat H_j}(y)+4d_{\hat H_j}(y)\Delta_1(\hat H_j)=O(n_\ast^4)$.
\end{enumerate}
\end{claim}
\begin{proof}[Proof of \cref{claim:cut-decomposition}]
For (i), following the first triple from $y$ leads either to an $\cM$-unsaturated intersection vertex other than $z$, to $z$, or to a vertex $u$ of a pair $uu'\in\cM$. In the last case the cycle follows $uu'$ and then a triple $\{u',w_2,a\}$. Since $u'\in V(\cM_0)$ and this triple belongs to $\Kc$, both $w_2$ and $a$ lie outside $V(\cM_0)$. Reaching $z$ and then traversing the pair $yz$ would return to $y$ after visiting at most six vertices, whereas the cycle spans $V(\hat H_j)\setminus\{x\}$. Thus, for sufficiently large $n_\ast$, the segment ends at an $\cM$-unsaturated intersection vertex $a\ne z$ and has Type I or II. The cycle determines the first triple encountered from $y$ away from $yz$ and the private vertex of that triple. In Type II, it also determines the next triple after $uu'$. These data uniquely determine the labelled path $P$.

For (ii), fix a path $P$ and a cycle $C_P$ counted by $\cyc(H_P\cap\Kc,M_P)$. The inverse of contraction is
\[
 C=\bigl(C_P\setminus\{\{a,z\}\}\bigr)\cup P\cup\{\{y,z\}\}.
\]
The fixed path $P$ restores all vertices deleted by contraction and any pair of $\cM$ removed by contraction, giving a cycle counted by $\cyc(\mathcal G-x,\cM\cup\{yz\})$ with initial path $P$. Thus contraction and expansion are inverse bijections for each segment, proving the stated identity.

For (iii), there are at most $2d_{\hat H_j}(y)$ choices for a Type I segment. For Type II, choose the first triple and its private vertex in at most $2d_{\hat H_j}(y)$ ways; the pair $uu'\in\cM$ is then determined. There are at most $\Delta_1(\hat H_j)$ choices for the triple through $u'$, and at most two choices for its private vertex. Multiplying these bounds gives at most $4d_{\hat H_j}(y)\Delta_1(\hat H_j)$ Type II segments. Adding the Type I bound and using $\Delta_1(\hat H_j)\le\binom{n_\ast-1}{2}$ gives the stated $O(n_\ast^4)$ bound.
\end{proof}

Call $(w,v)$ \emph{legal} for a path $P$ and a vertex $y'\in[n_\ast]\setminus(V(\cM_0)\cup\{x,y,z\})$ if $w,v,y'$ are distinct vertices of $V(H_P)\setminus V(M_P)$. Every vertex of $V(\cM_0)$ that remains in $H_P$ is $M_P$-saturated, so $V(\cM_0)\cap V(H_P)\subseteq V(M_P)$. Thus this legality condition also excludes $V(\cM_0)$, as required in the definition of legal directed candidates.

For a legal pair $(w,v)$ with $wv\in L_{\mathcal G}(y)$, we start with a cycle counted by $\comp_{H_P,M_P}(w;v\to y')$. Removing the pairs $az,vy'$, restoring $P$, and adding the triple $\{y,w,v\}$ and pair $y'z$ gives a cycle counted by $\cyc(\mathcal G-x,\cM\cup\{y'z\})$. The injection below allows us to sum these numbers of auxiliary cycles over the link of $y$ and the segments.
\begin{claim}\label{claim:intersection-injection}
For every vertex $y'\in[n_\ast]\setminus(V(\cM_0)\cup\{x,y,z\})$,
\[
 \cyc\bigl(\mathcal{G}-x,\cM\cup\{y'z\}\bigr)\ge
 \sum_P\ \sum_{\substack{(w,v)\text{ legal for }P\,,\ ywv\in \mathcal{G}}}
 \comp_{H_P,M_P}(w;v\to y').
\]
\end{claim}
\begin{proof}[Proof of \cref{claim:intersection-injection}]

Indeed, fix a path $P$ and a legal pair $(w,v)$ with $ywv\in\mathcal{G}$. Choose a cycle $C$ counted by $\comp_{H_P,M_P}(w;v\to y')$ and give it the root orientation determined by $a\to z$. Thus $C$ has vertex set $V(H_P)\setminus\{w\}$, uses triples from $(H_P\cap\Kc)-w$, and contains all pairs of $M_P\cup\{vy'\}$. Let $Q_1$ be the directed path along $C$ from $z$ to $v$ avoiding the pairs $az,vy'$, and let $Q_2$ be the directed path from $y'$ to $a$ avoiding these pairs. Thus the directed cyclic order of $C$ is
\[
 a\longrightarrow z\ \overset{Q_1}{\longrightarrow}\ v
 \longrightarrow y'\ \overset{Q_2}{\longrightarrow}\ a.
\]
We reconnect $Q_1,Q_2$ by forming
\[
 C'\coloneqq\bigl(C\setminus\{\{a,z\},\{v,y'\}\}\bigr)
 \cup P\cup\{\{y,w,v\},\{y',z\}\}.
\]
This restores $P$ from $y$ to $a$ and inserts $ywv$ with $w$ private. The resulting cyclic order is
\begin{equation}\label{eq:output-order}
 y'\longrightarrow z\ \overset{Q_1}{\longrightarrow}\ v
 \longrightarrow y\ \overset{P}{\longrightarrow}\ a
 \overset{\overleftarrow{Q_2}}{\longrightarrow}\ y'.
\end{equation}
The path $Q_2$ runs from $y'$ to $a$ in the original cycle, so it is traversed in reverse to return from $a$ to $y'$ in the resulting cycle; see \cref{fig:endpoint-migration}.

The chosen auxiliary cycle omits exactly $\{w\}\cup(V(P)\setminus\{a\})$ from $V(\hat H_j)\setminus\{x\}$. The new triple restores $w$, and the path $P$ restores the other missing vertices. Every triple of the restored path $P$ belongs to $\hat H_j\cap\Kc$, as does the new triple $ywv$. Thus $C'$ is a single mixed Hamilton cycle with cyclic order \eqref{eq:output-order}, counted by $\cyc\bigl(\mathcal{G}-x,\cM\cup\{y'z\}\bigr)$.

To prove that the construction is injective, orient $C'$ so that it traverses the pair $y'z$ from $y'$ to $z$. Its incoming triple at $y$ is exactly $vyw$, while its outgoing triple starts the path $P$. In the incoming triple at $y$, $w$ is the private vertex and $v$ is the intersection vertex other than $y$. To recover the segment, follow the outgoing triple at $y$. If its other intersection vertex is $\cM$-unsaturated, that intersection vertex is $a$ and the segment is of Type I. Otherwise the other intersection vertex is $u$ in a pair $uu'\in\cM$; follow this pair and the next triple to reach $a$, recovering a Type II segment. The private vertices of these triples determine the remaining segment labels. Once $P,w,v$ have been recovered, the input cycle is
\[
 C=\bigl(C'\setminus(P\cup\{\{y,w,v\},\{y',z\}\})\bigr)
 \cup\{\{a,z\},\{v,y'\}\}.
\]
Thus $C'$ determines the segment, the ordered pair, and the input cycle. In particular, different segments have disjoint images, proving the claim.
\end{proof}

We next obtain enough triples $ywv$ for the construction. To apply \cref{lem:simultaneous-link}, we fix the segment labels before revealing $L_{\hat H_j}(y)$. We allow the following choices, without requiring the triples specified by these labels to belong to $\hat H_j$.
\begin{enumerate}[label=\textup{Type \Roman*:},leftmargin=*]
\item Choose distinct $w_1,a\in[n_\ast]\setminus(V(\cM_0)\cup\{x,y,z\})$.
\item Choose ordered vertices $u,u'$ with $uu'\in\cM$, and distinct $w_1,w_2,a\in[n_\ast]\setminus(V(\cM_0)\cup\{x,y,z\})$.
\end{enumerate}
The definitions of $H_P$ and $M_P$ apply to each such choice, before checking whether the triples specified by the vertex labels of $P$ belong to $\hat H_j$.

Fix one such choice of segment labels and a target $y'\in[n_\ast]\setminus(V(\cM_0)\cup\{x,y,z\})$. Define $\Gamma$ as follows.
\begin{enumerate}[label=(\textit{\roman*}),ref=\roman*]
\item If $y'\notin V(H_P)\setminus V(M_P)$ or $d(H_P)=0$, then put $\Gamma=\binom{V(\hat H_j)\setminus\{y\}}2$.
\item If $y'\in V(H_P)\setminus V(M_P)$ and $d(H_P)>0$, then let $\Gamma$ consist of all pairs $wv\in\binom{V(\hat H_j)\setminus\{y\}}2$ meeting $([n_\ast]\setminus V(H_P))\cup V(M_P)\cup\{y'\}$, together with the remaining pairs satisfying
\[
 \max\left\{\comp_{H_P,M_P}(w;v\to y'),\comp_{H_P,M_P}(v;w\to y')\right\}
 <\frac{\cyc(H_P\cap\Kc,M_P)}{8d(H_P)}.
\]
\end{enumerate}
The two numbers correspond to choosing $w$ or $v$, respectively, as the private vertex of the triple $\{y,w,v\}$ to be inserted. The displayed inequality excludes a pair only if both choices give fewer than $\cyc(H_P\cap\Kc,M_P)/(8d(H_P))$ cycles. Since $H_P$ omits $y$, for $y'\in V(H_P)\setminus V(M_P)$ and $d(H_P)>0$ the numbers $\comp_{H_P,M_P}(w;v\to y')$ and $\comp_{H_P,M_P}(v;w\to y')$, as well as $\cyc(H_P\cap\Kc,M_P)$ and $d(H_P)$, are determined by $\hat H_j-y$ and the fixed labels. In the other cases, $\Gamma$ is the entire fixed set of pairs. Thus $\Gamma$ is determined by $\hat H_j-y$ and the fixed labels in every case, as required in \cref{lem:simultaneous-link}.

\begin{claim}\label{claim:intersection-link-bound}
With $\QQ$-probability $1-o(1)$, the following holds simultaneously for every $m_\ast\le j\le\binom{n_\ast}{3}$, every choice of data in Cases~\ref{item:migration-unsaturated-case} and~\ref{item:migration-saturated-case}, every labelled path $P$ of either type, and every target $y'\in[n_\ast]\setminus(V(\cM_0)\cup\{x,y,z\})$. If $\mathsf R$ holds, the triple through $y$ specified by $P$ belongs to $\hat H_j$, and $|\Gamma|\le\binom{n_\ast-|\hat B|-1}{2}/(\log\log\log n_\ast)^{1/16}$, then $|L_{\hat H_j}(y)\setminus\Gamma|\ge\frac{d_{\hat H_j}(y)}3$. 
\end{claim}
\begin{proof}[Proof of \cref{claim:intersection-link-bound}]
The required triple through $y$ is included as the prescribed triple in \cref{lem:simultaneous-link}: for Type I, prescribe $\{y,w_1,a\}$; for Type II, prescribe $\{y,w_1,u\}$. In Type II, the presence of the other triple $\{u',w_2,a\}$ is determined by $\hat H_j-y$. For each fixed time and choice of data in Cases~\ref{item:migration-unsaturated-case} and~\ref{item:migration-saturated-case}, there are at most $n_\ast^2$ choices of Type I segment labels, at most $2|\cM|n_\ast^3$ choices of Type II segment labels, and at most $n_\ast$ targets $y'$. Including the choices of $j$ and the case data therefore gives $n_\ast^{O(1)}$ applications of \cref{lem:simultaneous-link} with $h=1$. A union bound over these applications gives a failure probability at most $n_\ast^{O(1)}\exp\{-\Omega((\log n_\ast)\log\log\log\log n_\ast)\}=o(1)$, proving the claim.
\end{proof}

Fix a realisation in $\mathsf R$ satisfying \cref{claim:intersection-link-bound}, a size $j$ for which $\mathsf A_j$ holds, and a choice of $x,y,z,\hat B,\cM$ in Case~\ref{item:migration-unsaturated-case} or Case~\ref{item:migration-saturated-case} such that $\cyc\bigl(\mathcal G-x,\cM\cup\{yz\}\bigr)\ge\Xc(H_j)n_\ast^{-40}$. We now restrict the sum in \cref{claim:intersection-injection} to segments whose auxiliary cycle families are large enough to apply condition (\ref{item:migration-candidates}) in the definition of $\mathsf R$. Specifically, discard paths $P$ with $\cyc(H_P\cap\Kc,M_P)<\cyc\bigl(\mathcal{G}-x,\cM\cup\{yz\}\bigr)n_\ast^{-8}$.

Before applying the candidate estimates, we verify that the contracted frame and the candidate frames satisfy \eqref{eq:frame-bounds-final}. The contracted frame has host $H_P$, prescribed matching $M_P$, and root direction $a\to z$. For a legal candidate $(w;v\to y')$, the corresponding frame has host $H_P-w$ and prescribed matching $M_P\cup\{vy'\}$, with $vy'$ prescribed to be traversed from $v$ to $y'$ in the root orientation determined by $a\to z$; the cycles counted by $\comp_{H_P,M_P}(w;v\to y')$ use only triples of $(H_P\cap\Kc)-w$. The vertices deleted from $H_j$ are $\hat B\cup\{x,y,w_1,w\}$ in Type I and $\hat B\cup\{x,y,w_1,u,u',w_2,w\}$ in Type II. Since $|\hat B|\le1$, these sets have at most five and eight vertices, respectively. We count changes to the prescribed matching relative to $\cM_0$. The only pairs of $\cM_0$ that can be removed are $y_0z_0$ when $\hat B=\{y_0\}$ and $uu'$ in Type II. The added pairs are $az$ and, for this candidate, $vy'$. The auxiliary pair $yz$ does not belong to $\cM_0$, so replacing it by $az$ removes no additional pair of $\cM_0$. At most two pair directions are fixed.

For every remaining path $P$, the threshold used to discard paths gives $\cyc(H_P\cap\Kc,M_P)\ge\Xc(H_j)n_\ast^{-40}n_\ast^{-8}=\Xc(H_j)n_\ast^{-48}\ge\Xc(H_j)n_\ast^{-100}$. Thus \cref{claim:frame-count-transfer} gives \eqref{eq:frame-large} for the frame on $H_P$ with prescribed matching $M_P$. We may therefore apply condition (\ref{item:migration-candidates}) in the definition of $\mathsf R$ to the frame on $H_P$ with prescribed matching $M_P$ for every remaining path $P$.

For each remaining path $P$, let $E_P$ consist of the vertices in $[n_\ast]\setminus V(H_P)$, the vertices in $V(M_P)$, and the targets $y'\in V(H_P)\setminus V(M_P)$ for which the set $\Gamma$ defined above for $P,y'$ violates the size condition in \cref{lem:simultaneous-link}(iii). For a fixed target $y'$, we omit the summands indexed by paths $P$ with $y'\in E_P$. The next claim shows that, for most targets, this loses only a small proportion of the cycles counted by $\cyc(\mathcal G-x,\cM\cup\{yz\})$.
\begin{claim}\label{claim:intersection-target-weight}
For all but $O(n_\ast/(\log\log\log n_\ast)^{1/8})$ vertices $y'\in[n_\ast]\setminus(V(\cM_0)\cup\{x,y,z\})$, we have
\[
 \sum_{\substack{P\text{ remaining}\\y'\notin E_P}}\cyc(H_P\cap\Kc,M_P)
 \ge\left(1-O\left( \frac{1}{n_\ast^{4}} \right)-\frac{1}{(\log\log\log n_\ast)^{1/8}}\right)
 \cyc\bigl(\mathcal{G}-x,\cM\cup\{yz\}\bigr).
\]
\end{claim}
\begin{proof}[Proof of \cref{claim:intersection-target-weight}]
There are two losses to control. We first bound the number of cycles in the discarded path classes. Among the remaining paths, a target $y'$ also excludes the classes indexed by $P$ with $y'\in E_P$. We will bound $|E_P|$ for each remaining path and then sum over targets to show that this second loss is small for most targets.

By \cref{claim:cut-decomposition},
\[
 \sum_{P\text{ discarded}}\cyc(H_P\cap\Kc,M_P)
 \le O(n_\ast^4)\,n_\ast^{-8}\cyc\bigl(\mathcal G-x,\cM\cup\{yz\}\bigr)
 =O\left(\frac{1}{n_\ast^4}\right)\cyc\bigl(\mathcal G-x,\cM\cup\{yz\}\bigr).
\]
We now bound $|E_P|$ for a remaining path $P$. If a legal candidate $(w;v\to y')$ satisfies $\comp_{H_P,M_P}(w;v\to y')<\cyc(H_P\cap\Kc,M_P)/(8d(H_P))$, then $4d(H_P)\comp_{H_P,M_P}(w;v\to y')/\cyc(H_P\cap\Kc,M_P)<1/2<1-(\log\log\log n_\ast)^{-1/2}$ for sufficiently large $n_\ast$. Hence this candidate is $(\log\log\log n_\ast)^{-1/2}$-abnormal in $H_P$, and condition (\ref{item:migration-candidates}) in the definition of $\mathsf R$ bounds the number of these candidates by $O(n_\ast^3/\sqrt{\log\log\log n_\ast})$.

Each target $y'\in V(H_P)\setminus V(M_P)$ with more than $n_\ast^2/(\log\log\log n_\ast)^{1/4}$ legal ordered pairs $(w,v)$ satisfying $\comp_{H_P,M_P}(w;v\to y')<\cyc(H_P\cap\Kc,M_P)/(8d(H_P))$ contributes more than this many abnormal candidates to the total above. Consequently, the number of such targets is at most
\[
 \frac{(\log\log\log n_\ast)^{1/4}}{n_\ast^2}
 O\left(\frac{n_\ast^3}{\sqrt{\log\log\log n_\ast}}\right)
 =O\left(\frac{n_\ast}{(\log\log\log n_\ast)^{1/4}}\right).
\]
For every other target $y'\in V(H_P)\setminus V(M_P)$, at most $n_\ast^2/(\log\log\log n_\ast)^{1/4}$ legal ordered pairs satisfy this inequality. By the definition of $\Gamma$, a legal unordered pair $wv$ belongs to $\Gamma$ only if both $(w;v\to y')$ and $(v;w\to y')$ satisfy this inequality. The vertex restrictions also exclude pairs meeting $([n_\ast]\setminus V(H_P))\cup V(M_P)\cup\{y'\}$. Since $|[n_\ast]\setminus V(H_P)|\le7$ and $|M_P|\le s+1$, this union has at most $2s+10$ vertices. Each vertex belongs to at most $n_\ast-1$ pairs, so these restrictions exclude at most $(2s+10)(n_\ast-1)=O((s+1)n_\ast)$ pairs. Consequently, $|\Gamma|\le O(n_\ast^2/(\log\log\log n_\ast)^{1/4})+O((s+1)n_\ast)=o(n_\ast^2/(\log\log\log n_\ast)^{1/16})$. This satisfies the size condition in \cref{lem:simultaneous-link}, so \cref{claim:intersection-link-bound} gives $|L_{\hat H_j}(y)\setminus\Gamma|\ge d_{\hat H_j}(y)/3$.

The definition of $E_P$ and the preceding bound on exceptional targets give $|E_P|\le7+2(s+1)+O(n_\ast/(\log\log\log n_\ast)^{1/4})$. Since $s\le n_\ast^{1/10}$, this is $O(n_\ast/(\log\log\log n_\ast)^{1/4})$, uniformly in $P$. For each target $y'$, the sum over remaining paths $P$ with $y'\in E_P$ is the number of cycles lost by omitting these path classes. By \cref{claim:cut-decomposition}, the sum of $\cyc(H_P\cap\Kc,M_P)$ over the remaining paths is at most $\cyc(\mathcal G-x,\cM\cup\{yz\})$. Summing the loss over all targets and reversing the order of summation therefore gives
\[
\begin{aligned}
 \sum_{y'\in[n_\ast]}\ \sum_{\substack{P\text{ remaining}\\y'\in E_P}}\cyc(H_P\cap\Kc,M_P)
 &=\sum_{P\text{ remaining}}|E_P|\cyc(H_P\cap\Kc,M_P)\\
 &=O\left(\frac{n_\ast}{(\log\log\log n_\ast)^{1/4}}\right)\cyc\bigl(\mathcal{G}-x,\cM\cup\{yz\}\bigr).
\end{aligned}
\]
By Markov's inequality, at most
\[
 \frac{(\log\log\log n_\ast)^{1/8}}
 {\cyc\bigl(\mathcal G-x,\cM\cup\{yz\}\bigr)}
 \sum_{P\text{ remaining}}|E_P|\cyc(H_P\cap\Kc,M_P)
 =O\left(\frac{n_\ast}{(\log\log\log n_\ast)^{1/8}}\right)
\]
targets $y'\in[n_\ast]$ fail to satisfy
\[
 \sum_{\substack{P\text{ remaining}\\y'\in E_P}}\cyc(H_P\cap\Kc,M_P)
 \le\frac{\cyc\bigl(\mathcal G-x,\cM\cup\{yz\}\bigr)}{(\log\log\log n_\ast)^{1/8}}.
\]
For every other target, the discarded paths contribute only $O(n_\ast^{-4})\cyc\bigl(\mathcal G-x,\cM\cup\{yz\}\bigr)$. By the decomposition in \cref{claim:cut-decomposition}, the retained weight is
\begin{multline*}
  \sum_{\substack{P\text{ remaining}\\y'\notin E_P}}\cyc(H_P\cap\Kc,M_P)
  =\cyc\bigl(\mathcal G-x,\cM\cup\{yz\}\bigr)
  -\sum_{P\text{ discarded}}\cyc(H_P\cap\Kc,M_P)\\
  -\sum_{\substack{P\text{ remaining}\\y'\in E_P}}\cyc(H_P\cap\Kc,M_P).
\end{multline*}
Substituting the two preceding loss bounds gives the lower bound in \cref{claim:intersection-target-weight}.
\end{proof}

We now combine the injection in \cref{claim:intersection-injection} with the link bound in \cref{claim:intersection-link-bound} and the lower bound in \cref{claim:intersection-target-weight} for $\sum_{\substack{P\text{ remaining}\\y'\notin E_P}}\cyc(H_P\cap\Kc,M_P)$.
\begin{claim}\label{claim:intersection-count-comparison}
For every target $y'\in[n_\ast]\setminus(V(\cM_0)\cup\{x,y,z\})$ satisfying the inequality in \cref{claim:intersection-target-weight}, we have
\[
\begin{aligned}
 \cyc\bigl(\mathcal G-x,\cM\cup\{y'z\}\bigr)
 \ge\frac{d_{\hat H_j}(y)}{100d(\hat H_j)}\cyc\bigl(\mathcal G-x,\cM\cup\{yz\}\bigr)
 \ge c_0\cyc\bigl(\mathcal G-x,\cM\cup\{yz\}\bigr).
\end{aligned}
\]
\end{claim}
\begin{proof}[Proof of \cref{claim:intersection-count-comparison}]
We first compare $d(H_P)$ with $d(\hat H_j)$ so that the contributions from different paths can be summed using a common denominator. The host $H_P$ is obtained from $\hat H_j$ by deleting three vertices in Type I and six in Type II. Hence $|\hat H_j\setminus H_P|\le6\Delta_1(\hat H_j)=O(d(\hat H_j))$ by \cref{prop:path-regularity}(i), so
\[
 \frac{d(H_P)}{d(\hat H_j)}
 =\frac{|V(\hat H_j)|}{|V(H_P)|}\frac{|H_P|}{|\hat H_j|}
 =1+O\left(\frac1{n_\ast}\right)
\]
uniformly in $P$.

For each remaining path $P$ with $y'\notin E_P$, \cref{claim:intersection-link-bound} gives $|L_{\hat H_j}(y)\setminus\Gamma|\ge d_{\hat H_j}(y)/3$. For each unordered pair in $L_{\hat H_j}(y)\setminus\Gamma$, choose an order $(w,v)$ for which $\comp_{H_P,M_P}(w;v\to y')\ge\cyc(H_P\cap\Kc,M_P)/(8d(H_P))$; such an order exists by the definition of $\Gamma$. This order is legal. Since $wv\in L_{\hat H_j}(y)\setminus\Gamma$, the triple $\{y,w,v\}$ belongs to $\hat H_j$. Its vertices lie outside $V(\cM_0)$, so it also belongs to $\Kc$, and hence to $\mathcal G$. Thus the corresponding term occurs in the sum in \cref{claim:intersection-injection}. Summing these contributions in \cref{claim:intersection-injection} and using the preceding identity and \cref{claim:intersection-target-weight}, we obtain
\[
\begin{aligned}
 \cyc\bigl(\mathcal G-x,\cM\cup\{y'z\}\bigr)
 &\ge\sum_{\substack{P\text{ remaining}\\y'\notin E_P}}
 \frac{d_{\hat H_j}(y)\cyc(H_P\cap\Kc,M_P)}{24d(H_P)}\\
 &\ge\frac{1-o(1)}{24(1+o(1))}
 \frac{d_{\hat H_j}(y)}{d(\hat H_j)}\cyc\bigl(\mathcal G-x,\cM\cup\{yz\}\bigr)\\
 &\ge\frac{d_{\hat H_j}(y)}{100d(\hat H_j)}\cyc\bigl(\mathcal G-x,\cM\cup\{yz\}\bigr),
\end{aligned}
\]
for sufficiently large $n_\ast$. \Cref{fact:migration-common-event}(ii) gives $d_{\hat H_j}(y)/(100d(\hat H_j))\ge c_0$, as required.
\end{proof}

By \cref{claim:intersection-target-weight,claim:intersection-count-comparison}, the migration inequalities hold for all but $O(n_\ast/(\log\log\log n_\ast)^{1/8})$ targets. By \cref{fact:migration-common-event}(i) and \cref{claim:intersection-link-bound}, these conclusions hold simultaneously in both cases on an event of $\QQ$-probability $1-o(1)$, proving the lemma.
\end{proof}

We now prove \cref{thm:uniform-marginal}. For edges disjoint from $V(\cM_0)$, successive migrations would turn an excessively large family of mixed Hamilton cycles into comparably large families for almost all choices of the deleted vertex and added pair. These three vertices need not form an edge of the host. The estimates in \cref{prop:uniform-completion} for the two directions of each pair rule this out.

For an edge containing a $\cM_0$-saturated vertex $y_0$, the added pair contains the partner of $y_0$ in $\cM_0$. Every legal directed candidate avoids $V(\cM_0)$, so \cref{prop:uniform-completion} gives no estimate for the corresponding family of mixed Hamilton cycles. We instead partition the counted mixed Hamilton cycles according to their unique triple at $y_0$ and sum over $L_{H_j}(y_0)$. The migration comparisons and \cref{lem:simultaneous-link} ensure that a large auxiliary family would make sufficiently many summands comparably large. Since the full sum equals $\Xc(H_j)$, this gives the required upper bound.

\begin{proof}[Proof of \cref{thm:uniform-marginal}]
We first establish the simultaneous link estimate needed for edges containing a $\cM_0$-saturated vertex. For a time $m_\ast\le j\le\binom{n_\ast}{3}$ and a pair $y_0z_0\in\cM_0$, put $\cM\coloneqq\cM_0\setminus\{y_0z_0\}$. The following partition identifies the numbers of cycles that will be compared in the link estimate.

\begin{claim}\label{claim:marked-cycle-partition}
The total number of counted mixed Hamilton cycles satisfies
\begin{multline*}
 \Xc(H_j)=\sum_{\substack{w'v\in L_{H_j}(y_0)\\w',v\notin V(\cM_0)}}
 \bigl(\cyc\bigl((H_j\cap\Kc)-\{y_0,w'\},\cM\cup\{vz_0\}\bigr)+\cyc\bigl((H_j\cap\Kc)-\{y_0,v\},\cM\cup\{w'z_0\}\bigr)\bigr).
\end{multline*}
\end{claim}
\begin{proof}[Proof of \cref{claim:marked-cycle-partition}]
Fix an edge $y_0wu\in H_j\cap\Kc$. Since this edge is $\cM_0$-compatible, $w,u\notin V(\cM_0)$. Consider a cycle counted by $\Xc(H_j)$ that contains $y_0wu$ with $w$ private. In this cycle, the pair $z_0y_0$ and the triple $y_0wu$ form a segment from $z_0$ to $u$. Delete $y_0,w$ and replace this segment by the pair $z_0u$. Restoring the same segment is the inverse operation, so this gives a bijection from these cycles to those counted by $\cyc\bigl((H_j\cap\Kc)-\{y_0,w\},\cM\cup\{uz_0\}\bigr)$. Every counted cycle has a unique triple at $y_0$ and a unique private vertex on that triple. Partitioning the cycles by these choices proves \cref{claim:marked-cycle-partition}.
\end{proof}

For the link estimate, fix distinct vertices $w,u\notin V(\cM_0)$ and let $\Gamma\subseteq\binom{[n_\ast]\setminus\{y_0\}}2$ consist of pairs meeting $V(\cM_0)\cup\{w,u\}$, together with the remaining pairs $w'v$ satisfying
\begin{multline*}
 \cyc\bigl((H_j\cap\Kc)-\{y_0,w'\},\cM\cup\{vz_0\}\bigr)
 +\cyc\bigl((H_j\cap\Kc)-\{y_0,v\},\cM\cup\{w'z_0\}\bigr) \\
 <\frac{c_0^2}{2}\cyc\bigl((H_j\cap\Kc)-\{y_0,w\},\cM\cup\{uz_0\}\bigr).
\end{multline*}
For $w'v\in L_{H_j}(y_0)$ with $w',v\notin V(\cM_0)$, the sum on the left is the number of cycles counted by $\Xc(H_j)$ that contain $\{y_0,w',v\}$. The two terms correspond to choosing $w'$ or $v$, respectively, as its private vertex.

The numbers $\cyc\bigl((H_j\cap\Kc)-\{y_0,w'\},\cM\cup\{vz_0\}\bigr)$, $\cyc\bigl((H_j\cap\Kc)-\{y_0,v\},\cM\cup\{w'z_0\}\bigr)$, and $\cyc\bigl((H_j\cap\Kc)-\{y_0,w\},\cM\cup\{uz_0\}\bigr)$ are determined by $H_j-y_0$ and the fixed labels. Hence $\Gamma$ is determined by this same information. Apply \cref{lem:simultaneous-link} with $\hat B=\varnothing$, with $y_0$ as the vertex whose link is sampled, and with $h=0$ and $\mathcal P=\varnothing$. There are at most $2s\binom{n_\ast}{3}n_\ast^2$ choices of $j,(y_0,z_0),w,u$, including both orders of each pair in $\cM_0$. A union bound over these choices gives a failure probability at most $n_\ast^{O(1)}\exp\{-\Omega((\log n_\ast)\log\log\log\log n_\ast)\}=o(1)$. Consequently, with $\QQ$-probability $1-o(1)$, on $\mathsf R$ the bound $|L_{H_j}(y_0)\setminus\Gamma|\ge d_{H_j}(y_0)/3$ holds simultaneously for all these choices whenever $|\Gamma|\le\binom{n_\ast-1}{2}/(\log\log\log n_\ast)^{1/16}$.

By \cref{fact:migration-common-event,lem:private-migration,lem:intersection-migration} and the union bound above, with $\QQ$-probability $1-o(1)$, $\mathsf R$ holds, the two migration inequalities hold for every permitted choice, and the link bound just established holds simultaneously for every $j,y_0,z_0,w,u$ whenever $\Gamma$ satisfies its stated size bound. Fix a realisation with all these properties, and fix a size $m_\ast\le j\le\binom{n_\ast}{3}$ for which $\mathsf A_j$ holds. Choose $C_1$ sufficiently large for both migration lemmas. For a $\cM_0$-compatible edge, we partition the counted cycles containing it according to its private vertex. There are three choices if the edge avoids $V(\cM_0)$ and two if it contains one $\cM_0$-saturated vertex, since that vertex must be an intersection vertex. We bound the size of each resulting cycle family by $O(\Xc(H_j)/d(H_j))$.

Since $d(H_j)\le\binom{n_\ast-1}{2}$, we have $n_\ast^{-20}d(H_j)=O(n_\ast^{-18})=o(1)$. Thus a cycle family of size less than $\Xc(H_j)n_\ast^{-20}$ has size $o(\Xc(H_j)/d(H_j))$.

We first treat edges disjoint from $V(\cM_0)$, and then edges containing one $\cM_0$-saturated vertex.
\begin{claim}\label{claim:unmarked-completion-bound}
For every three distinct vertices $x,y,z\in[n_\ast]\setminus V(\cM_0)$, we have
\[
 \cyc\bigl((H_j\cap\Kc)-x,\cM_0\cup\{yz\}\bigr)
 =O\left(\frac{\Xc(H_j)}{d(H_j)}\right).
\]
\end{claim}
\begin{proof}[Proof of \cref{claim:unmarked-completion-bound}]
Choose distinct vertices $x,y,z\in[n_\ast]\setminus V(\cM_0)$ maximising $\cyc\bigl((H_j\cap\Kc)-x,\cM_0\cup\{yz\}\bigr)$. If $\cyc\bigl((H_j\cap\Kc)-x,\cM_0\cup\{yz\}\bigr)<\Xc(H_j)n_\ast^{-20}$, then the bound $n_\ast^{-20}d(H_j)=o(1)$ gives $\cyc\bigl((H_j\cap\Kc)-x,\cM_0\cup\{yz\}\bigr)=o(\Xc(H_j)/d(H_j))$.

In the remaining case, it suffices to find distinct vertices $x',y',z'\in[n_\ast]\setminus V(\cM_0)$ satisfying
\begin{multline*}
 c_0^3\cyc\bigl((H_j\cap\Kc)-x,\cM_0\cup\{yz\}\bigr)
 \le\cyc\bigl((H_j\cap\Kc)-x',\cM_0\cup\{y'z'\}\bigr)\\
 \le\left(1+O\left(\frac{1}{(\log\log\log n_\ast)^{1/2}}\right)\right)
 \frac{\Xc(H_j)}{2d(H_j)}.
\end{multline*}
We will obtain the lower bound by changing all three coordinates, using one migration of a private vertex and two migrations of an intersection vertex. The upper bound will follow from the candidate estimates when both orientations of the added pair give normal candidates.

For the lower bound, apply \cref{lem:private-migration} to change $x$ to $x'$, then \cref{lem:intersection-migration} to change $y$ to $y'$. Finally, apply \cref{lem:intersection-migration} with its parameters $(x,y,z)$ taken as $(x',z,y')$ and target $z'$; this replaces the prescribed pair $zy'$ by $z'y'$. All three applications use Case~\ref{item:migration-unsaturated-case}, and at each stage the number of cycles is at least $c_0^3\Xc(H_j)n_\ast^{-20}\ge\Xc(H_j)n_\ast^{-40}$ for large $n_\ast$. At each step, choose the target outside the exceptional set in the corresponding migration lemma; the exceptional set at the next step may depend on the vertices already chosen. For every sequence of targets chosen this way, the three comparisons give
\[
\begin{aligned}
 \cyc\bigl((H_j\cap\Kc)-x',\cM_0\cup\{y'z'\}\bigr)
 & \ge c_0\cyc\bigl((H_j\cap\Kc)-x',\cM_0\cup\{y'z\}\bigr) \\
& \ge c_0^2\cyc\bigl((H_j\cap\Kc)-x',\cM_0\cup\{yz\}\bigr)
 \\
 &\ge c_0^3\cyc\bigl((H_j\cap\Kc)-x,\cM_0\cup\{yz\}\bigr).
\end{aligned}
\]
By \cref{lem:private-migration,lem:intersection-migration}, there are at most $C_1n_\ast/(\log\log\log n_\ast)^{1/4}$ excluded choices of $x'$, then at most $C_1n_\ast/(\log\log\log n_\ast)^{1/8}$ excluded choices of $y'$ for each remaining $x'$, and at most $C_1n_\ast/(\log\log\log n_\ast)^{1/8}$ excluded choices of $z'$ for each remaining ordered choice $(x',y')$. At each of the three stages, the target must avoid $V(\cM_0)$ and at most three specified vertices, including targets chosen at earlier stages. The other two coordinates each have at most $n_\ast$ choices, so these vertex restrictions exclude at most $3(2s+3)n_\ast^2$ further ordered triples. Adding these bounds shows that the number of excluded ordered triples $(x',y',z')$ is at most
\[
 n_\ast^2\frac{C_1n_\ast}{(\log\log\log n_\ast)^{1/4}}
 +2n_\ast^2\frac{C_1n_\ast}{(\log\log\log n_\ast)^{1/8}}+3(2s+3)n_\ast^2
 =o(n_\ast^3).
\]
Thus the lower bound with factor $c_0^3$ holds for all but $o(n_\ast^3)$ ordered triples $(x',y',z')$ of distinct vertices outside $V(\cM_0)$.

On the other hand, consider the frame $\Lambda$ with $\hat B=\varnothing$ and $\cM=\cM_0$, with only the root direction fixed. The frame $\Lambda$ has $\Xc(H_j)$ cycles, so \cref{claim:frame-count-transfer} gives \eqref{eq:frame-large}. Apply condition (\ref{item:migration-candidates}) in the definition of $\mathsf R$ to the frame $\Lambda$. The map $(x';y'\to z')\mapsto(x';z'\to y')$ is a bijection of $\mathcal Q_\Lambda$. Thus the number of ordered triples $(x',y',z')$ for which at least one of these two candidates is $(\log\log\log n_\ast)^{-1/2}$-abnormal is at most $2|\mathcal Q_\Lambda|/\sqrt{\log\log\log n_\ast}=O(n_\ast^3/\sqrt{\log\log\log n_\ast})$. For every other ordered triple, both candidates are $(\log\log\log n_\ast)^{-1/2}$-normal. By \cref{fact:completion-identities}(i), adding their upper bounds gives
\begin{multline*}
 \cyc\bigl((H_j\cap\Kc)-x',\cM_0\cup\{y'z'\}\bigr)
 =\comp_{H_j,\cM_0}(x';y'\to z')+\comp_{H_j,\cM_0}(x';z'\to y')\\
 \le\left(1+\frac1{\sqrt{\log\log\log n_\ast}}\right)\frac{\Xc(H_j)}{2d(H_j)}.
\end{multline*}
The upper bound just proved and the lower bound with factor $c_0^3$ each exclude only $o(n_\ast^3)$ of the $\Theta(n_\ast^3)$ triples, so some triple satisfies the two-sided comparison displayed at the start of the proof.
Hence $\cyc\bigl((H_j\cap\Kc)-x,\cM_0\cup\{yz\}\bigr)=O(\Xc(H_j)/d(H_j))$. By maximality, this bound holds for every choice of distinct vertices $x,y,z\in[n_\ast]\setminus V(\cM_0)$.
\end{proof}
It remains to treat edges containing a $\cM_0$-saturated vertex. Fix a pair $y_0z_0\in\cM_0$ and put $\cM\coloneqq\cM_0\setminus\{y_0z_0\}$. We use \cref{claim:marked-cycle-partition} to bound the size of each auxiliary family and hence the edge marginals.
\begin{claim}\label{claim:marked-edge-bound}
For every edge $y_0wu\in H_j\cap\Kc$, we have
\[
 \cyc\bigl((H_j\cap\Kc)-\{y_0,w\},\cM\cup\{uz_0\}\bigr)
 =O\left(\frac{\Xc(H_j)}{d(H_j)}\right),
 \quad\text{and}\quad
 \Qc_{y_0wu}(H_j)=O\left(\frac{1}{d(H_j)}\right).
\]
\end{claim}
\begin{proof}[Proof of \cref{claim:marked-edge-bound}]
Fix an edge $y_0wu\in H_j\cap\Kc$. If $\cyc\bigl((H_j\cap\Kc)-\{y_0,w\},\cM\cup\{uz_0\}\bigr)<\Xc(H_j)n_\ast^{-20}$, then the bound $n_\ast^{-20}d(H_j)=o(1)$ gives $\cyc\bigl((H_j\cap\Kc)-\{y_0,w\},\cM\cup\{uz_0\}\bigr)=o(\Xc(H_j)/d(H_j))$. Otherwise, let $\Gamma$ be the set defined before \cref{claim:unmarked-completion-bound} for the labels $j,y_0,z_0,w,u$. We use Case~\ref{item:migration-saturated-case} of the migration lemmas, with $\hat B=\{y_0\}$ and $\cM=\cM_0\setminus\{y_0z_0\}$. First apply \cref{lem:private-migration} with $(x,y,z)=(w,u,z_0)$ and target $w'$ to change the deleted vertex from $w$ to $w'$, keeping the pair $uz_0$. Then apply \cref{lem:intersection-migration} with $(x,y,z)=(w',u,z_0)$ and target $v$ to replace $uz_0$ by $vz_0$, keeping $y_0$ deleted throughout. The number of cycles after the first migration is at least $c_0\Xc(H_j)n_\ast^{-20}\ge\Xc(H_j)n_\ast^{-40}$ for large $n_\ast$.

There are $O(n_\ast/(\log\log\log n_\ast)^{1/4})$ exceptional targets $w'$ for the first migration. For each other $w'$, there are $O(n_\ast/(\log\log\log n_\ast)^{1/8})$ exceptional targets $v$ for the second migration. Thus the number of exceptional ordered pairs is at most
\[
 n_\ast O\left(\frac{n_\ast}{(\log\log\log n_\ast)^{1/4}}\right)
 +n_\ast O\left(\frac{n_\ast}{(\log\log\log n_\ast)^{1/8}}\right)
 =O\left(\frac{n_\ast^2}{(\log\log\log n_\ast)^{1/8}}\right).
\]
For every other ordered pair of distinct vertices $w',v\in[n_\ast]\setminus(V(\cM_0)\cup\{w,u\})$, we have
\[
\begin{aligned}
 \cyc\bigl((H_j\cap\Kc)-\{y_0,w'\},\cM\cup\{vz_0\}\bigr)
 &\ge c_0\cyc\bigl((H_j\cap\Kc)-\{y_0,w'\},\cM\cup\{uz_0\}\bigr)\\
 &\ge c_0^2\cyc\bigl((H_j\cap\Kc)-\{y_0,w\},\cM\cup\{uz_0\}\bigr).
\end{aligned}
\]
Since the threshold defining $\Gamma$ is half this lower bound, each such ordered pair satisfying the bound yields an unordered pair $w'v\notin\Gamma$. Therefore, for every pair $w'v\in\Gamma$ disjoint from $V(\cM_0)\cup\{w,u\}$, both $(w',v)$ and $(v,w')$ belong to the exceptional set of ordered pairs. The set $V(\cM_0)\cup\{w,u\}$ has $2s+2$ vertices, each belonging to at most $n_\ast-1$ pairs. Hence at most $(2s+2)(n_\ast-1)=O(sn_\ast)$ pairs meet this set. Including these pairs gives $|\Gamma|\le O(n_\ast^2/(\log\log\log n_\ast)^{1/8})+O(sn_\ast)=o(n_\ast^2/(\log\log\log n_\ast)^{1/16})$.

The simultaneous link bound established before \cref{claim:unmarked-completion-bound} gives $|L_{H_j}(y_0)\setminus\Gamma|\ge d_{H_j}(y_0)/3$. By the definition of $\Gamma$, each pair in this set contributes at least $(c_0^2/2)\cyc\bigl((H_j\cap\Kc)-\{y_0,w\},\cM\cup\{uz_0\}\bigr)$ to the sum in \cref{claim:marked-cycle-partition}. Therefore,
\[
 \Xc(H_j)\ge\frac{d_{H_j}(y_0)}3\,\frac{c_0^2}{2}
 \cyc\bigl((H_j\cap\Kc)-\{y_0,w\},\cM\cup\{uz_0\}\bigr).
\]
By \cref{fact:migration-common-event}(ii), $d_{H_j}(y_0)\ge100c_0d(H_j)$, so
\[
 \cyc\bigl((H_j\cap\Kc)-\{y_0,w\},\cM\cup\{uz_0\}\bigr)
 \le\frac{6\Xc(H_j)}{c_0^2d_{H_j}(y_0)}
 \le\frac{3\Xc(H_j)}{50c_0^3d(H_j)}.
\]
Applying the same bound with $w$ and $u$ interchanged controls the other choice of private vertex. Combining the two bounds gives
\[
\begin{aligned}
 \Qc_{y_0wu}(H_j)
 &=\frac{\cyc\bigl((H_j\cap\Kc)-\{y_0,w\},\cM\cup\{uz_0\}\bigr)+\cyc\bigl((H_j\cap\Kc)-\{y_0,u\},\cM\cup\{wz_0\}\bigr)}{\Xc(H_j)}\\
 &=O\left(\frac1{d(H_j)}\right), 
\end{aligned}
\]
which proves the claim.
\end{proof}

The bounds in \cref{claim:unmarked-completion-bound,claim:marked-edge-bound} hold at the fixed size $j$. The event $\mathsf A_j$ implies $\Xc(H_j)>0$, so the edge marginals are defined. We consider the following three cases for an edge $e\in H_j$.
\begin{enumerate}[label=(\textit{\roman*})]
\item If $e=\{x,y,z\}\in H_j\cap\Kc$ contains no $\cM_0$-saturated vertex, then by \cref{fact:completion-identities}(ii),
\begin{multline*}
 \Xc(H_j)\Qc_{xyz}(H_j)
 =\cyc\bigl((H_j\cap\Kc)-x,\cM_0\cup\{yz\}\bigr)\\
 {}+\cyc\bigl((H_j\cap\Kc)-y,\cM_0\cup\{xz\}\bigr)
 +\cyc\bigl((H_j\cap\Kc)-z,\cM_0\cup\{xy\}\bigr).
\end{multline*}
The three summands correspond to choosing $x,y,z$, respectively, as the private vertex of $xyz$. By \cref{claim:unmarked-completion-bound}, each summand is $O(\Xc(H_j)/d(H_j))$. Summing these bounds and dividing by $\Xc(H_j)$ gives $\Qc_e(H_j)=O(1/d(H_j))$.
\item If $e\in H_j\cap\Kc$ contains exactly one $\cM_0$-saturated vertex, then \cref{claim:marked-edge-bound} gives $\Qc_e(H_j)=O(1/d(H_j))$.
\item If $e\in H_j\setminus\Kc$, then no counted $\cM_0$-respecting mixed Hamilton cycle uses $e$, so $\Qc_e(H_j)=0$.
\end{enumerate}

The constants in these estimates are absolute and independent of the edge and the size $j$. Choose an absolute constant $C_0\ge1$ large enough to bound all the marginals by $C_0/d(H_j)$. Since $|H_j|=j$, we have $d(H_j)=3j/n_\ast$. Also, $s\le n_\ast^{1/10}$ gives $n_\ast-s\ge2n_\ast/3$ for sufficiently large $n_\ast$. Therefore,
\[
 \max_{e\in H_j}\Qc_e(H_j)
 \le\frac{C_0}{d(H_j)}
 =\frac{C_0n_\ast}{3j}
 \le\frac{C_0(n_\ast-s)}{2j}.
\]
The size $j$ was arbitrary subject to $\mathsf A_j$, and the fixed realisation belongs to an event of $\QQ$-probability $1-o(1)$. Thus the bound holds simultaneously for every such $j$, proving \cref{thm:uniform-marginal}.\qedhere
\end{proof}

\section{Concluding remarks}
\label{sec:conclusion}
We present the proof for $3$-graphs to keep the cycle modifications transparent. For $r\ge3$, a loose Hamilton cycle is a spanning $r$-graph whose edges admit a cyclic ordering in which consecutive edges intersect in exactly one vertex, these intersection vertices are distinct, and non-consecutive edges are disjoint. For $r=2$, we interpret loose Hamilton cycles as ordinary Hamilton cycles.

The following theorem gives the hitting-time and counting conclusions for every fixed $r\ge2$.
\begin{theorem}\label{thm:counting-cycles}
Fix an integer $r\ge2$, and consider the random $r$-graph process $(\mathcal H_j)$ on $[n]$. Let $\tau$ be the first time at which the minimum degree is at least $2$ if $r=2$, and at least $1$ if $r\ge3$, and let $\tau_{\LH}$ be the first time at which a loose Hamilton cycle appears. As $n$ tends to infinity through multiples of $r-1$, with high probability, the following statements hold.
\begin{enumerate}[label=(\textit{\roman*})]
\item We have $\tau_{\LH}=\tau$.
\item The number of loose Hamilton cycles in $\mathcal H_\tau$ is $\exp\{o(n)\} \left(\frac{(r-1)\log n}{e^{r-1}}\right)^{n/(r-1)}$.
\end{enumerate} 
\end{theorem}

The conditioned-core estimates (\cref{thm:conditioned-core} for $r=3$), together with injective expansion of prescribed pairs, give lower bounds for the number of loose Hamilton cycles at the hitting time. These bounds reach the first-moment scale, as in Kahn's~\cite[Theorem~1.4]{Kahn2022} counting result for perfect matchings. Combining them with first-moment upper bounds gives \cref{thm:counting-cycles}(ii).

For $r=2$, the hitting-time conclusion is due to Bollob\'as~\cite{Bollobas1984} and Ajtai, Koml\'os, and Szemer\'edi~\cite{AjtaiKomlosSzemeredi1985}, and the leading exponential counting estimate is due to Glebov and Krivelevich~\cite[Theorem~2]{GlebovKrivelevich2013}. The present argument gives another proof without P\'osa's rotation--extension method. Since every vertex of a Hamilton cycle is an intersection vertex, only migration of an intersection vertex is needed. The entropy step uses bipartite perfect matchings and a random-order argument as in Radhakrishnan's proof of Bregman's theorem~\cite{Radhakrishnan1997}.
Using this framework, one can show that, with high probability, the number of Hamilton cycles in $\mathcal H_\tau$ is
\[
 \left(\frac{2\tau}{en}\right)^n\exp\left\{O\left(\frac{n}{\log n}\right)\right\}.
\]
This estimate retains the actual hitting time $\tau$ and reduces the logarithmic error from $o(n)$ to $O(n/\log n)$.

For fixed $r\ge4$, private blocks have size $r-2$, leading to the following changes from the $3$-uniform proof.
\begin{enumerate}[label=(\textit{\roman*}),nosep]
\item In the reduction, each selected edge is contracted to a prescribed pair by deleting its private block, chosen to contain the low-degree vertex. For $r=4$, this deletes two vertices. To keep these blocks separated, we exclude Berge paths of length at most four between low-degree vertices.
\item In the entropy step, the auxiliary matchings are $r$-uniform, and the comparison uses a Bernoulli product law with parameter $1/(r-1)$. Thus, for $r=4$, the limiting intersection and private proportions are $1/3$ and $2/3$.
\item To change the deleted block, apply private-vertex migration once for each of its $r-2$ vertices. Then apply intersection-vertex migration to change the added pair.
\end{enumerate}

Detailed proofs for the graph case $r=2$ and for every fixed $r\ge4$ are given in two separate drafts, available at \url{https://github.com/xliu2022/xliu2022.github.io/tree/main/Unpolished_Manuscripts}.
\section*{Acknowledgements}


We thank Wanfang Chen for exploring an extension of the proof from
$r=3$ to general uniformity with the assistance of GPT. Gaps in the
proposed extension were identified during Lean formalisation and
subsequently repaired with the assistance of GPT.

\section*{Declaration on the use of AI}
The authors used GPT~6~Pro to assist in discussing proof strategies,
checking proofs, improving exposition, and formalising proofs in Lean.

\bibliographystyle{custom-plainurl}
\bibliography{references}

\end{document}